\documentclass[11pt, a4paper, oneside]{Thesis} % Paper size, default font size and one-sided paper

\graphicspath{{./Pictures/}} % Specifies the directory where pictures are stored

\usepackage{amsmath}
\usepackage{amstext}
\usepackage{amsfonts}
\usepackage{amssymb}
\usepackage{amsbsy}
\usepackage{amsthm}
\usepackage[T1]{fontenc}

\usepackage{makeidx}
\makeindex

\usepackage{enumitem}

\newcommand{\NN}{\mathbb{N}}
\newcommand{\RR}{\mathbb{R}}

\newcommand{\QQ}{\mathbb{Q}}
\newcommand{\ZZ}{\mathbb{Z}}
\newcommand{\CC}{\mathbb{C}}

\newcommand{\FF}{\mathbb{F}}

\newcommand{\TT}{\mathbb{T}}

\newtheorem{convention}[theorem]{Convention}
\newtheorem{observation}[theorem]{Observation}

\theoremstyle{definition}
\newtheorem{construction}[theorem]{Construction}

\newcommand{\cA}{\mathcal{A}}
\newcommand{\cB}{\mathcal{B}}
\newcommand{\cC}{\mathcal{C}}
\newcommand{\cD}{\mathcal{D}}
\newcommand{\cE}{\mathcal{E}}
\newcommand{\cF}{\mathcal{F}}
\newcommand{\cG}{\mathcal{G}}
\newcommand{\cH}{\mathcal{H}}

\newcommand{\cK}{\mathcal{K}}
\newcommand{\cL}{\mathcal{L}}

\newcommand{\cP}{\mathcal{P}}

\newcommand{\cS}{\mathcal{S}}
\newcommand{\cT}{\mathcal{T}}
\newcommand{\cU}{\mathcal{U}}
\newcommand{\cW}{\mathcal{W}}
\newcommand{\cV}{\mathcal{V}}

\newcommand{\bB}{\mathbf{B}}

\newcommand{\bx}{\mathbf{x}}
\newcommand{\by}{\mathbf{y}}

\newcommand{\seq}[3]{ \left( {#1}_{#2} \right)_{#2 \in #3} }
\newcommand{\set}[3]{ \left\{ {#1}_{#2} \right\}_{#2 \in #3}}
\newcommand{\setsep}{\ : \ }
\newcommand{\imply}{\Longrightarrow}

\newcommand{\lc}{\operatorname{lc}}

\usepackage{stmaryrd}
\newcommand{\ci}[1]{\left\llbracket #1 \right\rrbracket}%{\left[\left[ #1 \right]\right]}
\newcommand{\fp}[1]{\left< #1 \right> } %{\left\{\left\{ #1 \right\}\right\}}

\newcommand{\abs}[1]{\left| #1 \right|}

\newcommand{\Top}{\operatorname{Top}}
\newcommand{\Hom}{\operatorname{Hom}}

\newcommand{\llim}[2]{\displaystyle{\operatorname*{\mathit{#1}--lim}_{\phantom{#1} #2} }\ }

\newcommand{\id}{\operatorname{id}}
\newcommand{\cl}{\operatorname{cl}}
\newcommand{\lin}{\operatorname{lin}}

\newcommand{\restrict}[1]{\big|_{#1}}

\newcommand{\FS}[1]{\mathrm{FS} \! \left( #1\right)}
\newcommand{\FP}[1]{\mathrm{FP} \! \left( #1\right)}
\newcommand{\FU}[1]{\mathrm{FU} \! \left( #1\right)}

\newcommand{\card}[1]{{\# #1}}

\newcommand{\ldiv}[2]{ #1\backslash #2}
\newcommand{\rdiv}[2]{#2 / #1}
\newcommand{\principal}[1]{{\cF_{#1}}}
\newcommand{\IP}{\mathsf{IP}}
\renewcommand{\C}{\mathsf{C}}

\newcommand{\IPst}{\IP^*}
\newcommand{\VIP}{\mathsf{VIP}}

\newcommand{\Ultrafilters}[1]{\mathfrak{Ult}\left( #1 \right)}
\newcommand{\Filters}[1]{\mathfrak{Filt}\left( #1 \right)}

\newcommand{\Theory}{\mathfrak{T}}
\newcommand{\Language}{\mathcal{L}}
\newcommand{\Structure}{\mathsf{S}}

\newcommand{\norm}[1]{\left\lVert #1 \right\rVert}
\newcommand{\scalar}[1]{\left< #1 \right>}

\newcommand{\defiff}{\mbox{if}}

\newcommand{\Kappa}{\mathrm{K}}

\newcommand{\floor}[1]{\lfloor #1 \rfloor}
\newcommand{\ceil}[1]{\lceil #1 \rceil}

\usepackage[square, numbers, comma, sort&compress]{natbib} % Use the natbib reference package - read up on this to edit the reference style; if you want text (e.g. Smith et al., 2012) for the in-text references (instead of numbers), remove 'numbers' 
\hypersetup{urlcolor=blue, colorlinks=true} % Colors hyperlinks in blue - change to black if annoying
\title{\ttitle} % Defines the thesis title - don't touch this

\begin{document}

\frontmatter % Use roman page numbering style (i, ii, iii, iv...) for the pre-content pages

\setstretch{1.3} % Line spacing of 1.3

% Define the page headers using the FancyHdr package and set up for one-sided printing
\fancyhead{} % Clears all page headers and footers
\rhead{\thepage} % Sets the right side header to show the page number
\lhead{} % Clears the left side page header

\pagestyle{fancy} % Finally, use the "fancy" page style to implement the FancyHdr headers

\newcommand{\HRule}{\rule{\linewidth}{0.5mm}} % New command to make the lines in the title page

% PDF meta-data
\hypersetup{pdftitle={\ttitle}}
\hypersetup{pdfsubject=\subjectname}
\hypersetup{pdfauthor=\authornames}
\hypersetup{pdfkeywords=\keywordnames}

%----------------------------------------------------------------------------------------
%	TITLE PAGE
%----------------------------------------------------------------------------------------

\begin{titlepage}
\begin{center}

\textsc{\LARGE \univnameA}\\ % University name
\textsc{\LARGE \univnameB}\\[1.5cm] % University name
\textsc{\Large Master Thesis}\\[0.5cm] % Thesis type

\HRule \\[0.4cm] % Horizontal line
{\huge \bfseries \ttitle}\\[0.4cm] % Thesis title
\HRule \\[1.5cm] % Horizontal line
 
\begin{minipage}{0.4\textwidth}
\begin{flushleft} \large
\emph{Author:}\\
\authornames % Author name - remove the \href bracket to remove the link
\end{flushleft}
\end{minipage}
\begin{minipage}{0.4\textwidth}
\begin{flushright} \large
\emph{Supervisors:} \\
\href{http://www.science.uva.nl/math/People/show_person.php?Person_id=Eisner-Lobova+T.}{\supnameAa}\\ % Supervisor name - remove the \href bracket to remove the link  
\href{http://staff.science.uva.nl/~alejan/}{\supnameAb} \\% Supervisor name - remove the \href bracket to remove the link  
\href{http://www.im.uj.edu.pl/instytut/pracownik?id=219}{\supnameB}
\end{flushright}	
\end{minipage}\\[3cm]
 
\large \textit{A thesis submitted in fulfilment of the requirements\\ for the degree of \degreename}\\[0.3cm] % University requirement text
\textit{in the}\\[0.4cm]
\deptnameA,\ \univnameA \\ 
\deptnameB,\ \univnameB \\[2cm]
%\groupname\\\deptname\\[2cm] % Research group name and department name
 
{\large \today}\\[4cm] % Date

\vfill
\end{center}

\end{titlepage}

%----------------------------------------------------------------------------------------
%	DECLARATION PAGE
%	Your institution may give you a different text to place here
%----------------------------------------------------------------------------------------

\Declaration{

\addtocontents{toc}{\vspace{1em}} % Add a gap in the Contents, for aesthetics

I, \authornames, declare that this thesis titled \emph{``\ttitle''} is my own. I confirm that:

\begin{itemize} 
\item[\tiny{$\blacksquare$}] This work was done wholly or mainly while in candidature for a research degree at these Universities.
%\item[\tiny{$\blacksquare$}] Where any part of this thesis has previously been submitted for a degree or any other qualification at this University or any other institution, this has been clearly stated.
\item[\tiny{$\blacksquare$}] Where I have consulted the published work of others, this is always clearly attributed.
\item[\tiny{$\blacksquare$}] Where I have quoted from the work of others, the source is always given. With the exception of such quotations, this thesis is entirely my own work.
\item[\tiny{$\blacksquare$}] I have acknowledged all main sources of help.
%\item[\tiny{$\blacksquare$}] Where the thesis is based on work done by myself jointly with others, I have made clear exactly what was done by others and what I have contributed myself.\\
\end{itemize}
 
Signed:\\
\rule[1em]{25em}{0.5pt} % This prints a line for the signature
 
Date:\\
\rule[1em]{25em}{0.5pt} % This prints a line to write the date
}

\clearpage % Start a new page

%----------------------------------------------------------------------------------------
%	QUOTATION PAGE
%----------------------------------------------------------------------------------------

\pagestyle{empty} % No headers or footers for the following pages

\null\vfill % Add some space to move the quote down the page a bit

\textit{``Thanks to my solid academic training, today I can write hundreds of words on virtually any topic without possessing a shred of information, which is how I got a good job in journalism."}

\begin{flushright}
Dave Barry
\end{flushright}

\vfill\vfill\vfill\vfill\vfill\vfill\null % Add some space at the bottom to position the quote just right

\clearpage % Start a new page

%----------------------------------------------------------------------------------------
%	ABSTRACT PAGE
%----------------------------------------------------------------------------------------

\addtotoc{Abstract} % Add the "Abstract" page entry to the Contents

\abstract{\addtocontents{toc}{\vspace{1em}} % Add a gap in the Contents, for aesthetics
	
	Ultrafilters are very useful and versatile objects with applications throughout mathematics: in topology, analysis, combinarotics, model theory, and even theory of social choice. Proofs based on ultrafilters tend to be shorter and more elegant than their classical counterparts. In this thesis, we survey some of the most striking ways in which ultrafilters can be exploited in combinatorics and ergodic theory, with a brief mention of model theory.

	In the initial sections, we establish the basics of the theory of ultrafilters in the hope of keeping our exposition possibly self-contained, and then proceed to specific applications. Important combinatorial results we discuss are the theorems of Hindman, van der Waerden and Hales-Jewett. Each of them asserts essentially that in a finite partition of, respectively, the natural numbers or words over a finite alphabet, one cell much of the combinatorial structure. We next turn to results in ergodic theory, which rely strongly on combinatorial preliminaries. They assert essentially that certain sets of return times are combinatorially rich. We finish by presenting the ultrafilter proof of the famous Arrow's Impossibility Theorem and the construction of the ultraproduct in model theory.		
}

\clearpage % Start a new page

%----------------------------------------------------------------------------------------
%	ACKNOWLEDGEMENTS
%----------------------------------------------------------------------------------------

\setstretch{1.3} % Reset the line-spacing to 1.3 for body text (if it has changed)

\acknowledgements{\addtocontents{toc}{\vspace{1em}} % Add a gap in the Contents, for aesthetics

Any advances that are made in this thesis would not be have been possible without the guidance and help from the supervisors under whom the author had the privilege to work. Many thanks go to Pavel Zorin-Kranich, who was de facto an informal supervisor of this thesis, to Miko\l{}aj Fr\k{a}czyk for the illuminating discussions and his non-wavering enthusiasm, and to the StackExchange community for providing and endless supply of answers to the endless stream of questions produced during our work. We are also grateful to Professor Bergelson for expressing an interest in our research, and providing some useful remarks.

Last, but not least, the author wishes to thank his close ones for the continual support and understanding during the time of writing of this thesis. 

The \LaTeX\ template to which this thesis owns its appearance was created by Steven Gunn and Sunil Patel, and is distributed on Creative Commons License CC BY-NC-SA 3.0 at \url{http://www.latextemplates.com}.
}
\clearpage % Start a new page

%----------------------------------------------------------------------------------------
%	LIST OF CONTENTS/FIGURES/TABLES PAGES
%----------------------------------------------------------------------------------------

\pagestyle{fancy} % The page style headers have been "empty" all this time, now use the "fancy" headers as defined before to bring them back

\lhead{\emph{Contents}} % Set the left side page header to "Contents"
\tableofcontents % Write out the Table of Contents

\mainmatter % Begin numeric (1,2,3...) page numbering

\pagestyle{fancy} % Return the page headers back to the "fancy" style

% Include the chapters of the thesis as separate files from the Chapters folder
% Uncomment the lines as you write the chapters

%\input{ Chapter1}

%\setcounter{chapter}{-1}
\chapter*{Introduction}
\addcontentsline{toc}{chapter}{Introduction}  
\label{I:chapter} 
\lhead{Chapter \ref{I:chapter}. \emph{Introduction}} 

	Ultrafilters are one of the most mysterious and surprising objects in mathematics. On the one hand, there is no explicit construction of an ultrafilter and even proof of their existence involves the axiom of choice. On the other hand, they turn out to have remarkable applications in a wide variety of branches of mathematics. In topology they are closely tied to Stone-\v{C}ech compactification. In analysis, they provide a notion  of a generalised limit, which is in many ways the best possible generalised limit that exists.  In model theory, they make construction of ultraproducts and saturated models possible, leading to worthwhile applications in non-standard analysis. Even the theory of social choice can benefit from application of ultrafilters. What is even more, the space of ultrafilters is a highly non-trivial object with rich algebraic and topological structure, well worth the study in its own right. With a shade of national pride, we also mention that ultrafilters were first discovered by the Polish mathematician Tarski in 1930s, cf. \cite{history-who-gave-C-W-tale}.

	The application of ultrafilters that interests us the most lies in ergodic theory. Running a little ahead of the exposition, we informally present the basic idea behind these applications. In ergodic theory, one classically considers Ces\'{a}ro averages, which in a simplest instance take the form
$$\lim_{N \to \infty} \frac{1}{N} \sum_{n=1}^{N} f\left(T^nx\right),$$
where $T$ is a measure preserving transformation of a compact probability space $X$, and $f:\ X \to \RR$ is a measurable function. Ultrafilters allow one to replace the Ces\'{a}ro averages by generalised limits along an idempotent ultrafilter $p$
$$\llim{p}{n} f\left(T^nx\right).$$
While typical results in ergodic theory would imply that certain sets of recurrence times are non-empty, or at best syndetic, ultrafilter methods show additional algebraic structure of these sets, such as $\IP^*$ or $\C^*$.
  
Behaviour of the mentioned averages depends on the algebraic properties of $p$, hence it becomes necessary to study the algebraic structure of ultrafilters. We also need to establish a connection with combinatoric to recover a concrete notion of largeness from considerations about ultrafilters. However, we dispose of the need to study Ces\'{a}ro averages, which frees us of much of the $\varepsilon/\delta$ management. On the whole, we are able to shift much of the burden from analysis to algebra, which often simplifies the reasonings and strengthens the conclusions, as well as provides a different point of view.

%The most significant contributions in this area are due to Bergelson, McCutcheon, Kra, Blass, Leibman \dots \todo{make sure not to miss people, double check spelling, then triple check}
Many important contributions to the area of our investigation were made by Bergelson, Blass, Hindman, Knutson, Kra, Leibman, McCutcheon, and others. The highly illuminating papers by these authors were the basis of our research. Very accessible expository papers by Bergelson were especially helpful and motivating. Of importance to our considerations was also the sole paper by Schnell \cite{Schnell} on ergodic theory. A comprehensive reference for algebraic structure is due to Hindman and Strauss \cite{hindman-strauss}. Topological and set-theoretic issues are thoroughly discussed by Comfort and Negrepontis \cite{comfort-negre-book}. Great expository material can be found in the thesis by Zirnstein \cite{zirnstein-thesis}. Most of the discussed results come from these sources, with some minor extensions and simplifications due to the author.

	The aim of this thesis is to provide a possibly self-contained introduction to the many ways in which ultrafilters are applicable to ergodic theory and combinatorial number theory, and provide a glimpse of their applications in other areas. We hope to convince the reader that ultrafilters are a versatile and powerful tool for dealing with problems in these fields. We do not assume previous knowledge of ultrafilters, and take care to keep the treatment self-contained. A degree of familiarity with abstract topology and functional analysis is necessary for the dynamical applications. A nodding acquaintance with ergodic theory and with combinatorics is useful, but not strictly indispensable.

	The thesis structure is as follows.
	
	In Chapter \ref{A:chapter} we introduce the preliminary material. We begin by defining filters and ultrafilters on arbitrary spaces. Subsequently, we introduce the natural topological structure on the space of ultrafilters, as well as the semigroup structure, if the underlying space is a semigroup. We close with some remarks on general left-topological semigroups.

	In Chapter \ref{C:chapter} we discuss applications in combinatorics, especially combinatorial number theory. Notions introduced there include $\IP/\IP^*$ sets and $\C/\C^*$ (central) sets, which will be important for dynamical applications. We provide the ultrafilter proof of Hindman's theorem, which is arguably one of the most elegant application of ultrafilters, as well as some immediate generalisations. Other discussed results include van der Waerden Theorem and Hales-Jewett Theorem, which have effectively the same proof in the ultrafilter approach.

	In Chapter \ref{B:chapter} we discuss applications in ergodic theory. We first discuss the easy example of maps on the torus, where we derive recurrence results along polynomials and more general functions. A significant amount of work goes into establishing the correct generalisation of polynomial maps.  Next, we proceed to applications in general dynamical systems, where we prove that certain set of return times are $\IP^*$ or $\C^*$. 

	In Chapter \ref{D:chapter} we present some applications of ultrafilters in apparently unrelated areas of mathematics: social choice theory and model theory. Our main purpose there is to show how multifarious applications of ultrafilters are; the reader interested solely in ergodic theory may disregard this chapter. We start by considering Arrow's theorem on voting procedures. Next, we develop ideas already present in this simple application to construct ultraproducts, which are an important object in model theory, important for the introduction of non-standard analysis.

\chapter{Preliminaries} % Main chapter title
\label{A:chapter} 
\lhead{Chapter \ref{A:chapter}. \emph{Preliminaries}} 

In this chapter, we build up the basics of the theory, necessary for later applications. To begin with, we define ultrafilters on arbitrary sets, and then proceed to introduce the additional structure that the space of ultrafilters carries. In particular, we show that the space of the ultrafilters on a discrete semigroup has a natural structure of a compact left-topological semigroup, and can be identified with the Stone-\v{C}ech compactification. We develop some basic theory of Stone-\v{C}ech compactifications and compact left-topological semigroups in an abstract way, avoiding reference to the concrete example of the ultrafilter space, partly for elegance, partly because we will require a pinch more generality in the applications to come. The notion of a generalised limit, defined here, will play an the most essential  role in the following chapters.

For most of the applications, it suffices to restrict one's attention to ultrafilters on simple spaces. The single most useful example is the natural numbers $\NN = \{1,2,3,\cdots\}$. Slightly more general ones are the integers $\ZZ$, the finite sets of natural numbers $\cP_{\text{fin}}(\NN)$, and Cartesian products thereof. The reader may always assume that the space $X$ is one of these special cases.

All of the presented results are widely known by now. The basic definitions can be found in any introductory text, and are provided in many of the cited papers. For aspects connected to topology and pure set theory, we refer to \cite{comfort-negre-book}. Also, many purely topological texts treat Stone-\v{C}ech compactification, possibly without identifying it with the ultrafilters; see for example \cite{engelking-topology}. For a detailed discussion of the algebraic structure, we refer to \cite{hindman-strauss}.

\section{Set theory: filters and ultrafilters}
\index{ultrafilter|(}
Throughout this section, let $X$ denote for an arbitrary set. We will later require the space $X$ to additionally have the structure of a discrete semigroup, but just yet we work in a fully general context. The main goal in this section is to introduce and analyse the notion of ultrafilters on $X$, but it will be useful to also define the related weaker notions of filters and families with the finite intersection properties. 

\begin{definition}[Finite intersection property]\label{def:finite-intersection}
\index{finite intersection property}
  Let $X$ be a set. A family $\cA \subset \cP(X)$ is said to have \emph{finite intersection property} if and only if for any finite subset $\cA_0 \subset \cA$ it holds the intersection $\bigcap \cA_0$ is non-empty.

\end{definition}

\begin{definition}[Filter]\label{def:filter}
\index{filter}
  Let $X$ be a set. A family $\cF \subset \cP(X)$ is said to be a \emph{filter} if and only if the following conditions are satisfied:
  \begin{enumerate}[label={(\textit{\roman*})}]
    \item\label{def:filter-prop:1} $\emptyset \not \in \cF,\ X \in \cF.$ 
    \item\label{def:filter-prop:2} If $A \in \cF$ and $A \subset B$ then $B \in \cF$.
    \item\label{def:filter-prop:3} If $A,B \in \cF$ then $A \cap B \in \cF$. 
  \end{enumerate}
	We denote the family of all ultrafilters on the set $X$ by $\Filters{X}$.
\end{definition}

\begin{definition}[Ultrafilter]\label{def:ultrafilter}
\index{ultrafilter}
  Let $X$ be a set. A family $\cU \subset \cP(X)$ is said to be an \emph{ultrafilter} if and only if $\cU$ is a filter and the following additional condition is satisfied:	
  \begin{enumerate}[label={(\textit{\roman*})},start=4]
    \item\label{def:ultrafilter-prop:4} If $A \cup B \in \cU$ then $A \in \cU$ or $B \in \cU$.
  \end{enumerate}
	We denote the family of all ultrafilters on the set $X$ by $\Ultrafilters{X}$.
\end{definition}

  We acknowledge that this notation is slightly non-standard. It is more frequent to denote the set of ultrafilters by $\beta X$, which has its roots in topology. This issue will be discussed in more detail.

\begin{observation}
\index{ultrafilter}
  If $\cU \in \Filters{X}$ is a filter the property \ref{def:ultrafilter-prop:4} in definition \ref{def:ultrafilter} is equivalent to either of the following conditions:
  \begin{enumerate}[label={(\textit{\roman*})},start=5]
    \item\label{def:ultrafilter-prop:8} If $A \cup B = X$ then $A \in \cU$ or $B \in \cU$.
    \item\label{def:ultrafilter-prop:5} If $\bigcup_{i=1}^n A_i \in \cU$ then $A_i \in \cU$ for some $i$.
    \item\label{def:ultrafilter-prop:6} If $C,D \in \cP(X) \setminus \cA$, then $C \cup D \in \cP(X) \setminus \cA$
    \item\label{def:ultrafilter-prop:7} If $C \not \in \cU$ then $C^c \in \cU$.
  \end{enumerate}
\end{observation}

\begin{proof}
  \begin{description}
   \item[ \ref{def:ultrafilter-prop:4} $ \iff $ \ref{def:ultrafilter-prop:8}] Since $X \in \cU$ by property  \ref{def:filter-prop:1}, the implication \ref{def:ultrafilter-prop:4} $ \imply $ \ref{def:ultrafilter-prop:8}  is clear. Conversely, if $A \cup B =: C \in \cU$, then $A \cup B \cup C^c = X$, so either $A \in \cU$ or $B \cup C^c \in \cU$ by property \ref{def:ultrafilter-prop:8}. If $A \in \cU$ then we are done. If $B \cup C^c \in \cU$, then $B =  (B \cup C^c) \cap C \in \cU$, so we are done as well.

    \item[ \ref{def:ultrafilter-prop:8} $ \iff $ \ref{def:ultrafilter-prop:7}] Since $C \cup C^c = X$, the implication \ref{def:ultrafilter-prop:8} $ \imply $ \ref{def:ultrafilter-prop:7}  is clear. Conversely, if $A \cup B = X$, then setting $C := A \setminus B$ we find $C^c = B \setminus A$. By \ref{def:ultrafilter-prop:7} we have $C \in \cU$ or $C^c \in \cU$. Since $A \supset C$ and $B \supset C^c$, the property  \ref{def:filter-prop:2} implies $A \in \cU$ or $B \in \cU$.

    \item[ \ref{def:ultrafilter-prop:4} $ \iff $ \ref{def:ultrafilter-prop:5}] Condition \ref{def:ultrafilter-prop:4} is a special case of \ref{def:ultrafilter-prop:5}, where $n = 2$, so the implication \ref{def:ultrafilter-prop:5} $ \imply $ \ref{def:ultrafilter-prop:4}  is clear. On the other hand, \ref{def:ultrafilter-prop:5} follows from \ref{def:ultrafilter-prop:4} by induction. The case $n = 2$ is clear. Suppose \ref{def:ultrafilter-prop:8} holds for all $n < n_0$ where $n_0 \geq 3$. If we then have $\bigcup_{i=1}^{n_0} A_i \in \cU$, then either $A_{n_0} \in \cU$ or $\bigcup_{i=1}^{n_0-1} A_i \in \cU$ by the case $n = 2$. If $bigcup_{i=1}^{n_0-1} A_i \in \cU$, then case $n = n_0-1$ implies that $A_i \in \cU$ for some $1 \leq i < n_0$. Thus, either way, $A_i \in \cU$ for some $i$, so the claim for $n = n_0$ follows. By induction,  \ref{def:ultrafilter-prop:5}  holds for all $n$.

    \item[ \ref{def:ultrafilter-prop:4} $ \iff $ \ref{def:ultrafilter-prop:6}] Putting $C := A^c$ and $D := B^c$ we see that the two conditions are equivalent.  

  \end{description}
\end{proof}

\begin{remark}
\index{ultrafilter}
	The family of ultrafilters can be more concisely defined, using the following characterisation. A family $\cU \subset \cP(X)$ is an ultrafilter if and only if for any partition $X = A_1 \cup A_2 \cup A_3$, exactly one of $A_i$ belongs to $\cU$. We prefer the more elaborate definition because it is more intuitive and easier to motivate. 
\end{remark}

Having defined the objects of our interest in this section, let us provide some basic examples. It is clear from the above definitions that ultrafilters are filters, and that filters have the finite intersection property, so examples of some of these classes automatically also give examples of other classes.

Because the finite intersection property does not impose any additional structure, a simple way to give an example of a set with the finite intersection property is to consider an arbitrary subset of a given filter. We will shortly see that these are essentially the only examples. 

\begin{example}[Cofinite sets]
\index{filter!cofinite filter}
	Define $\cF_{\text{cofin}}$ consist of all sets $A \in \cP(X)$ such that $\#A^c < \aleph_0$, assuming additionally that $\#X \geq \aleph_0$. It is clear by direct verification that $\cF_{\text{cofin}}$ is closed under the operation of taking supersets and under finite intersections, so $\cF_{\text{cofin}}$ is a filter.

	More generally, if $\aleph_0 \leq \kappa \leq \card{X}$ is a cardinal number, then the family $\cF_{\kappa}$ consisting of all sets $A \in \cP(X)$ such that $\card{A}^c < \kappa$ is a filter. Thus defined filters are never ultrafilters, because $X$ can be partitioned into two sets of equal cardinality.
\end{example}

\begin{example}[Density $1$ sets]
\index{density}
	Suppose that $\delta:\ \cP(X) \to [0,1]$ is a subadditive density\footnote{
	We require that $\delta$ satisfies $\delta(X) = 1,\ \delta(\emptyset) = 0$ and $\delta(A \cup B) \leq \delta(A) + \delta(B)$ for $A,B \in \cP(X)$.}, for example $\delta = d^*$, the upper Banach density on $X = \NN$. Then the family $\cF_\delta$ of sets $A \in \cP(X)$ such that $\delta(A^c) = 0$ forms a filter. Indeed, it clear that $\emptyset \in \cF_\delta$, that $X \in \cF_\delta$, and that if $A \subset B,\ A \in \cF_\delta$, then also $B \in \cF_\delta$. Finally, if $A,B \in \cF_\delta$, then 
	$$\delta((A\cap B)^c) = \delta(A^c \cup B^c) \leq \delta(A^c) + \delta(B^c) = 0,$$
	so also $(A\cap B) \in \cF_\delta$. Hence, $\cF_\delta$ satisfies the definition of a filter.
\end{example}

\begin{example}[Neighbourhoods]
\index{filter}
	Suppose that $\cT \subset \cP(X)$ is a topology. Let $x \in X$ be an arbitrary point. Then the set $\cF_x$ of open neighbourhoods of $x$, i.e. of $A \in \cT$ such that $x \in A$, is a filter. The filter properties are an immediate consequence of the topological axioms.

One can extend this example to allow non-open neighbourhoods, or neighbourhoods of more general sets than singletons.
\end{example}

\begin{example}[Restrictions and extensions]
	Suppose that $\cF \in \Filters{X}$ is a filter on $X$, and $Y \subset X$ is a subset. Consider the family $\cF|_Y = \{A \cap Y \setsep A \in \cF\}$. It is clear that $\cF|_Y$ satisfies all the defining properties of a filter, except possibly for the requirement $\emptyset \not \in \cF|_Y$. Hence, $\cF|_Y$ is a filter on $Y$, provided that $Y^c \not\in \cF$. If $\cF$ is an ultrafilter, then an easy argument shows that so is  $\cF|_Y$.

	Conversely, if $Z \supset X$ is a superset, then we extend $\cF \in \Filters{X}$ to $\cG \in \Filters{Z}$ by declaring for $C \in \cP(Z)$ that $C \in \cG$ if and only if $C \cap X \in \cF$. If $\cF$ is an ultrafilter, then so is $\cG$.
\end{example}

Note that the above examples are, in general, not ultrafilters but merely filters. We now introduce the simplest examples of ultrafilters. As it shall will shortly turn out, these are the only ultrafilters that can be explicitly described.

\begin{definition}[Principal ultrafilters.]
\index{ultrafilter!principal}
  For $x \in X$, the family $\{ A \in \cP(X) \setsep x \in A \}$ is an ultrafilter. We denote this ultrafilter by $\principal{x}$. Ultrafilters of this form are said to be \emph{principal}, and accordingly ultrafilters that are not of this form are said to be \emph{non-principal}.
\end{definition}

\begin{remark}
\index{ultrafilter!principal}
  Principal ultrafilters can be characterised as the ultrafilters that include singletons. Alternatively, from property \ref{def:ultrafilter-prop:5} it follows that an ultrafilter is principal if and only if it contains a finite set.

	The principal ultrafilter can be construed as the set of neighbourhoods of a given point in the discrete topology. This is essentially the only case when the set of neighbourhoods is an ultrafilter.	
\end{remark}

	Although we are not able to exhibit concrete examples of ultrafilters, we will prove an existence statement which will provide us with all the ultrafilters we need. As a first step, we show how a family with the finite intersection property can be extended to a filter. Among other applications, this allows one to specify a filter by providing less data: a generating family with the finite intersection property, instead of all the elements of the filter.
	
\begin{lemma}[Constructing filters]\label{lem:constructing-filters}
  Let $\cA \subset \cP(X)$ be a family with the finite intersection property. Then there exists a unique filter $\cF \subset \cP(X)$ which contains $\cA$ and is minimal with respect to this property among filters. Moreover $\cF$ can be explicitly described as:
  $$ \cF = \left\{ A \setsep \exists \cA_0 \subset \cA :\ \# \cA_0 < \aleph_0\ \wedge \ \bigcap \cA_0 \subset A \right\}$$
\end{lemma}
\begin{proof}
  Let $\cF$ be defined as above. We shall prove that $\cF$ is indeed a filter, and that it satisfies the required uniqueness property.

  We begin by the showing that, assuming that $\cF$ is a filter, it is the unique minimal filter containing $\cA$. Indeed, let $\cG \subset \cP(X)$ be a filter and suppose that $\cA \subset \cG$, and let us consider a set $A \in \cF$ with $A \supset \bigcap \cA_0 $. Since $\cG$ is closed under finite intersections, we have $ \bigcap \cA_0 \in \cG$. Since $\cG$ is closed under taking supsets, $A \in \cG$. Since $A$ was chosen arbitrarily, it follows that $\cF \subset \cG$. Thus, $\cF$ is minimal, and it remains to show that it is a filter.

  By definition of the finite intersection property, all the intersections of the form $\bigcap \cA_0$ where $\cA_0 \subset \cA$ and $\# \cA_0 < \aleph_0$ are non-empty. Thus, if $A \supset \bigcap \cA_0$, then $A \neq \emptyset$, and thus $\emptyset \not \in \cF$. Taking arbitrary $\cA_0$, we also find that $X \in \cF$.

  Let $A \in \cF$ and $B \supset A$. Then, we have:
  $$ B \supset A \supset \bigcap \cA_0 $$
  for some finite $\cA_0 \subset \cA$. It follows immediately that $B \in \cF$, and thus $\cF$ is closed under taking supersets.

  Suppose that $A,B \in \cF$. Then, there exist finite subsets $\cA_0,\cB_0 \subset \cA$ such that $A \supset \bigcap \cA_0$ and $\cB \supset \bigcap \cB_0$. Then the family $\cA_0 \cup \cB_0$ is a again a finite subset of $\cA$, and we have by de Morgane's rules: 
  $$A \cap B \supset \left(\bigcap \cA_0\right) \cap \left( \bigcap \cB_0\right) =  \bigcap \left(\cA_0 \cup \cB_0\right) \in \cF.$$
  Thus, $\cF$ is closed under taking finite intersections.

  It follows that $\cF$ satisfies all the defining properties of a filter.
\end{proof}

	We next give a convenient characterisation of ultrafilters in terms of maximality. It will lead directly to the existence result alluded to earlier. Moreover, it provides some intuition concerning the structure of ultrafilters.  

\begin{proposition}[Characterisation of ultrafilters]\label{lem:ultrafilter-characterisation}\label{A:prop:ultrafilter-characterisation}
   Let $\cA \subset \cP(X)$ be an arbitrary family of subsets. Then the following conditions are equivalent:
   \begin{enumerate}[label={(\textit{\arabic*})},start=1]
    \item\label{lem:ultr-char-cond:1} The set $\cA$ is a maximal family with finite intersection property, i.e. $\cA$ has finite intersection property and if $\cA' \supset \cA$ also has this property then $\cA' = \cA$.
    \item\label{lem:ultr-char-cond:2} The set $\cA$ is an ultrafilter.
   \end{enumerate}
\end{proposition}
\begin{proof}\qquad\
  \begin{enumerate} 
    \item[\ref{lem:ultr-char-cond:1} $\imply$ \ref{lem:ultr-char-cond:2}] Suppose $\cA$ has finite intersection property, and there is no proper supset of $\cA$ with this property. Since all fiters clearly have the finite intersection property, it follows from lemma \ref{lem:constructing-filters} that $\cA$ is in fact a filter. Thus, it remains to verify the defining property of an ultrafilter. Let us now consider an arbitrary set $C \subset X$ which does not belong to $\cA$. Since $\cA \cup \{C\}$ is then a proper supset of $\cA$, it cannot have the finite intersection property. Taking into account that $\cA$ is already closed under finite intersections, this means that there exists a set $A \in \cA$ such that $C \cap A = \emptyset$. This can be rewritten as $C^c \supset A$, so from $A \in \cA$ we conclude that $C^c \in \cA$. Thus, $C \not \in \cA$ implies $C^c \in \cA$, and hence $\cA$ satisfies all the defining properties of an ultrafilter. 

    \item[\ref{lem:ultr-char-cond:2} $\imply$ \ref{lem:ultr-char-cond:1}]
    Suppose that $\cA$ is an ultrafilter. Since ultrafilters are closed under finite intersections, $\cA$ has finite intersection property, so it remains to show that no proper supset of $\cA$ has this property. Suppose that $\cA \subsetneq \cB$ where $\cB \subset \cP(X)$ is an arbitrary family, and let $B \in \cB \setminus \cA$. Since $B \not \in \cA$, by the ultrafilter property, $B^{c} \in \cA \subset \cB$. Thus, $B,B^{c} \in \cB$, so evidently $\cB$ does not have the finite intersection property. Since $\cB$ was taken arbitrarily, the maximality of $\cA$ follows. 
  \end{enumerate}
\end{proof}

\begin{corollary}\label{cor:ultrafilters-existance-above-FI-sets}\label{A:cor:ultrafilters-existance-above-FI-sets}
	If $\cA \subset \cP(X)$ is a family with finite intersection property, then there exists an ultrafilter $\cU$ which contains $\cA$.
\end{corollary}
\begin{proof}
	Let us fix $\cA$ consider the class $\alpha \subset \cP(\cP(X))$ of all families $\cB \subset \cP(X)$ that contain $\cA$ and have the finite intersection property. We can consider  $\alpha$ as a partially ordered set, with the natural order given by the inclusion. We claim that each chain $\gamma \subset \alpha$ has an upper bound. In fact, an upper bound can be explicitly described as $\cC := \bigcup \gamma$. It is clear that this family satisfies $\cC \supseteq \cB$ for any $\cB \in \gamma$. That $\cC$ is a filter follows immediately from the fact that the defining conditions are of the inductive type. Thus, the partially ordered set $(\alpha, \subseteq)$ satisfies the assumptions of Kuratowski-Zorn Lemma. It follows that there exists a maximal element in $\alpha$, say $\cU$. As a consequence of the definition of $\alpha$,  $\cU$ is maximal with respect to the finite intersection property. By Proposition \ref{A:prop:ultrafilter-characterisation}, $\cU$ is an ultrafilter, and by definition of $\alpha$, $\cU$ contains $\cA$, so we are done. 
\end{proof}

\begin{corollary}\label{cor:ultrafilters-exist}\label{A:cor:ultrafilters-exist}
\index{ultrafilter!non-principal}
  There exist non-principal ultrafilters on any infinite space $X$.
\end{corollary}
\begin{proof}
	Let $\cA$ be the family of all sets $A$ of the form $A = X \setminus \{x\}$. Then clearly $\cA$ has the finite intersection property. In fact, the minimal filter $\cF$ corresponding to $\cA$ consists of all sets with finite complement, which has already been discussed. By the above corollary, there exists an ultrafilter $\cU$ which contains $\cA$. This ultrafilter contains no singletons, because it contains all their complements. Thus, it is a non-principal ultrafilter.
\end{proof}

\begin{remark}\index{axiom!axiom of choice}
	The proof of Corollary \ref{cor:ultrafilters-exist} depends ostensibly on the Axiom of Choice, embedded in Kuratowski-Zorn Lemma. One can show that the  Axiom of Choice is really necessary. In fact, it is consistent with Zermelo-Fraeknel Axioms that no non-principal ultrafilters exist. We assume the Axiom of Choice throughout.
\end{remark}

There is a more constructive way of proving existence of ultrafilters, which offers additional insight into their structure. We are not able to provide and explicit construction, and Axiom of Choice will have to be used at some stage. However, we can describe an ultrafilter by transfinite induction, where each step contains a binary choice. The picture that emerges is that of a limit object, obtained after a transfinite number of $\card{\cP(X)}$ steps, where each step can readily be comprehended. To avoid trivialities, we assume $\card{X} \geq \aleph_0$.

Intuitively speaking, the presented construction considers each subset of $X$ in some preassigned order, and decides whether or not to include a given set in the ultrafilter being constructed --- assuming this decision is not yet determined by the choices made previously. We keep track of the choices made at each step, because we will use this construction to find cardinalities of certain sets of ultrafilters.

\begin{construction}\label{A:constr:ultrafilter-by-induction}
Let $\alpha := \card{\cP(X)} = 2^{\card{X}} $ be the cardinality of the family of all subsets of $X$. We can enumerate all sets these subsets by ordinals less then $\alpha$: $\cP(X) = \{A_\iota\}_{\iota < \alpha}$. We stress that the enumeration of the sets in $\cP(X)$ is done in advance of any subsequent choices.

Let $\cF_0$ be a filter. We shall construct an ascending family of filters $\{\cF_\iota\}_{\iota \leq \alpha}$ such that $ \cF_\alpha$ will turn out to be an ultrafilter. Because the construction will involve a choice at each stage, and we want to keep track of these choices, let $\chi \in \{\top,\bot\}^\alpha$ be an arbitrary sequence. 

The filter $\cF_0$ is already given, which constitutes the base for the transfinite induction that we are about to perform. We need to show how to construct $\cF_{\beta+1}$ given $\{\cF_\iota\}_{\iota \leq \beta}$ and how to construct $\cF_\zeta$ given $\{\cF_\iota\}_{\iota < \zeta}$ for limit ordinal $\zeta$. At the step $\cF_\gamma$ is defined, the following invariant shall be satisfied:
\begin{equation}
  \iota < \gamma \imply A_\iota \in \cF_\gamma \vee A_\iota^c \in \cF_\gamma.
  \tag{$\ast$}
  \label{A:cond:002}
\end{equation}
Additionally, we keep track of a sequence of ordinals $\{\tau(\beta)\}_{\beta < \alpha}$.

\newcommand{\Iota}{\mathrm{I}}

Let us consider a ordinal of the form $\beta + 1$, and assume that $\{\cF_\iota\}_{\iota < \beta}$ are already constructed. We define the family $\Iota := \{ \iota < \alpha  \setsep \ A_\iota , A_\iota^c \not \in \cF_\beta \}$, and ordinal $\tau(\beta) := \min \Iota$, with the convention that if $\Iota = \emptyset$ then $\tau(\beta) = \alpha$. If $\Iota = \emptyset$, then the filter $\cF_\beta$ already satisfies the defining property of ultrafilters. Hence we may set $\cF_{\beta+1} := \cF_\beta$, and certainly the property \eqref{A:cond:002} holds for $\beta+1$.

Let us suppose that $\Iota \neq \emptyset$, so $\tau(\beta) < \alpha$ and $\cF_\beta$ is not yet an ultrafilter. Let us put $B := A_{\tau(\beta)}$ if $\chi_\beta = \top$ and $B := A_{\tau(\beta)}^c$ if $\chi_\beta = \bot$. By construction, it is clear that $B^c \not \in \cF_\beta$, and hence the family $\cB := \cF_\beta \cup \{B\}$ has finite intersection property. By Lemma \ref{lem:constructing-filters}, there exists the smallest filter that contains $\cB$; let $\cF_{\beta+1}$ be this filter.

It remains to check that \eqref{A:cond:002} is satisfied for $\beta+1$. Because of the construction of $\tau$, it suffices to show that if $\iota \leq \beta$, then $\iota \leq \tau(\beta)$, which amounts to the proving that $\tau(\beta) \geq \beta$. To prove this, we first note that the function $\tau :\ \beta+1 \to \alpha$ is strictly increasing. Indeed, it is weakly increasing because the family $\cF_\iota$ is ascending, and we have $\tau(\iota+1) \neq \tau(\iota)$ unless $\tau(\iota) = \alpha$. Because of the monotonicity of $\tau$ and the construction of the order on the ordinals, it follows that $\tau(\beta) \geq \beta$, which finishes the inductive step.

Suppose now that $\zeta$ is a limit ordinal. Then we set $\cF_\zeta := \bigcup_{\iota < \zeta} \cF_\beta$. It is clear that $\cF_\zeta$ is a filter because each term $\cF_\iota$ in the union is a filter, and the family is ascending. It is also immediate that the condition \eqref{A:cond:002} holds for $\zeta$, and that $\cF_\zeta \supset \cF_\iota$ for $\iota < \zeta$. This finishes the inductive step, and hence also the construction.

We claim that that $\cF_\alpha$ is an ultrafilter. Indeed, because of the condition \eqref{A:cond:002}, for any $\gamma < \alpha$ it holds that $A_\gamma \in \cF_{\gamma+1} \subset \cF_\alpha$ or $A_\gamma^c \in \cF_{\gamma+1} \subset \cF_\alpha$ . Because we already know that $\cF_\alpha$ is already known to be a filter, this concludes the proof. We denote this ultrafilter by $\cU^\chi$.
\end{construction}

In the above construction, an ultrafilter was specified by making $\card{\cP(X)}$ binary choices, or choosing $\chi \in \{\top,\bot\}^{\card{\cP(X)}}$. This suggests that the cardinality of the ultrafilters that can be constructed should be $2^\card{\cP(X)}$. Note that this is the cardinality of $\cP(\cP(X))$, which contains $\Ultrafilters(X)$, so certainly we can never construct more ultrafilters than this. It is not clear, however, that the construction does not terminate at an earlier step, leading to a smaller number of constructed objects. The following Proposition shows that this is not the case.

\begin{proposition}
\index{set theory!cardinality}

  The cardinality of the space of all ultrafilters on $X$ is $\# \Ultrafilters{ X} = 2^{2^{\# X}}$. Moreover, if $\cA$ is a family with finite intersection property and $\# \cA < \# \cP(X)$, then the family of ultrafilters $\cU \in \Ultrafilters{X}$ such that $\cU \supset \cA$ has cardinality $2^{2^{\# X}}$
\end{proposition}
\begin{proof}
  We will only prove the second claim, since the first follows by taking $\cA := \emptyset$.  We take $\cF_0$ to be the smallest filter containing $\cA$, and we retain the notation from the above Construction \ref{A:constr:ultrafilter-by-induction}. We also let $\delta := \# X$.

  First, we show that if $\beta < \alpha$ then $\tau(\beta) < \alpha$. Suppose otherwise, and for a proof by contradiction let $\beta < \alpha$ be such that $\tau(\beta) = \alpha$. Then, $\cF_\beta = \cU$ is already an ultrafilter. Let $\cB := \{ A_{\tau(\iota)}^{\chi(\iota)} \}_{\iota < \beta}$, and choose an ordinal $\gamma$ with $\# \cB, \# \cA \leq \gamma < \alpha$. It follows from the construction that $\cF_\beta$ is the smallest filter that contains $\cA \cup \bB$. From the characterisation in Lemma \ref{lem:constructing-filters}, it follows that $\cF_\beta$ consists of the intersections of finite subsets of $\cA \cup \cB$, of which there are $\gamma$. On the other hand, $\# \cF_\beta = \alpha$, because there is a bijection between $\cF_\beta \times 2$ and $\cP(X)$. This is the sought contradiction.

  We next show that the map $\chi \mapsto \cU^\chi$ is injective. For a proof by contradiction, suppose that $\cU^\chi = \cU^\psi$ for some $\chi \neq \psi$. Let $\beta := \min \{ \iota < \alpha \setsep \chi(\iota) \neq \psi(\iota)\}$. We may assume that $\chi(\beta) = \top$ and $\psi(\beta) = \bot$, and by the choice of $\beta$ we have $\chi|_\beta = \psi|_\beta$. Hence, $\tau_\chi(\iota) = \tau_\psi(\iota)$ for $\iota \leq \beta$, because the part of the construction that defines these ordinals depends only on the first $\beta$ choices. If we denote by $\tau(\beta)$ the common value of $\tau_\chi(\beta) $ and $\tau_\psi(\beta)$, it becomes clear that $A_{\tau(\beta)} \in \cU_\chi$, while $A_{\tau(\beta)}^c \in \cU_\psi$. Consequently, $\cU_\chi \neq \cU_\psi$.

\end{proof}

\index{ultrafilter|)}
\renewcommand{\bar}{\overline}
\section{Topology: ultrafilters as a topological space}
\index{topology|(}
Our next goal is to establish the link between the ultrafilters and topology. There two main objectives that we will accomplish in this section.

Firstly, we will show that given an ultrafilter on the space $X$, it is possible to construct a notion of generalised limit for sequences indexed by $X$. This  generalisation has number of desirable properties, most notable of which is that the limits along ultrafilters exist for all sequences with terms in a compact Hausdorff space.

Secondly, we define and study a natural topology on the space of ultrafilters $\Ultrafilters{X}$. We show that this topology is compact Hausdorff, which makes it remarkably well behaved. 

The generalised limits and limits in the sense of topology are closely related to the limits in the topological sense. In fact, they can be considered to be the same notion, modulo a number of innocuous identifications. This leads us to the conclusion that $\Ultrafilters{X}$ can be identified with the Stone-\v{C}ech compactification of $X$, usually denoted by $\beta X$, assuming that we take $X$ to be a discrete topological space. 

Throughout this section, $X$ denotes a topological space. When $X$ comes with no standard topology, as is arguably the case for $\NN$, we assume the topology of $X$ to be discrete. At some point we will entirely restrict to discrete topological spaces, but we do not do it just yet, in the hope of providing some motivating examples.

As has already been mentioned earlier, one of the main motivations behind the notion of a filter is that it can be used to construct generalised limits in a natural way. This is done in the following definition.

\begin{definition}[Generalised limits]\label{A:def:gen-limit}
\index{limit!generalised limit}
  Let $Z$ be a topological Hausdorff space, let $f :\ X \to Z$ be any map, and let $\cF$ be a filter. If there exists $z \in Z$ such that $$(\forall U \in \Top(Z)):\ (z \in U) \imply ( f^{-1}(U) \in \cF)$$  then we define $z$ to be the limit of $f$ along $\cF$. Symbolically, we denote this by:
  $$\llim{\cF}x f(x) = z.$$
  If no such $z$ exists, we leave the symbol $\llim{\cF}x f(x)$ undefined. 
\end{definition}

The above definition does not guarantee that the limit, if exists, is unique. In particular, if $Z$ is equipped with the trivial topology: $\Top(Z) = \{\emptyset,Z\}$, then for \emph{any} $z \in Z$ it holds that $\llim{\cF}x f(x) = z$. However, this situation is not worse than for the classical notion of a limit. We will now show that for most interesting spaces, the limit is in fact unique, hence the notation will not lead to confusion.

\begin{proposition}
\index{limit!generalised limit}
 If $Z$ is Hausdorff, and there exists $z \in Z$ such that $\llim{\cF}x f(x) = z$, then this $z$ is unique.
\end{proposition}
\begin{proof}
 For a proof by contradiction, suppose that $z,z' \in Z$ are two distinct points such that $\llim{\cF}x f(x) = z, z'$. Since $Z$ is $T_{2}$, there exist two open neighbourhoods $U$ and $U'$ of $z$ and $z'$ respectively, such that $U \cap U' = \emptyset$. Let $A := f^{-1}(U)$ and $A' := f^{-1}(U')$. By the definition of the limit we have $A \in \cF$ and $A' \in \cF$, so on one hand $A \cap A' \in \cF$, and on the other hand $A \cap A' = f^{-1}(U \cap U') = \emptyset$. This is a contradiction, since $\emptyset \not \in \cF$ by definition. Thus, no two distinct limits can exist.
\end{proof}

Let us now see how the above notion corresponds to some of the usual limits. We begin with limits of conventional sequences indexed by natural numbers.

\begin{example}\label{exp:limit-ordinary}
\index{limit!classical limit}
\index{filter!cofinite filter}
 Take $X = \NN$, arbitrary Hausdorff $Z$, and define $\cF_{\text{cofin}} = \{ A \in \cP(\NN) \setsep \# A^c < \aleph_0 \}$ to be the filter of cofinite sets. Then 
  $$\llim{\cF_{\text{cofin}}}n f(n) = \lim_{n\to\infty} f(n).$$
  In particular, the limit may or may not exist.
  
  Moreover, let $L \subset \NN$ be an infinite set, and $\cF_L := \{ A \cap L \setsep A \in \cF_{\text{cofin}} \}$. Then $$\llim{\cF_L}{n} f(n) = \lim_{\substack{n\to\infty \\ n \in L}} f(n).$$
\end{example}

On the topological space $X$ we already have the topological notion of a limit. In the following example, we show how to recover this limit as a special case of the generalised limit. 
\begin{example}\label{A:exple:limit-topo}
\index{limit!topological limit}
	Let $y \in X$, and let $\cF_y$ consist of all the open neighbourhoods of $y$. 
Then the generalised limit coincides with the classical limit as defined in general topology:
	$$ \llim{\cF_{y}}x f(x) = \lim_{x \to y} f(y).$$
\end{example}

	As a special case of the above definition, we can compute limits along principal ultrafilters, which correspond to taking limits in the discrete topology.
\begin{example}\label{A:exple:limit-principal}
\index{ultrafilter!principal}
  Let  $y \in X$ and let $\cU_{y}$ be the principal ultrafilter associated to the point. Then 	
  $$\llim{\cU_{y}}x f(x) = f(y).$$
  In particular, this limit always exists.
\end{example}

The following property of ultrafilters is extremely useful in applications. It is the principal reason why we will restrict to ultrafilters in most of the subsequent discussion.

\begin{proposition}[Existence of limits]\label{A:prop:U-limit-exists}
\index{limit!generalised limit}
 If $\cU$ is an ultrafilter, then $\llim{\cU}{x} f(x)$ exists any map $f :\ X \to Z$ into a compact Hausdorff space $Z$. 
\end{proposition}
\begin{proof}
 For a proof by contradiction, suppose that $\llim{\cU}{x} f(x)$ does not exist, meaning that for no $z \in Z$ is it true that $\llim{\cU}x f(x) = z$. Then, there exist open neighbourhoods $U_z$ of $z$, such that $f^{-1}(U_z) \not \in \cU$. Since $Z$ is compact, and $\{ U_z \setsep z \in Z\}$ is an open cover, there exists a finite subcover, which is by necessity of the form $\{ U_z \setsep z \in Z_0 \}$ for some finite $Z_0 \subset Z$. Let $A_z := f^{-1}(U_z)$. We have $A_z \not \in \cU$, and $$\bigcup_{z \in Z_0} A_z = f^{-1}\left(\bigcup_{z \in Z_0} U_z \right) = f^{-1}(Z) = X.$$ But since the set $Z_0$ is finite, this is a contradiction with the defining property of ultrafilters \ref{def:ultrafilter-prop:4}.
\end{proof}

It will be convenient to have the following description of the limit along an ultrafilter. It is more concrete, and easier to work with, than the one derived by using the definition verbatim.

\begin{proposition}[Characterisation of limits]\label{prop:characterisation-of-limits}
\index{limit!generalised limit}
 If $\cU$ is an ultrafilter, then 
  $$\{\llim{\cU}x f(x) \} = \bigcap_{A \in \cU} \cl f(A) .$$ 
\end{proposition}
\begin{proof}
 We will first show that $z := \llim{\cU}x f(x)$ lies in $\cl{ f(A) }$ for any $A \in \cU$. It will suffice to show that if $U \in \Top Z$ is a neighbourhood of $z$ then $U \cap f(A) \neq \emptyset$. But we know that $f^{-1}(U) \in \cU$, so $f^{-1}(U \cap f(A) ) \supset  f^{-1}(U) \cap A \in \cU$ and in particular $U \cap f(A)$ cannot be empty.

 Let us now show that if $w \in \bigcap_{A \in \cU} \cl{ f(A) } $, then $\llim{\cU}x f(x) = w$. Let $w \in Z$ be any such point, and let $U \in \Top Z$ be any open neighbourhood of $w$. For any $A \in \cU$ we have that $U \cap f(A) \neq \emptyset$, so $f^{-1}(U) \cap A \neq \emptyset$. Since $A$ was taken arbitrarily, it follows that $f^{-1}(U) \in \cU$. Thus, directly from the definition $z = \llim{\cU}x f(x)$. 
\end{proof}

It is natural to inquire into the connection between the generalised limits we just defined, and the more classical notion of a limit in a topological space. It turns out that this relation is rather close, and generalised limits can be realised as the classical limits for the properly chosen topology on $\Ultrafilters{X}$. The Definition \ref{A:def:gen-limit} suggests that the following sets should be open in the topology we are about to construct.

\begin{definition}[Base clopen sets]\label{def:ultrafilter-topology}
\index{ultrafilter!topological structure}
  For a set $A \in \cP(X)$, define $\bar{A} \in \cP(\Ultrafilters{X})$ to be the set:
  $$ \bar{A} = \{ \cU \in \Ultrafilters{X} \setsep A \in \cU \}$$
\end{definition}

We stress that right now the symbol $\bar{A}$ is not meant to denote closure, but merely the construction in the definition above. It so happens that these sets \emph{will} be closures in the topological sense (up to the natural identification of elements of $X$ with the related principal ultrafilters), but they will also be open sets, and indeed a basis for a topology. A reader accustomed to working with connected topological spaces may find this worrying at first, but a closer inspection shows that this situation merely indicates that the constructed topology will be highly disconnected. 

Before we pass on to using these sets to introduce a topology, let us note some of the convenient properties they satisfy.

\begin{proposition}[Properties of closure]\label{lem:properties-of-bar}
\index{ultrafilter!topological structure}
  The operation $A \mapsto \bar{A}$ defined above has the following properties.
  \begin{enumerate}[label={(\textit{\arabic*})},start=1]
   \item\label{lem:properties-of-bar:1} If $A \in \cP(X)$ then $\bar{A^c} = \bar{A}^c$.
   \item\label{lem:properties-of-bar:2} If $A,B \in \cP(X)$ then $\bar{A \cap B} = \bar{A} \cap \bar{B}$.
   \item\label{lem:properties-of-bar:3} If $A,B \in \cP(X)$ then $\bar{A \cup B} = \bar{A} \cup \bar{B}$.
  \end{enumerate}
\end{proposition}
\begin{proof}
  \begin{enumerate}[label={(\textit{\arabic*})},start=1]
   \item We need to show that an ultrafilter $\cU$ contains $A$ if and only if it does not contain $A^c$. One direction is clear: $\cU$ cannot contain both $A$ and $A^c$, since otherwise it would have to contain $A \cap A^c = \emptyset$ by property \ref{def:filter-prop:3}, which contradicts property \ref{def:filter-prop:1}. Conversely, since $A \cup A^c = X$, by property \ref{def:ultrafilter-prop:4}, the ultrafilter $\cU$ has to contain either $A$ or $A^c$.

   \item We need to show that an ultrafilter $\cU$ contains $A \cap B$ if and only if $p$ contains both $A$ and $B$. For one direction, note that if $\cU$ contains $A \cap B$ then $\cU$ contains all supersets of $A \cap B$ as well by the property \ref{def:filter-prop:2}, so it contains both $A$ and $B$. Conversely, if $\cU$ contains $A$ and $B$, then it contains $A \cap B$ by the property \ref{def:filter-prop:3}.
  
   \item Using the previous points, we find that:
    $$ \bar{A \cup B} 
      = \bar{(A^c \cap B^c)^c} 
      = \bar{(A^c \cap B^c)}^c
      = \bar{A^c} \cap \bar{ B^c}^c
      = (\bar{A}^c \cap \bar{ B}^c)^c
      = \bar{A} \cup \bar{B}.
    $$
  \vspace{-1cm}
  \end{enumerate}    
\end{proof}

We recall a classical result characterising families of sets that can be used to define a topology. Together with the above observations, it will immediately allows us to describe a topology on $\Ultrafilters{X}$.

\begin{theorem}\label{thm:topology-base-characterisation}
\index{topology!topological base}

\newcommand{\Space}{X}
\newcommand{\Base}{\cB}
\newcommand{\Element}{x}
\newcommand{\Topology}{\cT}

  Suppose that $\Space$ is a set and $\Base \subset \cP(\Space)$ is a family of sets such that the following conditions are satisfied:
  \begin{enumerate}
   \item $\bigcup \Base = \Space,$
   \item $(\forall A,B \in \Base) (\forall \Element \in A \cap B) (\exists C \in \Base) :\ \Element \in C \subset A \cap B$
  \end{enumerate}
 Then there exists a unique topology $\Topology$ on $\Space$ for which $\Base$ is a base. This is the coarsest topology for which all sets $B \in \Base$ are open. The open sets in this topology are precisely the sets of the form $\bigcup \Base_0$ for $\Base_0 \subset \Base$.
\end{theorem}

\begin{definition}
\index{ultrafilter!topological structure}
 We turn $\Ultrafilters{ X}$ into topological space by declaring the family $\{ \bar{A} \setsep A \in \cP(X) \}$ to be the base of the topology. By Theorem \ref{thm:topology-base-characterisation} and Proposition \ref{lem:properties-of-bar}, this indeed defines a topology.
\end{definition}

We shall now proceed to the study of the topology of $\Ultrafilters{X}$.  This topology turns out to have many desirable properties. Because the topology on  $\Ultrafilters{X}$ does not carry any connection to the topology on $X$, we will be assuming from now on that the topology on $X$ is discrete. Under this assumption, $\Ultrafilters{X}$ can be shown to be the maximal compactification of the discrete space $X$, in a sense that will be made precise soon. 

\begin{proposition}\label{prop:beta-X-is-T2}
\index{ultrafilter!topological structure}
 The topological space $\Ultrafilters{X}$ is Hausdorff.
\end{proposition}
\begin{proof}
 Let $\cU$ and $\cV$ be any distinct ultrafilters. By the characterisation of ultrafilters as the maximal families with finite intersection property in \ref{lem:ultrafilter-characterisation}, we see that neither of $\cU$ and $\cV$ is contained in the other. Thus, there exists sets $A,B \in \cP(X)$ such that $A \in \cU \setminus \cV$ and $B \in \cV \setminus \cU$. Now the ultrafilter property \ref{def:ultrafilter-prop:7} ensures that $A^c \in \cV$ and $B^c \in \cU$. If we now denote $A_1 := A \setminus B$ and $B_1 := B \setminus A$ then it follows that $A_1 \in \cU$ and $B_1 \in \cV$ and $A_1 \cap B_1 = \emptyset$. Thus, $\cU \in \bar{A}_1$ and $\cV \in \bar{B}_1$. Finally, $$\bar{A}_1 \cap \bar{B}_1 = \bar{A_1 \cap B_1} = \bar{\emptyset} = \emptyset,$$ by Lemma \ref{lem:properties-of-bar}. Thus, $\bar{A}_1$ and $\bar{B}_1$ are separating neighbourhoods for $\cU$ and $\cV$.
\end{proof}

\begin{proposition}\label{A:prop:beta-X-is-cmpct}
\index{ultrafilter!topological structure}
\index{topology!compactness}
 The topological space $\Ultrafilters{X}$ is compact.
\end{proposition}
\begin{proof}
 Let $\cC \subset \operatorname{Top}(\Ultrafilters{X}X)$ be an open cover of $\Ultrafilters{X}$. Replacing $\cC$ by a finer cover if necessary, we can assume $\cC$ consists only of base sets of the form $\bar{A}$ with $A \in \cP(X)$. Thus, we can find a family $\cA \subset \cP(X)$ such that $\cC = \{ \bar{A} \setsep A \in \cA \}$. For any $x \in X$ and the related principal ultrafilter $\principal{x}$ based at $x$, we know that $\principal{x} \in \bar{A}$ if and only if $x \in A$. Thus, $\cA$ is a cover of $X$. 
 
 I claim that one can find a finite subcover of $\cA$. For a proof by contradiction, suppose the sum $\bigcup \cA_0$ of any finite family $\cA_0 \subset \cA$ is not the full space $X$. Let $\cB := \{A^c \setsep A \in \cA \}$ denote the family of complements of sets in $\cA$. We can rephrase the above assumption by saying that for any finite family $\cB_0 \subset \cB$ we have $\bigcap \cB_0 \neq \emptyset$. Thus, $\cB$ has the finite intersection property. By corollary \ref{cor:ultrafilters-existance-above-FI-sets}, there exists an ultrafilter $\cU$ that contains $\cB$. By construction, for any $A \in \cA$ we have $A^c \in \cU$, so $\cU \in \bar{A}^c$. But this means that $\cU$ does not belong to any of the sets $\bar{A}^c$ in the cover $\cC$, which is a contradiction with $\cC$ being a cover.

 Let $\cA_0$ be the finite cover of $\cA$, whose existence we have just proved, and let $\cC_0 := \{ \bar{A} \setsep A \in \cA_0\}$ be the corresponding part of $\cC$. I claim that $\cC$ is then a cover of $\Ultrafilters{X}$. Indeed, let $\cU$ be any ultrafilter. Then $\bigcup \cA_0 = X \in \cU$, so by the ultrafilter property \ref{def:ultrafilter-prop:5} there exists $A \in \cA_0$ such that $A \in \cA$. Thus, $\cU \in \bar{A} \in \cC_0$, as desired.  
\end{proof}

\begin{corollary}
  The topological space $\Ultrafilters{X}$ is normal.
\end{corollary}
\index{ultrafilter!topological structure}
\begin{proof}
  It is well known that compact Hausdorff spaces are normal. \cite{engelking-topology} 
\end{proof}

As the above results show, the space of all ultrafilters $\Ultrafilters{X}$ is well-behaved from the topological point of view. However, it should be noted that this space is also large, as the following corollary shows.

\begin{corollary}
\index{ultrafilter!topological structure}
\index{set theory!cardinality}
  The topological space $\Ultrafilters{X}$ is not first countable, and in particular not metrizable.  
\end{corollary}
\begin{proof}
  The set $X$ is dense in $\Ultrafilters{X}$. If $\Ultrafilters{X}$ was first countable, then all point in $\cl X$ could be described as limits of sequences (indexed by $\omega$) with elements in $X$. The cardinality of such sequences is at most $(\# X)^{\aleph_0} \leq 2^{\# X}$. On the other hand, we have seen that $\Ultrafilters{X} = 2^{2^{\#X} } > 2^{\# X}$, hence $\Ultrafilters{X}$ cannot be first countable. It is known that metrizable spaces are first countable, so $\Ultrafilters{X}$ is in particular not metrizable.
\end{proof}

So far, we have studied the basic topological properties of  $\Ultrafilters{X}$. Note that there is a natural injective map $i :\ X \to \Ultrafilters{X}$ given by $i(x) = \cF_x$, the principal ultrafilter. We will now study the inclusion map in more detail, and show that $X$ can be considered as a subspace of $\Ultrafilters{X}$.

\begin{proposition}\label{prop:closures-are-the-same}
\index{ultrafilter!topological structure}
 If $A \in \cP(X)$ and $i:\ X \to \Ultrafilters{X}$ is the natural inclusion, then $\bar{A} = \cl i(A)$.
\end{proposition}
\begin{proof}
 For any $x \in A$ we have $A \in \principal{x}$, so $\principal{x} \in \bar{A}$ and consequently $i(A) \subset \bar{A}$. Since $\bar{A}$ is closed by the definition of topology on $\Ultrafilters{X}$, it follows that $\cl i(A) \subset \bar{A}$.

 Conversely, suppose that $\cU \in \bar{A}$, and let us consider any base neighbourhood of $\cU$ of the form $\bar{B}$ with $B \in \cP(X)$. Then, $A \in \cU$ and $B \in \cU$, so $A \cap B \in \cU$. Thus, for any $x \in A \cap B$ we have $\principal{x} \in i(A) \cap \bar{B}$, so in particular  $ i(A) \cap \bar{B}$ is not empty. Since $B$ was chosen arbitrarily, it follows that $\cU \in \cl{i(A)}$. Thus, $\bar{A} \subset \cl{i(A)}$. 

 Since we have inclusions $\cl i(A) \subset \bar{A} \subset \cl{i(A)}$, the sets $\bar{A}$ and $\cl{i(A)}$ are equal.
\end{proof}

\begin{corollary}\label{A:cor:inclusion-is-continuous}
\index{ultrafilter!topological structure}
  If $X$ is discrete, then the inclusion  $i:\ X \to \Ultrafilters{X}$ is a homeomorphism onto its image.
\end{corollary}
\begin{proof}
  It is clear that $i$ is injective, and continuous. By the above Proposition \ref{prop:closures-are-the-same}, if $\{\cF_x\}_{x \in A} = i(A)$ is an arbitrary set of principal ultrafilters, then:
  $$ \left( \cl \left\{\cF_x\right\}_{x \in A} \right) \cap i(X) = \bar{A} \cap i(X) =  \{\cF_x\}_{x \in A}.$$
  Hence, arbitrary subset of $i(X)$ is closed in the induced topology, and the topology of $i(X)$ is discrete, which finishes the proof.
\end{proof}
	
\begin{corollary}\label{prop:dense-image-in-beta-X}
\index{ultrafilter!topological structure}
 The image $i(X) $ of the standard inclusion $i :\ X \to \Ultrafilters{X}$ is dense.
\end{corollary}
\begin{proof}
 It suffices to apply the previous Proposition \ref{prop:closures-are-the-same} to the full space $X$ to find that:
  $$ \cl{i(X)} = \bar{X} = \Ultrafilters{X} .$$
\end{proof}

Our next step is to study the generalised limits from the topological perspective. The following proposition shows that the generalised limits can be though of essentially as ordinary limits in the space $\Ultrafilters{X}$.

\begin{proposition}\label{prop:U-lim-is-continuous}
\index{topology!continuity}
\index{limit!generalised limit}
 For a fixed map $f :\ X \to Z$ into a compact Hausdorff space, the map $\cU \mapsto \llim{\cU}x f(x)$ is continuous.
\end{proposition}
\begin{proof}
 For ease of notation, define $l(\cU) := \llim{\cU}x f(x)$. We need to prove that $l$ is a continuous map, i.e. that for any open set $W \in \Top Z$, the pre-image $l^{-1}(W)$ is open. For general topological reasons, it will suffice to show that for any $\cU \in l^{-1}(W)$ there exists a set $A \in \cP(X)$ such that $\cU \in \bar{A}$ and $l(\bar{A}) \subset W$. For any $A$ and ultrafilter $\cV \in \bar{A}$ we have by Proposition \ref{prop:characterisation-of-limits} that $l(\cV) \in \bar{f(A)}$, and hence $l(\bar{A}) \subset \cl f(A)$. Since $Z$ is normal, we can find $V \in \Top Z$ such that $\cl V \subset W$. Let $A := f^{-1}(V)$. By definition of the limit, $\cU \in \bar{A}$. By the above observation:
  $$l(\bar{A}) \subset \cl f(A) \ \subset \cl V \subset W $$
\end{proof}

The above considerations show that the space $\Ultrafilters{X}$ is a compactification of $X$ with the rather special property that many maps defined on $X$ can be naturally prolonged to $\Ultrafilters{X}$. This situation has been studied in much depth by topologists in the more general context of locally compact topological spaces.

\begin{definition}[\v{C}ech-Stone compactification]\label{A:def:beta-X}
\index{Cech-Stone compactification@\v{C}ech-Stone compactification}
\index{topology!compactness}
Let $X$ be a locally compact Hausdorff topological space. Let $Y$ be a compact Hausdorff topological space, and $i :\ X \to Y$ a continuous, injective map. Then the pair $(Y,i)$ is said to the \v{C}ech-Stone compactification of $X$ if and only if for any compact Hausdorff topological space $Z$ and continuous map $f:\ X \to Z$, there exists a unique continuous map $g:\ Y \to Z$ such that $f = g \circ i$.
\end{definition}

\begin{proposition}\label{A:prop:beta-X-is-unique}
\index{Cech-Stone compactification@\v{C}ech-Stone compactification}
 If the \v{C}ech-Stone compactification of $X$ exists, then it is unique up to unique isomorphism. More precisely, if $(Y,i)$ an $(Y',i')$ are two Stone-\v{C}ech compactifications, then there exists a unique isomorphism of topological spaces $u :\ Y \to Y'$ such that $u \circ i = i'$.
\end{proposition}
\begin{proof}
 Suppose that $(Y,i)$ and $(Y',i')$ are two \v{C}ech-Stone compactifications of $X$. Then, applying the definition of \v{C}ech-Stone compactification for $(Y,i)$ to the map $i':\ X \to Y'$, we find that there exists a unique map $g:\ Y \to Y'$ such that $g \circ i = i'$. Similarly, there exists unique $g' :\ Y' \to Y$ such that $g' \circ i' = i$. Then $g' \circ g :\ Y \to Y$ is such that $g' \circ g \circ i = g' \circ i' = i$. Applying the definition of compactification once more, this time to $(Y,i)$ and the map $i:\ X \to Y$ we conclude that $g' \circ g = \id_Y$. Likewise, we show that $g \circ g' = \id_{Y'}$. Thus, $g$ is an isomorphism between $(Y,i)$ and $(Y',i')$, in the sense that it is an isomorphism between $Y$ and $Y'$ and intertwines between $i$ and $i'$. Uniqueness of $g$ follows from uniqueness in the definition the compactification.
\end{proof}

The following theorem affirms existence \v{C}ech-Stone compactification in a situation much more general than we need in our applications.

\begin{theorem}
\index{Cech-Stone compactification@\v{C}ech-Stone compactification}
  If $X$ is a locally compact Hausdorff topological space, then there exists a \v{C}ech-Stone compactification of $X$.
\end{theorem}
\begin{proof}
  See \cite{engelking-topology}.
\end{proof}

\begin{definition}
\index{Cech-Stone compactification@\v{C}ech-Stone compactification}
 If $X$ is a locally compact Hausorff topological space, then we denote its \v{C}ech-Stone compactification by $(\beta X,i)$, with the understanding that $\beta X$ is defined only up to the unique isomorphism. If $f:\ X \to Z$ is a map to an arbitrary compact Hausdorff space, then we denote by $\beta f :\ \beta X \to Z$ the unique continuous extension such that $\beta f \circ i = f$.
\end{definition}

With this more general language, we can summarise many of the previous results on the topology of $\Ultrafilters{X}$ is a much more succinct form.

\begin{theorem}\label{beta-X-is-CS-compactification}
\index{Cech-Stone compactification@\v{C}ech-Stone compactification}
 Let $X$ be a discrete topological space. Then the space $\Ultrafilters{X}$ together with the natural inclusion map $i :\ X \to \Ultrafilters{X}$, is the \v{C}ech-Stone compactification of $X$.
\end{theorem}
\begin{proof}
 Let $f :\ X \to Z$ be any continuous map from $X$ to a compact space $Z$. Define $g :\ \Ultrafilters{X} \to Z$ by the formula $g(\cU) := \llim{\cU}{x} f(x)$. Then $g$ is continuous by proposition \ref{A:cor:inclusion-is-continuous}. By Example \ref{A:exple:limit-principal} we have that $g \circ i = f$.

  For uniqueness, suppose that $h:\ \beta X \to Z$ is another continuous function such that $h \circ i = f$. By the choice of $g$ and $h$, we have $h \restrict{i(X)} = g \restrict{i(X)}$. But $i(X)$ is dense in $\beta X$ by Proposition \ref{prop:dense-image-in-beta-X}, so $h = g$, as desired.
\end{proof}

\begin{remark}
  We now have two different notations for the space of ultrafilters on $X$, namely $\beta X$ and $\Ultrafilters{X}$. They are equivalent, but seem to carry slightly different intuitions. We will use the notation $\Ultrafilters{X}$ when topological structure is irrelevant, and consequently denote ultrafilters by $\cU,\cV,\cW,\dots$. This will be done when an ultrafilter is thought of as a family of sets with particular properties. When topological properties become important, especially when considering limits, we will prefer the notation $\beta X$ for the space of ultrafilters on $X$, and use $p,q,r,\dots$ to denote ultrafilters. It will usually be more helpful to think of ultrafilters as limit objects in this case. We keep in mind that an ultrafilter corresponds to a family of sets, but avoid notations like ``set $A \in p$'' for aesthetic reasons.
\end{remark}

Below, we list some of the properties of the extensions of maps provided by the \v{C}ech-Stone compactification. They are very useful when one is faced with the need to compute generalised limits, and mimic the analogous rules for classical limits.

\begin{proposition}\label{prop:beta-f-properties}
\index{Cech-Stone compactification@\v{C}ech-Stone compactification}
  \begin{enumerate}
  \item For any map $f:\ X \to Z$ to a compact Hausdorff space $Z$ we have $\llim{\cU}{x} f(x) = \beta f(\cU)$
  \item For any maps $f:\ X \to Z$ and $g:\ Z \to T$ to compact  Hausdorff spaces $Z,T$, we have $\beta( g \circ f) = g \circ \beta f$. In particular, $\llim{\cU}{x} g \circ f(x) = g( \llim{\cU}{x} f(x))$.
  \item For any maps $f:\ X \to Y$ and $g:\ Y \to Z$ where $X$ and $Y$ are discrete and $Z$ is compact Hausdorff, we have $\beta( g \circ f) = \beta g \circ \beta(i_Y \circ f)$, where $i_Y :\ Y \to \beta Y$ is the inclusion.
  \item For any maps $f:\ X \to Z$ and $g:\ X \to W$ where $X$ is discrete and $Z, W$ are compact Hausdorff, we have $\beta( f \times g) = (\beta f) \times (\beta g)$.
	\item For any maps $f :\ X \to Z$, $g:\ X \to W$ and $h:\ Z \times W \to T$, consider the map $c:\ X \to T$ given by $c(x) = h(f(x),g(x))$. Then $\beta c (x) = h (\beta f(x), \beta g(x))$.
 \end{enumerate}

\end{proposition}
\begin{proof}
 \begin{enumerate}
   \item It follows directly from how the limit was defined.
   \item It suffices to check that $g \circ \beta f$ satisfies the universal property: $ (g \circ \beta f) \circ i = g \circ f$. But this is clear, since $\beta f \circ i = f$, and composition is associative. 

   The statement about the limits follows directly from relation of $\llim{\cU}{x}$ to $\beta f$ from the previous point.
  \item It suffices to check that $\beta g \circ \beta(i_Y \circ f)$ satisfies the universal property: $ \beta g \circ \beta(i_Y \circ f) \circ i_X = g \circ f$. This can be done as follows:
  $$ \beta g \circ \beta(i_Y \circ f) \circ i_X = \beta g \circ i_Y \circ f = g \circ f $$ 
	\item Let $p :\ Z \times W \to Z$ and $q :\  Z \times W \to W$ be the standard projection maps. To verify that $\beta (f \times g) = \beta f \times \beta g$, it sufices to prove that $p \circ \beta  (f \times g) = \beta f$ and $q \circ \beta (f \times g) = \beta g$. From the previous points, we already know that:
 $$p \circ \beta  (f \times g) = \beta( p \circ   (f \times g) ) = \beta f $$
and likewise $p \circ \beta  (f \times g) = \beta g$, hence the claim follows.
	\item Follows immediately from the previous observations.
  \end{enumerate}
\end{proof}

The following special case of the above theorem shows that generalised limits have many of the properties the classical limits have.

\begin{corollary}
\index{Cech-Stone compactification@\v{C}ech-Stone compactification}
	Let $f,g:\ \NN \to \hat{\RR}$ be any maps, where $\hat{\RR} = \RR \cup \{+\infty,-\infty\}$. Then we have:
	\begin{align*}
	\llim{\cU}{x} \left( f(x) + g(x) \right) 
		&= \llim{\cU}{x}f(x) + \llim{\cU}{x} g(x), \\
	\llim{\cU}{x} \left( f(x) \cdot g(x) \right) 
		&= \llim{\cU}{x}f(x) \cdot \llim{\cU}{x} g(x), \\
	\end{align*}
	provided that the application of the operations $+$ and $\cdot$ does not lead to indeterminate symbols $\infty -\infty$, $0 \cdot (\pm \infty)$. Likewise, we have:
	\begin{align*}
	\llim{\cU}{x} \left( f(x) - g(x) \right) 
		&= \llim{\cU}{x}f(x) - \llim{\cU}{x} g(x), \\
	\llim{\cU}{x} \left( f(x) / g(x) \right) 
		&= \llim{\cU}{x}f(x) / \llim{\cU}{x} g(x), \\
	\end{align*}
	provided that the operations can be carried out.

\end{corollary}

\index{topology|)}
\section{Algebraic structure of filters and ultrafilters}
\newcommand{\Group}{S}
\index{semigroup}
\index{ultrafilter!semigroup structure}
We will presently show how to give $\Ultrafilters{X}$ the structure of a semigroup, assuming that $X$ is a semigroup. The derived structure will be natural, but not the only possible. There are in fact two competing and equally natural notions of semigroup structure, so one has to be careful when consulting the literature.

Throughout this section, we assume that $X$ is a semigroup. We also make $X$ into a topological space by declaring that the topology on $X$ is discrete. We have seen how to endow $\Ultrafilters{X}$ with a natural topological structure.

We begin by giving some algebraic definitions, needed to define the algebraic structure on $\Ultrafilters{X}$.

\renewcommand{\Group}{X}

\begin{definition}\label{A:def:inverses}
  For a set $A \subset \Group$ and $x \in \Group$, we define $\ldiv{x}{A}$ to be the set $\{ y \in \Group \setsep xy \in A \}$, and $\rdiv{x}{A}$ to be the set $\{ y \in \Group \setsep yx \in A \}$.

  Likewise, for a filter $\cF$ we define $\ldiv{\cF}{A}$ to be the set $\{ x \in \Group \setsep \rdiv{x}{A} \in \cF \}$, and $\cF$ we define $\rdiv{\cF}{A}$ to be the set $\{ x \in \Group \setsep \ldiv{x}{A} \in \cF \}$.
\end{definition}

\begin{remark}
  Note that for $\ldiv{\principal{x}}{A}$ we use $\rdiv{x}{A}$ rather than $\ldiv{x}{A}$. This makes sense, since this way we have $\ldiv{\principal{x}}{A} = \ldiv{x}{A}$. The analogous remark applies to $\rdiv{\principal{x}}{A}$.
\end{remark}
\begin{observation}
	Let $A,B \in \cP(X)$, $x \in X$ and $\cF \in \Filters{X}$. Then $\ldiv{x}{A} \cap\ldiv{x}{B} = \ldiv{x}{(A\cap B)}$ and $\rdiv{\cF}{A} \cap \rdiv{\cF}{B} = \rdiv{\cF}{(A\cap B)}$. Analogously, $\rdiv{x}{A} \cap\rdiv{x}{B} = \rdiv{x}{A\cap B}$ and $\ldiv{\cF}{A} \cap \ldiv{\cF}{A} = \ldiv{\cF}{A\cap B}$.
\end{observation}
\begin{proof}
	The conditions $y \in \ldiv{x}{A} \cap\ldiv{x}{B} $ and $y \in \ldiv{x}{A\cap B}$ are both equivalent to $xy \in A \cap B$. Likewise, the condition $y \in \rdiv{\cF}{A} \cap \rdiv{\cF}{B} $ is equivalent to $\ldiv{y}{A } \in \cF$ and $\ldiv{y}{B } \in \cF$. This in turn is equivalent to $\ldiv{y}{A } \cap \ldiv{y}{B } = \ldiv{y}{A \cap B} \in \cF$.

	The remaining part of the claim follows by exactly symmetric reasoning.
\end{proof}

\begin{observation}
	If $A \in \cP(X)$, $x \in X$ and $\cU \in \Ultrafilters{X}$, then $(\ldiv{x}{A})^c = \ldiv{x}{(A^c)}$ and $(\rdiv{\cU}{A})^c = \rdiv{\cU}{(A^c)}$. Analogously, $(\ldiv{\cU}{A})^c = \ldiv{\cU}{(A^c)}$.
\end{observation}
\begin{proof}
	The condition $y \in (\ldiv{x}{A})^c$ is equivalent to $xy \not \in A$, which is equivalent to $y \in \ldiv{x}{(A^c)}$. Likewise, the condition $y \in (\rdiv{\cU}{A})^c$ is equvalent to $\ldiv{y}{A} \not \in \cU$. Because $\cU$ is an ultrafilter, this says that  $(\ldiv{y}{A})^c = \ldiv{y}{(A^c)} \in \cU$, which is equivalent to $y \in \rdiv{\cU}{(A^c)}$. 

The remaining part of the claim follows by exactly symmetric reasoning.
\end{proof}

We are now ready to define the semigroup structure on $\Ultrafilters{X}$. Hopefully, the definition appears natural, at least on the formal level. It is also noticable that we could have formulated the definition differently, applying the semigroup operation on the reverse side. This choice is far from inconsequential, as shall be seen when we discuss the relation with topology.

\begin{definition}[Semigroup structure of $\Ultrafilters{ \Group}$]\label{A:def:beta-X-as-semigroup}
\index{semigroup}
\index{ultrafilter!semigroup structure}  For filters $\cF, \cG$ we define $\cF \cdot \cG$ to be the family of those sets $A \in \cP(\Group)$ for which the set $\rdiv{\cG}{A}$ belongs to $\cF$:
  $$ \cF \cdot \cG = \{ A \in \cP(X) \setsep \rdiv{\cG}{A} \in \cF \}.$$
\end{definition}

As always, we see how the definition applies in case of principal ultrafilters. 

\begin{example}\index{ultrafilter!principal}
  For principal ultrafilters we have $\principal{x} \cdot \principal{y} = \principal{x \cdot y}$. This follows by expanding the definitions:
  \begin{align*}
'    \principal{x} \cdot \principal{y} 
  &= \left\{ A \in \cP(X) \setsep \{ z \in X \setsep \{w \in X \setsep z \cdot w \in A \} \in \principal{y} \} \in \principal{x} \right\} \\
  &= \left\{ A \in \cP(X) \setsep x \in \{ z \in X \setsep y \in \{w \in X \setsep z \cdot w \in A \} \} \right\} \\
  &= \left\{ A \in \cP(X) \setsep x \cdot y \in A \right\} =  \principal{x \cdot y} 
  \end{align*}
\end{example}

So far, we have defined the operation $(\cF,\cG) \mapsto \cF \cdot \cG$	only as a map $\Filters{X} \times \Filters{X} \to \cP(\cP(X))$. Before we make $ \Filters{X}$ and $ \Ultrafilters{X}$ into semigroups, we need to check that the constructed operation satisfies a number of additional conditions. We begin by verifying that necessary closure properties. Afterwards, we check associativity. 

\begin{proposition}
	If $\cF,\cG \in \Filters{X}$, then $\cF \cdot \cG \in \Filters{X}$. Moreover, if $\cU, \cV \in \Ultrafilters{X}$ then $\cU \cdot \cV \in \Ultrafilters{X}$.
\end{proposition}
\begin{proof}
	We need to check a number of defining properties of filters.

	We clearly have $\emptyset \not\in \cF \cdot \cG$ and $X \in \cF \cdot \cG$. Moreover, if $A \in  \cF \cdot \cG$ and $B \supset A$ then $\rdiv{\cG}{B} \supset \rdiv{\cG}{A}$ and consequently $B \in \cF \cdot \cG$. Finally, if $A,B \in \cG$ then $\rdiv{\cG}{A \cap B} = \rdiv{\cG}{A} \cap \rdiv{\cG}{B} \in \cF$ and hence $A \cap B \in \cF \cdot \cG$.

	For the additional part, consider we already know that $\cU \cdot \cV \in \Filters{X}$, so it remains to check that ultrafilter property. Let $A \in \cP(X)$. Then either $\rdiv{\cV}{A} \in \cU$ or $\rdiv{\cV}{(A^c)} = (\rdiv{\cV}{A})^c \in \cU$, hence either $A \in \cU \cdot \cV$ or $A^c \in \cU \cdot \cV$, which finishes the proof.
\end{proof}

\begin{proposition}
	If $\cF,\cG,\cH \in \Filters{X}$, then $(\cF \cdot \cG) \cdot \cH = \cF \cdot (\cG \cdot \cH)$. 
\end{proposition}
\begin{proof}
	Let $A \in \cP(X)$. We show that $A \in (\cF \cdot \cG) \cdot \cH$ if and only if $A \in \cF \cdot (\cG \cdot \cH)$, using the following transformations.
\begin{align*}
	A \in (\cF \cdot \cG) \cdot \cH
	&\iff \{ x \in X \setsep \ldiv{x}{A} \in \cH \} \in \cF \cdot \cG \\
	&\iff \{ y \in X \setsep \ldiv{y}{\{ x \in X \setsep {\ldiv{x}{A}} \in \cH \}} \in \cG \} \in \cF \\
	&\iff \{ y \in X \setsep \{ x \in X \setsep \ldiv{yx}{A} \in \cH \} \in \cG \} \in \cF \\
	&\iff \{ y \in X \setsep \{ x \in X \setsep \ldiv{x}{\ldiv{y}{A}} \in \cH \} \in \cG \} \in \cF \\
	&\iff \{ y \in X \setsep \rdiv{\cH}{\ldiv{y}{A}} \in \cG \} \in \cF \\
	&\iff \{ y \in X \setsep \ldiv{y}{A} \in \cG \cdot \cH \} \in \cF \\
	&\iff A \in \cF \cdot (\cG \cdot \cH )
\end{align*}
\end{proof}

\begin{corollary}
\index{semigroup}
\index{ultrafilter!semigroup structure}
	The sets $\Filters{\Group}$ and $\Ultrafilters{\Group}$ with the action defined by \ref{A:def:beta-X-as-semigroup} are semigroups.
\end{corollary}

Having verified the semigroup structure of $\Ultrafilters{\Group}$, we proceed to describe the semigroup operation in more detail. Our main objective here is to find the connection between algebraic and topological structure.

\begin{lemma}[Semigroup structure of $\Ultrafilters{\Group}$ --- alternative description]\label{prop:beta-X-as-semigroup-II}
\index{ultrafilter!semigroup structure}
  For ultrafilters $\cU, \cV$ the set $\cU \cdot \cV$ coincides with the ultrafilter $\llim{\cU}{x} \llim{\cV}{y}\ i(x\cdot y) $.
\end{lemma}
\begin{proof}
	Let us take any $C \in \cU \cdot \cV$. By the definition, we have $\rdiv{\cV}{C} \in \cU$. For any $x \in \rdiv{\cV}{C}$, which we fix for the time being, we have $\ldiv{x}{C} \in \cV$. For $y \in \ldiv{x}{C}$ we have $x \cdot y \in C$, and hence $i(x\cdot y) \in \bar{C}$. Because $\ldiv{x}{C} \in \cV$ and $\bar{C}$ is closed, we have $\llim{\cV}{y}\ i(x\cdot y) \in \bar{C}$. Likewise, because $\rdiv{\cV}{C} \in \cU$ and the choice of $x$ was arbitrary, another limit transition yields: $\llim{\cU}{x} \llim{\cV}{y}\ i(x\cdot y) \in \bar{C}$. Because $C$ was arbitrary and the space $\Ultrafilters{X}$ is Hausdorff, we have $\cU \cdot \cV = \llim{\cU}{x} \llim{\cV}{y}\ i(x\cdot y) $
\end{proof}

\begin{proposition}
\index{limit!generalised limit}
  For any function $f:\ X \to Z$ into a compact Hausdorff space, $\llim{\cU}{x} \llim{\cV}{y} f(x\cdot y) = \llim{(\cU \cdot \cV)}{z} f(z)$.
\end{proposition}
\begin{proof}
  Using Proposition \ref{prop:beta-f-properties}, we can perform the following transformations (we denote the map $X \ni y \mapsto x \cdot y \in X$ by $\mu_x$):
  \begin{align*}
    \llim{\cU}{x} \llim{\cV}{y} f(x\cdot y) &= 
    \llim{\cU}{x} \llim{\cV}{y} (f \circ \mu_x)(y)  
     = \llim{\cU}{x}  \beta f \circ \beta(i \circ \mu_x)(\cV) 
    \\ &= \beta f ( \llim{\cU}{x}  \beta(i \circ \mu_x)(\cV) )
     = \beta f (  \llim{\cU}{x}  \llim{\cV}{y} i(x\cdot y) )
	 \\ & = \beta f (\cU \cdot \cV)	
     = \llim{(\cU \cdot \cV)}{z} f(z)
  \end{align*}
\end{proof}

\begin{remark}
  Even if $X$ is a commutative group, $\Ultrafilters{ X}$ is not in general neither commutative, nor is it a group. In fact, even in $\Ultrafilters{\ZZ}$, all elements except for the principal ultrafilter are non-invertible and do not commute with the remainder of $\Ultrafilters{\ZZ}$. Non-invertibility is straightforward to prove: it suffices to notive that if $\cU$ is non-principal, then so is $\cU \cdot \cV$ for any $\cV$. One needs a considerable amount of work to prove non-commutativity, so we refrain from further discussion on this purely negative result.
\end{remark}

\begin{corollary}
\index{ultrafilter!semigroup structure}
  The map $\mu: \Ultrafilters{ \Group} \times \Ultrafilters{ \Group} \to \Ultrafilters{ \Group}$ given by $(\cU,\cV) \to \cU \cdot \cV$ is continuous in the left argument.
\end{corollary}
\begin{proof}
  It suffices to show that if $\cU \cdot \cV \in \bar{A}$ for some $A \in \cP(\Group)$, then $\cU' \cdot \cV \in \bar{A}$ for $\cU'$ in some open neighbourhood of $\cU$. According to definition \ref{A:def:beta-X-as-semigroup} that the condition $\cU' \cdot \cV \in \bar{A}$ is equivalent to: $ \rdiv{\cV}{A} \in \cU' $. But this is equivalent to $\cU' \in \bar{\rdiv{\cV}{A} }$, so we have just exhibited the desired open neighbourhood and we are done.
\end{proof}

\begin{remark}
	In general, the map $\mu$ in the above corollary is not continuous in the right argument. In fact, if it was, a simple argument would prove commutativity of $\Ultrafilters{X}$. 

	If we had chosen the definition of the semigroup structure differently, the map $\mu$ would have been continuous in the right argument. If $X^{\operatorname{opp}}$ denotes the semigroup with action $x \cdot^{\operatorname{opp}} y := y \cdot x$, then $\Ultrafilters{X^{\operatorname{opp}}}^{\operatorname{opp}}$ is a compactification of $X$ with multiplication continuous in the right argument instead of the left argument.
\end{remark}

Because $\Ultrafilters{X}$ will now have both topological and algebraic structure, we introduce the relevant definition.

\begin{definition}[Topological semigroup]
\index{semigroup!topological semigroup}
  Let $\Group$ be a set, $(\Group,\cdot)$ a semigroup and $(\Group,\cT)$ a topological space. Then the triple $(\Group,\cdot,\cT)$ is said to be a topological semigroup (resp. right/left topological semigroup) if and only if the multiplication map $\mu:\ \Group \times \Group \ni (g,h) \mapsto g \cdot h$ is continuous (resp. continuous in the right/left argument).
\end{definition}

\begin{corollary}
\index{ultrafilter!semigroup structure}
  The space $\Ultrafilters{\Group}$ is a compact Hausdorff left-topological semigroup.
\end{corollary}
\begin{proof}
  This follows directly from combining the results obtained previously.
\end{proof}

Having established that $\Ultrafilters{X}$ is a compact left-topological semigroup, we now turn to study compact  left-topological semigroups in more generality. Although our key object of interest will be \v{C}ech-Stone compactifications of discrete groups, we keep the discussion general because we need to apply our results to slightly more involved semigroups, such as $(\Ultrafilters{ \NN})^k$. Until the end of this section, $S$ stands for a  compact left-topological semigroup, except some initial definitions.

We note in the passing that all proved statements have their analogues in right-topological semigroups. In fact, if $S$ is a right-topological semigroup, then one can form the semigroup $S_{\operatorname{opp}}$ on the same set with the same topology by declaring $x \cdot_{\operatorname{opp}} y = y \cdot x$. A moment's thought will convince the reader that $S_{\operatorname{opp}}$ is then a left-topological semigroup, to which our results apply.

\renewcommand{\Group}{S}

	Study of ideals, especially the minimal ones, turns out to be essential for understanding the structure of general semigroup. We note that the concept becomes trivial in groups, where the only ideals are the trivial ones --- much as ring ideals make little sense in fields.

\begin{definition}[Ideal]
\index{semigroup!ideal}
	  Let $\Group$ be a semigroup. Then a non-empty set $I \subset \Group$ is defined to be a \emph{left} (resp. \emph{right}) \emph{ideal}, if and only if $\Group \cdot I \subset I$ (resp. $I \cdot \Group \subset I$)\footnote{The operation is taken elementwise, so $A \cdot B = \{a\cdot b \setsep a \in A,\ b \in B\}$ for $A,B \subset S$}. If $I$ is both left and right ideal, we refer to it as a two-sided ideal. By \emph{principal left ideal} (resp. \emph{principal right ideal}) we mean the ideal $\Group \cdot x$ (resp. $x \cdot \Group$). The ideal is said to be minimal, \defiff\ there is no ideal properly contained in it.
\end{definition}

	We will mostly pay attention to left ideals, because they are well-behaved from the topological point of view, as shown in the following lemma. Note that we would not be able to prove the analogous statement for right ideals.
 
\begin{lemma}\label{lem:principal=>closed}
  Let $\Group$ be a compact left-topological semigroup. If $L = \Group \cdot x$ is a principal left ideal in $\Group$, then $L$ is closed. 
\end{lemma}
\begin{proof}
  Note that $L$ is the image of a compact space $\Group$ by the continuous map $\mu(\cdot,x)$. Thus, $L$ is compact, as the image of a compact space by a continuous map. Since $\Group$ is assumed to be Hausdorff, $L$ is hence closed.
\end{proof}

	The following lemma is useful for finding left ideals contained in chains of left ideals. The result does not depend on topology. The analogous statement is true for right and two-sided ideals, but we won't need those results.

\begin{lemma}\label{A:lem:ideals-intersection}
  If $\cL$ is any family of left ideals in a semigroup $\Group$, then $\bigcap \cL$ is either the empty set or a left ideal.
\end{lemma}
\begin{proof}
  It follows from direct transformations that:
  $$ \Group \cdot \bigcap_{L \in \cL} L \subset \bigcap_{L \in \cL} \Group \cdot L \subset \bigcap_{L \in \cL} L $$
  Thus, the set $\bigcap \cL$ is closed under multiplication on the right. Hence, as long as it is non-empty, it is a left ideal.
\end{proof}

We are now able to characterise minimal left ideals. Again, the result is independent of topology, and holds also for right ideals.  

\begin{proposition}\label{A:prop:min-ideal-characterisation}
  If $L$ is a left ideal in a semigroup $\Group$, then $L$ is minimal if and only if for any $x \in L$ we have $L = G \cdot x$.
\end{proposition}
\begin{proof}
  If $L$ is minimal, then it does not properly contain any left ideal (principal or otherwise), so we only need to prove the other implication. Let us thus take $L$ as described, and for an arbitrary $x \in L$ consider the left principal ideal $L' := \Group \cdot x$. Then $L' \subset L$, by minimality the assumption on $L$ we have $L' = L$. Thus, $L = \Group \cdot x$ for any $x \in L$, and all remaining claims follow readily.
\end{proof}

	We are finally able to prove an existence statement about minimal left ideals. The result depends heavily on the topology, even though the formulation contains no topological notions. The reader will easily convince himself that simple non-compact semigroups, such as $\NN^k$ or $\cP_{\operatorname{fin}}(\NN)$ contain no minimal ideals.

\begin{proposition}\label{A:prop:min-ideal-existence}
  If $L$ is a left ideal in a compact left-topological semigroup $\Group$, then $L$ contains a minimal left ideal.
\end{proposition}
\begin{proof}
  Let $\cL$ be the set of closed left ideals in $\Group$ that are contained in $L$. Then $\cL$ is non-empty, because it contains all principal ideals corresponding to elements of $L$. Let us consider the natural order induced by the inclusion on $\cL$.

  We claim that each chain $\cC \subset \cL$ has a lower bound $M$. In fact, we can just take $M := \bigcap \cC$. With such definition, it is  immediately clear that $M \subset C$ for any $C \in \cC$. What is more, $M$ is the intersection of a descending family of non-empty compact sets, hence it is non-empty and compact. Thus, by Lemma \ref{A:lem:ideals-intersection}, $M$ is a left ideal, and hence an element of $\cL$. 

  It now follows from Kuratowski-Zorn Lemma that $\cC$ contains a minimal element, say $L'$. By construction, $L'$ contains no proper closed left ideal, so by Lemma \ref{A:prop:min-ideal-characterisation}, $L'$ is a minimal left ideal. By definition, $L' \subset L$. Thus, $L'$ is the sought minimal ideal.
\end{proof}

	We shall now introduce the notion of idempotence, which is useful in study of semigroups. 

\begin{definition}[Idempotent]
\index{semigroup!idempotent}
	Let $\Group$ be a semigroup. An element $x \in \Group$ is said to be \emph{idempotent} if and only if $x\cdot x = x$.
\end{definition}

\begin{example}
	If $G$ is a group, on at least a cancellative monoid, the the only idempotent is the unit, $e$. Indeed, from $x \cdot x = x \cdot e$ it follows after cancelling $x$ that $x = e$.

	If $\cE$ is a Banach space, then the idempotent elements in the semigroup of bounded operators $\cB(\cE)$ are the projections onto closed subspaces of $\cE$.  
\end{example}

It is not clear at all that idempotents should exists. For example, $(\mathbb{N},+)$ contains no idempotents, and the only idempotent in $(\mathbb{N}, \cdot)$ is $1$. Generally, if $S$ is a cancellative monoid and $I$ is a proper ideal, then $I$ contains no idempotent. Hence, the following theorem, due to Ellis \cite{Ellis} might come as a welcome surprise.

\begin{theorem}[Existence of idempotents, Ellis] \label{thm:idempotents-exist-in-general} \label{A:thm:idempotents-exist-in-general}
\index{semigroup!idempotent}
  Let $\Group$ be a compact left-topological semigroup. Then $\Group$ contains an idempotent element.
\end{theorem}

    \newcommand{\GroupII}{T}

\begin{proof}
  We divide the proof in two steps.

  \begin{step}
    Among the compact sub-semigroups of $\Group$, there exists a minimal one $\Group'$.
  \end{step}
  \begin{proof}
    Consider the family of $\mathcal\Group$ of all compact sub-semigroups of $\Group$, including $\Group$ itself. To show that there is the minimal in $\mathcal\Group$ with respect to the order induced by inclusion, we use Kuratowski-Zorn Lemma. We need to verify that any chain $\cC \subset \mathcal\Group$ has a lower bound $M$. We just take $M := \bigcap \cC$, which is a subset of any $\GroupII \in \cC$ by definition. Since $M$ is the intersection of a descending family of compact non-empty sets, it is non-empty and compact. It is closed under the semigroup operation, since all $\GroupII \in \cC$ are such, and this is an inductive condition. Thus, $M$ belongs to $\mathcal\Group$, as desired. Now, Kuratowski-Zorn Lemma ensures the existence of the announced minimal element.
  \end{proof}

  \begin{step}
		Any minimal compact sub-semigroup $\Group'$ contains exactly one element.
  \end{step}
  \begin{proof}
    Let $L \subset \Group'$ be a minimal left ideal. We know that for any $x \in L$, we can write $L$ in the form $L = \Group' \cdot x$. Consider the set $\GroupII := \{ y \in \Group' \setsep x = y \cdot x \}$. Note that since $x \in L = \Group' \cdot x$, we have $x = y \cdot x$ for some $y \in \Group'$, and hence $\GroupII$ is non-empty. Since the map $y \mapsto y \cdot x$ is continuous, and $\GroupII$ is the preimage of $\{x\}$, $\GroupII$ is compact. Moreover, if $y,y' \in \GroupII$ then $y'yx = y'x = x$, so $y,y' \in \GroupII$. Thus, $\GroupII$ is a compact sub-semigroup of $\Group'$. 

	We have assumed that $\Group'$ is a minimal compact sub-semigroup of $\Group$. Hence, the above considerations show that $\GroupII = \Group'$. In particular, we have $x \in \GroupII$, which means precisely that $x \cdot x = x$. Thus, $\{x\}$ is a compact sub-semigroup of $\Group$. Using minimality again, we conclude that $\Group' = \{x\}$ consists of precisely one idempotent element.
  \end{proof}

We have established that $\Group$ has a one-element sub-semigroup $\{x\}$. In particular,	we have $x\cdot x = x$, and consequently $x$ is the sought idempotent.

\end{proof}

\begin{corollary}
\index{semigroup!idempotent}
  If $\GroupII \subset \Group$ is a closed sub-semigroup of $\Group$, then $\GroupII$ contains an idempotent.
\end{corollary}
\begin{proof}
  It suffices to note that $\GroupII$ is a compact Hausdorff left topological semigroup in its own right, and apply Lemma \ref{lem:hindman-preparation}. Of course, the property of being idempotent is independent of the semigroup in which we consider the element.
\end{proof}

\begin{corollary}\label{cor:idempotetnes-exist-in-beta-X}
  There exists idempotent ultrafilters in $\Ultrafilters{X}$ for any discrete semigroup $X$. Moreover, if $\GroupII \subset \Ultrafilters{X}$ is a closed sub-semigroup, then $\GroupII$ contains and idempotent.
\end{corollary}
\begin{proof}
  We know that $ \Ultrafilters{X}$ is a compact Hausdorff topological semigroup. Thus, the above Theorem \ref{thm:idempotents-exist-in-general} applies.
\end{proof}

The next object of our study are the two-sided ideals. More precisely, we  prove the existence of a unique two-sided ideal. Given that there generally exists a multitude of minimal left ideals and we were not able to guarantee existence of minimal right ideals at all, this may be a surprising fact. Moreover, this unique ideal has useful functional properties, as we will shortly see.

\begin{proposition}\label{A:prop:K-existence}
\index{semigroup!ideal!minimal ideal}
  There exists a unique minimal two-sided ideal in $\Group$.
\end{proposition}
\begin{proof}
	Let $\cL$ denote the family of all minimal left ideals if $\Group$, and define $K := \bigcup \cL$. We claim that $K$ is the sought ideal.

	We first show that if $I$ is a two sided ideal, then $K \subset I$. If $L \in \cL$ is a minimal ideal, then $I \cap L \neq \emptyset$, because $I \cdot L \subset I \cap L$. Because $L \cap I$ is non-empty, it is a left ideal. Because $L$ is minimal we have $L = L \cap I$, or simply $L \subset I$. Taking the union over all choices of $L$, we find that $K \subset I$.

	Because each $L \in \cL$ is itself a left ideal, $K$ is also a left ideal. It remains to see that $K$ is also a right ideal. We will in fact show more, namely that if $L \in \cL$ and $x \in \Group$, then we have $L \cdot x \in \cL$. Because $L \cdot x$ is clearly a left ideal, it remains to see that it is minimal. For this, let us take any element of $L \cdot x$, which is necessarily of the form $y \cdot x$ with $y \in L$, and notice that $ G \cdot y \cdot x = L \cdot x$ because $G \cdot y = L$.
\end{proof}

\begin{definition}\label{A:def:minimal}
\index{semigroup!ideal!minimal ideal}
    We denote be $\Kappa(\Group)$ the unique minimal two sided ideal of $\Group$. If $x \in \Kappa(\Group)$, then we refer to $x$ as \emph{minimal} element of $\Group$. Likewise, $x$ is called a minimal idempotent if $x \in \Kappa(\Group)$ and $x$ is idempotent. 
\end{definition}
\begin{remark}
  The use of the adjective \emph{minimal} in the above definition is customary, but we wish to note, following Hindman and Strauss \cite{hindman-strauss} that it would be more logical to refer to $\Kappa(\Group)$ as the \emph{smallest} two sided ideal. Indeed, the phrase \emph{minimal} suggests that other minimal ideals may exist. This being said, we accept the traditional notation.
\end{remark}

It will be useful to have a criterion for membership in $\Kappa(\Group)$. 

\begin{proposition}\label{A:prop:K-characterisation}
\index{semigroup!ideal!minimal ideal}
  Let $\Group$ be a compact left-topological semigroup, and fix some $x \in \Group$. Then the following conditions are equivalent:
\begin{enumerate}
\item\label{A:prop:K-char:cond:a} $ x \in \Kappa(\Group)$,
\item\label{A:prop:K-char:cond:b} $ x \in L$ for some minimal left ideal $L$,
\item\label{A:prop:K-char:cond:c} for any $y \in \Group$ there exists $z \in \Group$ such that $z \cdot y \cdot x = x$.
\end{enumerate}
\end{proposition}
\begin{proof}
	The equivalence of \ref{A:prop:K-char:cond:a} and \ref{A:prop:K-char:cond:b} follows directly from the proof of Proposition \ref{A:prop:K-existence}. 

	Suppose that \ref{A:prop:K-char:cond:b} holds, and let $y \in \Group$ as in \ref{A:prop:K-char:cond:c}. We have $y \cdot x = \Group \cdot x = L$. By the characterisation of minimal ideals, we have $L = \Group \cdot y \cdot x$, so $x \in L$ implies that there exists $z \in \Group$ such that $x = z \cdot y \cdot x$. Hence, \ref{A:prop:K-char:cond:c} holds.

	Finally, suppose that \ref{A:prop:K-char:cond:c} is satisfied, and consider the left ideal $L = \Group \cdot x$. Let us consider an arbitrary $x' \in L$, which is necessarily of the form $x' = y \cdot x$, and let $z$ be such that $z \cdot x' = z \cdot y \cdot x = x$. It follows that:
	$$ \Group \cdot x' \supset \Group \cdot z \cdot  x' = \Group \cdot x = L$$
	Because the other inclusion is clear, we have $L = \Group \cdot x'$, which implies minimality of $L$ by characterisation in Proposition \ref{A:prop:min-ideal-characterisation}.
\end{proof}

Because these are the idempotent elements of $\Kappa(\Group)$ that are of most importance, we derive another criterion for minimality of idempotents. We begin by introducing a partial order on the set of idempotents.

\begin{definition}\label{A:def:idempotents-order}
  Let $p,q \in \Group$ be idempotent elements of a compact left-topological semigroup $\Group$. Then we say that $p \leq q$ if and only if $pq = qp = p$. 
\end{definition}
\index{semigroup!idempotent}
\index{order}
\begin{lemma}
  The relation $\leq$ defined in \ref{A:def:idempotents-order} is a partial order.
\end{lemma}
\begin{proof}
  If is clear that the relation is reflexive: $p \leq p$ because $pp = p$. It is also clear that if $p \leq q$ and $q \leq p$ then $p = pq = q$, so the relation is weakly anti-symmetric. Finally, if $p \leq q$ and $q \leq r$ then we have:
  $$ pr = pqr = pq = p, \qquad rp = rqp = qp = p $$
  so $p \leq r$, proving transitivity.
\end{proof}

We are now in position to characterise minimal idempotents as the idempotents minimal with respect to the introduced order.
\begin{proposition}\label{A:prop:minimal-equivalent-def}
\index{semigroup!ideal!minimal ideal}
\index{order}
  Let $\Group$ be a compact left-topological semigroup, and let $p \in \Group$ be idempotent element. Then the following conditions are equivalent:
  \begin{enumerate}
   \item\label{A:prop:minimal-equivalent-def:cond:1} The idempotent $p$ is minimal in the sense of Definition \ref{A:def:minimal}.
   \item\label{A:prop:minimal-equivalent-def:cond:2} The idempotent $p$ is minimal with respect to the order in Definition \ref{A:def:idempotents-order}.
  \end{enumerate}
\end{proposition}
\begin{proof}
  \begin{description}
   \item[ \ref{A:prop:minimal-equivalent-def:cond:1} $ \imply $ \ref{A:prop:minimal-equivalent-def:cond:2}] Suppose that $p \in \Kappa(G)$ and that $q$ is an idempotent with $q \leq p$; we need to check that $q = p$. Let $L = S\cdot p$ be the left ideal generated by $p$. We see that $q = qp \in L$. It follows from Proposition \ref{A:prop:K-characterisation} that we have for some $r \in S$ the relation: $p = rqp$. Consequently:
  $$ p = rqp = rqpq = pq = q,$$
  which finishes the proof that $p = q$.
   \item[ \ref{A:prop:minimal-equivalent-def:cond:2} $ \imply $ \ref{A:prop:minimal-equivalent-def:cond:1}] Suppose that $p$ is minimal with respect to $\leq$, and consider the left ideal $L := S \cdot p$; we need to show that $L$ is minimal. Let $M \subseteq L$ be a minimal left ideal; we know that $M = S \cdot q$ for some idempotent $q$. Because $q \in S \cdot p$, we have $q p = q$. Let us consider the $r := pq = pqp$. It is clear taht $r$ is idempotent:
  $$ rr = (pqp)(pqp) = p(qp)(qp)p = pqqp = pqp = r.$$
  Moreover, we have $pr = rp = r$, so directly by the definition we have $r \leq p$. Because of minimality, we have $r = p$. Consequently, $p = pq \in M$, and by a previously shown characterisation we have $M = L$. 
  \end{description}
\end{proof}

\begin{corollary}
\index{semigroup!ideal!minimal ideal}
\index{order}
 Let $p$ be an idempotent in a a compact left-topological semigroup $\Group$. Then, there exists a minimal idempotent $q$ with $q \leq p$.
\end{corollary}

We now prove some useful properties of the minimal ideal $\Kappa(\Group)$, depending on $\Group$. We begin with a simple result on Cartesian products, and then consider a slightly more involved result on sub-semigroups. These facts will have unexpected combinatorial applications.

\begin{proposition}\label{A:prop:K-product}
\index{semigroup!ideal!minimal ideal}
\index{semigroup!product}
  If $S,T$ are compact left-topological semigroups, then $\Kappa(S \times T) = \Kappa(S) \times \Kappa(T)$. Moreover, if $S_1,\dots,S_r$ are compact left-topological semigroups, then $\Kappa(\prod_i S_i) = \prod_i \Kappa(S_i)$. 
\end{proposition}
\begin{proof}
  It is clear that $\Kappa(S) \times \Kappa(T)$ is a two sided ideal in $S \times T$, so $\Kappa(S \times T) \subset \Kappa(S) \times \Kappa(T)$. Conversely, consider any $x \in \Kappa(S),\ y \in \Kappa(T)$, and let $(s,t) \in \Kappa(S \times T)$ be arbitrary. Because of minimality of $x,y$, there exist $s' \in S,\ t' \in T$ such that $s\cdot s'\cdot x=x$ and $t\cdot t'\cdot y = y$.  Hence, $(x,y) = (s,t) \cdot (s'\cdot x, t'\cdot x) \in \Kappa(S \times T)$.
  
  The additional claim about products of more than two semigroups follows by a simple induction.
\end{proof}

\begin{proposition}\label{A:prop:K-subsgrp}
\index{semigroup!ideal!minimal ideal}
\index{semigroup!subsemigroup}

  Let $T \subseteq S$ be compact left-topological semigroups, and suppose that $T \cap K(S) \neq \emptyset$. Then $\Kappa(T) = T \cap \Kappa(S)$.
\end{proposition}
\begin{proof}
  The inclusion $\Kappa(T) \subseteq T \cap \Kappa(S)$ follows from the simple observation that $T \cap \Kappa(S) \subset T$ is a two sided ideal. It remains to prove the inverse inclusion.

  Let $x \in T \cap \Kappa(S)$. We will show that $x \in \Kappa(T)$. By Lemma \ref{A:prop:min-ideal-existence}, the (principal) ideal $Tx$ contains a minimal ideal, which is of the form $Te$ for some idempotent $e$. Because $x \in \Kappa(S)$, the ideal $Sx$ is minimal, so in particular from $e \in Te \subseteq Tx \subseteq Sx$ it follows that $Sx = Se$, and hence there exist $s \in S$ such that $x = se$. It now remains to notice that $xe = see = se = x$, and hence $x \in Te$. Consequently, $x \in Te \subset \Kappa(T)$, as desired.
\end{proof}

In the case when $\Group$ is commutative, the theory is especially well-behaved. Although the most important compact semigroups for us are the highly non-commutative ones like $\Ultrafilters{\NN}$, it is interesting in its own right to investigate the behaviour of $\Kappa(\Group)$ in a commutative setting.

\begin{proposition}
\index{semigroup!ideal!minimal ideal}
\index{semigroup!commutative}
  Suppose that $\Group$ is a commutative compact left-topological semigroup. Then $\Kappa(\Group)$ is compact. Moreover, there exists a unique idempotent $e \in \Kappa(\Group)$, and $\Kappa(\Group)$ is a group with $e$ as the identity.
\end{proposition}
\begin{proof}
  Because $S$ is commutative, the notions of a two-sided ideal and left ideal coincide. Hence, $\Kappa(\Group)$ is a minimal left ideal, and hence it is compact. Because $\Kappa(\Group)$ is a compact semigroup, there exists an idempotent $e \in \Kappa(G)$. If $f \in \Kappa(G)$ was another idempotent, then we would have $e = xf$ and $f = ye$ for some $x,y \in \Group$. Hence, it would follow that:
    $$ e = xf = xff = ef = fe = yee = ye = f.$$
  Thus, $e$ is unique. Finally, if $x \in \Kappa(G)$, we have $\Kappa(G) = \Group \cdot x$, so there exists $y \in \Group$ with $yx = e$. Moreover, we have $eyx = ee = e$ and $ey \in \Kappa(G)$, so  $ey$ is the inverse of $x$. Because $\Kappa(G)$ was already a semigroup, existance of inverses implies that it is a group.
\end{proof}

It is useful to be able to construct compact left-topological sub-semigroups of a given compact left-topological semigroup $S$. If $T_0 \subset S$ is a (non-compact) sub-semigroup, one might be tempted to cojecture that $T := \cl T_0$ is the (smallest possible) compact sub-semigroup containing $T_0$. Unfortunately, upon closer insection there turns out to be no reason to believe this happens in the general situation. Because the semigroup operation is only continuous in the one argument, one cannot argue that $\cl T_0$ is closed under the semigroup operation by means of continuity. Nevertheless, under the additional assumption of commutativity, $\cl T_0$ turns out to indeed be a sub-semigroup. We begin by some relevant definitions and observations. 

\newcommand{\Centre}{\mathrm{Z}}
\begin{definition}
\index{semigroup!centre}
  If $S$ is a semigroup, then by $\Centre(S)$ we denote the \emph{centre} of $S$, given by:
  $$\Centre(S) := \{ x \in S \setsep (\forall y \in S):\ xy = yx \}.$$
\end{definition}

\begin{observation}
\index{semigroup!centre}
\index{semigroup!product}
  Is $S,T$ are semigroups, then $\Centre(S \times T) = \Centre(S) \times \Centre(T)$.
\end{observation}
\index{semigroup!subsemigroup}
\begin{proof}
	It is clear that $(x,x'), (y,y') \in S \times T$ commute if and only if $x,y$ commute and $x',y'$ commute. The claim follows directly by fixing $(x,x')$ and taking $(y,y')$ arbitrary.
\end{proof}

\begin{observation}
\index{semigroup!centre}
\index{ultrafilter!algebraic structure}

If $X$ is a commutative semigroup, then $X \subset \Centre(\Ultrafilters{ X})$. If $X = (\NN,\cdot)$ or $X = (\NN,+)$, then it can be shown that $ \Centre(\Ultrafilters{ X}) = X$, but we don't prove this result.
\end{observation}

We are now ready to give a condition for the closure of a sub-semigroup to be a (compact) subsemigroup.

\newcommand{\E}{T}
\newcommand{\I}{I}
\begin{proposition}\label{A:prop:algebra-closure-of-sgrp-is-sgrp}
\index{semigroup!centre}
\index{semigroup!subsemigroup}
  Suppose that $S$ is a compact left-topological semigroup, and $T_0 \subset \Centre(S)$ is a sub-semigroup contained in the centre of $S$, and let $\E := \cl{\E_0}$ denote the closure of $\E_0$. Then, $\E$ is a (compact) sub-semigroup of $S$. Moreover, if $I_0 \subset \E_0$ is a (by necessity, two-sided) ideal and $I := \cl{I}$, then $I$ is a two sided ideal in $\E$.
\end{proposition}
\begin{proof}
  Let $p,q \in \E$. We need to check that $p \cdot q \in \E$. Because of continuity, we have $p \cdot q = \lim_{x \to p} x \cdot q$. Because $p \in \cl{\E_0}$, it suffices to restrict to $x \in \E_0$ when taking the limit. Because $\E_0 \subset Z(S)$, we have for $x \in \E_0$: $x \cdot q = q \cdot x$. Using continuity again, we can write $q \cdot x = \lim_{y \to q} y \cdot x$. Again, taking this limit we may restrict to $y \in \E_0$. Because $\E_0$ is a semigroup, we have $ y  \cdot  x \in \E_0$. Passing to the limit, we conclude that $q  \cdot  x \in \E$. Passing to the limit again, we finally find $p  \cdot  q \in \E$, as desired.

  For the additional part, note that if either $p \in I$ or $q \in I$, then we might restrict to $x \in I_0$ or $y \in I_0$. In either case, we would have $x \cdot y \in I_0$, and after limit transitions we conclude that $p \cdot q \in I$. Hence, $I$ is an ideal (in $\E$).
\end{proof}

The above results can be generalised somewhat, using the notion of topological centre. As previously, we begin with the necessary definition. Next, we make some simple observations.

\newcommand{\CentreTop}{\mathrm{Z}_{\operatorname{top}}}
\begin{definition}
\index{semigroup!centre}
  If $S$ is a left-topological semigroup, then by $\CentreTop(S)$ we denote the \emph{topological centre} of $S$, which is defined as the set of those $q \in S$ for which the right multiplication by $q$, i.e. the map $S \ni q \mapsto p \cdot q \in S$, is continuous.
\end{definition}

\begin{observation}
\index{semigroup!centre}
\index{semigroup!product}
  Is $S,T$ are left-topological semigroups, then $\Centre(S \times T) = \Centre(S) \times \Centre(T)$.
\end{observation}
\begin{proof}
  The claim follows immediately from characterisation of continuity on product spaces. 
\end{proof}

\begin{observation}
\index{semigroup!subsemigroup}
\index{semigroup!centre}
  If $S$ is a  left-topological semigroup, then $\Centre(S) \subset \CentreTop(S)$.
\end{observation}
\begin{proof}
  It suffices to notice that for the elements of the centre, left multiplication and right multiplication coincite, and left multiplication is continuous by assumption.
\end{proof}

\begin{observation}
\index{ultrafilter!algebraic structure}
\index{semigroup!centre}
  If $X$ is a discrete semigroup, then $X \subset \CentreTop( \Ultrafilters{X})$.
\end{observation}
\begin{proof}
  This follows directly from how semigroup structure is defined of $\Ultrafilters{X}$.
\end{proof}

The following is refinement of Proposition \ref{A:prop:algebra-closure-of-sgrp-is-sgrp}, with centre replaced by topological centre. 

\begin{proposition}\label{A:prop:algebra-closure-of-sgrp-is-sgrp-II}
\index{semigroup!centre}
\index{semigroup!subsemigroup}
Suppose that $S$ is a compact left-topological semigroup, and $\E_0 \subset \CentreTop	(S)$ is a sub-semigroup contained in the centre of $S$, and let $\E := \cl{\E_0}$ denote the closure of $\E_0$. Then, $\E$ is a (compact) sub-semigroup of $S$. Moreover, if $I_0 \subset \E_0$ is a (by necessity, two-sided) ideal and $I := \cl{I}$, then $I$ is a two sided ideal in $\E$.
\end{proposition}
\begin{proof}
    Essentially the same as in Proposition \ref{A:prop:algebra-closure-of-sgrp-is-sgrp}.
\end{proof}

\renewcommand{\E}{}
\renewcommand{\I}{}

 \section{Finitely additive measures}
\index{measure!finitely additive measure}
  Yet another way to look at ultrafilters is as a special kind of measures. More precisely, we show that there is a natural way to identify an ultrafilter $\cU \in \Ultrafilters{X}$ with a $\{0,1\}$-valued finitely additive measure on $X$. Note that this is a rather specific kind of a probabilistic measure, where each set is assigned measure either $0$ or $1$; we take the apportunity to cite an amusing way of putting this, found at the $n$-Category Caf\'{e }:
\begin{quotation}
If probability indicates your degree of belief, an ultrafilter is a probability measure for fundamentalists.
\end{quotation}
This approach leads to new intuitions, and allows us to borrow some new language  from measure theory.

We begin by defining a measure corresponding to an ultrafilter, and vice versa. Afterward, we prove that this correspondence is bijective.

\begin{definition}
\index{measure!finitely additive measure}
 Let $\cU$ be an ultrafilter. Then we define the associated finitely additive, $\{0,1\}$-valued measure $\mu$ on $X$ by the formula:
  $$ \mu(A) = 
    \begin{cases}
     1 & \text{ if } A \in \cU \\
     0 & \text{ if } A \not \in \cU
    \end{cases}
  $$
 Conversely, if $\mu$ is a finitely additive, $\{0,1\}$-valued measure on $X$ then we define the associated ultrafilter $\cU$ by the formula:
  $$ \cU = \{ A \in \cP(X) \setsep \mu(A) = 1\}. $$
\end{definition}

\begin{proposition}
\index{measure!finitely additive measure}
 The above definition gives a bijective correspondence between ultrafilters and finitely additive, $\{0,1\}$-valued measures.
\end{proposition}
\begin{proof}
  Let $\cU$ be an ultrafilter, and $\mu$ the associated measure. We need first to verify that $\mu$ is indeed what we declared it to be. It is immediate from the definition that it is $\{0,1\}$-valued. Property \ref{def:filter-prop:1} translates into $\mu(\emptyset) = 0$ and $\mu(X) = 1$. Finite additivity is equivalent to the statement that $\mu(A \cup B) = \mu(A) + \mu(B)$ for any disjoint sets $A,B$. Since $\cU$ is closed under finite intersections, it cannot be the case that $\mu(A) = \mu(B) = 1$. If $\mu(A)$ and $\mu(B)$ are $0$ and $1$ (in either order), then $A \cup B \in \cU$, and hence $\mu(A \cup B) = 1$, and the equality agrees. Finally, if $A,B \not \in \cU$ then by property \ref{def:ultrafilter-prop:6} we have $\mu(A \cup B) = 0$, so the equality agrees again.

  Similarly, let $\mu$ be a finitely additive, $\{0,1\}$-valued measure, and $\cU$ be the additional ultrafilter. Again, we need to verify that $\cU$ is really an ultrafilter. We have $\mu(X) = 0$ and $\mu(\emptyset) = 0$, so $X \in \cU$ and $\emptyset \not \in \cU$, so property \ref{def:filter-prop:1} holds. If $A \subset B$ and $A \in \cU$, then we can find $C$ such that $B = A \cup C$ where the sum is disjoint, namely $C:= B \setminus A$. We find that $\mu(B) = \mu(A) + \mu(C) \geq 1$, so $B \in \cU$ --- and thus property \ref{def:filter-prop:1} is satisfied. If $A \in \cP(X)$, then $\mu(A) \cup \mu(A^c) = \mu(X) = 1$, hence exactly one of $A,A^c$ belongs to $\cC$, and property \ref{def:ultrafilter-prop:6} is satisfied. Finally, if $A,B \in \cU$, then $1 = \mu(A) = \mu(A\cap B) + \mu(A \cap B^c)$, thus $A \cap B \in \cU$ as soon as $A \cap B^c \not \in \cU$. However, this follows directly from what has already been shown, since $A \cap B^c \subset B^c \not \in \cU$.
  
  It remains to show that the described relations are mutually inverse. Thus, suppose that $\cU$ is an ultrafilter, $\mu$ is the associated measure, and $\cV$ is the ultrafilter associated to the measure $\mu$. We then have:
  $$ A \in \cV \iff \mu(A) = 1 \iff A \in \cU, $$
  and thus $\cU = \cV$. In the same spirit, if $\mu$ is a finitely additive, $\{0,1\}$-valued measure, $\cU$ is the associated ultrafilter, and $\nu$ is the measure associated to $\cU$, then
  $$ \mu(A) = 1 \iff A \in \cU \iff \nu(A) = 1, $$
  and since the measures take only two values, $\mu = \nu$ and  we are done.
\end{proof}

It is fairly easy to describe the correspondence explicitly in the one concrete example of an ultrafilter we have.

\begin{example}
\index{ultrafilter!principal ultrafilter}
 If $\principal{x}$ is the principal ultrafilter associated to a point $x \in X$ then the associated measure $\mu$ is the point measure $\delta_x$ centered at $x$.
\end{example}

\begin{remark}
\index{limit!generalised limit}
\index{measure!finitely additive measure}

  The generalised limit can be thought of as the natural generalisation of the integral. Because there does not seem to be a standard way of defining the integral for finitely additive measures, we don't make a rigorous statement out of this observation. Instead, we note that if $\cU \in \Ultrafilters{X}$ is an ultrafilter with associated measure $\mu$, and $f:\ X \to \hat{\RR}$ is a function with $\llim{\cU}{x} f(x) = \alpha$, then for any $\varepsilon > 0$, for $\mu$-almost all $x$ we have $\alpha - \varepsilon < f(x) < \alpha + \varepsilon$, and we encourage the reader work out the details.

  Similarly, extending the group operation to the ultrafilters can be though of as a generalisation of convolution of measures. Again, we leave the details to the reader. 
\end{remark}

  Thinking of ultrafilters in terms of measures suggests that it makes sense to speak of statements being true $\cU$-almost everywhere, where we identify the  ultrafilter $\cU$ with the corresponding measure. For completeness' sake, we define a statement (involving the variable $x$) to be true $\cU$-almost everywhere (abbreviated $\cU$-a.e.) with respect to $x$ if and only if the set of those $x$ for which the statement is true is $\cU$-big, or belongs to $\cU$; alternatively, we may say that the statement is true for $\cU$-almost all $x$ (abbreviated $\cU$-a.e. $x$).

  This is a well behaved notion, and has the additional useful property that either a statement is true $\cU$-a.e. or it's negation is true is true $\cU$-a.e. Conjunction of finitely many statements that are true $\cU$-a.a. is again true $\cU$-a.a., but this rule does not hold for countable conjunction (as is the case in ordinary measure theory) since $\cU$ corresponds to merely a finitely additive measure (as opposed to countably additive one). A more subtle caveat is that this type of quantification is sensitive to ordering: For a statement $\phi$ in $x$ and $y$, one could first make the statement that for $\cU$-a.a. $y$, $\phi$ is true, and then say that this statement holds true for $\cU$-a.a. $x$. This is a useful statement, however it is not equivalent to the similar statement has reversed order of quantification. As an extreme example, one may say that for $\cU$-a.a. $n$, for $\cU$-a.a. $m$ it holds that $m > n$, but of course it is not true when quantification is interchanged. This issue can be traced back to the failure of Fubini's theorem's analogue for finitely additive measures.

\chapter{Combinatorial applications.}
\label{C:chapter} 
\lhead{Chapter \ref{C:chapter}. \emph{Ramsey theory}} 

In this chapter study the interplay between ultrafilters and combinatorics. On one hand, application of ultrafilters allows us to prove interesting statements about combinatorial objects, such as arithmetic progressions or sets of finite sums, especially the ones related to partitions. Most of the results derived here were first proved by more classical means, but ultrafilter approach tends to produce proofs that are more succinct and to be more amenable to generalisations.

As a warm-up, we prove the Ramsey theorems for graphs and hypergraphs. Our goal there is not so much derivation of the result, which is classical and not too difficult, but showing how ultrafilters can be applied in combinatorics, and how they produce significantly shorter and more elegant proofs. In the subsequent sections, we investigate some notions of largeness and their relation to ultrafilters. We begin with $\IP$-sets, which are fairly natural from the combinatorial viewpoint, and are connected to idempotent ultrafilters. These can be used to derive Hindman theorem. Secondly, we treat $\C$-sets, which are more special than $\IP$-sets and are easiest to define with help of ultrafilters. As an elegant application of $\C$-sets, we prove van der Waerden's theorem and Hales-Jewett theorem.

Perhaps the most important motivation behind our present inquiry is that it provides a connection between ultrafilters and notions of largeness which are of independent interest. The most important results we prove in Chapter \ref{B:chapter} assert that a given set of ``return times'' belongs to an ultrafilter, assuming that some additional conditions are satisfied. Given that ultrafilers are not a part of ``mainstream'' mathematics, these results are not immediately appealing in the basic form. However, statements including $\IP$-sets and $\C$-sets are much more easily understood.

Most of the results presented in this chapter can be found in \cite{hindman-strauss}. The proof of Ramsey's theorem is inspired by \cite{arrows-by-galvin}. Extremely accessible treatment of many of the topics discussed here can be found in \cite{hindman-survey-1} and \cite{hindman-survey-2}. A good discussion of Hindman's Theorem for general semigroup can be found in \cite{golan-hindman}. For the original ultrafilter proof of van der Waerden, see \cite{vd-waerden-with-ultrafilters}. For a discussion of partition theorems like Hales-Jewett, see \cite{partitions-of-words}.  

\section{Ramsey theorem}

	We begin by discussing simple applications of ultrafilters to Ramsey theory for graphs. These applications are basic insofar as they only use existence of a single ultrafilter, with no reference to any algebraic or topological properties. Intuitively speaking, for these applications ultrafilters may be seen as a limit objects attached to an infinite number of restrictions to subsequences. Indeed, the classical proofs of the discussed results follow by repeated transitions to subsequences.

	To formulate the theorems, we introduce some notation. The reader familiar with graph theory will find the following definition standard.

\begin{definition}\index{graph}
  An (undirected, simple) graph $G = (V,E)$ consists of a set of vertices $V$ and a set of edges $E \subset \binom{V}{2} := \{ \{u,v\} \setsep u,v \in V, u \neq v \}$. The full graph with vertex set $V$ is the one with the maximal possible set of edges: $E = \binom{V}{2}$.

  An $r$-colouring of the graph $G$ is an arbitrary map $c:\ E \to [r]$, where we . If $e \in E$ is an edge and $c(e) = \gamma$, we say that $e$ has colour $\gamma$. If $u, v \in V$, we write $c(u,v)$ rather than $c(\{u,v\})$ for the colour of the edge $\{u,v\}$; this way $c$ can be identified with a symmetric function on a subset of $V^2$.

  A subset $U \subset V$ is said to be monochromatic if and only if there is a colour $\gamma$ such that $c(e) = \gamma$ for all edges $e \in \binom{U}{2} \cap E$.
\end{definition}

	Note that a colouring of a graph is essentially the same as finite partition of the set of egdes. The difference is purely notational: we find it more convenient to speak of monochromatic sets than of sets whose edges which lie entirely in a one cell of a given partition.

We are now able to formulate and prove the classical Ramsey theorem for graphs. It is a fundamental result which lies at the very basis of Ramsey theory. 

\begin{theorem}[Ramsey theorem, infinitary version]\label{C:thm:ramsey-inf}
  Let $G $ be the full graph with an infinite set of vertices $V$, and let $c: \binom{V}{2} \to [r]$ be a colouring of the edges of $V$ with $r$ colours. Then there exists an infinite monochromatic subgraph of $V$.  
\end{theorem} \index{Ramsey}
\begin{proof}
  Let $\cU$ be a non-principal ultrafilter on the set of vertices $V$.

  Define $c(u) := \llim{\cU}{v} c(u,v)$ and $c := \llim{\cU}{v} c(v)$. Note that $c(u)$ is well defined because $c(u,v)$ makes sense for $\cU$-a.a. $v$. This says that for any $u$, the edge $c(u,v)$ has colour $c(u)$ for $\cU$-a.a. $v$, and the colour $c(u)$ is $c$ for $\cU$-a.a. $u$. Additionally, let $A := \{ u \in V \setsep c(u) = c \}$ denote the set of those $u$ whose typical edge colour $c(u)$ takes the typical value $c$. Also, let $A(u) := \{ v \in V \setsep c(u,v) = c(u) \}$ denote for each $u$ the set of the endpoints $v$ of edges $(u,v)$ whose colour $c(u,v)$ takes the typical value $c(u)$. Note that $A \in \cU$ and for any $u$ also $A(u) \in \cU$.

  We will now construct by induction a sequance of vertices $\seq{v}{n}{\NN}$ so that $c(v_i,v_j) = c$ for $i < j$. For $n=0$, we take $v_0$ to be any element of $A$. For $n = 1$, we take $v_1$ to be any element of $A \cap A(v_0)$. Generally, if $v_0,v_1,\dots,v_{n-1}$ have been constructed, let $v_n$ be any element of the set $A \cap A(v_0) \cap A(v_2) \cap \dots \cap A(v_{n-1})$, which is non the empty set because it belongs to $\cU$. Since $v_i \in A$ for any $i \in \NN$, we have $c(v_i) = c$. Since $v_j \in A(v_i)$ for $j \in \NN,\ j> i$, we have $c(v_i,v_j) = c(v_i) = c$. Thus, all edges between $v_n$ indeed have the colour $c$, as claimed. 
\end{proof}

	The above theorem speaks of infinite graphs. We can formulate a finitary version of the above theorem, replacing infinite graphs by arbitrary large finite graphs. Arguably, the finite version of the theorem is more concrete. 

\begin{theorem}[Ramsey theorem, finitary version]\label{C:thm:ramsey-fin}
  Suppose that a number of colours $r$ and size $M$ are fixed. Then, there exists $N$ (dependent on $M$ and $r$) such that if $G$ is a full graph with at least $N$ vertices, whose edges are coloured in $r$ colours, then $G$ contains a monochromatic subgraph with at least $M$ vertices. 
\end{theorem} \index{Ramsey}
\begin{proof}
  Suppose that for some fixed $r,\ M$, the claim fails. Let $\Theory$ be the theory of $r$-coloured graphs (the language contains $r$ binary relations $R_1, R_2, \dots, R_r$ corresponding to the colouring of the edges, and the axioms are such that each edge has exactly one colour). The statement that the claim fails for the fixed values of $r$ and $M$ means that for any $N$ we can find a model for $\Theory$ which has at least $N$ elements and has no $M$-element monochromatic subgraph --- both conditions easily expressible as first order sentences. From \L o\'s Theorem it follows that $\Theory$ needs to have a model in which all these sentences are true: this means that the model has to be infinite, and have no $M$-element monochromatic subgraphs. But this contradicts the infinitary version of the theorem, since the infinite model needs to have an infinite monochromatic subgraph, let alone $M$-element.
\end{proof}

\begin{remark}
  The finitary version of Ramsey theorem can also be deduced by a more classical compactness argument. In this argument, we start with a countably infinite vertex set $V = \seq{v}{i}{\NN}$, and use the failure of the finitary version to construct for each $N$ an $r$-colouring of $V_N = \seq{v}{i}{[N]}$ with no monochromatic subgraph of size $M$. %By assumption, on any initial segment $\seq{v}{i}{\NN}$ there exists a colouring.
The difficulty lies in combining the many finite colourings into one. This can be done by multiple passing to subsequences, using compactness of $[r]^V$, or using ultrafilter limits.
\end{remark}

We can repeat very simialar considerations for hypergraphs, which are a natural generalisation of graphs. We begin by giving the relevant definitions.

\begin{definition}\index{graph!hypergraph} \index{graph!colouring}
  A $k$-hypergraph $G = (V,E)$ consists of a set of vertices $V$ and a set of edges $E \subset \binom{V }{ k} := \{ \{u_1,u_2,\dots,u_k\} \setsep u_i \in V,\ i \neq j \imply u_i \neq u_j \}$. The full $k$-hypergraph with vertex set $V$ is the one with the maximal possible set of edges: $E = \binom{V}{ k}$.

  An $r$-colouring of the $k$-hypergraph $G$ is an arbitrary map $c:\ E \to [r]$. If $e \in E$ is an edge and $c(e) = \gamma$, we say that $e$ has colour $\gamma$. If $u_1,u_2,\dots,u_k \in V$, we write $c(u_1,u_2,\dots,u_k)$ rather than $c(\{u_1,u_2,\dots,u_k\})$ for the colour of the edge $\{u_1,u_2,\dots,u_k\}$; this way $c$ can be identified with a symmetric function on a subset of $V^k$.

  A subset $U \subset V$ is said to be monochromatic if and only if $U$ there is a colour $\gamma$ such that $c(e) = \gamma$ for all edges $e \in \binom{U }{ k}$.
\end{definition}

The follwing theorem is the obvious generalisation of Theorem \ref{C:thm:ramsey-inf} to hypergraphs. The proof is almost a verbatim copy of the the proof for graphs.

\begin{theorem}[Ramsey theorem for hypergraphs, infinitary version]
\index{Ramsey}\index{graph!hypergraph}
  Let $G $ be the full $k$-hypergraph on with an infinite set of vertices $V$, and let $c:\ \binom{V }{ k} \to [r]$ be a colouring of the edges of $V$ with $r$ colours. Then there exists an infinite monochromatic subgraph of $V$.  
\end{theorem}
\begin{proof}
  Let $\cU$ be a non-principal ultrafilter on the set of vertices $V$.

  We begin by extending the colouring $c$ to all subsets of $V$ with cardinality not greater than $k$ by the following inductive procedure; thus extended $c$ will describe the colour that is typical for a given subset of vertices. For $k$-element subsets, the map $c\ \binom{V}{k} \to [r]$ is already given. Suppose that for some $l < k$, the map $c:\ \binom{V}{l+1} \to [r]$ has  been defined. Then, for an $l$-element set distinct verices $f = \{u_1,u_2,\dots,u_l\}$, the colour $c(f \cup \{v\})$ is defined for all $v \not \in f$, so in particular for $\cU$-a.a. $v$. It therefore makes sense to define the typical colour for $f$ as: $c(f) := \llim{\cU}{v}c(f \cup \{v\})$. We additionally define the set $A(f) \in \cU$ to be the set where the typical colour for $f$ is realised: $A(f) := \{ v \in V \setminus f \setsep c(f \cup \{v\}) = c(f) \}$; by definition $A(f) \in \cU$. This finishes the inductive step. In particular, we have defined $c(\emptyset)$, which we will denote simply by $c$.

  We will now construct by induction a sequance of vertices $\seq{v}{n}{\NN}$ so that for all $f \subset \set{v}{i}{\NN}$ with $\# f \leq k$ we have $c(f) = c$. This will finish the proof, because the condition implies that in particular for edges $e \subset \set{v}{i}{\NN}$ we have $c(e) = c$, and therefore $\set{v}{i}{\NN}$ is monochromatic. Suppose that for some $n \geq 0$, the vertices $v_0,v_1,\dots,v_{n-1}$ have already been constructed so that for $f \subset \set{v}{i}{[n]}$ with $\# f \leq k$ we have $c(f) = c$ (we allow $n=0$, where no previous vertices are constructed). Let us define 
  $$B := \bigcap \{ A(f) \setsep f \subset \set{v}{i}{[n]},\ \# f \leq k \}.$$
  Since all the sets $A(f)$ in the intersection are $\cU$-big, also $B$ is $\cU$-big, and in particular non-empty. We claim that for any choice of $v_n$ in $B$, the required conditions will hold for $v_n$. We need to check that for any $f' \subset \set{v}{i}{[n+1]}$ with $\# f' \leq k$ we have $c(f') = c$. If $v_n \not \in f'$, this is satisfied by the inductive assumption, so we may assume that $f' = f \cup \{v_n\}$ with $f \subset \set{v}{i}{[n]}$ and $\# f < k$. Since $v_n \in B \subset A(f)$, we have, by the definition of $A(f)$, the equality $c(f') = c(f \cup \{v_n\}) = c(f)$. By the inductive assumption, $c(f) = c$. Thus, we have $c(f') = c$, as claimed. This finishes the inductive step, and thus the proof.
  
\end{proof}

	Just as before, we can make the theorem more concrete by referring to finite objects. We again prefer model-theoretic methods of deriving the finitary version, but a more classical solution is to use compactness arguments.

\begin{theorem}[Ramsey theorem for hypergraphs, finitary version]
\index{Ramsey}\index{graph!hypergraph}
  For any number of colours $r$ and size $M$, there exists $N$ (dependent on $M$ and $r$) such that if $G$ is a full $k$-hypergraph with at least $N$ vertices, whose edges are coloured in $r$ colours, then $G$ contains a monochromatic subgraph with at least $M$ vertices. 
\end{theorem}
\begin{proof}
  Suppose that for some fixed $r,M$, the claim fails. Let $\Theory$ be the theory of $r$-coloured $k$-hypergraphs (the language contains $r$ relations $R_1, R_2, \dots, R_r$ with $k$ arguments corresponding to the colouring of edges, and the axioms are such that each edge has exactly one colour). The statement that the claim fails for the fixed values of $r,M$ means that for any $N$ we can find a model for $\Theory$ which has at least $N$ elements and has no $M$-element monochromatic subgraph --- both conditions easily expressible as first order sentences. From \L o\'s Theorem it follows that $\Theory$ needs to have a model in which all these sentences are true: this means that the model has to be infinite (at least $N$ vertices for every $N$), and have no $M$-element subgraphs. But this contradicts the infinitary version of the theorem, since the infinite model needs to have an infinite monochromatic subgraph, let alone $M$-element.
\end{proof}

\section{$\IP$-sets and idempotent ultrafilters}
\renewcommand{\Group}{X}
\index{IP@$\IP$!IP-set@$\IP$-set}
\index{ultrafilter!idempotent|(}
	In this section we introduce and study an important notion of combinatorial largeness: $\IP$-sets and $\IP^*$-sets. These make sense in an arbitrary commutative semigroup %\footnote{
	%If truthe be told, things work just as well in a non-commutative 	setting, except one has to be careful about the order of operations. Rewriting this section in a non-commutative situation is on the top of the to-do list.}
, but as usual the most interesting example is provided by the natural numbers. In fact, one can even work with non-commutative semigroups and obtain similar results, but we restrict to the commutative case. On one hand these concepts are fairly natural, and tied closely to even more natural syndeticity. On the other hand, they bear a close relation to ultrafilters. This makes them a convenient bridge between ultrafilters and concrete mathematics.
	Our main result in this section is the ultrafilter proof Hindman's theorem. The presented proof is due to Galvin and Glazer, and is commonly considered to be one of the most elegant results in combinatorial number theory. The first known proof is due to Hindman \cite{Hindman-original} and uses only elementary tools --- at the price of being rather lengthy and complicated. We follow the treatment by Hindman-Strauss \cite{hindman-strauss} and Bergelson \cite{Berg-survey-IP}.

	Throughout this section, $X$ stands for a commutative discrete semigroup. As usual, we begin with some definitions.

\begin{definition}[Finite sums and products]\index{FS}\index{FP}\index{finite sum|see{PS}}\index{finite product|see{FP}} 
  Let $\bx = \seq{x}{n}{\NN}$ be a sequence of elements of a commutative semigroup $\Group$. If $\Group$ is written multiplicatively, i.e. $\Group = (\Group,\cdot)$, then we define the family of finite products of $\bx$ to be:
  $$ \FP{ \bx } := \left\{ \prod_{i \in I} x_i \setsep I \subset \NN,\ 0 < \card{I} < \aleph_0 \right\}.$$
  Likewise, if $\Group = (\Group,+)$ is written additively, then we define the family of finite sums of $x$ to be:
  $$ \FS{ \bx } := \left\{ \sum_{i \in I} x_i\setsep I \subset \NN,\ 0 < \card{I} < \aleph_0 \right\}.$$
\end{definition}

\begin{definition}[$\IP$-sets]
\index{IP@$\IP$!IP-set@$\IP$-set}
\index{FS}
  A set $A \subset \Group$ is said to be an additive $\IP$-set by definition if and only if $A \supset \FS{\bx}$ for some sequence $\bx$. ($\IP$ stands for idempotent, \cite{Berg-survey-IP}, or infinite-dimensional parellopiped \cite{tao-blog-3}). Similarly, $A$ is said to be a multiplicative $\IP$-set by definition if and only if $A \supset \FP{\bx}$ for some sequence $\bx$.
\end{definition}

\begin{remark}
  In the setting of an arbitrary commutative semigroup, the difference between $\FS{\cdot}$ and $\FP{\cdot}$ is purely notational. In practice, these notions are normally applied in the context where both $+$ and $\cdot$ have fixed conventional meanings, such as in $\NN$.

  Some authors require $\IP$-sets to be \emph{precisely} of the form $\FS{\bx}$ (respectively $\FP{\bx}$) for some sequence $\bx$. However, following the rationalisation of Bergelson et al. \cite{}, we require a weaker condition of inclusion, since begin an IP-set should be a notion of ``largeness'', and thus should be preserved under taking supersets. On a more practical note, it makes the characterisation which we will prove shortly more elegant.
\end{remark}

\begin{example}
	Consider $X = \NN$, and take $x_n := 10^n$. Then $\FS{\bx}$ consists precisely of the positive integers whose digits are only $0$'s and $1$'s.

	We never assume that $x_n$ have to be distinct, nor that their sums need to be distinct. If $e \in X$ is idempotent, then taking $x_n := e$ shows that $\{e\}=\FS{\bx}$ is an $\IP$-set. In this case it also happens that the principal ultrafilter $\principal{e}$ is idempotent. 
\end{example}

The definition of $\IP$-sets we provided used only elementary properties of $X$, without ever mentioning ultrafilters. The purpose of the following two lemmas is to establish a connectio between $\IP$-sets and idempotent ultrafilters. First, we show how algebraic structure of an ultrafilter implies combinatorial richness of its members.

\begin{lemma}\label{lem:hindman-preparation}\label{C:lem:IP<=in-idempotent}
\index{Hindman}
\index{ultrafilter!idempotent}

\index{IP@$\IP$!IP-set@$\IP$-set}
   If $(\Group,+)$ is a commutative semigroup, $\cU \in \Ultrafilters{X}$ is an idempotent ultrafilter and $A \in \cU$, then $A$ is an $\IP$-set.
\end{lemma}
\begin{proof}
  We will inductively construct a sequence $\seq{x}{i}{\NN}$ for which $\FS{\bx} \in A$. In $n$-th step, the initial fragment $\seq{x}{i}{[n]}$ is assumed to be given, and we construct $x_{n}$. At each step of the construction, we require that two conditions should be satisfied.  Firstly, for any $\emptyset \neq I \subset [n]$, we require that $\sum_{i \in I} x_i \in A$. Secondly, let $A_n := \displaystyle\bigcap_{I \subset [n]} \left(A - \sum_{i \in I} x_i \right)$;
\footnote{We remind that $A - x = \{y \in X \setsep y+x \in A\}$. Here $A - \sum_{i \in \emptyset}\dots := A$.}
  we require that $A_n \in \cU$.

  Suppose that for some $n \in \NN$, the initial fragment $\seq{x}{i}{[n]}$ has already been constructed so that the requirements are satisfied. We allow $n=0$, which corresponds to no elements being constructed, and the requirement that $A_0 = A \in \cU$. We wish to construct $x_n$. By the inductive assumption $A_n \in \cU$. Since $\cU + \cU = \cU$, we can equivalently express this condition by saying that the set $B:= \{x \in \Group \setsep A_n - x \in \cU\}$ belongs to $\cU$. Since $\cU$ is closed under intersections, we also have $A_n \cap B \in \cU$. In particular, $A_n \cap B \neq \emptyset$, so we may select $x_n \in A_n \cap B$. We claim that any such $x_n$ satisfies the requirements.

  Note that by definition we have $A_{n+1} = (A_n-x_n) \cap A_n$. Thus, $A_{n+1} \in \cU$ follows from $A_n-x_n \in \cU$ and $A_n \in \cU$. The first of these conditions follows from $x_n \in B$, and the definition of $B$, while the second is the inductive assumption.
  
  Similarly, note that any index set $\emptyset \neq I \subset [n+1]$ is either contained in $[n]$, or can be written as $I = \{x_n\} \cup I'$ with $I' \subset[n]$. We need to show that $\sum_{i \in I} x_i \in A$. In the case $I \subset [n]$, this follows form the inductive assumption, so suppose that $I =  \{x_n\} \cup I'$ with $I' \subset[n]$. Then the requirement can be equivalently expressed as: $x_n \in A - \sum_{i \in I'} x_i$. This follows immediately from $x_n \in A_n$.

\end{proof}

\begin{remark}
	Note that the above proof really shows more than is stated in the theorem formulation. Namely, at each step, we choose $x_n$ as an arbitrary element of a given set which is $\cU$-large. This means, that we can make additional requirements of $x_n$. Most importantly, if $X = \NN$ we can require that $x_n$ be arbitrarily large (with respect to $x_1,x_2,\dots,x_{n-1}$).
\end{remark}

As the next step, we prove a statement converse to the above Lemma  \ref{lem:hindman-preparation}: we show that given combinatorially rich set, one can find an algebraically interesting ultrafilter to which it belongs.

\begin{lemma}\label{lem:hindman-converse}\label{C:lem:IP=>in-idempotent}
\index{Hindman}
\index{ultrafilter!idempotent}
\index{IP@$\IP$!IP-set@$\IP$-set}
   If $(\Group,+)$ is a commutative semigroup, and $A \in \cP(X)$ is an $\IP$-set, then there exists an idempotent ultrafilter $\cU \in \Ultrafilters{X}$ such that $A \in \cU$.

  Moreover, if $\bx = \seq{x}{n}{\NN}$ is a sequence and $\sigma $ denotes the left shift on $X^\NN$ (so that $\sigma^m \bx = \left( x_{m+n} \right)_{n \in \NN}$) then there exists an idempotent ultrafilter $\cU \in \Ultrafilters{X}$ such that $\FS{\sigma^m \bx} \in \cU$ for all $m$.
\end{lemma}
\begin{proof}
  By taking $\bx$ with $\FS{\bx} \subset A$, we see that it will suffice to prove the second part of the statement. 

  Let us define $A_n := \FS{\sigma^n \bx}$, and take $\Gamma := \bigcap_{n \in \NN} \bar{A}_n$. Since for all $n$ we have $A_n \supseteq A_{n+1}$ and consequently $\bar{A}_n \supseteq \bar{A}_{n+1}$, the set $\Gamma$ is the intersection of a descending family of non-empty compact sets, and hence is a non-empty and compact. 

  We shall prove that $\Gamma$ is also a sub-semigroup of $\Ultrafilters{X}$. We need to check for that for $\cU,\cV \in \Gamma$ we have $\cU + \cV \in \Gamma$. It will suffice to show that for any $n \in \NN$, we have $\cU + \cV \in \bar{A}_n$, or equivalently $A_n \in \cU + \cV$. By definition of $\cU + \cV$, this is equivalent to:
  $$ B:= \{ y \in X \setsep A_n - y \in \cV \} \in \cU $$

  We claim that if $y \in A_n$, then $A_n - y \in \cV$. Indeed, any such $y$ can be expressed as $y = \sum_{i \in I} x_i $ with $\min I \geq n$. If we set $m := \max I + 1$, then arbitrary element of $A_m$ is of the form $z = \sum_{j \in J} x_j$ where $\min J > \max I$. In particular, $I \cap J = \emptyset$ and $\min I \cup J \geq n$, so $y + z = \sum_{i \in I \cup J} x_i \in A_n$. This means that $A_m \subset A_n - y$. Since $A_m \in \cV$ by $\Gamma \subset \bar{A}_m$, we hence have $A_n - y \in \cV$, as claimed.

  From the above claim it follows that: $ B \supset A_n $. Since $A_n \in \cU$, we thus have $ B \in \cU$, which finishes the proof that $\Gamma$ is a sub-semigroup.

  We have shown that $\Gamma$ is a compact sub-semigroup of $\Ultrafilters{X}$. By Corollary \ref{cor:idempotetnes-exist-in-beta-X}, $\Gamma$ contains an idempotent. By construction, this idempotent contains all sets $A_n$.
\end{proof}

Finally, we combine the two Lemmas and reach ultrafilter characterisation of ultrafilters.

\begin{corollary}[Characterisation of $\IP$-sets via ultrafilters]\label{cor:hindman-iff-verison}\label{C:cor:IP-characterisation-via-ultrafilters}
\index{Hindman}
\index{ultrafilter!idempotent}
\index{IP@$\IP$!IP-set@$\IP$-set}
	For a set $A \in \cP(X)$, the following conditions are equivalent:
\begin{enumerate}
\item The set $A$ is an $\IP$-set.
\item There exists an idempotent ultrafilter $\cU \in \Ultrafilters{X}$ such that $A \in \cU$.
\end{enumerate}
\end{corollary}

The characterisation in the above corollary justifies the resolution of $\IP$ to ``idempotent'', as in ``sets which are members of idempotent ultrafilters''. To see how one can justify the resolution involving parallelopiped, recall that the $n$-dimensional parallelopipeds are the figures given (up to a translation) by the formula:
	$$ P = \left\{ \sum_{i=1}^n t_i x_i \setsep t_i \in [0,1] \right\} 
	= \operatorname{Conv} \Big( \FS{x_1,\dots,x_n} \cup \{0\} \Big)
,$$
	where $x_1,x_2,\dots,x_n \in \RR^n$ are some linearly independent vectors and $\operatorname{Conv}$ denotes the convex hull.  
The author prefers to connect $\IP$-sets with idempotence, but acknowledges that opinions on this issue may differ.

\newcommand{\Idempotents}{E}
It is possible to refine the above characterisation slightly, to include a wider class of ultrafilters. 

\begin{observation}\label{C:obs:IP-characterisation-via-ultrafilters-refined}
\index{Hindman}
\index{ultrafilter!idempotent}
\index{IP@$\IP$!IP-set@$\IP$-set}
  Let $\Idempotents$ denote the set of all idempotent ultrafilters on $X$. For a set $A \in \cP(X)$, the following conditions are equivalent:
\begin{enumerate}
\item The set $A$ is an $\IP$-set.
\item There exists an ultrafilter $\cU \in \cl \Idempotents$ such that $A \in \cU$.
\end{enumerate}
Moreover, if $\cU$ is an ultrafilter such that for any $A \in \cU$ is an $\IP$-set, then $\cU \in \cl E$.
\end{observation}

\begin{proof}
  If $A$ is an $\IP$-set, then we already know by \ref{C:cor:IP-characterisation-via-ultrafilters} that there exists $\cU \in \Idempotents$ such that $A \in \cU$. In the other direction, the condition that $A \in \cU$ for some $\cU \in \cl \Idempotents$ is equivalent to $\bar{A} \cap \cl \Idempotents \neq \emptyset$. By the definition of the closure, this means that $\bar{A} \cap  \Idempotents \neq \emptyset$, which in turn is equivalent to $A$ being an $\IP$-set by \ref{C:cor:IP-characterisation-via-ultrafilters}.

  For the additional part, suppose $\cU$ is such that $A \in \cU$ implies that $A$ is $\IP$. It follows that that any base neighbourhood of $\cU$ of the form $\bar{A}$ has non-trivial intersection with $ \Idempotents$. Hence, $\cU \in \cl \Idempotents$, as claimed.
\end{proof}

Having established the above characterisation, we are able to derive partition regularity of $\IP$-sets. 

\begin{definition}[Partition regularity]\label{C:def:partition-regularity}
\label{partition regular}
	Let $\cA \subset \cP(X)$ be a family of sets. Then $\cA$ is said to be \emph{partition regular} if and only for any $A \in \cA$ and any finit partition $A = A_1 \cup A_2 \cup \dots \cup A_k$, we have $A_i \in \cA$ for some $i$.
\end{definition}

\begin{corollary}[Partition regularity of $\IP$ sets]\label{C:cor:IP-part-reg}
	If $A \in \cP(X)$ is an $\IP$-set, and $A = A_1 \cup A_2 \cup \dots \cup A_k$ is a finite partition of $A$, then $A_i$ is an $\IP$-set for some $i$. Equivalently, the family of $\IP$-sets is partition regular.
\end{corollary}
\begin{proof}
	Since $A$ is an $\IP$-set, by Lemma \ref{C:lem:IP=>in-idempotent}, it belongs to some idempotent ultrafilter $\cU$. By the ultrafilter property \ref{def:ultrafilter-prop:4}, $\cU$ contains $A_i$ for some $i$. By Lemma \ref{C:lem:IP<=in-idempotent}, $A_i$ contains an IP-set.
\end{proof}

We now formulate the celebrated Hindman's theorem, in the many versions in which it appears. We strip the formulation from the fancy terminology to make sure we avoid trivial $\IP$ sets. Additionally, since this is an end result, we prefer to keep it as transparent as possible.  

\begin{theorem}[Hindman; integer version]
\index{Hindman}
\index{ultrafilter!idempotent}
\index{IP@$\IP$!IP-set@$\IP$-set}
	Suppose that $\NN =  A_1 \cup A_2 \cup \dots \cup A_k$ is a finite partition of the set of natural numbers. Then, there exists an increasing sequence of integers $\bx = \seq{x}{n}{\NN}$ and an index $i$ with $\FS{\bx} \subset A_i$.  Likewise, there exists an increasing sequence of integers $\by = \seq{y}{n}{\NN}$ and an index $j$ with $\FP{\by}\subset A_j$.  
\end{theorem}
\begin{proof}
	This follows immediately from Corollary \ref{C:cor:IP-part-reg}.
\end{proof}

\newcommand{\balpha}{\boldsymbol{\alpha}}

 Another semigroup which is useful in applications consists in the finite subsets of $\NN$, $\cP_{\operatorname{fin}}(\NN)$, with the union of sets as the semigroup operation. For this special case, because the operation is the union, we use the notation $\FU{\balpha}$ for the finite unions of the sequence of finite sets $\seq{\alpha}{n}{\NN}$. Note that this semigroup is badly non-cancellative --- in fact, each $\alpha \in \cP_{\operatorname{fin}}(\NN)$ is idempotent. This makes the notion of $\IP$-sets, as we defined it, essentially useless, since each non-empty subset of $\cP_{\operatorname{fin}}(\NN)$ is $\IP$. Interesting structure in  $\FU{\balpha}$ only emerges when additional conditions are imposed on $\alpha_n$.

	For $\alpha,\beta \in \cP_{\operatorname{fin}}(\NN)$ we say that $\alpha < \beta$ if and only if $\max \alpha < \min \beta$.

\begin{theorem}[Hindman; finite sets version]\label{C:thm:Hindman-sets}
	Suppose that $\cP_{\operatorname{fin}}(\NN) =  A_1 \cup A_2 \cup \dots \cup A_k$ is a finite partition of the family of finite sets of natural numbers. Then, there exists a sequence of finite sets $\balpha = \seq{\alpha}{n}{\NN}$ with $\alpha_n < \alpha_{n+1}$, an index $i$ with $\FU{\balpha} \subset A_i$.
\end{theorem}
\begin{proof}	
	We first notice that there exists an idempotent ultrafilter $\cU \in \Ultrafilters{\cP_{\text{fin}}(\NN)}$ such that for any $A \in \cU$, for any $\beta \in \cP_{\text{fin}}(\NN)$, we can find $\alpha \in A$ with $\beta < \alpha$. Let $\gamma_n := \{n\}$ Applying the construction from Lemma \ref{C:lem:IP=>in-idempotent}, we find that there exists an idempotent ultrafilter $\cU$ such that for any $m$ we have $\FU{\sigma^m\boldsymbol{\gamma}} \in \cU$, were $(\sigma^m\gamma)_{n} = \gamma_{n+m} = \{n+m\}$. We show that this choice of $\cU$ works. Let $A \in \cU$ be arbitrary, and let $\beta \in \cP_{\text{fin}}(\NN)$ with $m := \max \beta + 1$. Consider the set $A' := \FU{\sigma^m\boldsymbol{\gamma}} \cap A \in \cU$. Clearly, if $\alpha \in A'$ is arbitrary, then $\min \alpha \geq m$, so $\alpha > \beta$. It follows that $\cU$ satisfies the required properties.

	Let $\cP_{\text{fin}}(\NN) =  A_1 \cup A_2 \cup \dots \cup A_k$ be a partition, and let $\cU$ be the ultrafilter just constructed. We can find $i$ with $A_i \in \cU$. Revisiting the proof of Lemma \ref{C:lem:IP<=in-idempotent} and the subsequent Remark, we find that one can construct a sequence $\balpha = \seq{\alpha}{n}{\NN}$ such that $\FU{\balpha} \subset A_i$, and additionally $\alpha_n < \alpha_{n+1}$ for all $n$.
\end{proof}

\begin{remark}
  We have derived the Hindman theorem for a general commutative semigroup, and seen several special cases. One might suspect that the general case of the theorem should be significantly more difficult than the cases of particular semigroups. Somewhat surprisingly, this turns out not to be the case. In fact, it is possible to derive from Hindman's Theorem for finite sets \ref{C:thm:Hindman-sets} the partition regularity of $\IP$-sets as in Corollary \ref{C:cor:IP-part-reg}, which lies at the foundation of our subsequent applications. We refer to \cite{Berg-survey-IP} for details. However, operating with general semigroups does not increase the complexity of the reasoning significantly, so we have no reason to work specifically with $\cP_{\text{fin}}(\NN)$.   
\end{remark}

In case $X = \NN$, we have two natural structures of a semigroup: additive and multiplicative. Hence, if $\NN = A_1 \cup A_2 \cup \dots \cup$ is a finite partition, then we can find $i$ such that $A_i$ is additively $\IP$, and $j$ such that $A_j$ is multiplicatively $\IP$. It is natural to ask if one can find $i$ such that $A_i$ is both additively and multiplicatively $\IP$. The answer turns out to be positive, as shown in \cite{Berg-survey-IP}. 

\begin{theorem}
\index{Hindman}
\index{ultrafilter!idempotent}
\index{IP@$\IP$!IP-set@$\IP$-set}
\index{partition regular}
  Suppose that $\NN =  A_1 \cup A_2 \cup \dots \cup A_k$ is a finite partition of the set of natural numbers. Then, there exist increasing sequences of integers $\bx = \seq{x}{n}{\NN}$, $\by = \seq{y}{n}{\NN}$ index $i$ with $\FS{\bx} \subset A_i$ and $\FP{\by}\subset A_i$.  
\end{theorem}
\begin{proof}
  Let $\Idempotents$ denote the set of all additively idempotent ultrafilters on $\NN$. We claim that $\cl \Idempotents$ is a left multiplicative ideal. Clearly, $\Idempotents$ is non-empty. Because the multiplication is continuous in the left argument, it will suffice to check that $\cl \Idempotents$ is closed under multiplication by $\NN$. Let us consider $\cU \in \cl \Idempotents$, and $n \in \NN$; we need to show that $n \cdot \cU \in \cl \Idempotents$. Thanks to characterisation in \ref{C:obs:IP-characterisation-via-ultrafilters-refined}, it will suffice to show that if $B \in n \cdot \cU$ then $B$ is an additive $\IP$-set. Let $A := \ldiv{n}{B}$; we know that $A \in \cU$ so $A$ contains a set $\FS{\bx}$. Hence, $B$ is an additive $\IP$-set, because it contains the set $\FS{n\bx}$ (operation taken coordinate-wise).

  Because $\cl E$ is a closed left multiplicative ideal, it is clearly also a compact multiplicative sub-semigroup. Hence, by theorem \ref{A:thm:idempotents-exist-in-general}, $\cl E$ contains a multiplicative idempotent $\cV$. Let $i$ be such that $A_i \in \cV$. Because $\cV \in \cl E$, we know from \ref{C:obs:IP-characterisation-via-ultrafilters-refined} that $A_i$ is an additive $\IP$-set. Because $\cV$ is multiplicatively idempotent, from \ref{C:cor:IP-characterisation-via-ultrafilters} it follows that $A_i$ is a multiplicative $\IP$-set. Finally, we notice that $\cV$ is clearly not principal, so the sequences $\bx$ and $\by$ from the theorem formulation can be assumed to be increasing. 
\end{proof}

We finish this section by introducing the notion of $\IP^*$-sets. We will not yet make much use of it in this chapter, but in dynamical applications this concept will be crucial. We begin by giving a defining the operation $\mathsf{A} \mapsto \mathsf{A}^*$ in more generality than really needed. 

\index{IP@$\IP$}
\begin{definition}\label{C:def:A-star-sets}
	Let $\cA \subset \cP(X)$ be a family of sets. Then $\cA^*$ is defined to be the family of all $B \in \cP(X)$ such that for any $A \in \cA$ we have $A \cap B \neq \emptyset$.
	In particular, $B \in \cP(X)$ is an $\IP^*$-set if and only if for any $\IP$-set $A$ we have $A \cap B \neq \emptyset$.
\end{definition}

We make some simple observations about the operation we just defined. Apart from developing some intuition, we aim at an application to an ultrafilter characterisation of $\IP^*$.

\begin{proposition}\label{C:prop:star-properties}
\index{IP*-set@$\IP^*$-set}
	\begin{enumerate}
	\item\label{C:prop:star-prop:1} Let $\cA,\cB \subset \cP(X)$. If $\cA \subset \cB$ then $\cB^* \subset \cA^*$.
	\item\label{C:prop:star-prop:2} Let $I$ be a set,  $\cA_i \subset \cP(X)$ for $i \in I$. Then we have $\left( \bigcup_{i \in I} \cA_i \right)^* = \bigcap_{i \in I} \cA_i^*$.
	\item\label{C:prop:star-prop:3} If $\cA \subset \cP(X)$ is partition regular and $X \in \cA$, then $\cA^* \subset \cA$ and $\cA^*$ is closed under finite intersections. If $\cF \in \Filters{X}$, then $\cF \subset \cF^*$. If $\cU \in \Ultrafilters{X}$, then $\cU^* = \cU$. 
	\end{enumerate}
\end{proposition}
\begin{proof}
	\begin{enumerate}
	\item[\ref{C:prop:star-prop:1}.] Follows directly from the definition, since universal quntification over $\cB$ leads to a stronger condition than universal quantification over $\cA$.
	\item[\ref{C:prop:star-prop:2}.] For $B \in \cP(X)$, the requirement $B \in \left( \bigcup_{i \in I} \cA_i \right)^*$ is equivalent to $(\forall i \in I) (\forall A \in \cA_i):\ A \cap B \neq \emptyset$. For a fixed $i$, the condition $(\forall A \in \cA_i):\ A \cap B \neq \emptyset$ is equivalent to $B \in \cA_i$  
so the condition $B \in \bigcap_{i \in I} \cA_i^*$. Hence, $B \in \left( \bigcup_{i \in I} \cA_i \right)^*$ is equivalent to  $(\forall i \in I) B \in \cA^*_i$, which is what was to be shown.
	\item[\ref{C:prop:star-prop:3}.] If $A \in \cA^*$, then $A \cap A^c = \emptyset$, so $A^c \not \in \cA$. Because $X = A \cup A^c$, it follows that $A \in \cA$. Thus, $\cA^* \subset \cA$. For closure under finite intersections, let $B_1,B_2 \in \cA^*$ and $A \in \cA$; we need to verify that $A \cap B_1 \cap B_2 \neq \emptyset$. Because $\cA$ is partition regular, either $A\cap B_1 \in \cA$ or $A \setminus B_1 \in \cA$. The latter is impossible, bacause $A \setminus B_1$ is disjoint from $B_1 \in \cA^*$. Hence, $A\cap B_1 \in \cA$, and $A \cap B_1 \cap B_2 \neq \emptyset$ because $B_2 \in \cA^*$. 

	Likewise, if $A \in \cF$, then for $B \in \cF$ we have $A \cap B \in \cF$ so in particular $A \cap B \neq \emptyset$. Thus, $A \in \cF^*$, and consequently $\cF \subset \cF^*$. Finally, because ultrafilters are partition regular filters, we have by the previous assertion $\cU^* \subset \cU \subset \cU^*$, so indeed $\cU = \cU^*$.
	\end{enumerate}
\end{proof}

The practical consequence of the above Proposition is that we can charaterise $\IP^*$-sets in terms of ultrafilters, just as we did for $\IP$-sets. Note that in particular a finite intersection of $\IP^*$-sets is again an  $\IP^*$-sets, and intersection of an $\IP$-set with and $\IP^*$-set is again an $\IP$-set.

\index{IP*-set@$\IP^*$-set}

\begin{corollary}[Characterisation of $\IP^*$-sets via ultrafilters]\label{cor:IPst-characterisation}
	For a set $A \in \cP(X)$, the following conditions are equivalent:
\begin{enumerate}
\item The set $A$ is $\IP^*$.
\item For any idempotent ultrafilter $\cU \in \Ultrafilters{X}$ it holds that $A \in \cU$.
\end{enumerate}
\end{corollary}
\begin{proof}
	Let $\set{\cU}{i}{I}$ be the set of all idempotent ultrafilters on $\cU$. Then Corollary \ref{C:cor:IP-characterisation-via-ultrafilters} implies that the family of all $\IP$-sets is $\bigcup_{i \in I} \cU$. Using the above Proposition \ref{C:prop:star-properties}, it follows that the family of all $\IP$-sets is:
	$$\Big( \bigcup_{i \in I} \cU_i \Big)^* = \bigcap_{i \in I} \cU_i^* =  \bigcap_{i \in I} \cU_i $$
	The above characterisation follows immediately.
\end{proof}

To finish this section, we give a simple example of $\IP^*$-sets.

\index{IP*-set@$\IP^*$-set}

\begin{proposition}\label{prop:C:IP*-includes-finite-index-subgroups}
	Suppose that $X$ is a commutative group, and that $Y \subset X$ is a subgroup of finite index: $\card{X/Y} < \infty$. Then $Y$ is an $\IP^*$-set.
\end{proposition}
\begin{proof}
	It suffices to show that for any idempotent ultrafilter $\cU \in \Ultrafilters{X}$ we have $Y \in \cU$. Let us fix such $U$. Because $r:= \card{X/Y} < \infty$, we can partition $X$ into finitely many disjoint sets $Y_i$ of the form $Y_i = x_i + Y$ for some $x_i \in X$, $i \in [r]$. Because $\cU$ is an ultrafilter, for some $j$ we have $Y_j \in \cU$. Because $\cU$ is idempotent, it follows that $Y_j - x \in \cU$ for $\cU$-a.a. $x$. In particular, there exists $x \in Y_j$ such that $Y_j - x \in \cU$. Because for any such $x$ we have $Y_j = x + Y$ it follows that $Y = Y_j - x \in \cU$, as desired.
\end{proof}

\index{ultrafilter!idempotent|)}
\section{$\C$-sets and minimal idempotent ultrafilters}
\index{C*-set@$\C^*$-set}
\index{C-set@$\C$-set|(}

Another important notion of largeness is provided by $\C$-sets (also known as central sets
\footnote{We will usually avoid referring to $\C$-sets as ''central sets'', because this would leave no satisfactory name for the $\C^*$ sets. It is frequent in literature to use names $\text{central}$ and $\text{central}^*$, but we dislike the latter on aesthetic grounds.}) 
and $\C^*$-sets. Much like the $\IP$-sets are related to idempotent ultrafilters, the central sets are related to \emph{minimal} ultrafilters. We are again able to derive a strong combinatorial result by consideration of ultrafilters: this time these are van der Waerden's theorem and Hales-Jewitt theorem. As in the previous section, $X$ denotes a discrete (not necessarily commutative) semigroup throughout this section. However, we will often need to specialise to concrete semigroups.

The following definition of $\C$-sets is obviously inspired by the ultrafilter characterisation of $\IP$-sets. 

\begin{definition}\label{C:C-definition}
	Let $A \in \cP(X)$ be a set. Then $A$ is a $\C$-set if and only if there exists a minimal idempotent\footnote{This means that $\cU + \cU = \cU$ and $\cU \in \Kappa(\Ultrafilters{X})$, where $\Kappa(\Ultrafilters{X})$ denotes the minimal two sided ideal in $\Ultrafilters{X}$} ultrafilter $\cU \in \Ultrafilters{X}$ such that $A \in \cU$.
\end{definition}
\begin{observation}
\index{C*-set@$\C^*$-set}
	Let $A \in \cP(X)$ be a set. Then $A$ is a $\C^*$-set if and only if for any a minimal idempotent ultrafilter $\cU \in \Ultrafilters{X}$ we have $A \in \cU$.	
\end{observation} 

It is unfortunate that $\C$-sets do not allow such a natural definition in terms of the basic semigroup structure as $\IP$-sets do. However, there does exist an equivalent definition in terms of dynamical systems, at least in the most important case $X = (\NN,+)$. The dynamical definition of centrality is due to Furstenberg in \cite{Furstenberg-central}, and was introduced a long time before the connection to ultrafilters was discovered. 

The following result is in no way easy. We cite it (without a proof) as an additional motivation behind the study of $\C$-sets.

\begin{theorem}
	Let $A \subset \NN$. Then, the following conditions are equivalent:
	\begin{enumerate}
	\item The set $A$ is $\C$-set in the sense of Definition \ref{C:C-definition} 
	\item There exists a (topological) dynamical system $(X,T)$, uniformly recurrent point $y \in X$, point $x \in X$ proximal to $y$ and open neighbourhood $U \ni y$ such that $A$ has the form:
	$$ A = \{ n \in \NN \setsep T^nx \in x \}.$$
	\end{enumerate}
\end{theorem}
\begin{proof}
	See \cite{Berg-survey-minimal}.
\end{proof}

The above theorem can be generalised to characterise $\C$-sets in arbitrary semigroups by dynamical properties. However, we will not use them much in our applications.

Armed in the arithmetic preliminaries from Chapter \ref{A:chapter}, we are able to prove the following theorem about arithmetic progressions in central sets with remarkably little work. This result is similar in the spirit to Lemma \ref{C:lem:IP<=in-idempotent}: algebraic properties of an ultrafilter imply combinatorial richness of a set. To make notation more succinct, we begin with a definition.

\newcommand{\AP}{\mathsf{AP}}
\begin{definition}
	Let $A \in \cP(\NN)$ be a set. We say that $A$ is $\AP_r$ if and only if $A$ contains an arithmetic progression of length $r$, i.e. a configuration $\{a,a+b,\dots,a+(r-1)b\}$ for some $a \in \NN$ and $b \in \NN$. Moreover, $A$ is said to be $\AP$ if it is $\AP_r$ for all $r$, i.e. it contains arbitrarily long arithmetic progressions.
\end{definition}

\newcommand{\NNz}{\NN_0}

\begin{theorem}\label{C:thm:C->AP}
  Let $\cV \in \Ultrafilters{ \NN }$ be a minimal idempotent, and suppose that $A \in \cV$. Then, $A$ is $\AP$-set, i.e. for any $r \in \NNz$, the set $A$ contains an arithmetic progression of length $r$.  
\end{theorem}
\begin{proof}
  Consider the semigroup $S := \prod_{i \in [r]} \beta \NN$, and the sets:
  \begin{align*}
    E_0 &:= \{ \left(a+ib\right)_{i\in [r]} \setsep a \in \NN, b \in \NN_0 \} \\
    I_0 &:= \{ \left(a+ib\right)_{i\in [r]} \setsep a \in \NN, b \in \NN \} \\
  \end{align*}
  Note that with these definitions, existence of an arithmetic progression $a,a+b,\dots,a+(r-1)b$ in a set $B \in \cP(\NN)$ is equivalent to existence of a vector $ \left(a+ib\right)_{i\in [r]} \in B^{\times n} \cap I_0$. The same argument shows that if  $B \in \cP(\NN)$ is non-empty, then $B^{\times n}$ contains a common element with $E_0$, namely any constant sequence.

  It is clear that $E_0$ is a semigroup, and that $I_0 \subset E_0$ is an ideal. What is more, $\Centre(S) = \prod_{i \in [r]} \Centre(\Ultrafilters{\NN}) = \prod_{i \in [r]} \NN$, so clearly $E_0 \subset \Centre(S)$. It follows by Proposition \ref{A:prop:algebra-closure-of-sgrp-is-sgrp} that $E := \cl E$ is a semigroup and $I := \cl I_0$ is an ideal in $E$. 

  Let $\delta :\ \Ultrafilters{\NN} \to S$ be the diagonal map $\cU \mapsto (\cU)_{i \in [r]}$. We note that $\delta(\cU) \in E$ for any $\cU$, or equivalently $E_0 \cap U \neq \emptyset$ for any open neighbourhood $U \ni \cU$. For a proof, consider any neighbourhood of $\delta(\cU)$, which can be assumed to be of the form $\prod_{i} \bar{B}_i$, because of how topology on $\Ultrafilters{\NN}$ and product topologies are defined. Taking $B := \bigcap_i B_i$, we may further restrict to the neighbourhood $\bar{B}^{\times n}$. But now for any $b \in B$ we have $\delta(b) \in \bar{B}^{\times n} \cap E$, so the intersection is indeed non-empty.

  Because of Lemma \ref{A:prop:K-product}, we have $\delta(\cU) \in \Kappa(S)$ for $\cU \in \Kappa(\Ultrafilters{\NN})$. In particular, $\Kappa(S) \cap E \neq \emptyset$, and it follows from Lemma \ref{A:prop:K-subsgrp} that we in fact have $\Kappa(S) \cap E = \Kappa(E)$. Because $I$ is an ideal, by the definition of $\Kappa$ we have:
  $$I \supseteq \Kappa(E) = \Kappa(S) \cap E \supseteq \delta(\Kappa(\Ultrafilters{\NN})). $$
  
  Let us now consider the minimal idempotent $\cV$, and set $A \in \cV$. Then $\bar{A}^{\times r}$ is a neighbourhood of $\delta(\cV)$ in $S$. The above considerations show that $\delta(\cV) \in I$. Hence $\bar{A}^{\times r} \cap I_0 \neq \emptyset$, which is, as we noted above, equivalent to existence of an arithmetic progression of length $r$ in $A$.
\end{proof}

The reason why the above theorem is important is that it allows us to prove the classical van der Waerden theorem. 

\begin{theorem}[van der Waerden]
	Suppose $\NN = A_1 \cup A_2 \cup \dots \cup A_k$ is a finite partition of the natural numbers. Then, for some $i$, the set $A_i$ is $\AP$, i.e. contains arbitrarily long arithmetic progressions.	
\end{theorem}
\begin{proof}
	Let $\cV$ be an arbitrary minimal idempotent ultrafilter. There exists $i$ such that $A_i \in \cV$. It follows from the above theorem that $A_i$ is $\AP$.
\end{proof}

It comes as no surprise that the above theorem also has a finite version. The derivation of this finite version is standard and can be done in many ways, so we omit the proof.

\begin{theorem}[van der Waerden, finite version]
	Let length $r$ and number $k$ be fixed. Then there exists $N \in \NN$ such that for any partition $[N] = A_1 \cup A_2 \cup \dots \cup A_k$ into $k$ pieces, there exists $i$ such that the set $A_i$ is $\AP_r$, i.e. contains an arithmetic progression of length $r$.	
\end{theorem}
\begin{proof}
	Left as an exercise to the reader.
\end{proof}

Using the finite version of the van der Waerden Theorem, we can derive the following elegant corollary.
\begin{corollary}
	The family of subset of $\NN$ which are $\AP$ is partition regular.
\end{corollary}
\begin{proof}
	Suppose that $A$ is $\AP$, and that $A = A_1 \cup A_2 \cup \cdots \cup A_k$ is a partition. We will show for any $r$ that one of $A_i$ is $\AP_r$. It will follow immediately that one of $A_i$ is $\AP_r$ for arbitrarily large $r$, and hence also $\AP$.

	Let $r$ be fixed. Let $N$ be such that whenever $[N]$ is partitioned into $k$ parts, one of the parts is $\AP_r$. Because $A$ is $\AP$, it contains an arithmetic progression of length $N$, say $P = \{ a + t b\}_{t=0}^{N-1}$. Let $P_i := P \cap A_i$. Because scaling and shifts do not alter the propery of being an arithemetic progression, it follows than that one of $P_i$ is $\AP_r$. Thus, also $A_i$ is $\AP_r$, which finishes the proof.
\end{proof}

\begin{remark}
	We note that, unlike in the case of $\IP$-sets in Lemma \ref{C:lem:IP=>in-idempotent}, Theorem \ref{C:thm:C->AP} does not admit a converse: it is not at all true that if $A$ is an $\AP$-set then $A \in \cU$ for some minimal idempotent $\cU$. For example, the set $2 \NN + 1$ contains an infinite arithmetic progression, but is not even $\IP$. One can notice that  $2 \NN + 1$  is in fact a translate of $2\NN$ (which is even $\IP^*$). With a little more work, it is possible to find sets with $\AP$ but not being translates of $\IP$-sets.
\end{remark}

	A theorem closely related to van der Waerden theorem is attributed to Hales and Jewett.	While van der Waerden theorem concerned the commutative semigroup $(\NN,+)$, Hales-Jewett speaks of a highly non-commutative situation of words over a finite alphabet. To begin with, we make some definitions.

\begin{definition}
	Let $\Sigma$ be a set, referred to from now on as ``the alphabet''. The free semigroup $F(\Sigma)$ generated by $\Sigma$ is the set of all non-empty sequences $w:\ [n] \to \sigma$, $n \in \NN_1$, together with the operation of concatenation\footnote{If $w = a_1a_2\dots a_n$ ($a_i \in \Sigma$) and $w' = a'_1a'_2\dots a'_{n'}$ ($a'_i \in \Sigma$) then the concatenation of $w$ and $w'$ is the word $ww' = a_1a_2\dots a_na'_1\dots a'_{n'}$} $(w,w') \mapsto ww'$.	
\end{definition}
\begin{remark}
	This definition is justified by a unique factorisation property. It is not hard to discover that if $f:\ \Sigma \to S$ is any map from $\Sigma$ to a semigroup $S$, then there exists a unique map $\tilde{f} :\ F(\Sigma) \to S$ such that $\tilde{f} \circ \iota = f$, where $\iota$ is the natural inclusion map.
\end{remark}

A variable world is the analogue of an affine non-constant function $f(v) = av + b$. 

\begin{definition}
	Let $v$ be a variable (formally, we just need $v \not \in \Sigma$). Then a variable word $w(v)$ is any element of $F(\Sigma \cup \{v\}) \setminus F(\Sigma)$, i.e. a word over the alphabet enriched by $v$, in which $v$ appears at least once.
\end{definition}

	If $w(v)$ is a variable word, then for $a \in \Sigma$, $w(a)$ is the word in $F(\Sigma)$ obtained from $w(v)$ by substitution of $a$ for the variable:
$$ 
	w(a)_i = \begin{cases}
		w(v)_i & \text{ if } w(v)_i \neq v, \\
		a & \text{ if } w(v)_i = v.
	\end{cases}
$$

	A combinatorial line is the analogue of an arithmetic progression. It is obtained from a variable word in the same way an arithmetic progression is obtained from an affine function.
\begin{definition}
	A combinatorial line is a set of the form $\{ w(a) \setsep a \in \Sigma \}$, where $w(v)$ is a variable word.
\end{definition}

Having introduced the notation, we are able to state the Hales-Jewett theorem. Its formulation is, as we have emphasised, analogous to van der Waerden's theorem. Quite surprisingly, the proof of the new result can be obtained from the earlier proof almost by verbatim repetition. The reader will notice that most of the parts are precisely the same, except one replaces arithmetic sequences by combinatorial lines, and instead of Proposition \ref{A:prop:algebra-closure-of-sgrp-is-sgrp} we need to use the more refined Proposition \ref{A:prop:algebra-closure-of-sgrp-is-sgrp-II}.

\begin{theorem}[Hales-Jewett]\label{C:thm:Hales-Jewett}
  Let $F(\Sigma) = A_1 \cup A_2 \cup \dots \cup A_k$ be a finite partition of the space of finite words over an alphabet $\Sigma$. Then one of the cells $A_i$ contains a combinatorial line.
\end{theorem}
\begin{proof}
  Consider the semigroup $S := \prod_{i \in [r]} \Ultrafilters{F(\Sigma)}$, and the sets:
  \begin{align*}
    E_0 &:= \{ \left( w(c) \right)_{c \in \Sigma} \setsep w \in F(\Sigma \cup \{v\} ) \} \\
    I_0 &:= \{  \left( w(c) \right)_{c \in \Sigma} \setsep w \in F(\Sigma \cup \{v\} ) \setminus F(\Sigma) \} \\
  \end{align*}
  Note that with these definitions, existence of a combinatorial line $\left\{ w(c) \right\}_{c \in \Sigma} $ in a set $B \in \cP(F(\Sigma))$ is equivalent to existence of a vector $ \left( w(c) \right)_{c \in \Sigma} \in B^{\times n} \cap I_0$. The same argument shows that if  $B \in \cP(F(\Sigma))$ is non-empty, then $B^{\times n}$ contains a common element with $E_0$, namely any constant sequence.

  It is clear that $E_0$ is a semigroup, and that $I_0 \subset E_0$ is an ideal. We need to check that the same is true of $E$ and $I$. We know that, $\CentreTop(S) = \prod_{i \in [r]} \Centre(\Ultrafilters{\NN}) = \prod_{i \in [r]} \NN$, so clearly $E_0 \subset \Centre(S)$. It follows by Proposition \ref{A:prop:algebra-closure-of-sgrp-is-sgrp-II} that $E := \cl E$ is a semigroup and $I := \cl I_0$ is an ideal in $E$. 

  Let $\delta :\ \Ultrafilters{\NN} \to S$ be the diagonal map $\cU \mapsto (\cU)_{i \in [r]}$. We note that $\delta(\cU) \in E$ for any $\cU$, or equivalently $E_0 \cap U \neq \emptyset$ for any open neighbourhood $U \ni \cU$. For a proof, consider any neighbourhood of $\delta(\cU)$, which can be assumed to be of the form $\prod_{i} \bar{B}_i$, because of how topology on $\Ultrafilters{\NN}$ and product topologies are defined. Taking $B := \bigcap_i B_i$, we may further restrict to the neighbourhood $\bar{B}^{\times n}$. But now for any $b \in B$ we have $\delta(b) \in \bar{B}^{\times n} \cap E$, so the intersection is indeed non-empty.

  Because of Lemma \ref{A:prop:K-product}, we have $\delta(\cU) \in \Kappa(S)$ for $\cU \in \Kappa(\Ultrafilters{F(\Sigma)})$. In particular, $\Kappa(S) \cap E \neq \emptyset$, and it follows from Lemma \ref{A:prop:K-subsgrp} that we in fact have $\Kappa(S) \cap E = \Kappa(E)$. Because $I$ is an ideal, by the definition of $\Kappa$ we have:
  $$I \supseteq \Kappa(E) = \Kappa(S) \cap E \supseteq \delta(\Kappa(\Ultrafilters{F(\Sigma)})). $$
  
  Let us now consider the minimal idempotent $\cV$, and set $A \in \cV$. Then $\bar{A}^{\times r}$ is a neighbourhood of $\delta(\cV)$ in $S$. The above considerations show that $\delta(\cV) \in I$. Hence $\bar{A}^{\times r} \cap I_0 \neq \emptyset$, which is, as we noted above, equivalent to existence of an arithmetic progression of length $r$ in $A$.
\end{proof}
\index{C-set@$\C$-set|)}

\chapter{Applications in ergodic theory.}

\label{B:chapter} 
\lhead{Chapter \ref{B:chapter}. \emph{Ergodic theory}} 

In this chapter, we study the applications of ultrafilters to ergodic theory. We will prove that certain sets of ``return times'' have the combinatorial structure of $\IP^*$-sets of $\C^*$-sets. 

We begin by some general considerations about polynomials. We introduce the notion of discrete derivative, which allows us to characterise the polynomials only in terms of the additive structure. This leads to the notion of polynomial maps between general commutative groups. 

Our first application is to polynomial maps on a torus; it can also be construed as a polynomial recurrence result for rotations. By explicit computation of generalised limits, we are able to show $\IP^*$-set property of certain interesting sets. These results are meant to foreshadow subsequent applications to general dynamical systems.

Thanks to the characterisation of polynomials in terms of discrete derivatives, we are able to introduce a generalisation of the notion of polynomials, which we refer to as ``almost polynomials'', for want of a better name. These are very closely related to $p$-$\VIP$-systems, and extend the so-called ``generalised polynomials''. Adapting proofs for standard polynomials, we obtain very similar recurrence results.

Finally, we turn to applications to general dynamical systems. We re-derive and strengthen Khintchine's theorem: instead of a statement about a given set of returns being merely syndetic, we show the $\IP^*$ property. We then derive some results similar to those of Schnell, except we deal with minimal idempotents rather than general ones. 

We make extensive use of papers by Bergelson, McCutcheon and Knutson, like \cite{BergMc2010-SzThm-GenPoly}, \cite{BergKnuMc2006} and \cite{Berg-RamseyErgo-update}. The paper by Schnell \cite{Schnell}, re-deriving results by Bergelson, Furstenberg and McCutcheon, is also very relevant to our inquiry. The surveys by Bergelson \cite{Berg-survey-IP} and \cite{Berg-survey-minimal} were also very helpful.

%%%%%%%%%%%%%%%%%%%%%%%%%%%%%%%%%%%%%%%%%%%%%
% SECTION
%%%%%%%%%%%%%%%%%%%%%%%%%%%%%%%%%%%%%%%%%%%%%
\section{Polynomials and discrete derivative}

In this section we will study the properties of polynomial maps. Given that polynomial maps are among the simplest maps one can imagine, interest in them hardly needs justification. We will take a rather different approach that is common in algebra. For our purposes, a polynomial map will first and foremost be a particularly regular map, and the algebraic aspects will play a secondary role. To begin with, we define polynomials in the simplest possible situation.

\begin{definition}[Polynomial]\label{def:polynomial}
\index{polynomial}
  A map $f:\ \ZZ \to \ZZ$ is said to be a \emph{polynomial} if and only if $f$ is a polynomial with coefficients in $\QQ$ in the usual sense (i.e. $f$ is of the form $f(x) = \sum_i q_i x^i$) and $f(\ZZ) \subset \ZZ$.
\end{definition}

\begin{remark}
  These polynomials include, but are not restricted to, polynomials with integer coefficients. An example of $f:\ \ZZ \to \ZZ$ which is a polynomial, but not a polynomial with integer coefficients, is $f(x) = \frac{x(x+1)}{2}$. We shall shortly see that the assumption that the coefficients of $f$ lie in $\QQ$ is not restrictive, in the sense that the definition would not change if we allowed more general coefficients, for instance in $\CC$. 
\end{remark}

One of our objectives is to extend the notion of a polynomial to maps between a commutative semigroup and a commutative group.\footnote{A reason for interest in such extensions is that a dynamical system can be construed as a measure preserving action of the additive \emph{semigroup} $\NN$. Results about polynomial recurrence then become statements about polynomial maps in $\NN$. It is natural to inquire into generalisations of such statements to measure preserving actions of more general (commutative) semigroups.} Hence, we need to understand what characterises polynomials in terms of the additive structure. The reader will recall that polynomial in $\RR$ or $\CC$ are characterised by the vanishing of sufficiently high derivatives. To make use of this insight in the discrete setup, the notion of the discrete derivative will be useful.

\newcommand{\stepi}{a}
\newcommand{\stepii}{b}
\begin{definition}[Discrete derivative operator]\label{def:Delta-operator}
\index{discrete derivative}
  For a function $f:\ X \to Y$ from a commutative semigroup $(X,+)$ to a commutative group $(Y,+)$, we define for $\stepi \in X$ the \emph{discrete derivative} $\Delta_\stepi f :\ X \to Y$ by the formula $\Delta_\stepi f(x) := f(x+\stepi) - f(x)$. Occasionally, we also refer to $\Delta_\stepi f$ as the \emph{finite difference}\footnote{Some authors refer to the expression $f(x+a) - f(x)$ as the finite difference, and to $\frac{f(x+a)-f(x)}{a}$ as the discrete derivative. However, we use these two terms interchangeably.}.

If $R$ is a domain (commutative ring with unit) of characteristic $0$, then polynomials in $R[x]$ can be identified with a subset of functions. Because $\Delta_a f$ is a polynomial whenever $f$ is a polynomial, we will refer to derivatives of polynomials again as polynomials without further mention. 
  
\end{definition}
\begin{remark}
\index{polynomial}
\index{discrete derivative}
	Note that in finite rings it may happen that a polynomial is not uniquely determined by its values. For example, in $\FF_p[x]$, the polynomial $x^p-x$ and the $0$ polynomial give rise to the same map, but are clearly distinct as polynomials. In general, definitions similar to the one above make sense for arbitrary commutative rings with unit, but we restrict to characteristic $0$ domains for the ease of presentation. In particular, we wish to avoid having to make a distinction between polynomials and polynomial maps.   
\end{remark}

Before we make use of the introduced notion of the discrete derivative, we point out some of the elementary properties.

\begin{proposition}\label{lem:Delta-properties-algebra}
\index{semigroup}
\index{discrete derivative}
	Let $f,g:\ X \to Y$ be maps from a semigroup $X$ to a group $Y$, and let $a,b \in X$. Then, the following properties hold true:
  \begin{enumerate}
   \item $\Delta_\stepi (f+g) = \Delta_\stepi f+ \Delta_\stepi g$.
   \item $\Delta_\stepi (f \cdot g) = \Delta_\stepi f \cdot \Delta_\stepi g + \Delta_\stepi f \cdot g + f \cdot \Delta_\stepi g$.
   \item $\Delta_\stepi \Delta_\stepii f = \Delta_\stepii \Delta_\stepi f = \Delta_{a+b} f - \Delta_a f - \Delta_b f$.
  \end{enumerate}

\end{proposition}
\begin{proof}
  All the equalities follow from direct substitution into the definition.
\end{proof}

We recall some standard notation related to polynomials. The reader will surely find these standard, but we give a detailed definition to avoid ambiguities.

\begin{definition}[Degree and leading coefficient]
\index{polynomial!degree}
\index{polynomial!leading coefficient}
  If $R$ is an arbitrary commutative ring with unit and $f \in R[x]$ is a non-zero polynomial, then $\deg f$ stands for the polynomial degree of $f$ in $x$, and $\lc f$ stands for the leading coefficient. We take $\deg 0 := -\infty$ and $\lc 0 := 0$ by definition, so generally $\deg f \in \NN \cup \{-\infty\}$. Additionally, when speaking of degrees, we assume the convention that if $\deg f < k$ then $\deg f - k := -\infty$, and also for any $k$ we have $-\infty \pm k = -\infty$.

  With these conventions, for any $f \in R[x]$ we have the decomposition:
  $$ f(x) = \lc f \cdot x^{\deg f} + g,$$
  where $\deg g \leq \deg f - 1$.
\end{definition}

Much like with the standard derivative, application of the discrete derivative to a polynomial decreases the degree by $1$, as shown in the following lemma.

\begin{observation}\label{B:lem:Delta-properties-deg-and-lc}
\index{polynomial!degree}
\index{polynomial!leading coefficient}
  If $R$ is a characteristic $0$ domain and $f \in R[x]$ is a non-zero polynomial, then for any $a \in R \setminus \{0\}$ we have $\deg \Delta_\stepi f = \deg f - 1$ and $\lc \Delta_\stepi f = \deg f \cdot \stepi \cdot \lc f$, with the understanding that $-\infty \cdot 0 = 0$.
\end{observation}
\begin{proof}
  We proceed by induction on $\deg f$. The case $\deg f = -\infty$, i.e. $f = 0$, is clear. In the case when $\deg f = 0$, we have that $f(x) = c \in R \setminus \{0\}$ is a constant polynomial, so $\Delta_\stepi f = 0$ for any $\stepi$, hence the claim holds. 

  Suppose now that $\deg f \geq 1$, and the claim holds for all polynomials of degree strictly smaller than $\deg f$. We can write $f$ in the form $f(x) = \lc f \cdot x^{\deg f} + g(x)$, where $\deg g < \deg f$. We then have:
  \begin{align}
   \Delta_a f(x) &= \lc f \sum_{k=0}^{\deg f} \binom{{\deg f} }{k} a^k x^{{\deg f}-k} - x^{\deg f} + \Delta_a g(x)
\\& = \deg f \cdot  a \cdot \lc f \cdot x^{{\deg f}-1} + \left( \lc f \cdot \sum_{k=2}^{\deg f} \binom{\deg f}{ k} a^k x^{{\deg f}-k} + \Delta_a g(x) \right).
\end{align}
 
  By inductive assumption, $\deg \Delta_a g(x) \leq {\deg f}-2$, and hence the expression in the parenthesis has degree at most ${\deg f}-2$. Since $\deg f \cdot a \cdot \lc f \neq 0,$ we have $$\deg \left( \deg f \cdot a \cdot \lc  \cdot f x^{\deg f -1} \right) = \deg f -1.$$ It follows that $\deg \Delta_a f = \deg f - 1 $ and $\lc \Delta_a f = \deg f \cdot a \cdot \lc f$, as desired.

\end{proof}
\begin{remark}
\index{polynomial!polynomial over a ring}
\index{polynomial!degree}
\index{polynomial!leading coefficient}

  In finite characteristic, it can happen that for a polynomial $f$ we have $f(x+a) - f(x) = 0$ as polynomials, even though $\deg f \gg 1$. For instance, in $\FF_p$ we have for $f(x) = x^p - x$: 
  $$f(x+a) - f(x) = (x+a)^p - (x+a) - x^p + x = \sum_{k=1}^p \binom{p }{ k} a^k x^{p-k} - a = a^p - a = 0.$$
\end{remark}

The above lemma suggests the following generalisation of the notion of polynomials to maps between commutative (semi)groups. 

\begin{definition}[Polynomials in general groups] \label{B:def:poly-general-groups}
\index{polynomial}
\index{semigroup}
  Let $(X,+)$ be a commutative semigroup, and let $(Y,+)$ be commutative group, written additively. We define polynomials $X \to Y$ inductively, as follows:
\begin{enumerate}
\item The unique polynomial of degree $-\infty$ is the zero map $x \mapsto 0_Y$.
\item The polynomials of degree $0$ are the non-zero constant maps $x \mapsto c$.
\item A map $f :\ X \to Y$ is a polynomial of degree $d \geq 1$ if and only if for any $a \in X$, the map $\Delta_a f$ is a polynomial of degree at most $d-1$.
\end{enumerate}
\end{definition}

From Observation \ref{B:lem:Delta-properties-deg-and-lc} it follows that for a characteristic $0$ domain, the standard polynomials in $R[x]$ are polynomials in the sense of the above definition. More generally, if $R \subset S$ is an extension of characteristic $0$ domains, and $f \in S[x]$ is such that $f(R) \subset R$, then the same lemma shows that $f$ is a polynomial in the above sense. We shall now make the correspondence between polynomials in $R[x]$ and polynomial maps $R \to R$ more precise.

\begin{lemma}
\index{polynomial!polynomial over a ring}
  Let $R$ be a characteristic $0$ domain, with field of fractions $Q$. Suppose that $f:\ R \to R$ is a polynomial in the sense of Definition \ref{B:def:poly-general-groups}. Then $f \in Q[x]$, i.e. $f$ can be represented as a polynomial all of whose coefficients lie in $Q$. Moreover, any such polynomial is a combination of the polynomials $\binom{x}{n} := \frac{x^{\underline{n}}}{n!}$ for $n \in \NN$ with coefficients in $R$. Here, $x^{\underline{n}} := \prod_{k=0}^{n-1} (x-k)$.
\end{lemma}
\begin{proof}
  Let us denote $e_n(x) := \binom{x }{ n}$. By a direct computation, we check that $\Delta_1 e_n = e_{n-1}$ for $n \geq 1$, and $\Delta_1 e_0 = 0$. Indeed, we have for $n \geq 1$:
$$\Delta_1 e_n(x) 
= \frac{(x+1)^{\underline{n}}-x^{\underline{n}}}{n!} 
= \frac{x^{\underline{n-1}}((x+1)-(x-n+1))}{n!} 
= e_{n-1}(x).$$
  Let us now take a polynomial $f$ as described in the assumptions. We show by induction on $\deg f $ that $f$ lies in the $R$-linear span of $e_i$. The case $\deg f \leq 0$ is immediate, so let us suppose $\deg f  \geq 1$ and the claim holds for polynomials of lower degrees. By Lemma \ref{B:lem:Delta-properties-deg-and-lc}, we find that $\deg \Delta_1 f = \deg f - 1$, so by the inductive assumptions, we can write $\Delta_1 f$ in the form:
$$ \Delta_1 f = \sum_{i=1}^{\deg f} c_i e_{i-1},$$
where $c_i \in R$. Let us consider the polynomial $g := \sum_{i=1}^{\deg f} c_i e_{i} \in K[x]$. Because of the preliminary observation, we have:
$$ \Delta_1 g =  \sum_{i=1}^{\deg f} c_i e_{i-1} = \Delta_1 f.$$
Hence, $\Delta_1( f - g) = 0$, and Lemma \ref{B:lem:Delta-properties-deg-and-lc} ensures that $\deg (f-g) \leq 0$. In other words, there exists a constant $c_0 \in K$ such that $f = g + c_0$. Evaluation at $0$ yields $c_0 = f(0) \in R$. Because $e_0 = 1$, we now have the expression:
$$ f =  \sum_{i=0}^{\deg f} c_i e_{i}.$$
Hence, $f$ is a combination of $e_0,e_1,\dots,e_{\deg f}$ with coefficients in $R$, as claimed.
\end{proof}

\begin{corollary}
  \index{polynomial!polynomial over a ring}
  Let $R$ be a characteristic $0$ domain, with field of fractions $Q$, and let $S$ be a ring containing $Q$ as a subring. If $f \in S[x]$ is a polynomial such that $f(R) \subset R$, then $f \in Q[x]$, and moreover is a combination of the polynomials $\binom{x }{ n}$ with coefficients in $R$.
\end{corollary}

Note that the above lemma and corollary contain implications only in one direction: there is no guarantee that the map $x \mapsto \binom{x }{ n}$ should preserve the ring $R$. However, for $R = \ZZ$ we have a full characterisation.

\begin{corollary}\label{B:cor:form-of-polynimials-Z-to-Z}
\index{polynomial!polynomial over a ring}
	Let $K$ be a characteristic $0$ field. Then, the polynomials $f \in K[x]$ such that $f(\ZZ) \subset \ZZ$ are precisely the combinations of the polynomials $\binom{x }{ n}$ for $n \in \mathbb{N}$, with integer coefficients.
\end{corollary}
\begin{proof}
	The above theorem shows that if $f \in K[x]$ is such that $f(\ZZ) \subset \ZZ$, then $f$ is a combination of $\binom{x }{ n}$ for $n \in \mathbb{N}$. Conversely, we show that $\binom{x }{ n} \in \ZZ$ for any $n \in \NN$ and  $x \in \ZZ$. If $x \in \NN$, then $\binom{x }{ n}$ has the combinatorial interpretation of the number of ways to choose $n$ elements out of $x$ elements, and hence surely is an integer. For general $x$, we note that the statement that $\binom{x }{ n}$ is an integer is equivalent to the statement that $n! | x^{\underline{n}}$, which depends only on the equivalence class of $x$ modulo $n!$. Hence, it suffices to check that $\binom{x }{ n}$ for $n!$ consecutive values of $x$, which we have already done.
\end{proof}

\begin{remark}
  The assumption of commutativity is essential for our considerations. It is natural to ask if the theory can be extended to a non-commutative setting. There seems to be little hope of developing a theory for general non-commutative (semi)groups. However, Leibman \cite{Lei-Poly-Groups} proposed a fairly successful theory of polynomials in general nilpotent groups. We do not go into more details on this matter, because for our applications the commutative context is more than sufficient.
\end{remark}

We will now introduce the symmetric discrete derivative. Although the standard discrete derivative is more natural, the following variation will be more useful for our purposes. We take time to develop some algebraic properties before we move on to applications in the consecutive sections.

\newcommand{\DeltaS}{\bar{\Delta}}
\newcommand{\Val}{\mathrm{V}}
\begin{definition}[Symmetric finite derivative]
\label{def:DeltaS-operator}
\index{discrete derivative!symmetric}
  For a function $f:\ X \to Y$ from a semigroup $(X,+)$ to a group $(Y,+)$, we define for $\stepi \in X$ the \emph{symmetric discrete derivative} $\DeltaS_\stepi f :\ X \to Y$ by the formula $\DeltaS_\stepi f(x) := f(x+\stepi) - f(x) - f(\stepi)$.

  Moreover, we define the $k$-fold symmetric discrete derivative:	
 $$\DeltaS^k f(x_0,x_1,\dots,x_k) := \DeltaS_{x_1} \DeltaS_{x_{2}}\dots\DeltaS_{x_k} f(x_0).$$

  If $R$ is a characteristic $0$ domain and $f \in R[x]$, then $\DeltaS_a f \in R[x]$ for any $a$, so we reserve the right to refer to $\DeltaS_a f $ as a polynomial in this situation. Moreover, it is true that $\DeltaS^k f(x_0,x_1,\dots,x_k) \in R[x_0,x_1,\dots,x_k]$.
\end{definition}

\begin{observation}\label{B:lem:DeltaS-commute}
\index{discrete derivative!symmetric}
	The symmetric discrete derivatives commute: $\DeltaS_\stepi \DeltaS_\stepii f = \DeltaS_\stepii \DeltaS_\stepi f$. 
\end{observation}
\begin{proof}
  Both terms are equal to $\Delta_{a}\Delta_{b} f(x) - \DeltaS_a f(b) = \Delta_{b}\Delta_{a} f(x) - \DeltaS_b f(a)$.
\end{proof}

The following observation justifies the used terminology and motivates the above definition.
\begin{observation}\label{B:lem:DeltaS-symmetric}
	The $k$-fold symmetric discrete derivative $\DeltaS^k f :\ X^k \to Y$ is symmetric (i.e. invariant under the permutation of arguments).
\end{observation}
\begin{proof}
  Because the operators $\DeltaS_{x_i}$ commute, the value of $\DeltaS^k f(x_0,x_1,\dots,x_k)$ is invariant under the permutation of $x_1, x_{2}, \dots, x_k$. From the definition of $\DeltaS$ it also follows that for any $a,b \in X$ and $g:\ X \to X$ we have $\DeltaS_a g(b) = \DeltaS_b g(a)$. Applying this rule to $g = \DeltaS_{x_{1}}\DeltaS_{x_{2}} \dots \DeltaS_{x_k} f$, $a = x_0$ and $b = x_1$, we see that $\DeltaS^k f(x_0,x_1,\dots,x_k)$ is invariant under swapping of $x_0$ and $x_1$. Since any permutation can be expressed as a composition of permutations already considered, the claim follows.
 \end{proof}

	It is possible to derive an explicit formula for the $k$-fold finite difference. Having an explicit formula is often useful; in particular, it can be used to re-derive some of the previous two observations immediately.

\begin{proposition}[Explicit finite difference]
\label{B:lem:DeltaS^k-explicit}
\index{discrete derivative!symmetric}
  The symmetric finite difference is given by the formula:
  $$ \DeltaS^k f(x_0,x_1,\dots,x_k) = \sum_{\emptyset \neq I \subset [k+1]} (-1)^{k+1-\#{I}} f\left( \sum_{i \in I} x_i \right).$$
\end{proposition}
\begin{proof}
  We prove the claim by induction on $k$. If $k = 0$, then the claim clearly holds. Suppose we want to prove the claim for $k$, while we know it holds for $k-1$. We can explicitly transform:
	\begin{align*}
		\DeltaS^k f(x_0,\dots,x_k) &= \DeltaS_{x_1} \DeltaS_{x_2} \dots \DeltaS_{x_k} f (x_0) 
	=  \sum_{\emptyset \neq I \subset [k]} (-1)^{k-\#{I}} \DeltaS_{x_k} f\left( \sum_{i \in I} x_i \right)
	\\ &= 
	\sum_{\emptyset \neq I \subset [k]} (-1)^{k-\#{I}} f\left( x_k + \sum_{i \in I} x_i \right) 
	\\& \quad - \sum_{\emptyset \neq I \subset [k]} (-1)^{k-\#{I}} f\left( \sum_{i \in I} x_i \right) 
	- \sum_{\emptyset \neq I \subset [k]} (-1)^{k-\#{I}} f\left( x_k \right)
	\\&= \sum_{\substack{ \emptyset \neq J \subset [k+1] \\ k \in J }} (-1)^{k+1-\#{J}} f\left( \sum_{i \in J} x_i \right) 
	\\& \quad + \sum_{\substack{ \emptyset \neq J \subset [k+1] \\ k \not \in J }} (-1)^{k+1-\#{J}} f\left( \sum_{i \in J} x_i \right) 
	- f\left( x_k \right)
	\\&= \sum_{ \emptyset \neq J \subset [k+1]}(-1)^{k+1-\#{J}} f\left( \sum_{i \in J} x_i \right). 
	\end{align*}
	This formula is the one we were aiming for, which finishes the proof of the inductive claim.
\end{proof}

Having derived an explicit formula for $k$-fold symmetric finite difference, our next step is to find a relation to the standard (non-symmetric) finite difference. This is established in the following lemma.

\begin{lemma}\label{B:lem:DeltaS-and-Delta}
\index{discrete derivative!symmetric}
\index{discrete derivative}
	If $X$ is a commutative monoid with neutral element $0_X$, and $f:\ X \to Y$ is a map to a commutative group $Y$, then the following relation holds:
	\begin{equation}
	\DeltaS^{k} f(x_0,x_1,\dots x_k) - (-1)^k f(0_X) = \Delta_{x_0}  \Delta_{x_1} \dots  \Delta_{x_k}f (0_X).
	\label{eq:01:lem:2.14}
	\end{equation}
\end{lemma}
\begin{proof}
	For $k = 0,1$, the formula can be verified directly. For $k \geq 2$, we proceed by induction. Using the claim for $1$ and for $k-1$, we conclude that:
  \begin{align*}
  \DeltaS^k f(x_0,x_1,\dots,x_k) &= \DeltaS^{k-1} f(x_0+x_1,\dots,x_k) - \DeltaS^{k-1} f(x_0,\dots,x_k) 
  \\& \quad  - \DeltaS^{k-1} f(x_1,\dots,x_k)
  \\&= \Delta_{x_0+x_1}  \Delta_{x_2} \dots  \Delta_{x_k}f (0) + (-1)^{k-1}f(0)
  \\& \quad - \Delta_{x_0}  \Delta_{x_2} \dots  \Delta_{x_k}f (0) - (-1)^{k-1}f(0)
\\& \quad  - \Delta_{x_1}  \Delta_{x_2} \dots  \Delta_{x_k}f (0) - (-1)^{k-1}f(0)
  \\&= \left( \Delta_{x_0+x_1}  - \Delta_{x_0} - \Delta_{x_1}  \right) \Delta_{x_2} \dots  \Delta_{x_k}f (0) + (-1)^{k}f(0)
  \\&= \Delta_{x_0} \Delta_{x_1} \Delta_{x_2} \dots  \Delta_{x_k}f (0) + (-1)^{k}f(0)
  \end{align*}
	This finishes the inductive proof.
\end{proof}
\begin{remark}
	The assumption of $X$ being a monoid is not restrictive. If $X$ is merely a commutative semigroup, one can make $X$ into a monoid by artificially adding a neutral element $0_X$. One can then extend $f$ by assigning any value to $f(0_X)$.
\end{remark}

\begin{corollary}\label{lem:2.14}
\index{polynomial!polynomial over a ring}
  If $R$ is a characteristic $0$ domain and $f \in R[x]$ is a polynomial, then the polynomial $\DeltaS^{k} f(x_0,x_1,\dots x_k) - (-1)^k f(0) \in R[x_1,\dots,x_k]$ is divisible by $x_i$ for any $i$. In particular, it has degree at most $\deg f - k$ in any variable $x_i$, and has the constant term equal to $0$. Moreover, it holds true that $\DeltaS^{\deg f} f = (-1)^{\deg f} f(0)$.
\end{corollary}
\begin{proof}
	From Bezout's theorem, it follows that in general that for $g \in R[x]$, $\Delta_y g(x) = g(x+y) - g(x)$ is divisible by $y$, as polynomials. From the above lemma, it follows that $\DeltaS^{k} f(x_0,x_1,\dots x_k) - (-1)^k f(0)$ is divisible by $x_0$. By symmetry, it is divisible by $x_i$ for all $i$. Since the total degree of this polynomial is at most $\deg f$, the degree in any of the $k+1$ variables cannot exceed $\deg f - k$. The last assertion is an immediate consequence of taking $k = \deg f$.
\end{proof}

\begin{corollary}
\index{polynomial}
\index{discrete derivative}
	Let $f:\ X \to Y$ be a polynomial map. If $X$ is a monoid, then 
	$$\DeltaS^k f(x_0,\dots,x_k) = (-1)^k f(0_X),$$ 
	for $k \geq \deg f$. In general, if $X$ is only a semigroup, there exists a constant $C(f) \in Y$ such that
	$$\DeltaS^k f(x_0,\dots,x_k) = (-1)^k C(f),$$ 
	for $k \geq \deg f$.
\end{corollary}

%%%%%%%%%%%%%%%%%%%%%%%%%%%%%%%%%%%%%%%%%%%%%
% SECTION
%%%%%%%%%%%%%%%%%%%%%%%%%%%%%%%%%%%%%%%%%%%%%
\section{Polynomial maps to the torus}
\index{torus}
Having deepened our understanding of polynomials, we now turn to a simple example of an explicit computation of a generalised limit. We begin with a general case, and then proceed to draw some surprising conclusions.
 %application to the polynomial maps on the torus. We start with and explicit computation of certain generalised limits, and then proceed to draw some surprising conclusions.

\newcommand{\aalpha}{}
\begin{proposition}\label{B:lem:p-lim-of-polynomials-I}
\index{polynomial}
\index{limit!generalised limit}
\index{torus}
  Let $f: \ X \to T$ be a polynomial map from commutative monoid to compact commutative group $T$, and let $p \in \beta X$ be an idempotent ultrafilter. Then we have: 
  $$ \llim{p}{n} \aalpha f(n) = \aalpha f(0).$$
  In particular, if $f(0) = 0$, then: $$ \llim{p}{n} \aalpha f(n) = 0.$$ 
\end{proposition}
\begin{proof}
  The proof follows by induction on the degree of $f$. For $\deg f \leq 0$ the claim is trivial. Thus, let $\deg f > 0$, and suppose that the claim holds for all polynomials of smaller degrees. Let $\lambda$ denote the limit $ \llim{p}{n} \aalpha f(n)$ --- our goal is to show that $\lambda = \aalpha f(0)$. Note that $f(n+m) = \DeltaS f(m,n) + f(m) + f(n)$. Since $p$ is idempotent, we have:
  \begin{align*}
  \lambda = \llim{p}{n} \aalpha f(n) 
  &=  \llim{p}{m} \llim{p}{n} \aalpha f(n+m) 
  \\&= \llim{p}{m} \llim{p}{n} \left( \aalpha \DeltaS f(m,n) + \aalpha f(m) + \aalpha f(n) \right)
  \\&= \llim{p}{m} \llim{p}{n} \aalpha \DeltaS f(m,n) + 2 \lambda.
  \end{align*}
  For a fixed $m$, the polynomial $ \DeltaS f(m,n)$ in the variable $n$ has degree strictly smaller than $\deg f$. Likewise, for fixed $n$, $ \DeltaS f(m,n)$ is a polynomial in $m$ of degree strictly smaller than $n$. Thus, the inductive assumption applies:
  $$  \llim{p}{m} \llim{p}{n} \aalpha \DeltaS f(m,n) =  \llim{p}{m} \aalpha \DeltaS f(m,0) = \aalpha \DeltaS f(0,0) = - \aalpha f(0).$$
  Hence, the above computation leads to:
  $$ \lambda = 2 \lambda - \aalpha f(0).$$
  This is equivalent to $\lambda = \aalpha f(0)$, which was our claim.
\end{proof}

We can make the above result more concrete by applying it to a particular choice of spaces and explicitly describing polynomial maps. Our choice is to investigate polynomials $\ZZ \to \TT$, but similar considerations are possible for other choices; in particular we can derive multi-dimensional analogues by considering polynomials $\ZZ^k \to \TT^l$.

\begin{corollary}\label{cor:p-lim-of-polynomials-II}
\index{polynomial}
\index{limit!generalised limit}
\index{torus}
	Let $p \in \beta \ZZ$ be a fixed idempotent ultrafilter. For any $\alpha \in \TT$ and polynomial $f:\ \ZZ \to \ZZ$ we have:
	$$ \llim{p}{n} \alpha f(n) = \alpha f(0).$$
	Moreover, for any $\alpha_i \in \TT$, $1 \leq i \leq d$, we have:
  $$ \llim{p}{n} \sum_{i=1}^d \alpha_i n^i = 0.$$
\end{corollary}

\begin{remark}
\index{polynomial}
\index{limit!generalised limit}
\index{torus}
  The above corollary is a particular property of idempotent ultrafilters as opposed to general ultrafilters. As we will see, limits along arbitrary ultrafilters do not show nearly as much regularity.
\end{remark}

A useful consequence of the above results is the following approximation result. It speaks of real valued polynomials, which are a very natural object to study.

\newcommand{\eps}{\varepsilon}
\begin{corollary}[Integral approximation]
\index{polynomial}
\index{limit!generalised limit}
\index{torus}
  Let $g \in \RR[x]$ be a polynomial with real coefficients with $g(0) = 0$. For any $\varepsilon > 0$, consider the set of those integers which are mapped by $g$ to $\varepsilon$-almost integers: 
  $$ A_\varepsilon := \{ n \in \ZZ \setsep \operatorname{dist}(g(n),\ZZ) < \varepsilon \}.$$
  Then the set $A_\varepsilon$ is an $\IP^*$-set.
\end{corollary}
\begin{proof}
  If $\pi:\ \RR \to \TT$ denotes the standard projection, we have the relation:
  $$\operatorname{dist}(g(n),\ZZ) = \operatorname{d}(\pi(g(n)),0),$$ 
  where $\operatorname{d}$ denotes the standard distance in $\TT$. If $g(x) = \sum_{i=1}^d g_i x^i$, then $\pi(g(n)) = \sum_{i=1}^d \pi(g_i) n^i$. From the above Corollary \ref{cor:p-lim-of-polynomials-II}, it follows that for any idempotent ultrafilter $p$ we have $ \llim{p}{n} \sum_{i=1}^d \pi(g_i) n^i = 0$. If follows that the set $A_\varepsilon$ is $p$-large for any $\eps > 0$. Since $A_\eps$ is $p$-large, it is an $\IPst$ set in view of $p$ being arbitrary. 
\end{proof}

In all of the above results, we relied on the assumption that the ultrafilter $p \in \beta X$ used for taking limits was an idempotent: $p+p = p$. It is natural to ask if anything specific can be said about limits along arbitrary ultrafilters. It turns out that for limits these limits can exhibit fairly arbitrary behaviour, as we see shortly. 

We will use the following classical equidist results, due mostly to Weyl. Similar results can be proved in more generality. For a derivation of these results, see \cite{Ergo-by-Einsiedler-Ward}

\begin{theorem}[Weyl]\label{B:thm:Weyl-equidistribution}
\index{equidistribution}
\index{Weyl}
\index{torus}
  Let $\alpha \in \RR$ 	be irrational. Then the sequence $n \alpha \pmod{1},\ n \in \NN$ is equidistributed in $\TT$.

  More generally, if $g:\RR \to \RR$ is a polynomial with at least one irrational coefficient except for the constant term, then the sequence $g(n) \pmod{1},\ n \in \NN$ is equidistributed in $\TT$.
\end{theorem}

A generic tool for extending eqidistribution results is the following criterion. In particular, it allows one to generalise results about one dimensional equidistribution into higher dimensioins.

\begin{theorem}[Weyl equidistribution criterion]\label{B:thm:Weyl-criterion}
\index{Weyl}
\index{equidistribution}
\index{torus}
\index{Fourier analysis}

Let $\seq{\alpha}{n}{\NN}$ be a sequence with terms in $\TT^d$. Then the following conditions are equivalent:
\begin{enumerate}
 \item The sequence $\seq{\alpha}{n}{\NN}$  is equidistributed.
  \item For any $k \in \ZZ^d$ one has: $$ \lim_{N\to \infty} \sum_{n=1}^N e^{2 \pi i k \cdot \alpha_n}=0,$$ where $k \cdot \gamma := \sum_{i\in [d]} k_i \gamma_i$.
\end{enumerate}
\end{theorem}

Using the above equidistribution results, we are in position to describe make the aforementioned statements about limits along general ultrafilters. The following example can be juxtaposed with Lemma \ref{B:lem:p-lim-of-polynomials-I}. 

\begin{example}
\index{torus}
\index{ultrafilter!generic ultrafilter}
Let $\seq{\alpha}{i}{[d]} \in \TT^d$ be a sequence with at least one irrational entry. By Weyl Theorem \ref{B:thm:Weyl-equidistribution}, the sequence $\phi(n) := \sum_{i=1}^d \alpha_i n^i$ is equidistributed in $\TT$. In particular, for any fixed $\gamma \in \TT$, the sets $A_\varepsilon := \{ n \in \ZZ \setsep d( \phi(n), \gamma) < \varepsilon \}$ are nonempty for $\varepsilon > 0$, and hence the family of set $\cA := \{ A_\varepsilon \setsep \varepsilon > 0 \}$ trivially has the finite intersection property. Applying Lemma \ref{A:cor:ultrafilters-existance-above-FI-sets}, we conclude that $\cA$ is contained in some ultrafilter $p$, for which we necessarily have $\llim{p}{n} \phi(n) = \gamma$.
\end{example}

The above result concerns a single polynomial of arbitrary degree. Even more can be said for linear polynomials. It is clear that for a fixed ultrafilter $p$ the map $\lambda_p :\ \TT \ni \alpha \mapsto \llim{p}{n} n \alpha \in \TT$ is additive, in the sense that $\lambda_p(\alpha + \beta) = \lambda_p(\alpha) + \lambda_p(\beta),\alpha,\beta \in \TT$. We have shown that for idempotent $p$, the map $\lambda_p$ is identically $0$. Similar statement is true if $p = \beta f(q)$ for polynomial $f:\ \ZZ \to \ZZ$ with $f(0) = 0$, which is a consequence of Lemma \ref{B:lem:p-lim-of-polynomials-I}. It is natural to ask if any additional restriction can be placed on $\lambda_p$ for arbitrary $p$. It turns out this is not the case, as the below observation shows.

\begin{proposition}\label{B:prop:p-lim-arbitrary-linear}
\index{torus}
\index{ultrafilter!generic ultrafilter}
\index{additive map}
  Let $\phi:\ \mathbb{T} \to\ \mathbb{T}$ satisfy $\phi(\alpha + \beta) = \phi(\alpha) + \phi(\beta)$. Then, there exists an ultrafilter $p \in \beta \ZZ$ such that $\phi = \lambda_p$, where $\lambda_p$ is defined by $\lambda_p(\alpha) = \llim{p}{n} n \alpha$.
\end{proposition}
\begin{proof}
  Let $\Gamma_\alpha := \{ p \in \beta \ZZ \setsep \lambda_p(\alpha) = \phi(\alpha) \}$. The claim is equivalent to existence of $p$, such that $p \in \Gamma_\alpha$ for all $\alpha$, hence it will suffice to show that $\bigcap_{\alpha \in \TT} \Gamma_\alpha \neq \emptyset$. Because the map $p \mapsto \lambda_p(\alpha)$ is continuous for any fixed $\alpha$, the sets $\Gamma_\alpha$ are closed. Thus, because $\beta \ZZ$ is compact, it will be enough to show that the finite intersections of the form $\bigcap_{\alpha \in A} \Gamma_\alpha$, with $A \subset \TT$ and $A$ --- finite, are non-empty.
  
  Let $\tilde{A} \subset [0,1) \subset \RR$ denote the set corresponding to $A$ under the natural identification\footnote{The natural projection map $\pi:\ \RR \to \TT = \RR/\ZZ$ maps $[0,1)$ to $\TT$ bijectively. Some authors identify $\TT$ and $[0,1)$ implicitly, but in this case the distinction is important.} of $\TT$ and $[0,1)$. Consider the $\QQ$-linear space $ V := \lin_\QQ \left( \tilde{A} \cup \{1\} \right)$. Let $\tilde{A}_0 \subset \RR$ be a basis of $V$, so that $1 \in \tilde{A}_0$ and any element of $\tilde{A}$ is a $\QQ$-linear combination of elements of $\tilde{A}_0$. Putting $\tilde{A}_1 := \frac{1}{N} \tilde{A}_0$ for properly chosen integer $N$, we can assure that $\tilde{A}_1$ is  $\QQ$-linearly independent, $\frac{1}{N} \in  \tilde{A}_1 $ and each element of $\tilde{A}$ is a $\ZZ$-linear combination of elements of $ \tilde{A}_1 $. Finally, let us write $\tilde{A}_1 = \tilde{B} \cup \{1/N\}$, and let $B \subset \TT$ be the projection of $\tilde{B}$. At is clear that $\Gamma_{1/N} \cap \bigcap_{\alpha \in B} \Gamma_\alpha \subset \bigcap_{\alpha \in A} \Gamma_\alpha$, so it will suffice the former set is non-empty. 
  Because $N \phi(1/N) = \phi(1) = 0$, we have $\phi(1/N) = k/N$ for some $k$. Hence $p \in \Gamma_{1/N}$ if and only if $k+N\ZZ$ is $p$-large. Let us enumerate $B = \seq{\beta}{j}{J}$. Again, a classical theorem ensures the equidistribution of the sequence of vectors $\left(  m N\beta_j \right)_{j \in J} $ (for $m \in \ZZ$) in $\TT^J$, because of the $\QQ$-linear independence of $\{1\} \cup \left\{ N\beta_j\right\}_{J}$. It follows that the vectors $\left(  (m N+k)\beta_j \right)_{j \in J} $ ($m \in \ZZ$) are also equidistributed, and in particular form a dense set.

Hence, there exists a sequence $\seq{m}{t}{\NN}$ such that $\lim_{t \to \infty} (Nm_t+k) \beta_j = \phi(\beta_j)$ for all$j \in J$. It follows that any ultrafilter $p$ for which $\{ Nm_t+k \setsep t \in \NN\}$ is $p$-large, belongs to $\Gamma_{1/N} \cap \bigcap_{\alpha \in B} \Gamma_\alpha$. Since such ultrafilters clearly exists, this finishes the proof.
\end{proof}

The above lemma shows that the class of the maps $\alpha \mapsto \llim{p}{n} n \alpha$ for $p \in \beta \ZZ$ is rather rich: Any map $\TT \to \TT$ which satisfies the necessary condition of being additive can be represented in this form for some $p$.

\index{automatic continuity}
A natural question arises as to the richness of the class of additive maps $\TT \to \TT$. The obvious examples are ``multiplication'' maps $\alpha \mapsto k \alpha$ for some fixed $k \in \ZZ$. It is difficult to think of a different example, and there is a good reason for this. We state the following proposition without the proof, which can be obtained by the suitable adaptation of the classical reasoning for Cauchy functional equation. 

\begin{proposition}\label{B:prop:additive+measurable->multiplication}
\index{torus}
\index{additive map}
  Let $\phi:\ \TT \to \TT$ be an additive map. Then, the following conditions are equivalent:
  \begin{enumerate}
   \item\label{B:prop:a+m->mult:cond:measurable} The map $\phi$ is Lebesgue measurable.
   \item\label{B:prop:a+m->mult:cond:continuous} The map $\phi$ is continuous.
   \item\label{B:prop:a+m->mult:cond:multiplication} The map $\phi$ is of the form $\phi(\alpha) = k \alpha$ for some $k \in \ZZ$.
  \end{enumerate}
\end{proposition}

Of course, the condition \ref{B:prop:a+m->mult:cond:multiplication} implies \ref{B:prop:a+m->mult:cond:measurable}. The implication from \ref{B:prop:a+m->mult:cond:continuous} to \ref{B:prop:a+m->mult:cond:multiplication} is relatively straightforward, and can be deduced from the similar fact for additive maps $\RR \to \RR$. The implication from \ref{B:prop:a+m->mult:cond:measurable} to \ref{B:prop:a+m->mult:cond:continuous} is an example of a more widely discussed phenomenon known as automatic continuity. Much research into this area was done by Frech\'{e}t, Sierpi\'{n}ski and Steinhaus, and more recently by Weil, as is well discussed for example by Rosendal in \cite{AutomaticContinuity-by-Rosendal}. 

One can show by the suitable adaptation of the classical reasoning for Cauchy functional equation that any different additive maps $\TT \to \TT$ is not Lebesgue measurable at any interval. It is relatively straightforward to show that a continuous additive map $\TT \to \TT$ has to be a multiplication by an integer. 

To complete the picture, let us consider the maps $\alpha \mapsto \llim{p}{n} f(n) \alpha$, where $f:\ \ZZ \to \ZZ$ is a fixed polynomial map and $p$ ranging over $\beta \ZZ$. At first, one might again hope that given a non-constant polynomial $f:\ \ZZ \to \ZZ$, any additive map $\TT \to \TT$ is of the form $\alpha \mapsto \llim{p}{n} f(n) \alpha$ for appropriately chosen $p$. Our earlier result shows that this is indeed true for $f(n) = n$. However, taking $f(n) = 2n$ or $f(n) = n^2$ and evaluating $\llim{p}{n} f(n) \alpha$ at $\alpha = \frac{1}{2}$ we see that this naive hope is not realised. However, a slightly weaker statement is true, as shown in the following result.
 
\begin{proposition}\label{B:prop:gen-lim-for-polynomials}\label{B:prop:p-lim-arbitrary-polynomial}
\index{polynomial}
\index{torus}
\index{ultrafilter!generic ultrafilter}
\index{additive map}
  Let $f:\ \ZZ \to \ZZ$ be a non-constant polynomial map. Let $A = \set{\alpha}{i}{I} \subset \TT$ be a sequence such that $A \cup \{1\}$ is linearly independent\footnote{To be precise, we should specify that $\alpha_i \neq \alpha_j$ for $i \neq j$, and that to consider linear independence we take representatives in $[0,1)$. We hope that nevertheless it is clear to the reader what is meant.} over $\QQ$, and let $B = \set{\beta}{i}{I}  \subset \TT$ be arbitrary. Then, there exists an ultrafilter $p$ such that $\llim{p}{n} f(n) \alpha_i = \beta_i$ for all $i \in I$.

\end{proposition}
\begin{proof}
  Define $\Gamma_i := \{ p \in \beta \ZZ \setsep \llim{p}{n} f(n) \alpha_i = \beta_i\}$. It is clear that $\Gamma_i$ are closed, and that the claim will follow once we prove that $\bigcap_{i \in I} \Gamma_i \neq \emptyset$. Because $\beta \ZZ$ is compact, it will suffice to show that the finite intersections $\bigcap_{i \in I_0} \Gamma_i$ ($I_0 \subset I$, finite) are non-empty. Once again, Theorem \ref{B:thm:Weyl-equidistribution} ensures that $\left( f(n) \alpha_i \right)_{i \in I_0} \in \TT^{I_0}$ is equidistributed, hence dense. It follows that there exists a sequence $\seq{n}{t}{\NN}$ such that $\lim_{t \to \infty} f(n_t) \alpha_i = \beta_i$ and consequently there exists an ultrafilter $p$ with $\llim{p}{t} f(n_t) \alpha_i = \beta_i$ for $i \in I_0$. This ultrafilter $p$ lies in $\bigcap_{i \in I_0} \Gamma_i$, which finishes the proof.

\end{proof}

\begin{remark}
\index{polynomial}
\index{torus}
\index{ultrafilter!generic ultrafilter}
\index{additive map}
  Proceeding along similar lines as in Proposition \ref{B:prop:p-lim-arbitrary-linear}, one can modify the above Proposition \ref{B:prop:p-lim-arbitrary-polynomial} to the following statement:

  \textit{Given a  a non-constant polynomial map $f$, and an addive map $\phi:\ \TT \to \TT$, and a set $C \subset \TT$ such that $1$ does not lie in $\QQ$-linear span\footnote{Again, we identify the set $C \subset \TT = \RR/\ZZ$ with the set of representatives of its elements in $[0,1) \subset \RR$.} of $C$, we can find an ultrafilter $p$ such that $\llim{p}{n} f(n) \alpha = \phi(\alpha)$ for $\alpha \in C$.}

  Somewhat regrettably, we cannot take $C = \TT$ in the above statement. 
\end{remark}

To close this section, we use the results obtained so far to obtain somre results about the group structure and cardinality of $\beta \NN$, foreshadowing the latter developements. We begin by re-deriving the formula for the cardinality of $\beta \NN$ is a short way, and show that the idempotent ultrafilters constitute a very small part of $\beta \NN$ in certain sense.

\begin{example}
\index{torus}
\index{ultrafilter!cardinality}

  Let $A = \set{\alpha}{i}{I} \subset \TT$ be such that $1 \cup A$ is $\QQ$-linearly independent, and $\# A = \mathfrak{c}$. We can consider the map from $\Phi:\ \beta \NN  \to \TT^I$, given by $p \mapsto \left( \llim{p}{n} \alpha_\iota n \right)_{i \in I}$, which can easily be verified to be a morhphism of compact commutative semigroups.

	Let us consider the image of $\Phi$, $\Phi(\beta \NN)$. Proposition \ref{B:prop:gen-lim-for-polynomials} asserts that for any choice of $\beta_i \in \TT,\ i \in I$ there exists $p \in \beta \NN$ such that $\llim{p}{n} \alpha_i n = \beta_i$. It follows that for this choice of $p$ we have $\Phi(p) = \seq{\beta}{i}{I}$. Since $\beta_i$ were chosen arbitrarily, we conclude that $\Phi$ is surjective: $\Phi(\beta \NN) = \TT^I$.  In particular, we see that  $\card{ \beta \NN } \geq \card{\TT^I} = \mathfrak{c}^\mathfrak{c} = 2^\mathfrak{c}$. Because the reverse inequality is obvious, we have  $\# \beta \NN = 2^\mathfrak{c}$.

  By Lemma \ref{B:lem:p-lim-of-polynomials-I}, it holds for any idempotent $p \in \beta \NN$ and integer polynomial $f$ with $f(0) = 0$ that:
 $$\Phi(\beta f(p)) = \llim{p}{n} \left(\alpha_i f(n)\right)_{i \in I} =  \left(0\right)_{i \in I} =: 0.$$
 On the other hand, let us consider $\Gamma := \{ p \in \beta \NN \setsep \Phi(p) = 0 \}$. Because the map $\Phi$ is continuous, $\Gamma$ is compact. Because $\Phi$ is a semigroup homomorphism and $\{0\}$ is a semigroup, $\Gamma$ is a semigroup. Moreover, $\Gamma$ contains the idempotent ultrafilters, and even ultrafilters of the form $\beta f(p)$ for $f$ --- polynomial with $0 \mapsto 0$. (We will see that the function $f$ in the last statement can be chosen from an even richer family.) In particular, $\Gamma$ contains the smallest compact semigroup that contains the idempotents.

  We will call a subset of $T \subset \beta \NN$ a \emph{generalised translate} of $\Gamma$ if it is equal to $\Gamma$, or if it consists of a single ultrafilter $p \in \beta \NN$, or if it is of the form $T_1 + T_2$ where $T_1,\ T_2$ are generalised translates of $\Gamma$ constructed earlier. Hence, we are considering sets like $\Gamma + p$, $p + \Gamma$, $p+\Gamma + q$, $p + \Gamma + q + \Gamma$, and so on. It is easily shown by structural induction that if $T$ is a generalised translate of $\Gamma$, then the image $\Phi(T)$ consists of a signle element. It follows by a short argument that $\beta \NN$ cannot be covered by less than $2^\mathfrak{c}$ generalised translates of $\Gamma$. 
\end{example}

%%%%%%%%%%%%%%%%%%%%%%%%%%%%%%%%%%%%%%%%%%%%%
% SECTION
%%%%%%%%%%%%%%%%%%%%%%%%%%%%%%%%%%%%%%%%%%%%%
\section{Almost polynomials}
\index{polynomial!almost polynomial|(}

\renewcommand{\AP}{\cA}
\newcommand{\NAP}{\cA_{0}}
\newcommand{\ap}{almost polynomial}
\newcommand{\nice}{admissible}

\newcommand{\Szemeredi}{Szemer{\'{e}}di}
\newcommand{\Sarkozy}{S{\'{a}}rk{\"{o}}zy}
\newcommand{\Furstenberg}{F{\"{u}}rstenberg}

Polynomials are an extremely well behaved class of functions, satisfying a range of recurrence results. For a trivial example, we have Lemma \ref{B:lem:p-lim-of-polynomials-I} which describes the form of the limit $\llim{p}{x} f(x)$ (for $f$ --- polynomial, $p$ --- idempotent ultrafilter), together with the corollaries concerning the approximation of real-valued polynomials by integers. For a more serious example, one can consider \Furstenberg-\Sarkozy's theorem, asserting that a set of integers with positive Banach density contains two elements differing by the value of any integral polynomial which maps $0$ to $0$. A very profound result which one might hope to generalise is Polynomial \Szemeredi Theorem (see for example \cite{BergLeib-polySzemeredi}), which itself is a generalisation of the classical \Szemeredi Theorem for arithmetic progressions. It is fairly natural to search for a generalisation of the notion of a polynomial which preserves some of the regularity used in the proofs of these results.

In this section, we introduce the notion of ``\ap{s}'' to formalise the idea that a function shows similar behaviour to a polynomial with respect to taking finite differences, but relativised with respect to a chosen ultrafilter $p$. An almost polynomial is essentially synonymous with $p$-$\VIP$ system (modulo a constant term and possibly the generality of definition), but we prefer to use a name that has an intuitive content. Additionally, we take a different point of departure, and only later will it become apparent that our definition is closely related to the classical one. We begin by building some general theory, which will mostly be applied to maps $\ZZ \to \ZZ$. 

The following definition is inspired by Definition \ref{B:def:poly-general-groups}.  

\begin{definition}[Almost polynomials] \label{B:def:gen-poly} \index{almost polynomial}
  Let $(X,+)$ be a commutative semigroup, and let $(Y,+)$ be commutative group, and let $p \in \beta X$ be an ultrafilter. We define the family of $p$-\ap{s} $X \to Y$ inductively, as follows:
\begin{enumerate}
\item A map $f :\ X \to Y$ is a $p$-almost polynomial of degree $-\infty$ if and only if $f = 0_Y$ $p$-a.e.
\item A map $f :\ X \to Y$ is a $p$-almost polynomial of degree $0$ if and only if there is a constant $c \in Y \setminus \{0_Y\}$ such that $f = c$ $p$-a.e.
\item A map $f :\ X \to Y$ is a polynomial of degree at most $d \geq 1$ if and only if for $p$-almost all $a \in X$, the map $\Delta_a f$ is an \ap\ of degree at most $d-1$.
\end{enumerate}
We denote the collection of all $p$-\ap{s} by $ \AP^p(X,Y)$. If $f \in \AP^p(X,Y)$, then $\deg^p f$ denotes the degree of $f$ as $p$-almost polynomial. If $f \not \in \AP^p$, then we define $\deg^p f = + \infty$, so that the statement $\deg^p f \leq d$, $d \in \NN$, includes the assumption that $f \in \AP^p$.
\end{definition}

\begin{convention}
Throughout this section, $X = (X,+)$ will stand for a commutative semigroup, $(Y,+)$ will stand for a commutative group, and $p \in \beta X$ will stand for an ultrafilter on $X$, unless specified otherwise. We abbreviate the notation $\AP^p(X,Y)$ to $\AP^p$ or even $\AP$ when no confusion is possible. Likewise, we omit $p$ in $\deg^p f$ and similar expressions.
\end{convention}

We have seen that the constant term plays an essential role for properties of ordinary polynomials. The following definition gives the right generalisation of the constant term for the \ap{s}.

\begin{definition}\label{B:def:C-f}
   For a map $f:\ X \to Y$, ultrafilter $p \in \beta X$, and integer $d \in \NN$ we define:
$$
  C_d^p(f) := (-1)^d \llim{p}{m_0,\dots,m_d} \DeltaS^d f(m_0,m_1,\dots,m_d),
$$ 
with the understanding that the topology of $Y$ is discrete, and if the limit does not exist in $Y$, then $C_d^p(f)$ remains undefined. If $p$ is understood from the context, we skip the upper index and write simply $C_d(f)$.
\end{definition}

We now prove several results which show why $C_d(f)$ is an interesting object.

\begin{proposition}\label{B:lem:C-vs-gen-poly}
  In the situation of the above definition, for $d \geq 1$ it holds for $p$-a.a. $a \in X$ that $C_{d}(f) = -C_{d-1}(\DeltaS_a f)$, in the sense that is one of the limits converges then so does the other, and the values agree. In particular, $f \in \cA$ and $\deg^p f \leq d$ if and only if $C_d(f)$ exists.
\end{proposition}
\begin{proof}
  The first part of the statement follows from the observation that:
  $$ - \llim{p}{a} C_{d-1}(\DeltaS_a f) = (-1)^d \llim{p}{a} \llim{p}{m_0,\dots,m_{d-1}} \DeltaS^d f(m_0,m_1,\dots,m_{d-1},a) $$
  which is the same as the definition of $C_{d}(f)$, up to renaming variables and using the symmetry of $\DeltaS^d f$. Because $Y$ is given the discrete topology, we for $p$-a.a. it holds that  $-C_{d-1}(\DeltaS_a f) = C_{d}(f)$, assuming either limit exists.

For the second part of the statement, we use induction. The case $d = 0$ is clear, so let us suppose $d \geq 1$. Then, for $p$-a.a. $a \in X$, the statement that $\deg^p f \leq d$ is equivalent to $\deg^p \DeltaS_a f \leq d-1$, which by induction is equivalent to existence of $C_{d-1}(\DeltaS_a f) = C_d(f)$.
\end{proof}

\begin{lemma}\label{B:lem:C-independent-of-d}
  If $f \in \AP$ and $\deg f \leq d$, and $p \in \beta X$ is idempotent, then $C_{d}(f) = C_{\deg f}(f)$. 
\end{lemma}
\begin{proof}
  We will show that for $d \geq \deg f$ it holds that $C_{d+1}(f) = C_d(f)$. Once this is shown, the rest of the claim follows by simple induction.

  We begin by writing out the formula for $C_{d+1}(f)$. It will be convenient to distinguish two of the variables by giving them different names. Note that we may shuffle the variables at the first step because of the symmetry of $\DeltaS^{d+1}f$.
\begin{align*} 
C_{d+1}(f) &= (-1)^{d+1} \llim{p}{a}\llim{p}{b} \llim{p}{m_1,\dots,m_d} \DeltaS_{a} \DeltaS_{m_{1},m_2 \dots, m_d} f (b)
\\ &= (-1)^{d+1} \llim{p}{a}\llim{p}{b} \llim{p}{m_1,\dots,m_d} 
\DeltaS_{m_{1},m_2 \dots, m_d} f (b+a) \\
&+  (-1)^{d} \llim{p}{a}\llim{p}{b} \llim{p}{m_1,\dots,m_d}  \DeltaS_{m_{1},m_2 \dots, m_d} f (a) \\
&+  (-1)^{d} \llim{p}{a}\llim{p}{b} \llim{p}{m_1,\dots,m_d}  \DeltaS_{m_{1},m_2 \dots, m_d} f (b). 
	\end{align*}
Note that in the first of the three resulting summands, $a$ and $b$ occur only in the expression $a+b$, so using the idempotence of $p$, this can be condensed to:
	$$ (-1)^{d+1} \llim{p}{n} \llim{p}{m_1,\dots,m_d} 
\DeltaS_{m_{1},m_2 \dots, m_d} f (n) = - C_d(f).$$
The remaining two limits are equal to $C_d(f)$, because the limits over non-occurring variables can be cancelled. The claim follows:
	$$ C_{d+1}(f) = - C_d(f) + C_d(f) + C_d(f) = C_d(f). \qquad \qedhere$$ 
\end{proof}

\begin{remark}
  Note that the proof relies on the idempotence of $p$ already for $\deg f = 0$. Indeed, we have:
  $$ C_1^p(f) = -\llim{p}{m,n} \DeltaS_m f(n) = -\llim{p}{m,n}\left( f(n+m) - f(n) - f(m) \right) = -C_0^{p+p}(f) + 2 C_0^p(f),$$
  which is not the same as $C_0^p(f)$ in general, unless $p+p = p$. For a concrete example, take $X = \NN$, $Y = \ZZ$ and $f(n) = n \cdot \chi_{2 \NN}(n)$ (that is $f(2m) = 2m$ and $f(2m+1) = 0$), and let $p$ be such that $p \in \bar{2\NN +1}$. Then, $f = 0$ $p$-a.e., so $\deg f = -\infty$ and $C_0^p(f) = 0$. However, because $(p + p) \in \bar{2 \NN }$, we have $f(n) = n$ for $(p+p)$-a.a. $n$. Hence $f$ is definitely \emph{not} constant almost everywhere with respect to $p+p$, and $ C_1^p(f) = -C_0^{p+p}(f)$ remains undefined. 
\end{remark}

\begin{convention}
  From this point on we assume that $p$ is idempotent, except when explicitly mentioned otherwise.
 
  The above lemma shows that $C_d(f)$ does not depend on $d$, provided that $d$ is large enough for $C_d(f)$ to be defined. Hence, we shorten the notation to $C(f)$ when $d$ is immaterial.
\end{convention}

  Polynomials which map $0$ to $0$ exhibit particularly nice properties with respect to recurrence. The following is the analogue for \ap{s}.

\begin{definition}\label{B:def:gen-poly-admissible}
\index{polynomial!almost polynomial!admissible}
 An \ap\ $f$ is said to be \nice\ if and only if $C(f) = 0$. We denote the set of \nice\ \ap\ by $\NAP^p$, or $\NAP$ if $p$ is understood from the context.
\end{definition}

\begin{example}	
If $f:\ X \to Y$ is a polynomial map in the sense of Definition \ref{B:def:poly-general-groups}, then $f$ is an \ap, and $\deg f \leq \deg^p f$, where $\deg$ stands for the degree of $f$ as a polynomial. Moreover, in this case $f$ is \nice\ if and only if $f(0) = 0$, because for any $a \in X$ we have $\Delta_a f(0) = - f(0)$.
\end{example}

\begin{lemma}\label{B:lem:gen-poly-admissible-vs-C}
  If $f:\ X \to Y$ is \ap\ from commutative group $X$ to commutative group $Y$ then $C(f)$ is the unique constant such that $f - C(f)$ is an \nice\ \ap.
\end{lemma}
\begin{proof}
This is a direct consequence of how $\DeltaS_a$ acts on constants, and linearity of $C$.
\end{proof}

\begin{proposition}
  Let $f,g:\ X \to Y$ be maps such that $f = g$ $p$-a.e. Then $\deg f = \deg g$ and $C(f) = C(g)$.
\end{proposition}

\begin{proof}
	The claim will follow immediately, once we show that for any $d \in \NN$ we have the equality:
	$$ \llim{p}{m_0,\dots,m_d} \left( \DeltaS^d f(m_0,m_1,\dots,m_d) - \DeltaS^d g(m_0,m_1,\dots,m_d) \right) = 0, $$
	which implies in particular that $C_d(f) = C_d(g)$ provided that either constant is defined.
	Using the formula for $\DeltaS^d$ from Lemma \ref{B:lem:DeltaS^k-explicit}, we observe that it suffices to show that for any index set $I \subset \{m_0,m_1,\dots,m_d\}$ we have the equality:
	$$
	\llim{p}{m_0,\dots,m_d} \left( f \left( \sum_{i \in I} m_i \right) - g \left( \sum_{i \in I} m_i \right) \right) = 0. 
	$$	
	Note that for $i \not \in I$, the expression whose limit we are taking is independent of $m_i$, and thus the operation of taking $\llim{p}{m_i}$ is just the identity, and may thus be dropped. Using the symmetry, we may assume that $I = \{0,1,\dots,r\}$ for some $r$, which reduces the problem to showing that:
	$$
	\llim{p}{m_0,\dots,m_r} \left( f\left( \sum_{i =1}^r m_i \right) - g \left( \sum_{i =1}^r m_i \right) \right) = 0. 
	$$	
	Using the idempotence of $p$, we see that this is equivalent to:
	$$
	\llim{q}{n} \left( f( n ) - g(n) \right) = 0, 
	$$	
	or, more naturally, $f = g$ $q$-a.e., where $q = p+p+\dots+p$, where $r$ copies are in place. Thanks to the fact that $p$ is idempotent, we have $q = p$, so we arrive at the assumption. Hence, the claim holds. 
\end{proof}

\begin{remark}
	If it was not the case that $p$ is idempotent, we would need a much stronger condition, of $f$ and $f'$ being equal $p$-a.e., $(p+p)$-a.e., and generally $(p+p+\dots+p)$-a.e. for any number of repetitions of $p$. It should be taken as a strong hint that in our considerations, the assumption of idempotence cannot be weakened.
\end{remark}

\begin{observation}
	If $f,f' \in \AP$, then $f+f' \in \AP$. Moreover, $C(f+f') = C(f) + C(f')$ and $\deg(f+f') \leq \max\{\deg f, \deg f'\}$.
\end{observation}
\begin{proof}
	Let $d := \max\{\deg f, \deg f'\}$. It is clear from the previous considerations that:
	$$ C_d(f+f') = C_d(f) + C_d(f') = C(f) + C(f'). $$ 
Hence, the claim follows.
\end{proof}

\begin{observation}
	Suppose that $f \in \AP(X,Y)$ is an almost polynomial and $g \in \Hom(Y,Z)$ is a morphism of commutative groups. Then $g \circ f \in \AP(X,Z)$ is again an almost polynomial. Moreover, $\deg g \circ f \leq \deg f$ and $C(g \circ f) = g(C(f))$.
\end{observation}
\begin{proof}
	It is clear that for $d := \deg f$ we have:
	$$ C_d(g \circ f) = g\left( C_d(f) \right).$$
	Hence, the claim follows.
\end{proof}

\begin{lemma}\label{B:lem:gen-poly-closure}
	Let $f \in \AP(X,Y)$ and $f' \in \AP(X,Y')$, where $Y,Y'$ are both commutative groups. Suppose additionally that a bi-additive\footnote{By bi-additive map we mean that with one argument fixed, the map is a morphism of semigroups.} map $Y \times Y' \ni (y,y') \mapsto y \cdot y' \in Z$ is defined, where $Z$ is a commutative group. Then, $f \cdot f' :\ X \to Z$ (defined pointwise), is an almost polynomial. Moreover, $\deg f \cdot f' \leq \deg f + \deg f'$ and $C(f \cdot f') = C(f) C(f')$.
\end{lemma}
\begin{proof}
We begin by considering a few special cases. In the case when $\deg f = -\infty$ or $\deg f' = -\infty$, then $f\cdot f' = 0_Z$ $p$-a.e., and hence $f\cdot f' \in \AP(X,Z)$ and $\deg f \cdot f' = -\infty$. In the case when $\deg f = 0$, there exists a constant $c \in Y$ such that $f = c$ for $p$-a.e.. It follows that $f \cdot f' = c \cdot f'$ $p$-a.e.. Because the map $y' \mapsto c \cdot y'$ is a morphism, the claim follows from the above observation. The same reasoning applies when $\deg f' = 0$.

	For the general case, we use induction on the degrees $(\deg f, \deg f')$. Proving the theorem for $f,f'$, we may assume that the claim holds for $g,g'$ with $\deg g + \deg  g' < \deg f + \deg f'$. Because of the above considerations, we may assume that $\deg f > 0$ and $\deg f' > 0$.

	To show that $f\cdot f' \in \AP(X,Z)$, it will suffice to check that $\DeltaS_a (f\cdot f') \in \AP$ for $p$-a.a. values of $a$. Directly by writing out the formulas, one can verify that:
	\begin{equation}
	\DeltaS_a (f\cdot f') =
		\left( \DeltaS_a f + f(a) \right) \cdot \left( \DeltaS_a f' + f' \right)
		+ \left( \DeltaS_a f + f \right) \cdot f(a) + f \cdot \DeltaS_a f'.
	\label{B:lem:product-eq:01}
	\end{equation} 
	From the inductive Definition \ref{B:def:gen-poly}, it follows that for $p$-a.a. values of $a$ we have the expected bounds for degrees: $\deg \left( \DeltaS_a f + f(a) \right) \leq \deg f - 1$, $\deg \left( \DeltaS_a f' + f' \right) \leq \deg f' - 1$, $ \deg \left( \DeltaS_a f + f \right) \leq \deg f $, $\deg \DeltaS_a f' \leq \deg f' - 1$. This means that the inductive assumption can be applied to each of the three products in equation \eqref{B:lem:product-eq:01}. It follows that $f \cdot f' \in \AP(X,Z)$. More precisely, we have:
\begin{align*}
	\deg \DeltaS_a (f\cdot f') &\leq \max\Big\{
		\deg \left( \DeltaS_a f + f(a) \right) \cdot \left( \DeltaS_a f' + f' \right),
		\\ & \qquad \deg \left( \DeltaS_a f + f \right) \cdot f(a) , \ \deg f \cdot \DeltaS_a f' \Big\}
  \\ & \leq \deg f + \deg f' - 1.
	\end{align*} 
Hence $\deg f\cdot f' = \deg \DeltaS_a (f\cdot f') +1 \leq \deg f + \deg f' $, which proves one part of the claim.

For the final part of the claim, we again use equation \eqref{B:lem:product-eq:01} and the inductive assumption. Note that we have:
	 $$C(\DeltaS_a f + f)= C(\DeltaS_a f) + C(f) = -C(f) + C(f) = 0.$$
and likewise $C(\DeltaS_a f' + f') = 0$. This identity, together with the inductive assumption, allow us to perform the following computation:
	\begin{align*}
	C\left(\DeltaS_a \left(f\cdot f'\right)\right) &=
		C\left( \left( \DeltaS_a f + f\left(a\right) \right) \cdot \left( \DeltaS_a f' + f' \right) \right)
		+ C\left( \left( \DeltaS_a f + f \right) \cdot f\left(a\right)\right)
  	   \\& \quad  + C\left( f \cdot \DeltaS_a f'\right) \\
&=
		C \left( \DeltaS_a f + f\left(a\right) \right) \cdot C\left( \DeltaS_a f' + f' \right) 
		+ C \left( \DeltaS_a f + f \right) \cdot C\left(f\left(a\right)\right)
  	   \\& \quad  + C\left( f \right) \cdot C\left(\DeltaS_a f'\right)
\\&= - C\left(f\right)C\left(f'\right)
	\end{align*} 
Hence, we have $C(f\cdot f') = - C\left(\DeltaS_a (f\cdot f')\right) = C(f)C(f')$.
\end{proof}

\begin{corollary}
	Suppose that $(Y,+,\cdot)$ is a ring. If $f,f' \in \NAP$ and $g \in \AP$, then $f+f' \in \NAP$ and $f \cdot g \in \NAP$. In other words, $\NAP$ constitutes an ideal in $\AP$.
\end{corollary}
\begin{proof}
	By previous results we have $C(f+f') = C(f) + C(f') = 0$ and $C(f \cdot g) = C(f)C(g) = 0$.
\end{proof}

To finish our considerations on \ap{s}, we cite a fundamental structure theorem for such maps. The following result is taken from \cite{BergMc2010-SzThm-GenPoly}, and needs to be modified slightly to fit our treatment. We denote by $\cP_{d}(S)$ the family of non-empty subsets of $S$ with cardinality at most $d$. 
\begin{theorem}\label{B:thm:gen-poly-characterisation-finite-subsets}
  Let $X$ be a commutative semigroup and $Y$ --- a commutative group, let $p \in \beta X$ be idempotent, let $f:\ X \to Y$ be a map, and fix $d \in \NN$. Then, the following conditions are equivalent:
  \begin{enumerate}
   \item\label{B:thm:gen-poly-char-fs:cond:1} The map $f$ is an \ap\ with degree at most $d$: $f \in \NAP^p(X,Y)$ and $\deg^p f \leq d$.
   \item\label{B:thm:gen-poly-char-fs:cond:2} There exists a map $u:\cP_d(X) \to Y$ and a constant $c \in Y$ such that the following formula is satisfied for any $r \in \NN$:
    $$ \llim{p}{a_1,\dots,a_r} f\left( \sum_{i=1}^r a_i\right) - \sum_{\substack{ \alpha \in \cP_{\!d}\left( X \right) \\ \alpha \subset \{a_i\}_{i=1}^r}} u(\alpha) = c.$$
  \end{enumerate}
  Moreover, if the above conditions are satisfied, then $c = C(f)$.
\end{theorem}
\begin{proof}[Partial proof]
  \begin{description}
   \item[\ref{B:thm:gen-poly-char-fs:cond:2} $\imply$ \ref{B:thm:gen-poly-char-fs:cond:1}.]
   If \ref{B:thm:gen-poly-char-fs:cond:2} holds, then we can replace $f\left( \sum_{i=1}^r a_i\right)$ by the expression $\sum_{\substack{ \alpha \in \cP_{\!d}\left( X \right), \alpha \subset \{a_i\}_{i=1}^r}} u(\alpha) + c$ under the generalised limit $\llim{p}{a_1,\dots,a_r}\!\!$. This allows us to compute the limit $$\llim{p}{x_0,\dots,x_d} \DeltaS^d f(x_0,\dots,x_d),$$ using the explicit formula from Lemma \ref{B:lem:DeltaS^k-explicit}. Once we prove that the above limit exists, it will follow from previous considerations that $f \in \NAP^p$ with $\deg^p f \leq d$, and the value of the limit equals $(-1)^{d}C^p(f)$. For brevity of notation, let $u(\emptyset) := c$. We can  compute that:
  \begin{align*}
    \llim{p}{x_0,\dots,x_d} \DeltaS^d f(x_0,\dots,x_d) 
  &= \llim{p}{x_0,\dots,x_d} \sum_{\substack{ I \subset [d+1] \\ I \neq \emptyset }} (-1)^{d+1-\card{I}} f\left( \sum_{i \in I} x_i \right) \\ 
  &= \llim{p}{x_0,\dots,x_d} \sum_{\substack{ I \subset [d+1] \\ I \neq \emptyset }} (-1)^{d+1-\card{I}} \sum_{\substack{\alpha \in \cP_d(X)\cup\{\emptyset\} \\ \alpha \subset \{x_i\}_{i \in I}}}u(\alpha). \\
%  &= \llim{p}{x_0,\dots,x_d} \sum_{\substack{\alpha \in \cP_d(X)\cup\{\emptyset\}}} u(\alpha) \sum_{\substack{ I \subset [d+1] \\ \{x_i: i \in I\} \supset \alpha,\  I \neq \emptyset  }} (-1)^{k+1-\card{I}} \\
%  &= (-1)^{d+1} c + \llim{p}{x_0,\dots,x_d} \sum_{\substack{\alpha \in \cP_d(X)\cup\{\emptyset\}}} u(\alpha) \sum_{\substack{ I \subset [d+1] \\ \{x_i: i \in I\} \supset \alpha }} (-1)^{d+1-\card{I}} 
  \end{align*}
  The restriction $I \neq \emptyset$ in the above sum is awkward, but we can dispose of it by artificially adding the term $(-1)^{d+1} c$ to both sides. After that, we can easily change the order of summation, leading to:
  \begin{align*}
    \llim{p}{x_0,\dots,x_d} \DeltaS^d f(x_0,\dots,x_d) - (-1)^d c 
    &= \llim{p}{x_0,\dots,x_d} \sum_{\substack{\alpha \in \cP_d(X)\cup\{\emptyset\}}} u(\alpha) \sum_{\substack{ I \subset [d+1] \\ \{x_i: i \in I\} \supset \alpha }} (-1)^{d+1-\card{I}}.
  \end{align*}
  For a fixed $\alpha$ the inner sum over $I$ can be computed explicitly. If we allow $I(\alpha) := \{ i \setsep x_i \in \alpha \}$ then the sum can be rewritten as:
  \begin{align*}
    \sum_{\substack{ I \subset [d+1] \\ \{x_i: i \in I\} \supset \alpha }} (-1)^{d+1-\card{I}} 
    = \sum_{I(\alpha) \subset I \subset[d+1]} (-1)^{d+1-\card{I}}
    = (-1)^{d+1-\card{I(\alpha)}} \sum_{J \subset[d+1]\setminus I(\alpha)} (-1)^{\card{J}}.
  \end{align*}
  If we denote $m := d+1-\card{I(\alpha)}$ and group the terms in the above sum with respect to $j := \card{J}$, we see that:
  $$ \sum_{J \subset[d+1]\setminus I(\alpha)} (-1)^{\card{J}} = \sum_{ 0 \leq j \leq m} (-1)^{j} \binom{m}{j} = (1 - 1)^m = 0.$$
  Note that we rely on $m > 0$, which is a consequence of $\card{\alpha} \leq d$. Hence, the previously considered limit trivialises:
    \begin{align*}
    \llim{p}{x_0,\dots,x_d} \sum_{\substack{\alpha \in \cP_d(X)\cup\{\emptyset\}}} u(\alpha) \sum_{\substack{ I \subset [d+1] \\ \{x_i: i \in I\} \supset \alpha }} (-1)^{d+1-\card{I}} = 0_Y.
  \end{align*}
  This leads to the desired formula, finishing the proof of this implication: 
\begin{align*}
    \llim{p}{x_0,\dots,x_d} \DeltaS^d f(x_0,\dots,x_d) = (-1)^d c.
\end{align*}

\item[\ref{B:thm:gen-poly-char-fs:cond:1} $\imply$ \ref{B:thm:gen-poly-char-fs:cond:2}.]   
  See \cite{BergMc2010-SzThm-GenPoly}. \qedhere
  \end{description}
\end{proof}

The above result can lead to shorter proofs of some of our claims, most notably Lemma \ref{B:lem:gen-poly-closure}. We choose a different perspective, relying more on induction than explicit structure theorems. The function $u$ in the above theorem is sometimes referred to as the generating function for $f$. One of the consequences of the implication we proved is a practical way of verifying that a given function $f$ is an \ap: it suffices to find the corresponding generating function $u$ and check the relation in \ref{B:thm:gen-poly-char-fs:cond:2}. A word of warning is due at this point: it is not the case that for an arbitrary choice of the function $u:\ \cP_d(X) \to Y$ one can find a corresponding \ap\ $f \in \NAP(X,Y)$.

%%%%%%%%%%%%%%%%%%%%%%%%%%%%%%%%%%%%%%%%%%%%%
% SECTION
%%%%%%%%%%%%%%%%%%%%%%%%%%%%%%%%%%%%%%%%%%%%%
\section{Integer almost polynomials}

We now turn to applications of the previously introduced theory to maps $\ZZ \to \ZZ$. Similar results can be obtained for multivariate polynomials $\ZZ^k \to \ZZ^l$, but we sacrifice some of the generality for the sake of simplicity.

Our goal is to generalise previous results about integer-valued polynomials. We would like, in particular, to allow for non-rational coefficients. Since multiplication on the torus by a non-integer is not well defined, we need to incorporate a notion of the integral part. It seems to be the most morally justified to use the ``closest integer'' function, rather that the ``floor'' function, since the former is better behaved, as shall be seen in the considerations below.

\begin{definition}\index{integer part}
  For $x \in \RR$ by $\ci{x}$ we denote the closest integer to $x$, given by $\ci{x} := \left\lfloor{x + \frac{1}{2}}\right\rfloor$. By $\fp{x}$ we denote the ``fractional part'': $\fp{x} := x - \ci{x}$. 
\end{definition}

The following result is the aforementioned generalisation of Lemma \ref{B:lem:p-lim-of-polynomials-I} to the context of generalised polynomials. 

\begin{proposition}\label{B:lem:p-lim-of-gen-poly}
  If $f:\ \ZZ \to \ZZ$ is an almost polynomial then for any $\alpha \in \TT$ it holds that 
  $$ \llim{p}{n} \alpha f(n) = \alpha C(f). $$
	In particular, if $f$ is \nice, then:
  $$ \llim{p}{n} \alpha f(n) = 0.$$	
\end{proposition}
\begin{proof}
  We reason in full analogy to the case of ordinary polynomials, and use induction on $d$. Let $\lambda$ denote the sought limit $\llim{p}{n} \alpha f(n)$. If $d \leq 0$, then $\alpha f(n) = C(f)$ for $p$-a.a. $n$, so clearly $\lambda = \alpha C(f)$. Thus, we may suppose $f$ is an \ap\ of degree $d \geq 1$, and the claim holds for \ap{s} of all smaller degrees.

  Using idempotence of $p$ and elementary transformations, we find:
  \begin{align*}
  \lambda = \llim{p}{n} \alpha f(n) 
  &=  \llim{p}{m} \llim{p}{n} \alpha f(n+m) 
  \\&= \llim{p}{m} \llim{p}{n}\alpha f(n) + \llim{p}{m} \llim{p}{n}\alpha f(m) + \\ &\phantom{\mbox{} = \mbox{}} \llim{p}{m} \llim{p}{n}  \alpha \DeltaS f(n,m) 
	\\&= 2 \lambda + \llim{p}{m} \llim{p}{n}  \alpha \DeltaS_m f(n) .
  \end{align*}
  Using Lemma \ref{B:lem:gen-poly-closure}, for $p$-a.a. $m$, the expression $\DeltaS_m f(n)$ is an \ap\ in $n$ with degree strictly smaller than $d$, with $C(\DeltaS_m f(n)) = - C(f)$. Thus, the inductive assumption applies, and:
  $$  \llim{p}{m} \llim{p}{n}  \alpha \DeltaS_m f(n) = \llim{p}{m} \alpha C(\DeltaS_m f(n)) = - \alpha C(f) .$$
  Hence, the above computation leads to:
  $$ \lambda = 2 \lambda - \alpha C(f).$$
  Thus, $\lambda = \alpha C(f)$, as claimed. 
\end{proof}

\begin{corollary}\label{B:cor:gen-poly-adm-closure-wrt-floor}
  For any scalars $\alpha_i \in \RR$, and any \ap{s} $f_i$ of degree $d_i$, where $i = 1,2, \dots, N$, consider the function $g:\ \ZZ \to \RR$ given by the formula:
  \begin{equation}    g(n) :=  \sum_{i=1}^N \alpha_i f_i(n).
    \label{eq:claim-002}
  \end{equation}
	Suppose additionally that $\abs{ \sum_{i=1}^N \alpha_i C(f_i) } < \frac{1}{2}$. Then, we have:
 $$\llim{p}{n} \fp{g(n)} = \sum_{i=1}^N \alpha_i C(f_i).$$
\end{corollary}
\begin{proof}
  From the previous theorem, we know that both $\llim{p}{n} \fp{g(n)}$ and $\sum_{i=1}^N \alpha_i C(f_i)$ represent the same element of $\TT$, and lie in $(-\frac{1}{2},\frac{1}{2})$. Hence, they are equal.
\end{proof}

The above Proposition \ref{B:lem:p-lim-of-gen-poly} describes behaviour of limits of \ap{s}. However, it does not give any description of $\AP$ other than the somewhat indirect one in Definition \ref{B:def:gen-poly}. We shall now give an operation that can be used to obtain \ap{s} that are not ordinary polynomials.

\begin{lemma}
  For arbitrary scalars $\alpha_i \in \RR$, and functions $f_i \in \AP$, where $i = 1,2, \dots, N$, consider the function $g:\ \ZZ \to \RR$ given by the formula:
  \begin{equation}
    g(n) :=  \sum_{i=1}^N \alpha_i f_i(n).
    \label{eq:claim-003}
  \end{equation}
  Then, $\ci{g} \in \AP$ and $\deg g \leq \max_i d_i$. What is more, if additionally we assume that $\abs{ \sum_{i=1}^N \alpha_i C(f_i) } < \frac{1}{2}$, then $\ci{g} \in \NAP$.
\end{lemma}
\begin{proof}
  Let us denote $d := \max_i \deg f_i$; we will apply induction with respect to $d$. The case $d = -\infty$ is clear, because then for all $i$ we have $f_i = 0$ $p$-a.e..
  If $d = 0$, then for all $i$ it holds that $f_i = C(f_i)$ $p$-a.e., and hence:
  \begin{equation*}   
    g = \sum_{i=1}^N \alpha_i C(f_i) \quad \text{$p$-a.e.}
  \end{equation*}
	Thus, $g$ is constant $p$-a.e., and so is $\ci{g}$. It follows that $\ci{g} \in \AP$ and $\deg \ci{g} = 0$. Moreover, if the additional assumption holds, then $g \in \left(-\frac{1}{2},\frac{1}{2}\right)$, and hence $\ci{g} = 0$ $p$-a.e., and consequently $\ci{g} \in \NAP$.

  For the inductive step with $d \geq 1$, let us note that:
  \begin{align*}
	\DeltaS_m \ci{g} &= \DeltaS_m \left( g- \fp{g } \right) 
	=\DeltaS_m g - \DeltaS_m  \fp{ g } 
	= \ci{ \DeltaS_m g} + \fp{\DeltaS_m g} -  \DeltaS_m  \fp{ g } 
	\\ &= \ci{ \DeltaS_m g} + \fp{\DeltaS_m \fp{g} } -  \DeltaS_m  \fp{ g }
	= \ci{ \DeltaS_m g} - \ci{ \DeltaS_m  \fp{ g } }.
  \end{align*} 
  Note that we have $\DeltaS_m g = \sum_{i=1}^N \alpha_i \DeltaS_m f_i$. Because $\deg \DeltaS_m f_i = \deg f_i - 1$ for $p$-a.a. $m$ (provided that $\deg f_i \geq 1$) the inductive assumption can be applied to conclude that $\ci{ \DeltaS_m g} \in \AP$ and $\deg \ci{ \DeltaS_m g} \leq d-1$. Moreover, because for any $m$ it holds (pointwise) that $\abs{ \DeltaS_m  \fp{ g } } < \frac{3}{2}$, we conclude that $\ci{ \DeltaS_m  \fp{ g } } \in \{-1,0,1\}$, and consequently $\ci{ \DeltaS_m  \fp{ g } } \in \cA$ with $\deg \ci{ \DeltaS_m  \fp{ g } } \leq 0$. Therefore, we have for $p$-a.a. $m$:
  $$ \deg \DeltaS_m \ci{g} \leq \max \left\{ \deg \ci{ \DeltaS_m g} , \deg \ci{ \DeltaS_m  \fp{ g } } \right\} \leq d-1. $$
  Hence $\ci{g} \in \AP$ and $\deg \ci{g} = \deg \DeltaS_m \ci{g} + 1 \leq d$.

  Let us now suppose that the additional assumption holds, so that we have $$\gamma := \sum_{i=1}^N \alpha_i C(f_i) \in \left( -\frac{1}{2},\frac{1}{2} \right).$$ Thanks to the above Corollary \ref{B:cor:gen-poly-adm-closure-wrt-floor}, we have $\llim{p}{n} \fp{g(n)} = \gamma$.  
 We can then compute:
    \begin{align*}
      \llim{p}{m} \llim{p}{n} \DeltaS_m \fp{g}(n) &= \llim{p}{m} \llim{p}{n} \left( \fp{g(n+m)} - \fp{g(n)} - \fp{g(m)} \right) = - \gamma.
    \end{align*}
    Because $\abs{\gamma} < \frac{1}{2}$, the closest integer map $\ci{\cdot}$ is continuous at $\gamma$, and thus: 
    \begin{align*}
      \llim{p}{m} \llim{p}{n} \ci{ \DeltaS_m  \fp{ g } (n)} =  \ci{-\gamma} = 0.
    \end{align*}
  Hence, for $p$-a.a. $m$ we have $\ci{ \DeltaS_m  \fp{ g } } \in \NAP$. (Let us remark that this part of the proof works under a weaker assumption $\sum_{i=1}^N \alpha_i C(f_i) \not \in \ZZ + \frac{1}{2}$.) 

  For the other term, we notice that for $p$-a.a. $m$ it holds that $C(\DeltaS_m f_i) = - C(f_i)$, and hence the the additional assumption also implies that $ \ci{ \DeltaS_m g} \in \NAP$. Therefore:
  $$ C(\ci{g}) = - \llim{p}{n} C(\DeltaS_m \ci{g}) = - \llim{p}{n}\left(  C( \ci{\DeltaS_m g}) - C( \ci{\DeltaS_m \fp{g}}) \right) = 0 .$$
\end{proof}

\begin{corollary}
	For any constants $\alpha_i \in \RR$ with $i =0,1,\dots,N$, the function $f:\ \ZZ \to \ZZ$ given by:	
	$$ g(n) := \sum_{i = 0}^N \alpha_i n^i, \qquad\qquad f(n) := \ci{ g(n) } $$	
is an \ap. Moreover, $f$ is an \nice\ \ap, provided that $\abs{\alpha_0} < \frac{1}{2}$.
\end{corollary}

\begin{corollary}
\begin{enumerate}
 \item The class of \ap\ is closed under taking sums, products, and the operation $f \mapsto \ci{ \alpha f + \beta}$ for $\alpha, \beta \in \RR$.

 \item The class of \nice\ \ap\ is closed under sums, multiplication by an \ap, and the operation $f \mapsto \ci{ \alpha f + \beta}$ for $\alpha, \beta \in \RR$ with $\abs{\beta} < \frac{1}{2}$.

 \item	Any function constructed by applying these operations, starting with ordinary polynomials, is an (\nice) \ap\ regardless of the choice of the idempotent ultrafilter $p$. In particular, if $f:\ \ZZ \to \ZZ$ is thus constructed \nice\ \ap\ and $\alpha \in \RR$ is a constant, then the set 
	$$A_\eps := \{ n \in \ZZ \setsep \operatorname{dist}(\alpha f(n),\ZZ ) <\eps\}$$ 
is $\IPst$ for all $\varepsilon > 0$.
\end{enumerate}
\end{corollary}

\begin{remark}
	In the above results, it was essential that if $f$ is \nice\ then so is  $\ci{ \alpha f + \beta}$ for $\alpha, \beta \in \RR$, provided that $\abs{\beta} < \frac{1}{2}$. It is natural to inquire what happens in the case of more general values of $\beta$. If $\beta = b + \beta_0$, with $\abs{\beta_0} < \frac{1}{2}$ and $b \in \ZZ$, then $\ci{ \alpha f + \beta} =\ci{ \alpha f + \beta_0} + b$, so although $\ci{ \alpha f + \beta}$ is not \nice, it can be made \nice\ by subtracting a constant. This essentially reduces the question to $\beta = \frac{1}{2}$, or equivalently to considering the almost polynomial $\floor{\alpha f}$. It is clear from above considerations that, for a fixed idempotent $p$, either $\floor{\alpha f}$ is \nice, or $\floor{\alpha f} + 1$ is \nice. It is, however, not the case that one of those functions is \nice\ for any idempotent $p$. This is an obstacle to sets of recurrence or good approximation being $\IPst$.
\end{remark}

\begin{example}
	Let $\alpha \in \RR$ be irrational. Then there exist two idempotent ultrafilters $p$ and $q$, such that $n \mapsto \floor{\alpha n} = \ci{\alpha n - \frac{1}{2}}$ is \nice\ with respect to $p$, and  $n \mapsto \ceil{\alpha n}  = \ci{\alpha n + \frac{1}{2}}$ is \nice\ with respect to $q$. Moreover, for any idempotent ultrafilter, exactly one of these functions is \nice.

	In particular, for any constants $\beta \in \RR \setminus \ZZ$, $c \in \{0,1\}$, and $\eps > 0$, the set of $n \in \ZZ$ such that $ \beta \left( \floor{\alpha n} + c \right)$ is $\eps$-close to an integer:
	$$A_\varepsilon(c) := \{ n \in \ZZ \setsep \operatorname{d}\left( \beta \left( \floor{\alpha n} + c \right), \ZZ\right) < \eps \}$$
	is not an $\IPst$ set, although $A_\eps(0) \cup A_\eps(1)$ is an $\IPst$ set. 
\end{example}
\begin{proof}
	The key observation is that $\fp{\alpha n}$ approaches $0$ along idempotent ultrafilters, and the limit value can be approached either from above or from below.

	More precisely, note that since $\alpha$ is irrational, a standard result shows that the sequence $\{ \fp{ \alpha n } \}_{n \in \NN}$ is equidistributed in $\TT$. For a sequence $\seq{\eps}{i}{\NN}$ with $\eps_i > 0$ and $\sum_{i} \eps_i < \frac{1}{2}$, we may choose $n_i \in \NN$ such that $\fp{\alpha n_i} \in (0,\eps_i)$. For $n \in \FS{\seq{n}{i}{\NN}}$ we then have $\fp{\alpha n} \in \left(0,\sum_{i} \eps_i\right) \subset \left( 0, \frac{1}{2} \right)$. By Lemma \ref{C:lem:IP<=in-idempotent}, there exists an idempotent $p$ such that $\FS{\seq{n}{i}{\NN}} \in p$, and in particular $\fp{\alpha n} \in \left( 0, \frac{1}{2} \right)$ for $p$-a.a. $n$. It follows that $\ci{\alpha n - \frac{1}{2}} = \ci{\alpha n}$ for $p$-a.a. $n$, and hence the function $n \mapsto \ci{\alpha n - \frac{1}{2}}$ is \nice\ with respect to $p$.

	Likewise, repeating the construction, but choosing $\eps_i < 0$ with $\sum_{i} \eps_i > -\frac{1}{2}$, we arrive at an idempotent $q$ such that $\ci{\alpha n + \frac{1}{2}} = \ci{\alpha n}$ for $q$-a.a. $n$. Hence, the function $n \mapsto \ci{\alpha n + \frac{1}{2}}$ is \nice\ with respect to $q$.

	The claim about either of the functions $\ci{\alpha n \pm \frac{1}{2}}$ being \nice\ is a direct consequence of the observation that for any $n$ either $\ci{\alpha n - \frac{1}{2}} = \ci{\alpha n}$ or $\ci{\alpha n + \frac{1}{2}} = \ci{\alpha n}$. 
\end{proof}

  Another question that naturally comes to mind is whether the recurrence properties that have been considered so far are a special feature of (almost) polynomials, or if there is a wider class of functions for which analogous results hold. We will show that for a function increasing more slowly than linear, there always exists an idempotent $p$ such that the limit $\llim{p}{n} \alpha f(n)$ is in general non-zero. We extract the following technical lemma before we proceed with the proof.

\begin{lemma}
  Let $\eps_i$ and $m_i$ be sequences such that $\frac{2\eps_i}{m_i} > \frac{3}{m_{i+1}}$. Denote $A_i := \{ \alpha \in \TT \setsep m_i \alpha \in (\gamma - \eps_i, \gamma+ \eps_i)$. Then, $\bigcap_{i \in \NN}A_i$ is not the empty set. 
\end{lemma}
\begin{proof}
  Let $B_k := \bigcap_{i \leq k} A_i$. We claim that $B_k$ contains an interval of length $\frac{2 \varepsilon_k}{m_k}$. For $k = 0$, this is clear. Suppose for some $k$ the claim holds. Then, $B_{k+1} = B_k \cap A_{k+1}$. From the form $A_{k+1}$ has, it is immediate that $\TT$ can be partitioned into $m_{k+1}$ intervals $I_1,I_2,\dots,I_{m_k}$ of length $\frac{1}{m_{k+1}}$ such that $J_s := A_{k+1} \cap I_s$ is an interval of length $\frac{2\eps_{k+1}}{m_{k+1}}$. By inductive assumption, $B_k$ contains an interval of length at least $\frac{2 \varepsilon_k}{m_k} > \frac{3}{m_{k+1}}$. This means that there exists $s$ such that $B_k \supset I_s \supset J_s$, and the claim follows.

  Because $B_k$ is a descending family of compact non-empty sets, $\bigcap_{i \in \NN}A_i = \bigcap_{k \in \NN}B_k \neq \emptyset$.

\end{proof}

\begin{proposition}
  Let $f : \NN \to \RR$ be such that $\lim_{n \to \infty} f(n) = \infty$. Suppose additionally that $\lim_{n \to \infty} f(n) - f(n+1) = 0$; for example $f(n) = o(n)$ and $f$ is increasing. Then, there exists an idempotent $p$ such for any $\gamma \in \TT$ there exists $\alpha \in \TT$ such that $\llim{p}{n}{\alpha \ci{f(n)}} = \gamma$. Moreover, $\alpha$ can be chosen in any interval of positive measure.
\end{proposition}
\begin{proof}
  Let us fix a sequence $\eps_i$ with $\lim_{i \to \infty} \eps_i = 0$. Because of the assumption on $f$, it is easy to construct an increasing sequence of integers $n_i$ such that $\ci{f}$ is constant on $[n_i,n_i + n_{i-1} + \dots + n_1]$. What is more, $n_i$ can also be chosen to be increasing steeply enough so that the assumptions of the Lemma above are satisfied for $m_i := \ci{f(n_i)}$. Under this assumption, it is clear that $\{ \ci{f(n)} \setsep n \in \FS{\seq{n}{i}{\NN}} \} = \{ \ci{f(n_i)} \setsep i \in \NN \} = \seq{m}{i}{\NN}$. Let $A_i := \{ \alpha \in \TT \setsep m_i \alpha \in (\gamma - \eps_i, \gamma+ \eps_i)$, as in the lemma above, and let $\alpha \in \bigcap_{i \in \NN}A_i$. Because $\alpha \in A_i$, we have $d( \alpha m_i, \gamma ) < \eps_i$. It follows that $\lim_{i \to \infty} \alpha m_i = \gamma$. Hence, for any $p$ such that $\FS{\seq{n}{i}{\NN}} \in p$, we have $\llim{p}{n} \alpha f(n) = \gamma$.
\end{proof}

We conjecture that similar results should be true for $f$ with any order of growth which is polynomially bounded, but different than polynomial. More precisely, if $f: \NN \to \RR$ is an increasing function such that for some integer $k$ we have\footnote{We say that a function $f$ has order of growth $\omega(g)$ if $\lim_{n \to\ \infty} \frac{f(n)}{g(n)} = \infty$. Likewise, we say that $f$ has order of growth $o(g)$ if $\lim_{n \to\ \infty} \frac{f(n)}{g(n)} = 0$. These definitions are normally only applied to positive and monotonous $g$. Assuming thatt $f$ and $g$ are positive and monotonous, the conditions $f = \omega(g)$ and $g = o(f)$ are equivalent.
}
 $f = \omega(n^k)$ and $f = o(n^{k+1})$, we believe that there exist $\alpha \in \TT$ an idempotent $p \in \beta \NN$ such that $\llim{p}{n}{\alpha \ci{f(n)}} \neq 0$ (possibly under some additional assumption, such as $f$ being restriction to $\NN$ of a function which is analytic, or belongs to a Hardy field).

%\section{Comparison with classical results}
To close this section, we compare our considerations with more well-established notions, and offer some examples. The class of (\nice) \ap{s} is, as the perspicacious reader might have already observed, closely related to the more classical notions of (admissible) generalised polynomials.

\begin{definition}[Generalised polynomials]\index{polynomial!generalised polynomial}
	The family of \emph{generalised polynomials} is the smallest family $\cG$ of maps $\ZZ \to \ZZ$ such that the following are satisfied:
	\begin{itemize}
		\item generalised polynomials extend ordinary polynomials: $\ZZ[x] \subset \cG$;
		\item generalised polynomials form an algebra: if $g,h \in \cG$ then $g\cdot h, g+h \in \cG$;
		\item generalised polynomials are closed under the floor map: if $\seq{g}{i}{[n]} \in \cG^n$ and $\seq{\alpha}{i}{[n]} \in \RR$ then $\floor{\sum_{i \in [n]} \alpha_i g_i} \in \cG$.
	\end{itemize}
\end{definition}

\begin{definition}[Admissible generalised polynomials]
\index{admissible generalised polynomial}
	The family of \emph{admissible generalised polynomials} is the smallest family $\cG_a \subset \cG$ of maps $\ZZ \to \ZZ$ such that the following are satisfied:
	\begin{itemize}
		\item polynomials vanishing at $0$ are admissible: $ x\ZZ[x] \subset \cG_a$;
		\item admissible generalised polynomials form an ideal in $\cG$: if $g \in \cG_a, h \in \cG$ then $g\cdot h \in \cG_a$, and if $g,h \in \cG_a$ then $g+h \in \cG_a$;
		\item generalised polynomials are closed under a ``shifted'' floor map: if $\eps \in (0,1)$, $\seq{g}{i}{[n]} \in \cG_a^n$ and $\seq{\alpha}{i}{[n]} \in \RR$ then $\floor{\sum_{i \in [n]} \alpha_i g_i + \eps} \in \cG_a$.
	\end{itemize}
\end{definition}

Thanks to Lemma \ref{B:lem:gen-poly-closure}, it is visible that generalised polynomials are \ap{s}, and admissible generalised polynomials are \nice\ \ap{s}. Hence, our results naturally yield results in the more classical terms.

One might wonder whether the classes we define here are really more general. It turns out that they indeed are, as the following examples show.  We stress that the following ideas are strongly inspired by $\IP$-systems, and more generally $\VIP$-systems, whose domain is the the family of finite sets of natural numbers, $\cP_{\operatorname{fin}}(\NN)$. More detailed discussion of such examples can be found in \cite{BergKnuMc2006}.

\begin{example}[Base change]\label{B:exple:base-change}
\index{IP*-set@$\IP^*$-set}
  Consider the map $f:\ \NN \to \NN$ defined by the condition that for $\alpha \in \cP_{\operatorname{fin}}(\NN_0)$ we have:
  $$ f\left( \sum_{i \in \alpha} 2^{i} \right) = \sum_{i \in \alpha} 3^i.$$
  Note that the above definition makes sense, because each integer has a unique binary expansion. Descriptively, $f(n)$ is the value one obtains by writing $n$ base $2$, and then reinterpreting this as an expansion base $3$. For $n \in \NN$, let $\alpha(n)$ denote the unique set for which we have $n = \sum_{i \in \alpha(n)} 2^{i}$, so that we have the relation:
  $$ f(n) = \sum_{i \in \alpha(n)} 3^i.$$
  It is easy to see that we have the linear relation:
  $$ f(n + m) = f(n) + f(m),$$
  provided that $\alpha(n) \cap \alpha(m)$ are disjoint. Note that this condition is satisfied as soon as $2^k | n$ for some $k$ with $2^k > m$. Because $2^k \NN$ is $\IP^*$ by Proposition \ref{prop:C:IP*-includes-finite-index-subgroups}, we have for any idempotent $p$:
  $$ \llim{p}{m}\llim{p}{n} \DeltaS^2 f(n,m) = \llim{p}{m}\llim{p}{n} \big(f(n + m) -  f(n) + f(m) \big) = 0.$$
  As a consequence, $\deg^p f = 1$ for arbitrary idempotent $p$. 
\end{example}

  Above, we exploited the uniqueness and existence of the binary expansion. Below we show how the same idea can be applied to more general bases. Additionally, there was nothing special about base $3$: we could have selected an arbitrary sequence $(b_i)_i$ in place of $\left( 3^i\right)_i$. Because the only allowable digits base $2$ are $0$ and $1$, above we could conveniently identify binary expansion of a number with a set of its non-zero digits; for general bases we need to proceed differently.

\begin{example}\label{B:exple:base-change-generalised}
\index{IP*-set@$\IP^*$-set}
  Suppose that we are given a sequence $\seq{a}{i}{\NN_0}$ and a sequence $\seq{d}{i}{\NN_0}$, such that each $n \in \NN$ has unique expansion:
  $$ n = \sum_{i \in \NN_0} \mu_i(n) a_i,$$
  with $\mu_i(n) \in [d_i]$ for all $i \in \NN_0$. Concretely, one can take $a_i := a^i$ and $d_i := a$ for some $a \geq 2$, leading to the expansion base $a$. Consider an arbitrary sequence $\seq{b}{n}{\NN_0}$ and a define the map $f:\ \NN \to \NN$ given by the formula:
  $$ f(n) = \sum_{i = 0}^\infty \mu_i(n) b_i.$$
  Assume additionally that the sets $A_i := \{ n \in \NN \setsep \mu_i(n) = 0 \}$ are $\IP^*$. It is easy to verify that this condition is satisfied if for all $i$ we have $a_i < a_{i+1}$ and $a_i | a_{i+1}$. Then, for a fixed $m$ and $n$ belonging to the $\IP^*$-set $\bigcap_{i: \mu_i(m) \neq 0} A_i$ it holds that for any $i$ either $\mu_i(m) = 0$ or $\mu_i(n) = 0$. Consequently, for such $m,n$ we have $\mu_i(n+m) = \mu_i(n) + \mu_i(m)$ for all $i$, and consequently $f(n+m) = f(n) + f(m)$. Hence, for any idempotent $p$ it holds that
  $$ \llim{p}{m}\llim{p}{n} \DeltaS^2 f(n,m) = 0,$$
  and thus $\deg^p f = 1$ (except for degenerate choices of $b$, leading to $\deg^p f \leq 0$).

  As a special case, for any $a \geq 2$ and $b \geq 1$, we may choose $a_i = a^i$ and $b_i = b^i$. Then the map $f$ described by reinterpreting expansion base $a$ as expansion base $b$ is an almost polynomial of degree $1$. Note that for $b = 1$, the value $f(n)$ is the sum of digits of $n$ base $a$.
\end{example}

	We can yet another example of class of \ap{s}, which is based on a somewhat more peculiar positional system. While the previous examples are well known, to the best of our knowledge the following example new.

\begin{example}[Fibonacci base]\label{B:exple:base-change-Fibonacci}
\index{IP*-set@$\IP^*$-set}

  Let $f_i$ be the $i$-th Fibonacci number (starting with $f_0 = 1,\ f_1 = 2$). It is a classical fact attributed to Zeckendorf (cf. \cite{Fibonacci-by-Zeckendorf}) that any integer $n$ can be represented in the form:
  $$ n = \sum_{i \in \NN_0} \mu_i(n) f_i,$$
  where $\mu_i(n) \in \{0,1\}$ and for no $i$ does it hold that $\mu_i(n) = \mu_{i+1}(n) = 1$. Such representation is often referred to as Fibonacci base or Zeckendorf expansion, and has been studied in some detail, see for example \cite{Fibonacci-by-Hoggatt} or \cite[pp. 295-296]{Knuth-ConcreteMathematics}.

Suppose that we can show that $A_i := \{ n \in \NN \setsep \mu_i(n) = 0 \}$ are $\IP^*$. Using the same arguments as previously, we can check that given an idempotent $p$ and a fixed $m$, for $p$-many $n$ it holds that $\mu_i(n+m) = \mu_i(n) + \mu_i(m)$ for all $i$. Consequently, we can derive that any function of the form $\sum_i \mu_i \cdot b_i$ is \ap\ of degree $1$.

  We now show that the sets $A_i $ indeed are $\IP^*$. For a proof by contradiction, suppose that $B$ is an $\IP$-set with $A_k \cap B = \emptyset$, i.e. such that $\mu_k(n) = 1$ for all $n \in B$. Fix a sufficiently large integer $j$, and for $n \in \NN$ let $t(n)$ denote the ``tail'' of $n$, obtained by restricting to the $j$ terminal digits:
  $$ t(n) := \sum_{i \in [j]} \mu_i(n) f_i.$$
  Let $p$ be an arbitrary idempotent in $\bar{B}$. Because $t(n)$ takes only finitely many values, the limit $a := \llim{p}{n} t(n)$ exists. Because $p$ is idempotent, we can easily find $n,m \in B$ such that $t(n) = t(m) = t(n+m) = a$. We can write $n = n' + a,\ m = m' + a, n+m = s' + a$, where $t(n') = t(m') = t(s') = 0$. We then have the relation: 
  $$ n' + m' + a = s'.$$
  It is not difficult to convince oneself that $\mu_i(n'+m') = 0$ if $i < j-2$, so we can write $n'+m' = r' + b$ with $t(r') = 0$ and $b$ having at most one non-zero digit at position $j-1$ or $j-2$. Consequently, we have the relation:
  $$ r' + a = s' - b.$$
  Using the relation $f_i + \sum_{t=0}^{s-1} f_{i+1+2t} = f_{i+2s}$, we conclude that $\mu_i(s' - b) = 0$ for $i < j-3$. On the other hand, $\mu_k(r'+a) = 1$, which leads to a contradiction, provided that $j > k+3$. This finishes our considerations for the Fibonacci base.
\end{example}

  We conjecture that similar reasonings should work for more general positional systems. In particular, some recursively defined sequences other than the Fibonacci sequence can be used to construct other positional systems in analogous manner. 

 We can go yet a step further and construct a fairly general class of ``automatic'' functions, which turn out to be degree $1$ \ap{s}. We need a preliminary concerning automata. The following definition is taken from \cite{AutSeq}.

\begin{definition}\index{automaton}
	A \emph{deterministic finite automaton with output} $\cA$ (DFAO or automaton, for short) over a finite finite alphabet $\Sigma$, consists of the following data:
\begin{enumerate}
\item set of states $Q$, with a distinguished initial state $q_{init}$;
\item transition function $\tau:\ Q \times \Sigma \to Q$;
\item output function $\lambda:\ Q \times \Sigma \to \NN_0$.
\end{enumerate}
\end{definition}

The intuition behind an automaton is the following. The automaton starts in the state $q_{init}$. A sequence $\alpha_0,\alpha_1,\alpha_2,\dots$ of symbols from $\Sigma$ is provided on input. The automaton accepts them one by one, and if it accepts a symbol $\alpha$ when it is in state $q$, then it passes to state $q' = \tau(q,\alpha)$. After each such transition, the symbol $\lambda(q,\alpha)$ is produced on output.

Let $a,b \in \NN_2$ be fixed bases, and for $n \in \NN$, let $n = \sum_{i=0}^\infty \mu_i(n) a^i$ be the unique decomposition of $n$ in base $a$. Suppose that $\cA = (Q,\tau,\lambda)$ is an automaton over the alphabet $\Sigma = [a]$. We can think of $\cA$ as generating a function $f_\cA:\ \NN \to \NN$ in the following way. We begin by indentifying $n$ with a sequence of its digits base $a$, then we apply $\cA$ and interpret the result as a number base $b$. More formally, let us fix $n \in \NN$ with expansion $n = \sum_{i=0}^\infty \mu_i a^i$; we will define $f_\cA(n)$. First, we denote the consecutive steps $q_0 := q_{init}$ and $q_{i+1} := \tau(q_i,\mu_i)$. Next, we denote the outputs $\lambda_i := \lambda(q_i,\mu_i)$; note that $\lambda_i$ depends implicitly on $n$. Finally, we put:
$$ f_\cA(n) := \sum_{i=0}^\infty \lambda_i b^i.$$ 
We call a function $f:\ \NN \to \ZZ$ an \emph{automatic function} if it is of the form $f_\cA$ for some automaton $\cA$. 

We leave the following result without proof, which is not difficult, but rather technical.
\begin{proposition}
	Let $\cA = (Q,q_{init},\tau,\lambda)$ be an automaton, and let $f$ be the corresponding function. Suppose that the map $\tau(\cdot,a):\ Q \to Q$ is bijective for any $a \in \Sigma$. Then $f$ is an \ap\ with degree at most $1$ with respect to any idempotent $p$.
\end{proposition}

  As the reader might have noticed, we do not actually prove that the \ap{s} we just presented are not generalised polynomials. Because generalised polynomials are always polynomially bounded while functions constructed above need not be, one can show that not all of these \ap{s} presented are generalised polynomials. We believe that in general none of the above \ap{s} are generalised polynomials, except for some degenerate cases, but we have no way of showing this rigorously.	

  The examples presented above are somewhat far-fetched: it is not clear that anyone would be interested in thus defined functions in the first place. There turn out to be more natural examples of \nice\ \ap{s} that are not admissible generalised polynomials. Note that the proof of Lemma \ref{B:lem:gen-poly-closure} in fact shows that for a fixed $g \in \cG$, there are just finitely many values that $C^p(g)$ can take, depending on the idempotent $p$. It follows that for $\seq{g}{i}{[n]} \in \cG$ and $\seq{\alpha}{i}{[n]} \in \RR$ we have $\ci{\sum_{i \in [n]} \alpha_i g_i } \in \NAP^p$ for all $p$, provided that $\abs{\alpha_i}$ are small enough that for any $p$ we have $\abs{\sum_{i \in [n]} \alpha_i C^p(g_i)} < \frac{1}{2}$. For an explicit example, $\floor{\frac{\floor{\pi n}}{43} + \frac{1}{2e} }$ lies in $\NAP^p$ for any $p$, but probably not in $\cG_a$. That being said, for $\eps \in (0,1)$, the map $\floor{\frac{\floor{\pi n + \eps}}{43} + \frac{1}{2e} }$ lies in $\cG_a$, so the difference does not appear to be very significant.

  To close the discussion about \ap{s}, we stress some pitfalls and oddities that one can encounter.

  Firstly, the property of being an \ap\ depends on $p$, even though the examples we encountered so far did not take $p$ into consideration. We have noted that for bounded maps $b_i :\ \ZZ \to \ZZ$, $i \in [r]$, the map $f$ given by $f(n) := \sum_{i \in [r]} b_i(n) n^k$ is an \ap\ with respect to any $p$. In fact, if $b_i^p := \llim{p}{n} b_i(n)$ and $f^p(n) := \sum_{i \in [r]} b^p_i n^k$ then $f(n) = f^p(n)$ for $p$-a.a. $n$, and consequently $f$ and $f^p$ are indistinguishable as members of $\AP^p$. However, for different $p$ the polynomials $f^p$ may very well be different. In particular, it may well be that $\deg^p f \neq \deg^q f$ and $C^p(f) \neq C^q(f)$ for $p \neq q$. 

  In fact, given two different idempotents $p$ and $q$ and arbitrary \ap{s} $f^p \in \AP^p,\ f^q \in \AP^q$, one can construct $f \in \AP^p \cap \AP^q$ such that $f(n) = f^p(n)$ for $p$-a.a. $n$ and $f(n) = f^q(n)$ for $q$-a.a. $n$. This can be achieved quite simply. Note that there exist a set $A$ which is $p$-large but not $q$-large, as well as a set $B$ which is $q$-large but not $p$-large. Taking $f:= \chi_A \cdot f + \chi_B \cdot g$ we can easily verify that $f$ has the mentioned relation to $f^p$ and $f^q$.

  Another aspect we wish to stress is that it is not the case that \ap{s} have the order of growth expected of polynomials. Of course, this is not much of a surprise, given that we can modify a member of $\AP^p$ on a $p$-small set without changing its properties as an almost polynomial. At this point, one might yet be hoping that some notion of order of growth relative to $p$ would work. However, even with naturally defined \ap{s} which can be approximated by ordinary polynomials up to a constant factor, this hope fails. For example, the map given by $\sum_i \mu_i a^i \mapsto \sum_i \mu_i b^i$ from Example \ref{B:exple:base-change-generalised} has degree $1$, but has the order of growth%
\footnote{We say that a function $f$ has the order of growth $\Theta(g)$ if there exist constants $C_1,C_2 > 0$ such that $C_1 g(n) < f(n) < C_2 g(n)$ for sufficiently large $n$.}
  $\Theta\left(n^{\ln b/ \ln a}\right)$, instead of the expected $\Theta(n)$. The \ap{} given by the formula $f(n) = \ci{\alpha n \fp{\beta n}}$ with $\alpha, \beta \in \RR \setminus \QQ$ is clearly
\footnote{We say that a function $f$ has the order of growth $\mathrm{O}(g)$ if there exist a constant $C > 0$ such that $f(n) < C g(n)$ for sufficiently large $n$. In particular, $f$ is $\Theta(g)$ if and only if $f$ is $\mathrm{O}(g)$ and $g$ is $\mathrm{O}(f)$.}
$\mathrm{O}(n)$ and one can even check that for any idempotent $p$ we have $\llim{p}{n} \frac{f(n)}{n} = 0$. However, explicit computation of $\DeltaS f$ shows that $\deg^p f = 2$, so $f$ is far from satisfying the expected approximation $\Theta(n^2)$.

\index{polynomial!almost polynomial|)}
\section{Dynamical applications}

  We shall now see how theory developed so far can be applied to measure preserving preserving systems. For this purpose, we will be considering averages of powers of operators on Banach spaces. The link between Banach spaces and dynamical systems uses the Koopman operator given by $U_T(f) = f \circ T$. Mostly, we will be interested in the Hilbert spaces $L^2$, but for the time as as far as it is possible we use develop our methods in the most general context.

\newcommand{\slim}[2]{\operatorname{s}\!-\llim{#1}{#2}}
\begin{definition}\label{B:def:operators-p-lim}
\index{limit!operator limit}
\index{limit!generalised limit}
  Let $\cE$ be a reflexive Banach space, let $\seq{A}{n}{X} \in \cB(\cE)^X$ an bounded sequence of (continuous linear) operators on $\cE$, indexed by a set $X$, and let $p \in \beta X$ be an arbitrary ultrafilter. By the standard symbol $\llim{p}{n} A_n$ we denote the generalised limit taken in the weak topology. The limit exists because of reflexivity of $\cE$ and Banach-Alaoglu theorem. 
  
  Since strong convergence implies weak convergence, we do not intorduce additional symbol for the strong limit. When convergence is strong, we will note this explicitly.
\end{definition}

	We prove some basic properties of limits related to commutativity.

\begin{lemma}\label{B:lem:operators-commutativity}
\index{limit!operator limit}
\index{limit!generalised limit}
  Let $\seq{A}{n}{X}$ and $\seq{B}{n}{X}$ be bounded sequences of operators on a reflexive Banach space $\cE$, and suppose that $A_n$ commutes with $B_n$ for all $n$. Then, for any ultrafilter $p \in \beta X$, the limits $\llim{p}{n} A_n$ and $\llim{p}{n} B_n$ is commute.	
  
  In particular, if $\seq{A}{n}{\ZZ}$ is a uniformly bounded sequences of operators on a Hilbert space $\cH$, and $A_n$ is normal for each $n$, then the limit $\llim{p}{n} A_n$ is normal.
\end{lemma}
\begin{proof}
	Let $L:=\llim{p}{n} A_n$ and $M:=\llim{p}{n} B_n$. By direct computation, using separate weak continuity of operator multiplication, we check that:
	$$ LM = \llim{p}{m} A_m \llim{p}{n} B_n = \llim{p}{m} \llim{p}{n} A_m B_n = \llim{p}{m} \llim{p}{n} B_n A_m = ML.$$

  Because the adjoint is continuous in the weak topology, the additional claim follows by applying the previous part with $\cE = \cH$ and $B_n = A_n^*$.
\end{proof}

We now discuss a very special case of theorems that shall be considered afterwards. The obtained result is not of much interest on its own, but serves as a motivation for what follows. In applications, we will be mostly interested in the limits of powers of the Koopman operator of a dynamical system.

\begin{proposition}\label{B:prop:operator-p-lim-linear-is-projection}\label{B:lem:p-lim-U^n=P}
\index{limit!operator limit}
\index{limit!generalised limit}
	Let $A \in \cB(\cE)$ be a power-bounded\footnote{An operator $A$ is \emph{power-bounded} if the sequence $\norm{A^n}$, $n \in \NN$ is bounded.} operator on a reflexive Banach space $\cE$, and let $p \in \beta \NN$ be an idempotent ultrafilter. Then $P:=\llim{p}{n} A^n$ is idepotent: $P^2 = P$.

	In particular, if $A \in \cB(\cH)$ is a normal operator on a Hilbert space $\cH$ with $\norm{A} \leq 1$, then $P:=\llim{p}{n} A^n$ is an orthogonal projection.
\end{proposition}
\begin{proof}
	Using idempotence of $P$, we first transform:
	$$P=\llim{p}{n}A^n = \llim{p}{m} \llim{p}{n} A^{n+m} =  \llim{p}{m} \llim{p}{n} A^{m}U^{n}.  $$
	Using the fact that operator multiplication is separately continuous in the weak topology, we can transform further:
	$$ P = \llim{p}{m} \llim{p}{n} A^{m}A^{n} =  \llim{p}{m} A^{m} \llim{p}{n} A^{n} = P^2.$$
	Therefore, $P = P^2$, as required.

	For the additional part of the claim, note that is $A$ is normal, then so is $P$, thanks to Lemma \ref{B:lem:operators-commutativity}. Hence, $P$ is an orthogonal projection thanks to the well-known criterion.
\end{proof}

The following lemma shows how to transform statements like the above into more concrete results about recurrence. We give the most general formulation first, and then apply it to the situation at hand. To begin with, we briefly recall the relevant definitions.

\newcommand{\fM}{\mathfrak{M}}
\newcommand{\X}{\mathsf{X}}

\begin{definition}\label{B:def:measure-preserving-system}
\index{measure preserving system}
  A \emph{measure preserving system} is a quadruple $\X = (X,\fM,\mu,T)$ where $X$ is a compact topological space, $\fM$ is a $\sigma$-algebra on $X$, $\mu$ is a probability measure on $\fM$ and $T: X \to X$ is a  measure preserving transformation. The Koopman operator $U_T \in \cB(L^2(X,\mu))$ associated to $T$ is the operator given by $U_T(f) = f \circ T$. In general, $U_T$ is an isometry. If $T$ is invertible, then $U_T$ is unitary and $U_T^{-1} = U_{T^{-1}}$.
\end{definition}

\begin{lemma}\label{B:lem:Khintchine-preparatory}
\index{limit!operator limit}
\index{limit!generalised limit}
\index{Khintchine}

  Let $(X,\fM,\mu)$ be a measure space. Let $\seq{T}{n}{X}$ be a family of measure preserving invertible transformations, and let $U_n \in \cB(L^2(X,\mu))$ be the associated Koopman operators. Suppose that the limit $P := \llim{p}{n}{U_n}$ is a projection. Finally, let $A \in \fM$ be such that $\mu(A) > 0$.  Then:
  $$ \llim{p}{n} \mu(A \cap T^{-1}_nA ) \geq \mu(A)^2. \qedhere $$
\end{lemma}
\begin{proof}
	Let $1_A$ denote the characteristic function of $A$, let $1_X$ denote the constant function $1$. Note that $U_n 1_X = 1_X$, and thus also $P 1_X = 1_X$. We can now translate the statements about measures of sets into statements about scalar products, in particular $\mu(A) = \scalar{1_A,1_X}$ and $\mu(A \cap T_{n}^{-1}A ) = \scalar{1_A,U_n 1_A}$. It follows that:
	\begin{align*}
	 \llim{p}{n} \mu(A \cap T_{n}A ) &= \llim{p}{n} \scalar{1_A,U^n 1_A} = \scalar{1_A, \llim{p}{n} U^n 1_A} 
	\\&= \scalar{1_A, P 1_A} = \norm{P 1_A}^2 = \norm{P 1_A}^2 \norm{P 1_X}^2 
	\\& \geq \scalar{P 1_A ,P 1_X}^2 = \scalar{1_A ,P 1_X}^2
	= \scalar{1_A , 1_X}^2 = \mu(A)^2
	\end{align*}
\end{proof}

\begin{corollary}[Khintchine]\label{B:cor:khintchine-linear}
\index{limit!operator limit}
\index{limit!generalised limit}
\index{Khintchine}
\index{IP@$\IP$!IP-set@$\IP$-set}

	Let $(X,\fM,\mu,T)$ be a measure preserving system, let $A \in \fM$ be such that $\mu(A) > 0$, and let $p \in \beta \ZZ$ be an idempotent ultrafilter. Then:
	$$ \llim{p}{n} \mu(A \cap T^{-n}A ) \geq \mu(A)^2 .$$
	In particular, we have:
	$$ \limsup_{n \to \infty} \mu(A \cap T^{-n}A ) \geq \mu(A)^2 .$$
	Moreover, for any $\eps > 0$, the set of return times:
	$$ R_\eps := \{ n \in \ZZ \setsep \mu(A \cap T^{-n}A ) > \mu(A)^2 - \eps \} $$
	is $p$-large, and therefore is an $\IPst$ set.
\end{corollary}
\begin{proof}
	The first statement is an immediate application of the above preparatory Lemma \ref{B:lem:p-lim-U^n=P}. The additional parts of the statement are just equivalent ways of expressing the convergence, and quantifying over all idempotents $p$.
\end{proof}

Our next goal is to give more general theorems describing when the operators of the form $\llim{p}{n}{U^{f(n)}}$ are projections. We will need some preliminary results. 

The following decomposition is a classical theorem. It will be important for applications of minimal ultrafilters. A detailed proof and discussion can be found in Eisner's \cite{Tanja-operators}. With the theory on compact semigroups developed in Chapter \ref{A:chapter} we could re-derive it without too much work, but it would take us too far afield, so we merely cite it instead.

\begin{theorem}[Jacobs-Glicksberg-de Leeuw decomposition]\label{B:thm:JGL-decomposition}
\index{Jacobs-Glicksberg-de Leeuw decomposition}
  Let $\cE$ be a reflexive Banach space, and let $A \in \cB(\cE)$ be an operator with $\norm{A} \leq 1$. Then, $\cE$ decomposes into the direct sum $\cE_r \oplus \cE_s$, where:
  \begin{align*}
    \cE_r &:= \bar\lin\{ f \in \cE \setsep (\exists \gamma \in \CC,\ \abs{\gamma} = 1)\ A f = \gamma f \}, \\
    \cE_s &:= \{ f \in \cE \setsep 0 \in \cl^{\operatorname{weak}}\{A^nf\}_{n \in \NN} \}.
  \end{align*}
  The minimal idempotent $Q$ in the semigroup generated by $A$ is the orthogonal projection onto $\cE_r$.
\end{theorem}

The above decomposition allows us to consider the operator limits of powers of $A$ on the two spaces $\cE_r$ and $\cE_s$ independently. The situation of $\cE_r$ is especially simple, as the following observation shows.

\begin{observation}\label{B:lem:p-lim-on-E_r}
\index{limit!operator limit}
\index{limit!generalised limit}
\index{polynomial!almost polynomial!admissible}
  Let $p \in \beta \ZZ$, and $f \in \NAP^p(\ZZ,\ZZ)$ be such that $f > 0$ $p$-a.e.. Let $\cE$ be a reflexive Banach space, let $A \in \cB(\cE)$ be an operator with $\norm{A} \leq 1$, and let $\cE = \cE_r \oplus \cE_s$ be the Jacobs-Glicksberg-de Leeuw decomposition of $\cE$ with respect to $A$. Then the limit $P := \llim{p}{n}{A^{f(n)}}$, restricted to $\cE_r$, is the identity operator $I_{\cE_r}$.
\end{observation}
\begin{proof}
  We first note that for $f \in \cE$ such that $Af = \gamma f$ with $\abs{\gamma} = 1$, we have $P f = f$. This is true because $A^{f(n)} x = \gamma^{f(n)} x$, and by assumption on $f$ and Lemma \ref{B:lem:p-lim-of-gen-poly} we have $ \llim{p}{n}  \gamma^{f(n)}  = 1$. Because such $f$ span $\cE_r$, it follows that $P|_{\cE_r} = I_{\cE_r}$.
\end{proof}

  It follows form the above observation that the situation on $\cE_r$ is clear in the most generality we can hope for. On $\cE_s$ we only have the simple result provived by Proposition \ref{B:lem:p-lim-U^n=P}. This is as much as we are able to say for general (reflexive) Banach spaces. 

To make further progress we need to restrict to Hilbert spaces, which allows us to use the version of van der Corput Lemma for generalised limits. This lemma is of vital importance for many inductive arguments.

\begin{lemma}[van der Corput, \cite{Schnell}]\label{B:lem:v-d-Corput}
\index{van der Corput}
\index{limit!operator limit}
\index{limit!generalised limit}
  Let $\cH$ be a Hilbert space, let $(X,+)$ be a semigroup, and let $\seq{x}{n}{X} \in \cH^X$ be a bounded family indexed by elements of $X$, and let $p \in \beta X$ be an idempotent ultrafilter. Suppose additionally that $\llim{p}{m} \llim{p}{n} \scalar{x_{n+m},x_n} = 0$. Then it also holds true that $\llim{p}{n} x_n = 0$. 
\end{lemma}
\begin{proof}
	Let us denote $y :=  \llim{p}{n} x_n$. As an immediate application of idempotence of $p$, we notice that for any positive interger $s$ we have can also express $y$ as:
$$y = \llim{p}{n_1,\dots,n_s} x_{n_1+\dots + n_s}.$$
 Likewise, we notice that the condition $\llim{p}{m} \llim{p}{n} \scalar{x_{n+m},x_n} = 0$, together with idempotence of $p$, implies that we have for any $r,s \geq 1$:
$$\llim{p}{m_1,\dots,m_r} \llim{p}{n_1,\dots,n_s} \scalar{x_{n_1+\dots + n_s+m_1 +\dots+m_r },x_{n_1+\dots + n_s}} = 0.$$

For any $N$ we therefore have: 
$$ y = \frac{1}{N} \sum_{s=1}^N \llim{p}{n_1,\dots,n_s} x_{n_1+\dots + n_s}. $$
In particular, because norm is semi-continuous from below in the weak topology:
\begin{align*} \norm{y}^2 &= \frac{1}{N^2} \norm{ \llim{p}{n_1,\dots,n_N} \sum_{s=1}^N  x_{n_1+\dots + n_s}}^2 \\
&\leq  \frac{1}{N^2} \llim{p}{n_1,\dots,n_N} \norm{  \sum_{s=1}^N  x_{n_1+\dots + n_s}}^2 \\
&=  \frac{1}{N^2} \sum_{r,s=1}^N \llim{p}{n_1,\dots,n_N} \scalar{ x_{n_1+\dots + n_s}, x_{n_1+\dots + n_r}}.
\end{align*}
As a direct application of the remark about scalar products, if $r \neq s$ we have  $$\llim{p}{n_1,\dots,n_N} \scalar{ x_{n_1+\dots + n_s}, x_{n_1+\dots + n_r}} = 0.$$ It allows us to simplify the above expression:
\begin{align*} \norm{y}^2 
&\leq \frac{1}{N^2} \sum_{s=1}^N \llim{p}{n_1,\dots,n_N} \norm{x_{n_1+\dots + n_s}}^2 \\
&=  \frac{1}{N}  \llim{p}{n} \norm{x_n}^2.
\end{align*}
Because $N$ was chosen arbitrarily and $\llim{p}{n} \norm{x_n}^2$ is a constant independent of $N$, it follows that $\norm{y}$ has to be equal to $0$. Thus, $y = 0$, as desired.
\end{proof}

With this tool we are able to obtain a general result on limits along \emph{minimal} idempotents. Under some additional assumptions, we are able to identify the limit quite explicitly as the minimal projection. Somewhat surprisingly, the case of degree $1$ almost polynonials seems to be the most problematic. The following lemma is somewhat unsatisfactory --- in a sense, it formalises an induction procedure, but does not secure the basic step. To avoid repetition, we introduce the situation which will be common in a number of consecutive results. 

	\begin{convention}\label{B:situation:Fp}
We let $\cH$ denote a Hilbert space, $p \in \beta \ZZ$ an ultrafilter, and $U$ a fixed unitary operator on $\cH$. Let $Q$ denote the minimal projection generated by $U$, as in Theorem \ref{B:thm:JGL-decomposition}. Finally, let the class of maps $\cF^p$ be defined by:
  $$ \cF^p = \{ f \in \NAP^p(\ZZ,\ZZ) \setsep  \llim{p}{n} U^{f(n)} = Q \}. $$	
	\end{convention}

\begin{proposition}\label{B:lem:min-rec-generic}	
\index{polynomial!almost polynomial!admissible}
\index{recurrence}
\index{ultrafilter!minimal}
	Assume notation as in {B:situation:Fp}, and suppose that $f \in \NAP^p$ is such that for $p$-a.a. $a$ we have $\DeltaS_a f \in \cF^p$. Then $f \in \cF^p$.
\end{proposition}

\begin{proof}
	Denote $P := \llim{p}{n}U^{f(n)}$; our goal is to show that $P = Q$. It is clear that $P$ is normal, as a limit of normal operators, thanks to Lemma \ref{B:lem:operators-commutativity}. What is more, all operators that appear throughout the proof arise as limits of powers of $U$, and hence by Lemma \ref{B:lem:operators-commutativity} they commute with one another. We will use this fact without further mention. 

  By Corollary \ref{B:lem:p-lim-on-E_r}, we already know that $Q|_{\cH_r} = I_{\cE_r}$. Hence, it will suffice to show that $Q|_{\cH_s} = O_{\cH_s}$. 

  Let us consider a fixed $x \in \cH_s$; our goal is to show that $\llim{p}{n} U^{f(n)}x = 0$. Using van der Corput Lemma \ref{B:lem:v-d-Corput}, it will suffice to show that $\llim{p}{m} \llim{p}{n} \scalar{U^{f(n)}x , U^{f(n+m)}x} = 0$. This can be established easily enough by algebraic manipulation:
 \begin{align*}
    \llim{p}{m} \llim{p}{n} \scalar{U^{f(n)}x , U^{f(n+m)}x}  &=
    \llim{p}{m} \llim{p}{n} \scalar{U^{f(n)}x , U^{\DeltaS_mf(n) + f(n) + f(m)}x} \\ &= 
    \llim{p}{m} \llim{p}{n} \scalar{U^{-f(m)}x , U^{\DeltaS_mf(n)}x} \\ &= 
    \llim{p}{m} \scalar{U^{-f(m)}x , \llim{p}{n} U^{\DeltaS_mf(n)}x} \\ &= 
    \llim{p}{m} \scalar{U^{-f(m)}x , Q x} = 0.  \qedhere
  \end{align*}
\end{proof}

The above result shows that once we identify members $f \in \cF^p$ with $\deg^p f = 1$, we are given a criterion for members of $\cF^p$ with higher degrees. More precisely, it becomes evident that if $f \in \NAP^p$ and $\DeltaS^{\deg f - 1}_{a_1,a_2,\dots} f \in \cF^p$ for $p$-a.a. $a_1,a_2,\dots$, then $f \in \cF^p$. We start with a result in this direction.

\begin{lemma}\label{B:lem:min-rec-2}
\index{ultrafilter!minimal}
\index{recurrence}
	With notation as in \ref{B:situation:Fp}, suppose that $p$ is minimal, and let $f:\ \ZZ \to \ZZ$ be the identity map $f(n) = n$. Then $f \in \cF^p$.
\end{lemma}
\begin{proof}
	Consider the semigroup morphism $\beta \NN \to \cB(\cH)$ given by $p \mapsto \llim{p}{n}U^n$. Let $\cS$ be the image of $\beta \NN$. It is clearly a semigroup, it is compact, and the semigroup $\{ U^n \setsep n \in \NN \}$ is dense in it. It is a consequence of previously shown results that $\cS$ is commutative. Let $\cK := \Kappa(\cS)$ be the minimal (two-sided) ideal. A relatively simple arguemnt shows that $\cK = Q \cS$ and that $\cK$ is a group with $Q$ as the identity; see \cite{Tanja-operators} for details.

	Let $I \subset \beta \NN$ be the set of those ultrafilters $q$ for which $\llim{q}{n}U^n \in \cK$. Because $q \mapsto \llim{q}{n}U^n$ is a morphism of compact semigroups and $\cK$ is a two-sided ideal, it follows that $I$ is a two-sided ideal. This means that $\Kappa(\beta \NN) \subset I$. 

	Let us return to the minimal idempotent $p$, and denote $P := \llim{p}{n}U^n$. Because $p$ is idempotent, we already know that $P|_{\cH_r} = I_{\cH_r}$. Because $p \in \Kappa(\beta \NN)$, the above considerations show that $P \in Q \cS$, so  $P|_{\cH_s} = O_{\cH_s}$. Hence, $P = Q$.
\end{proof}

\begin{lemma}\label{B:lem:Q(U^k)=Q}
\index{semigroup!ideal!minimal ideal}
	For a unitary operator $V$ on $\cH$, let $Q(V)$ denote the minimal projection as in Jacobs-Glicksberg-de Leeuw decomposition. Then it holds that $Q(V) = Q(V^k)$.
\end{lemma}
\begin{proof}
	We proceed by induction on $k$, with the case $k = 1$ being trivially satisfied. 

	Let us additionally denote by $\cS(V)$ the compact semigroup generated by $V$. We notice that we have the following decomposition of $\cS(V)$:
	$$ \cS(V) = \bigcup_{l=0}^{k-1} V^l \cS(V^k).$$
	The inclusion $ V^l \cS(V^k) \subset \cS(V)$ is clear. For the reverse inclusion we first note that:
	$$ \cS_0(V) = \bigcup_{l=0}^{k-1} V^l \cS_0(V^k),$$
	where $\cS_0(V)$ denotes the (non-compact) semigroup generated by $V$, and then take closures of both sides. 

	From the above observation, it follows that for some $0 \leq l < k$ we have $Q(V) \in V^l \cS(V^k)$. Because $Q(V)^k = Q(V)$, we conclude that  $Q(V) \in V^{kl} \cS(V^k)^k \subseteq \cS(V^k)$. Because $Q(V)$ is a minimal projection in $\cS(V) \supseteq \cS(V^k)$, we conclude that  $Q(V) \in  \Kappa(\cS(V^k))$. Finally, because $Q(V^k)$ is the unique idempotent in $\Kappa(\cS(V^k))$, we conclude that $Q(V^k) = Q(V)$.
\end{proof}

\begin{corollary}
  \index{semigroup!ideal!minimal ideal}
 	With notation as in \ref{B:situation:Fp}, assume that $p$ is minimal, and let $f:\ \ZZ \to \ZZ$ be the linear map $f(n) = kn$, $k \in \NN$. Then $f \in \cF^p$.
\end{corollary}

\begin{proof}
	From the above Lemma \ref{B:lem:Q(U^k)=Q}, it follows that:
	$$ \llim{p}{n} U^{kn} = Q(U^k) = Q(U) = Q.$$
	Hence, the claim follows.
\end{proof}

\begin{proposition}\label{B:prop:min-rec-weighted-sum-of-digits}

	With notation as in \ref{B:situation:Fp}, assume that $p$ is minimal, and let $f:\ \ZZ \to \ZZ$ be the base-changing map described in Example \ref{B:exple:base-change-generalised}, defined for some fixed $a \in \NN_2,\ b_i \in \ZZ$ by:
	$$ f\left(\sum_{i} \mu_i a^i\right) = \sum_i \mu_i b_i,$$
	where $\abs{\mu_i} < a$ and all $\mu_i$ have the same sign\footnote{We need a minor alteration to account for the domain changing from $\NN$ to $\ZZ$, but it is easy to see that this alteration does not lead to any significant problems.}. Then $f \in \cF^p$, provided that we have:
    $$ \{ \llim{q}{n}U^{f(n)} \setsep q \in \beta \ZZ \} \cap Q \cS \neq \emptyset. $$
\end{proposition}
\begin{proof}
	Let $H := \bigcap_{k \in \NN} \bar{a^k \ZZ} \subset \beta \ZZ$. Note that $H$ is compact, because it is the intersection of compact sets. Moreover, is is a semigroup, because $\bar{a^k \ZZ}$ are all semigroups by Proposition \ref{A:prop:algebra-closure-of-sgrp-is-sgrp}. If $m \in \ZZ$ is fixed then for $n$ divisible by a sufficiently large power of $a$ (dependent on $m$) we have:
	$$ f(n+m) = f(n) + f(m).$$
It follows that for $q \in H$ and arbitrary $p$ we have:
	$$ \llim{p}{m} \llim{q}{n} \left( f(n+m) - f(n) - f(m) \right) = 0.$$
Denote $\Phi(p) := \llim{p}{n} U^{f(n)}$ for $p \in \beta \ZZ$. The above shows that for $q \in H$ we have:
	$$ \Phi(p+q) = \llim{p}{m} \llim{q}{n} U^{f(n+m)} 
	= \llim{p}{m} \llim{q}{n} U^{f(n)+f(m)} = \Phi(p) \Phi(q).$$
In particular, $\Phi$ restricted to $H$ is a morphism of semigroups. 

Let $q \in \beta \ZZ$ be such that $\Phi(q) \in Q \cS$, which exists by the additional assumption. For any $k$ we can find $c_k$ such that $q_k := q + c_k \in \bar{a^k \ZZ}$. It is clear that $\Phi(q_k) \in Q \cS$, and hence a simple compactness argument shows that $\Phi(H) \cap Q \cS \neq \emptyset$. Consequently the minimal ideal in the semigroup $\Phi(H)$ is $Q\Phi(H)$.

Consider the two sided ideal $\Phi^{-1}(Q\Phi(H)) \cap H$. Clearly, it contains $\Kappa(H) = \Kappa(\beta \ZZ) \cap H$. We know that $H$ contains all the idempotents. Hence, if $q$ is a minimal idempotent, we have $\Phi(q) \in Q \Phi(H)$. Finally, because $Q$ is the only idempotent in the $Q \Phi(H)$, we conclude that $\Phi(q) = Q$. But this means precisely that $q \in \cF^p$.
\end{proof}

	Having characterised some degree $1$ polynomials in $\cF^p$ for $p$ --- minimal idempotent, we are able to derive a description elements of $\cF^p$ of higher degrees.

\begin{theorem}\label{C:thm:Fp-partial-characterisation}
	With notation as in \ref{B:situation:Fp}, for $p$ minimal idempotent the following are true:
	\begin{enumerate}
	\item If $f:\ \ZZ \to \ZZ$ is a standard polynomial with $f(0) = 0$, then $f \in \cF^p$.
	\item If $f \in \cF^p$ and $g \in \NAP^p$ are such that $\deg f < \deg g$, then $f + g \in \cF^p$.
	\item If $f \in \NAP^p$, then $n^{\deg^p f + 1} + f \in \cF^p$.
	\item If $f:\ \ZZ \to \ZZ$ is a ''weighted sum of digits'' as in Proposition \ref{B:prop:min-rec-weighted-sum-of-digits} and $g:\ \ZZ \to \ZZ$ is a standard polynomial, the $f \circ g \in \cF^p$.
\end{enumerate}
\end{theorem}

If fact, we have not been able to find an example of an \nice\ \ap\ such outside $\cF^p$. This leads us to state the following conjecture, for which the above results constitute a motivation.

\begin{conjecture}
	If $p$ is a minimal idempotent then $\cF^p = \NAP^p$.
\end{conjecture}

The reason for interest in the above considerations is that we can apply them to general measure-preserving systems. The resulting theorem is similar to Khintchine's, except it speaks of recurrence along (generalised) polynomials. It is the immediate consequence of the above Theorem \ref{C:thm:Fp-partial-characterisation} together with Lemma \ref{B:lem:Khintchine-preparatory}.

\begin{corollary}
	Let $(X,\fM,\mu,T)$ be a measure preserving system, and $A \in \fM$ be such that $\mu(A) > 0$, and let $p \in \beta \ZZ$ be a minimal idempotent, and let $f 	\in \cF^p$. Then:
	$$ \llim{p}{n} \mu(A \cap T^{-f(n)}A ) \geq \mu(A)^2 .$$
	In particular, we have:
	$$ \limsup_{n \to \infty} \mu(A \cap T^{-f(n)}A ) \geq \mu(A)^2.$$
	Moreover, for any $\eps > 0$, the set of return times:
	$$ R_\eps := \{ n \in \ZZ \setsep \mu(A \cap T^{-f(n)}A ) > \mu(A)^2 - \eps \} $$
	belongs to $p$, and is therefore is an $\C^*$ set\footnote{For relevant definitions, see \ref{C:C-definition} and \ref{C:def:A-star-sets}}.
\end{corollary}

%%%%%%%%%%%%%%%%%%%%%%%%%%%%%%%%%%%%%%%%%%%%%
% SECTION
%%%%%%%%%%%%%%%%%%%%%%%%%%%%%%%%%%%%%%%%%%%%%
\section{Some classical results}
Let us compare the above result with that of Schnell \cite{Schnell}, which in turn is largely inspired by results of Bergelson et al. \cite{BergKnuMc2006}. A special case of the main theorem of \cite{Schnell} is the following:

\begin{theorem}[Schnell]
  Let $(U_i)_{i=1}^m$ be a family of commuting unitary operators on a Hilbert space $\cH$. Let $p \in \beta \ZZ^n$ be an idempotent, and $f_i :\ \ZZ^d \to \ZZ^n$ be polynomials with $f(0) = 0$. Then, the operator $P := \llim{p}{n} \prod_{i=1}^m U^{f_i(n)}$ is a projection operator. 
\end{theorem}

Our argument is essentially a variation on the methods employed by Schnell. Let us consider the still more specialised case of the result, with $m = d = 1$. On one hand, our result is weaker insofar as it needs the ultrafilter to be minimal in addition to being idempotent. On the other hand, it is also stronger insofar as it identifies the limit explicitly, and works for generalised polynomials. %Of course, restricting to a single unitary operator, one dimension, and minimal idempotents makes the presented result palpably more elementary. Out aim is chiefly to give the reader a flavour and motivation for more technical and advanced results. 

\newcommand{\IPlim}[1]{{\IP\!\!-\!\!\lim_{\!\! #1}} }
Other results which deserve a mention involve $\IP$-limits. These limits are extensively used in ergodic theory, most notably in \cite{BergKnuMc2006}, \cite{FK-IP-limits}, \cite{BFM-IP-limits}. As we will see, these limits are strongly related to ultrafilter limits along idempotents.

For brevity, and to establish a better correspondence with existing literature, we denote $\cF := \cP_{\operatorname{fin}}(\NN)$. We turn $\cF$ into a semigroup by taking the group operation to be the set union, as usual. Additionally, we assume the topology of $\cF$ to be discrete, wherever relevant. Recall that the notation $\alpha < \beta$ for $\alpha,\beta \in \cF$ is a shorthand for $\max \alpha < \min \beta$. We are now in position to define the $\IP$-limit.
\begin{definition}[$\IP$-limit]\label{C:def:IP-lim}\index{IP@$\IP$!IP-limit@$\IP$-limit}
    Let $Z$ be a topological space, and let $\seq{x}{\alpha}{\cF}$ be a sequence of elements of $Z$, indexed by $\cF$. Then we say that $\displaystyle{ \IPlim{\alpha} x_\alpha = y}$ if and only if for any $U \in \Top Z$ with $y \in U$ there exists $\alpha_0 \in \cF$ such that for any $\alpha \in \cF$ with $\alpha > \alpha_0$ it holds that $x_\alpha \in U$.
\end{definition}
To make the notion of $\IP$-limit useful, one needs to define a proper way of passing to subsequences. Note that to extract a subsequence from a sequence $\seq{x}{i}{\NN}$, one normally begins by choosing a sequence of indices $\seq{i}{n}{\NN}$ with $i_n < i_{n+1}$ for all $n \in \NN$, and then looks at the sequence $\left( x_{i_j} \right)_{j \in \NN}$. The following definition provides the right index set for the subsequence of a set-indexed sequence.
\begin{definition}[$\IP$-ring]\index{IP@$\IP$!IP-ring@$\IP$-ring}
  A family $\cF_1 \subset \cF$ is said to be an \emph{$\IP$-ring} if and only if it is of the form: $\cF_1 = \FU{\balpha}$ for some sequence $\balpha = \seq{\alpha}{n}{\NN}$ with $\alpha_n < \alpha_{n+1}$ for all $n \in \NN$, where $\FU{\balpha}$ denotes the set of finite unions. 
\end{definition}
Given an $\IP$-ring $\cF_1 = \FU{\balpha}$, there is a natural way to identify $\cF_1$ with $\cF$, much like there is a natural identification between $\NN$ and a subsequence $\seq{i}{n}{\NN}$. The correspondence is given by the map $\beta \mapsto \alpha_\beta := \bigcup_{i \in \beta} \alpha_i$. Note that the surjectivity is simple, while injectivity relies on the condition that $\alpha_n < \alpha_{n+1}$ for all $n$. This identification leads to the natural extension of Definition \ref{C:def:IP-lim}.
\begin{definition}[$\IP$-limit along $\IP$-ring]\index{IP@$\IP$!IP-limit@$\IP$-limit}\index{IP@$\IP$!IP-ring@$\IP$-ring}
  Let $Z$ be a topological space, and let $\seq{x}{\alpha}{\cF}$ be a sequence of elements of $Z$, indexed by $\cF$. Suppose that $\cF_1 = \FU{\balpha}$ is an $\IP$-ring. Then we define the $\IP$ limit of $x_\alpha$ along $\cF_1$ to be:
  $$ \IPlim{\alpha \in \cF_1} x_\alpha := \IPlim{\beta} x_{\alpha_\beta},$$
  with the understanding that if the expression on the right is undefined then so is the expression on the right.
\end{definition}

An important consequence of Hindman's Theorem is that $\IP$-limits of sequences in a compact space behave much like ordinary limits, as exemplified by the following proposition.
\begin{proposition}
  Let $Z$ be a compact metrizable space, and let $\seq{x}{\alpha}{\cF}$ be a sequence of elements of $Z$, indexed by $\cF$. Then there exists an $\IP$-ring $\cF_1$ such that the limit $ \displaystyle{\IPlim{\alpha \in \cF_1}} x_\alpha $ exists.
\vspace{-\lineskip}
\end{proposition}
\begin{proof}[Sketch of proof.]
  We show that given an open cover $\cC \subset \Top Z$, we can construct an $\IP$-ring $\cF'$ such that there exists $U \in \cC$ such that $x_\alpha \in U$ for all $\alpha \in \cF'$. First, because of compactness, we may assume without loss of generality that $\cC$ is finite. Next, we may consider for each $U \in \cC$ the set $\cA_U := \{\alpha \in \cF \setsep x_\alpha \in U\}$. Clearly, $\cA_U$ for $U \in \cC$ form a finite partition of $\cF$. Hence, by Hindman's Theorem \ref{C:thm:Hindman-sets} we find that one of the cells of this partition, say $\cA_{V}$, contains an $\IP$-ring, say $\cF'$. Directly by construction, $x_\alpha \in V$ for all $\alpha \in \cF'$, as desired.

  We leave it to the reader to apply the above procedure to construct the $\IP$-ring mentioned in the assertion. One can do it inductively, by considering the covers consisting of balls with radii descending to $0$.
\end{proof}

There is a natural link between $\IP$-sets in $\NN$, $\IP$-limits and idempotent ultrafilters. To begin with, we introduce an alternative way of viewing $\IP$-sets, which is in the author's humble opinion a major motivating factor for the study of $\IP$-limits. Recall that for a sequence $\bx$ with elements in an additive group, the set $\FS{\bx}$ can be described as the values of the $\sum_{i \in \alpha} x_i$ for $\alpha \in \cF$. This motivates the following definition.

\begin{definition}[$\IP$-systems]\index{IP@$\IP$!IP-system@$\IP$-system}
  Let $X$ be a commutative semigroup. An $\IP$-system in $X$ is a map $x:\ \cF \to X$ such that $x(\alpha \cup \beta) = x(\alpha) \cup x(\beta)$ whenever $\alpha \cap \beta = \emptyset$.
\end{definition}
It is clear that $\IP$-sets in $X$ are precisely the sets of values of $\IP$-systems. Considering $\IP$-systems gives a clearer understanding of the structure, and is slightly more general, since a given $\IP$-set can potentially correspond to different $\IP$-systems.

Possibly the most frequent and probably the most basic way in which $\IP$-limits occur is in the expressions of the type
$$ \IPlim{\alpha \in \cF_1} x_{n(\alpha)} ,$$
where $n$ is an $\IP$-system and $\cF_1$ is an $\IP$-ring. Such limits are essentially equivalent to limits along idempotent ultrafilters, as shown in the following proposition.
\begin{proposition}\label{C:prop:IP-lim=p-lim}
  Let $Z$ be a metrizable topological space, $\seq{x}{n}{\NN}$ a sequence with elements in $Z$, and let $y \in Z$. Let $n:\ \cF \to \NN$ be an $\IP$-system. Denote the the corresponding $\IP$-sets $A_k = \{ n(\alpha) \setsep \alpha \in \cF,\ \min \alpha > k\}$ and $A := A_0$. The following conditions are equivalent:
  \begin{enumerate}[label={(\textit{\arabic*})},start=1]
   \item\label{C:prop:IP-lim=p-lim:cond:1} There exists an $\IP$-ring $\cF_1$ such that $\displaystyle{\IPlim{\alpha \in \cF_1}} x_{n(\alpha)} = y$.
   \item\label{C:prop:IP-lim=p-lim:cond:2} There exists an idempotent ultrafilter $p \in \bigcap_k \bar{A}_k$ such that $\llim{p}{n} x_n = y$.
  \end{enumerate}
\end{proposition}
\begin{proof}
\begin{description}
 \item[   \ref{C:prop:IP-lim=p-lim:cond:1} $\imply$    \ref{C:prop:IP-lim=p-lim:cond:2}] 
  Suppose that $\cF_1 = \FU{\balpha}$ is such that  $\displaystyle{\IPlim{\alpha \in \cF_1}} x_{n(\alpha)} = y$. By Lemma \ref{C:lem:IP=>in-idempotent} there exists an idempotent $p$ such that for any $k$ the set $\FU{\sigma^k\alpha}$ is $p$-large (where $\sigma^k\alpha = \left(\alpha_{k+l} \right)_{l \in \NN}$, as before). It is clear that $\llim{p}{n} x_n = y$ and that $p$ satisfies the remaining conditions.

 \item[   \ref{C:prop:IP-lim=p-lim:cond:2} $\imply$    \ref{C:prop:IP-lim=p-lim:cond:1}] Fix a metric $\rho$ on $Z$. Let $B_k$ denote the set $\{ n \in \NN \setsep \rho(x_n,y) < 1/k \}$. Because the sets $B_k$ are $p$-large, an application of Lemma \ref{C:lem:IP<=in-idempotent} shows that one can construct a sequence of integers $\mathbf{m} = \seq{m}{i}{\NN}$ such that $\FS{\sigma^k\mathbf{m}} \subset A \cap B_k$ for any $k$. Moreover, because $p \in A_l$ for any $l$, we can ensure that $m_i = \sum_{j \in \alpha_i} n(i)$ with $\alpha_i < \alpha_{i+1}$ for any $i$. It follows that $\cF_1 := \FU{\balpha}$ is the sought $\IP$-ring. %Given such $p$, we construct a descending sequence of $\IP$-rings $\cF = \cF_0 \supset \cF_1 \supset \cF_2 \supset \dots$ and finite sets $\gamma_i$ such that if $\alpha > \gamma_i$ and $\alpha \in \cF_i$ then $\rho(y,x_{n(\alpha)}) < \frac{1}{i}$, and additionally the set $A_i := \{ n(\alpha) \setsep \alpha \in \cF_i \}$ is $p$-large. Suppose $\cF_i$ is already constructed. Then  
\end{description}
\end{proof}

We have seen that there is ample justification for interest in when limits of unitary operators are projections. We restricted our attention to powers of unitary operators, but in literature one frequently encounters (unitary) actions of general (semi-)groups. The following definition should be construed as a generalisation of the assignment $n \mapsto U^n$ for a unitary operator $U$. We could have stated the definition in much more general terms, but for our purposes the following will suffice.
\begin{definition}[Unitary action]\index{unitary action}
  Let $X$ be a commutative semigroup, and let $\cH$ be a Hilbert space. A \emph{unitary action} of $X$ on $\cH$ is a map $x \mapsto U_x$, such that for any $x, y \in X$ it holds that $U_{x+y} = U_x U_y$.
\end{definition}

We are now in position to state some noteworthy results. The simplest among them is the following.
\begin{theorem}[\cite{FK-IP-limits}]
  Let $\cH$ be a separable Hilbert space, let $X$ be a commutative group, and let $x \mapsto U_x$ be a unitary action of $X$ on $\cH$. Suppose that $x :\ \cF \to X$ is an $\IP$-system and $\cF_1$ is an $\IP$-ring such that the following limit exists:
  $$ P := \IPlim{\alpha \in \cF_1} U_{x(\alpha)}.$$
  Then, $P$ is an orthogonal projection.
\end{theorem}

\begin{remark}
  In the case of actions of the integers, the above theorem is equivalent to Lemma \ref{B:lem:p-lim-U^n=P}, modulo an application of principle \ref{C:prop:IP-lim=p-lim}.
\end{remark}

Several generalisations of the above result are possible. Firstly, $\IP$-system in the above statement can be replaced by a so called $\IP$-polynomial. We don't define this notion rigorously, but merely remark that the relation in which $\IP$-polynomials stand to $\IP$-systems is similar to the relation of polynomials to linear functions. For a precise definition, see \cite{BergKnuMc2006}.

\begin{theorem}[\cite{BFM-IP-limits}]
  Let $\cH$ be a separable Hilbert space, and let $x \mapsto U_x$ be a unitary action of $\ZZ^k$ on $\cH$. Suppose that $x : \cF \to X$ is an $\IP$-polynomial and $\cF_1$ is an $\IP$-ring such that the following limit exists:
  $$ P := \IPlim{\alpha \in \cF_1} U_{x(\alpha)}.$$
  Then, $P$ is an orthogonal projection.
\end{theorem}

The above theorem can be generalised further, to allow for $\mathsf{FVIP}$ systems in place of $\IP$-polynomials. The relevant theorem is due to Bergelson, H{\aa}land Knutson and McCutcheon. Because definition of $\mathsf{FVIP}$ exceeds the scope of our investigation, we do not formulate the the theorem. We refer the reader to \cite{BergKnuMc2006}.

% *****************************************

\renewcommand{\cF}{\mathcal{F}}

\chapter{Applications in voting \& model theory}
\label{D:chapter} 
\lhead{Chapter \ref{D:chapter}. \emph{Voting \& models}} 

\section{Voting \& Arrow's theorem}

Yet another way to view ultrafilters is through the prism of voting procedures. Many strengths of this approach lie more in the intuitively appealing picture than in rigorous results, which should be borne in mind throughout this section. Whenever non-standard terminology is used, the goal is purely expository, and more theoretically inclined reader may disregard these superfluous details.

\newcommand{\rel}{\prec}
\newcommand{\relaggr}{\prec_{\text{soc}}}

Let us begin by introducing a situation which will essentially remain fixed throughout this section. We consider a population $X$ (where $X$ is a possibly infinite set, with no extra structure), which is voting on candidates from a set $C$ (again, no additional structure is required; in practical applications $C$ is finite, but we don't make this restriction). Each voter $x \in X$ has some preference between the candidates, which are expressed by a total order $\rel_x$, i.e. $a \rel_x b$ if and only if $x$ prefers $b$ to $a$. Note that we do not include any notion of \emph{strength} of preference in our picture, nor do we allow a voter to be undecided between two options. Moreover, we assume each voter to be rational to have preferences that form a total order: if for a voter $x$ a candidate $b$ is preferable to candidate $a$, and candidate $c$ is preferable to $b$, then $c$ is also preferable to $a$.

The goal of the vote is to establish an aggregated preference. More precisely, a \emph{social welfare function} (also known as \emph{preference aggregation rule}) is a function that assigns to the family of preferences $\seq{\rel}{x}{X}$ a total order $\relaggr$ which we consider to be the outcome of the vote, i.e. the preference of the society as a whole, or the aggregated preference of the voters. 

There are several conditions that a preference aggregation rule could be expected to satisfy:
\begin{description}\index{social welfare function}
	\item[(M)] Monotonicity (also known as Positive Association of Social and Individual Values) --- if a candidate moves up in individual rankings, then his final position does not fall. 
  Formally, let $a \in C$ be a candidate, and $(\rel_x)_{x\in X}, (\rel'_x)_{x\in X}$ be two individual preferences. Suppose that for any voter $x$ the following holds: if $b,c \in C \setminus \{a\}$ then $b \rel_x c$ if and only if $b \rel_x c$ ; moreover, if $b \rel_x a$ then also $b \rel_x' a$. We require that in this situation for any $b \in C$ with $b \relaggr a $ we have $b \relaggr' a$.
	\item[(NI)] Non-imposition (also known as Unanimity) --- if the vote is unanimous, then the aggregated preference is the same as the individual preference of the voters. Formally, if there is a universal total order $\rel_*$ on $C$ such that for all candidates $a,b \in C$ and all voters $x$ it holds that $ a \rel_x b$ if and only if $ a \rel_* b $, then also for all candidates $a,b$ it holds that $a \relaggr b$ if and only if $ a\rel_* b$. 

	\item[(IIA)] Independence of Irrelevant Alternatives --- relative ranking of any two given individuals is independent of preferences concerning other individuals. Formally, let $a,b \in C$ be two candidates, and let  $(\rel_x)_{x\in X}, (\rel'_x)_{x\in X}$ be two individual preferences such that for any voter $x$ we have $a \rel_x b $ if and only if $a \rel_x' b$. Then we require that $a \relaggr b $ if and only if $a \relaggr' b$.
\end{description}

Note that in presence of (IIA), the conditions (M) and (NI) are equivalent to apparently stronger conditions given below.

\begin{description}
	\item[(M')] If $a \in C$ is a candidate and $\rel_x$ and $\rel_x'$ are two individual preferences, such that for any other candidate $b \in C$ the condition $b \rel_x a$ implies $b \rel_x' a$, then the condition $b \relaggr a$ implies $b \relaggr' a$ for any $b \in C$.

	\item[(NI')] If for some two candidates $a,b \in C$ and all voters $x \in X$ it holds that $ a \rel_x b$, then also  $a \relaggr b$. 
\end{description}

The reason for giving the more complicated and weaker assumptions is that in social choice theory, these are more commonly accepted and easier to justify. In fact, most preference aggregation rules encountered in practice satisfy (M) and (NI), but fail to satisfy (IIA). The cause this state of affairs will become clear as soon as Arrow's theorem is formulated.

A voter $x \in X$ is called a \emph{dictator} if he alone controls the election. More precisely, $x$ is a dictator if and only if for any individual preferences and any pair of candidates $a,b \in C$ it holds that $a \relaggr b$ if and only if $a \rel_x b$. This means that the social preference is always identical to the preference of $x$, even if the entire rest of the society holds precisely opposite preferences to $x$. It is often %$\text{sometimes}^{\text{\textsf{[citation needed]}}}$
 thought that dictatorship should be avoided, at least on the grounds that in that case voting does not contribute any new information. The following celebrated theorem due to Arrows shows that, in most practical situations, this can only be accomplished is we sacrifice some of the desirable properties mentioned above. We mostly follow the approach by Galvin \cite{arrows-by-galvin}, see also \cite{arrows-by-tao} and \cite{arrows-by-tao}.

\begin{theorem}\label{D:thm:arrows-finite}
\index{Arrows|(}
\index{social choice}
  For a \emph{finite} set of voters $X$ ranking candidates from a set $C$ with $\# C \geq 3$, any {preference aggregation rule} that satisfies conditions (IIA), (M) and (NI) is necessarily dictatorial.
\end{theorem}

We will derive Arrow's theorem from the following more general result, which does not require finiteness of the set of the voters.

\begin{theorem}\label{D:thm:arrows-general}
  Let $X$ be an arbitrary set of voters, ranking candidates from a set $C$ with $\# C \geq 3$. For a fixed preference aggregation rule, define $\cD$ to be the family of those sets of voters who have control over the election:
    $$\cD := \{A \in \cP(X) \setsep  (\forall a,b \in C)\ ((\forall x \in A)\ a \rel_x b) \imply (a \relaggr b)\}.$$
  If the assumptions (IIA), (M) and (NI) are satisfied, then $\cD$ is an ultrafilter.
\end{theorem}

\begin{remark}
  It is easily seen that $x$ is a dictator if and only if $\cD = \principal{x}$ is the principal ultrafilter corresponding to $x$.
\end{remark}
\begin{remark}
  The assumption $\# C \geq 3$ in the above theorems is essential, since for $\# C = 2$ and finite $X$ (with odd cardinality) a simple majority vote is non-dictatorial and satisfies (M), (NI) and (IIA). More generally, if $\# C = 2$, one can construct a fairly general weighted voting procedure. For $2$ alternatives, there are just two possible preferences, and it will be convenient to label them ``YES'' and ``NO''. Let us attach to each voter $x \in X$ a weight $w_x \geq 0$, and define some threshold $0 < t < \sum_{x\in X} w_x$. We declare that the society chooses ``YES'' if and only if $\sum_{x \in Y} w_x > t$, where $Y$ is the set of those voters who chose ``YES''. It is clear that the conditions (M) and (NI) are satisfied, as well as trivially (IIA). As long as $w_x < t$ for all $x$, this scheme is not dictatorial.

  We use this opportunity to stress that in the formulation of Arrow's theorem we do not in any way require that voters should be ``equal'', nor do the candidates have to be treated ``equally''.
\end{remark}

\begin{proof}[Proof of Theorem \ref{D:thm:arrows-general}]
  By the (NI) property together with (IIA), we have $X \in \cD$. Also, clearly $\emptyset \not \in \cD$. 

  If $A \in \cD$ and $B \supset A$, then it is also visible that $B \in \cD$, since quantifying over $B$ gives a stronger condition than quantifying over $A$. 

  We now check that if $A,B \in \cD$ then $A \cap B \in \cD$. For a proof by contradiction, suppose that $A \cap B \not\in \cD$. Then, there are some candidates $a,b \in C$, such that for some individual preferences $\rel_x$ we have $a \rel_x b$ for all $x \in A \cap B$, but the aggregate preference is in favour of $a$: $b \relaggr a$. By (M), we may assume without loss of generality that $b \rel_y a$ for $y \in X \setminus A\cap B$, since moving $a$ up on some preference lists cannot harm his final position. Let us now consider another candidate $c \in C \setminus \{a,b\}$, whose existence is guaranteed by $\# C \geq 3$. By (IIA), we can assign any preference between $c$ and $a,b$ (consistent with already existing preferences between $a$ and $b$) without changing the relation $b \relaggr a$. Let us consider the following assignment of preferences:
\begin{align*}	
  &\text{for } x \in A \cap B :\ &a \rel_x c \rel_x b, \\
  &\text{for } x \in A \setminus B :\ &b \rel_x a \rel_x c, \\
  &\text{for } x \in B \setminus A :\ &c \rel_x  b \rel_x a, \\
  &\text{for } x \in X \setminus (A \cup B) :\ &\text{whatever.}  
\end{align*}

  With this assignment, we have for $x \in A$: $a \rel_x c$. Since we assumed $A \in \cD$, it follows that $a \relaggr c$. Similarly, we have for $x \in B$: $c \rel_x b$. Since $B \in \cD$, it follows that $c \relaggr a$. It follows that $a \relaggr c \relaggr b$, contradicting the assumption $b \relaggr a$. 

\newcommand{\cDa}{\cD_{a,*}}
\newcommand{\cDb}{\cD_{*,b}}

 At this point, we have shown that $\cD$ is a filter. We will now proceed to show the ultrafilter property, namely that for any partition $X = A \cup B$ with $A \cap B = \emptyset$, either $A \in \cD$ or $B \in \cD$. For that purpose, we will provide an alternative description of $\cD$. For a fixed pair of candidates $a,b \in \cD$, let $\cD_{a,b}$ denote the family of those sets of voters who have control the choice between $a$ and $b$, in the sense that if they prefer $b$ to $a$, then the collective preference is also in favour of $b$ over $a$:
  $$\cD_{a,b} := \{A \in \cP(X) \setsep (  (\forall x \in A)\ a \rel_x b) \imply (a \relaggr b)\}$$
 It is clear by the definition of $\cD$ that:
  $$ \cD = \bigcap_{a,b} \cD_{a,b}. $$

  We claim that all the sets $\cD_{a,b}$ are in fact equal to one another, and hence also to $\cD$.

  We first show for $a,b,c \in C$, distinct, that $\cD_{a,b} \subset \cD_{a,c}$. For a proof, let us take $A \in \cD_{a,b}$ and show that $A \in \cD_{a,c}$. We need to prove that for any individual preferences such that $a \rel_x c$ for all $x \in A$, we have $a \relaggr c$. Because of (M), we may assume that for $x \in X \setminus A$ we have $c \rel_x a$. By (IIA), we may assume any preference between $a,c$ and b. Let us consider the following assignment of preferences:
  \begin{align*}	
    &\text{for } x \in A :\ & a \rel_x b \rel_x c, \\
    &\text{for } x \in X \setminus A :\ & b \rel_x c \rel_x a. \\
  \end{align*}
  Since $A \in \cD_{a,b}$ and for $x \in A$ we have $a \rel_x b$, it follows that $a \relaggr b$. Since we have $b \rel_x c$ for all $x \in X$, it follows that $b \relaggr c$. Combining these, we conclude that $a \relaggr c$, as desired. Because the choice of $A$ was arbitrary, it follows that indeed $\cD_{a,b} \subset \cD_{a,c}$. 

  By a symmetric argument we can also verify the inverse inclusion, so for  arbitrary distinct $a,b,c \in C$ we have  $\cD_{a,b} = \cD_{a,c}$. By the same reasoning, but with preferences inverted, we also find $\cD_{a,b} = \cD_{c,b}$. Finally, we also note that:
    $$ \cD_{a,b} = \cD_{a,c} = \cD_{b,c} = \cD_{b,a}$$
  Hence, we conclude that for any $a,b,a',b'\in C$ with $a \neq b$ and $a' \neq b'$ (but possibly $\{a,b\} \cap \{a',b'\} \neq \emptyset$ ) we have $\cD_{a,b} = \cD_{a',b'}$. It follows by taking intersection over all $a',b'$ that $\cD_{a,b} = \cD$.

  After this preliminary work, it will suffice to show that $\cD_{a,b}$ has the property that for any partition $X = A \cup B$ with $A \cap B = \emptyset$, we have $A \in \cD_{a,b}$ or $B \in \cD_{a,b}$. But this is relatively simple. Consider the preference such that $a \rel_x b$ if $x \in A$ and $b \rel_x a$ if $x \in B$. It follows from (M) that if $a \relaggr b$, then $A \in \cD_{a,b}$, and if $b \relaggr a$ then $B \in \cD_{a,b}$. Thus, $\cD_{a,b} = \cD$ has the ultrafilter property.
\end{proof}

\index{Arrows}
\begin{proof}[Proof of Arrows Theorem \ref{D:thm:arrows-finite}]
  Under the assumptions of Arrows Theorem \ref{D:thm:arrows-finite}, the conditions of Theorem \ref{D:thm:arrows-general} are clearly satisfied, so the family $\cD$ defined as in the formulation of Theorem is an ultrafilter. Since $X$ is finite, the only possible ultrafilters are the principal ones, so $\cD$ is the principal ultrafilter $\principal{x}$ corresponding to some voter $x \in X$. In particular $\{x\} \in \cD$, so $x$ is the sought dictator. 
\end{proof}

The following result is a converse of Theorem \ref{D:thm:arrows-general}.

\begin{proposition}\label{D:prop:arrows-inverse}
  Let $\cU \in \Ultrafilters{X} $ be an ultrafilter. Define preference aggregation rule by declaring $a \relaggr b$ to be equivalent to $\{ x \in X \setsep a \rel_x b \} \in \cU$. This gives a well defined preference aggregation rule that satisfies the conditions (IIA), (M) and (NI).
\end{proposition}
\begin{proof}
  \newcommand{\asym}[1]{\operatorname{asym}_{a,b}\left(#1\right)}
  \newcommand{\trans}[1]{\operatorname{trans}_{a,b,c}\left(#1\right)}
  \newcommand{\total}[1]{\operatorname{total}_{a,b}\left(#1\right)}

  Let us note that $a \relaggr b \in \{\top,\bot\}$ is given by $(a \relaggr b) = \llim{\cU}{x}{(a \rel_x b)}$ (where the limit is taken in $\{ \top, \bot \}$ with discrete topology. Equivalently, if one identifies binary relations with elements of $\{\top,\bot\}^{C \times C}$ with the natural Tichonoff/pointwise convergence topology, then $(\relaggr) = \llim{\cU}{x}{(\rel_x)}$.

  We will first check that $\relaggr$ is indeed a total order. For this, we need to check a number of conditions, namely antisymmetry, transitivity and totality. This can be done directly, but we will pursue a slightly more sophisticated approach. For $a,b,c \in C$, and binary relation $\rel$ on $C$ denote the sentences 
  \begin{align*}
    \asym{\rel} &:= \neg ( (a \rel b) \wedge (b \rel a)) \\
    \trans{\rel} &:= ((a \rel b) \wedge (b \rel c)) \imply (a \rel c) \\
    \total{\rel} &:= (a \rel b) \vee (b \rel a)
  \end{align*}
  
  A binary relation $\rel$ on $C$ is a total order if and only if the sentences $\asym{\rel}$, $\trans{\rel}$ and $\total{\rel}$ are true for any $a,b,c$. These are clearly quantifier free sentences in first order logic, true whenever $\rel$ is a total order. Consider any such sentence $\phi(\rel)$, viewed as a map from $\{\top,\bot\}^{C \times C}$ to $\{\top,\bot\}$. Since $\phi(\rel)$ depends only on finitely many ``coordinates'', it is clearly continuous. From this continuity and the description of $\relaggr$ it follows that:
  $$ 
    \phi(\relaggr) = \phi\left( \llim{\cU}{x}{(\rel_x)} \right) = \llim{\cU}{x}{\phi(\rel_x)} =  \llim{\cU}{x}(\top) = \top
  $$
  In particular  $\asym{\relaggr}, \trans{\relaggr}$ and $\total{\relaggr}$ are true for any $a,b,c$, and hence $\relaggr$ is a total order.\footnote{The advantage of this approach is that it does not rely too much on the form of the conditions that define the order. The same proof works for weak orders, equivalence relations, and generally all relations that can be described by conditions of the form $(\forall a,b,c,\dots ) \phi_{a,b,c,\dots }(\rel)$.}

  The condition (M) holds for this preference aggregation rule directly by the definition, relying chiefly on the fact that $\cU$ is closed under taking supsets. Likewise, the condition (NI) holds, because $X \in \cU$. Finally, the condition (IIE) holds, because the definition of $a \relaggr b$ makes no mention of any other candidates.  
\end{proof}

One can check that the construction of an ultrafilter from a preference aggregation rule in Theorem \ref{D:thm:arrows-general} and  the construction of a preference aggregation rule from an ultrafilter in Proposition \ref{D:prop:arrows-inverse} are mutually inverse. 

\begin{corollary}
  For a fixed set of candidates $C$, with $\# C \geq 3$, there is a bijective correspondence between ultrafilters on $X$ and preference aggregation rules satisfying for a vote of the population $X$ on candidates from a set $C$ with $\# C \geq 3$ that satisfy conditions (M), (IIE) and (NI).
\end{corollary}

The sentiment that ultrafilters can be thought as a voting system is was expressed by Tao in some of his expositor materials \cite{}. The subsequent applications in model theory can be thought of as a generalisation of this idea.
\index{Arrows|)}

\section{Ultrapowers}
\index{model theory!ultraprower|see{ultraproduct}}
\index{model theory!ultraproduct}

\newcommand{\Space}{S}

\newcommand{\TotalOrders}{\mathrm{TotOrd}}

\newcommand{\NNp}{\NN}

We give a brief overview of the foundations of model theory. Because we believe the basics of model theory would be familiar to a working mathematician, at least on the intuitive level, we do not go into much detail. For a more detailed discussion, we refer to any number of introductory materials on model theory, such as \cite{models-by-marker} or \cite{models-by-hodges}. A very accessible treatment, which includes the ultrafilter construction and \L{}o\'{s} theorem roughly in the form presented here, is provided by the lecture notes \cite{models-by-clark}. 

To express various mathematics, one first need a language:

\begin{definition}[Language]\index{model theory!language}
  A language $\cL$ consists of the following data:
  \begin{enumerate}
   \item For each $n \in \NNp$, a family of $n$-argument function symbols (usually denoted by $f(x_1,x_2,\dots,x_n)$).
   \item For each $n \in \NNp$, a family of $n$-argument relation symbols (usually denoted by $R(x_1,x_2,\dots,x_n)$).
   \item A family of constant symbols  (usually denoted by $c$).
  \end{enumerate}
\end{definition}

We stress that a function symbol is not a function, but merely a symbol used to denote a function. Likewise for relations and constants.

The assignment of a meaning to a symbol goes by the name of interpretation, and is formalised as follows.

\begin{definition}[Structure]\index{model theory!structure}
  For a language $\Language$, a $\Language$-structure $\Structure$ consists of the following data:
  \begin{enumerate}
    \item The underlying space $\Space$
    \item For each $n$-argument function symbol $f$, a function $f_\Structure :\ \Space^n \to \Space$.
    \item For each $n$-argument relation symbol $R$, a relation $R_\Structure :\ \Space^n \to \{ \top, \bot \}$.
    \item For each constant symbol $c$, a constant $c_\Structure \in \Space$.
  \end{enumerate}
\end{definition}

Apart from the symbols specific to the language, we also need logical symbols and variables to construct mathematically meaningful entities. The set of variables will be fixed and denoted by $\seq{x}{i}{\NN}$, but in practice different symbols can be used for increased notational convenience (for example, $x,y,z,\dots$).

Terms in the given language are the well formed expressions that describe elements of the underlying space. They are the basic building blocks for more complicated expressions. This is made precise by the following definition.

\begin{definition}[Terms]
\index{model theory!term}
  The set of \emph{terms} over the language $\Language$ is defined to be the smallest family of expressions such that:
  \begin{enumerate}
   \item Every constant symbol $c$ is a term.
   \item Every variable symbol $x$ is a term.
   \item If $f$ is an $n$-argument function symbol and $t_1,t_2,\dots,t_n$ are terms, then the expression $f(t_1,t_2,\dots,t_n)$ is a term.
  \end{enumerate}
\end{definition}

Formulas are the well formed expressions that describe statements that can either true of false, after the interpretation. This is made precise by the following definition.

\begin{definition}[Formulas]\index{model theory!formula}
  The set of \emph{formulas}  over the language $\Language$ is defined to be the smallest family of expressions such that:
  \begin{enumerate}
    \item If $R$ is an $n$-argument relation, and $t_1,t_2,\dots,t_n$ are terms, then $R(t_1,t_2,\dots,t_n)$ is a formula.
    \item If $\phi$ and $\psi$ are formulas, then the following are formulas: $(\neg \phi)$, $(\phi \wedge \psi)$, $(\phi \vee \psi)$, $(\phi \Rightarrow \psi)$ and $(\phi \Leftrightarrow \psi)$.
    \item If $\phi$ is a formula and $x$ is a variable, then the following are formulas: $(\forall x)\ \phi$ and $(\exists x)\ \phi$.
  \end{enumerate}
\end{definition}

\begin{remark}
  Formulas that are logically equivalent will not normally be distinguishable. For instance $\phi \vee \psi$ and $\neg (\neg \phi \vee \neg \psi)$ will play the same role. Using the standard logical identities, we can restrict the vocabulary of logical symbols to $\neg, \vee$ and $\exists$, which we will implicitly use in proofs using structural induction. Conversely, we may treat any additional logical symbols, such as the disjoint alternative $\underline{\vee}$, to be just shorthands for their definitions in terms of the more fundamental symbols.

  Note that the above allows formulas such as $(\forall x_1) ((\forall x_1)\ \phi)$. The convention is then to bind the variable to the most nested quantifier, so the formula in question is logically equivalent to $(\forall x_2) ((\forall x_1)\ \phi)$, provided that $x_2$ does not appear in $\phi$. However, we will never use formulas of that kind in practice.
\end{remark}

A formula is allowed to contain free variables, i.e. such that are not bound by a quantifier. We will sometimes write $\phi(x_1,x_2,\dots,x_k)$ instead of $\phi$ if $x_1,x_2,\dots,x_k$ are unbound variables, to highlight this. It follows that formulas cannot yet be assigned a truth or false within any $\Language$-structure. Sentences are the type of formulas to which a logical value can be ascribed. Some level of vagueness is allowed, because we do not define what it means for a variable to be bound by a quantifier, relying on the intuitive understanding of the reader.

\begin{definition}[Sentence]
\index{model theory!sentence}
  A formula $\phi$ is said to be a sentence if and only if it contains no free variables. 
\end{definition}

We now give a definition of how sentences are interpreted inside structures. We allow ourselves some vagueness also at this point, because we merely formalise the skill of interpreting formulas that the reader obviously possesses. One point to keep in mind is that the quantifiers are always interpreted to run over the underlying space (hence, no quantification over sets of elements, or elements of some external sets, is possible). A formal definition uses induction over complexity of formulas, and can be found for instance in \cite{models-by-marker}.

\begin{definition}[Interpretation]
\index{model theory!interpretation}
  If $\Structure$ is an $\Language$-structure and $\phi$ is a sentence over $\Language$, then $\phi$ corresponds to a statement $\phi_\Structure$ obtained by replacing all function symbols $f$ by the corresponding functions $f_\Structure$, all relation symbols $R$ by the corresponding relations $R_\Structure$, all constant symbols $c$ by the corresponding constants $c_\Structure$, replacing each quantifier of the form $\forall x$ or $\exists x$, where $x$ is a variable, by $\forall x \in S$ or $\exists x \in S$, and (finally) interpreting the logical symbols in the standard way.

  If the statement $\phi_\Structure$ is true, we say that $\phi$ is true in $\Structure$, which we express by writing $\Structure \models \phi$. More generally, if $\Phi$ is a set of sentences, then we say that $\Phi$ is true in $\Structure$ if $\phi$ is true in $\Structure$ for any $\phi \in \Phi$; we express this by writing $\Structure \models \Phi$.

  Finally, if $\phi(x_1,x_2,\dots,x_n)$ is a formula in $n$ unbound variables $x_1,x_2,\dots,x_n$, then for $a_1,a_2,\dots,a_n \in S$, we say that $\phi(a_1,a_2,\dots,a_n)$ is true, or $\Structure \models \phi(a_1,a_2,\dots,a_n)$, if and only if the substitution procedure just described, combined with replacing $x_i$ by $a_i$, yields a true sentence.
\end{definition}

\newcommand{\Axioms}{\mathfrak{A}}
\begin{definition}[Theory]
\index{model theory!theory}
  A \emph{theory} $\Theory$ over the language $\Language$ is a set of sentences over $\Language$. Some authors require theories to be consistent and closed under logical consequence; but we pose no such restriction. 

  A $\Language$-structure $\Structure$ is said to be a model of $\Theory$ if and only if all sentences from $\Theory$ are true in $\Structure$, i.e. if $\Structure \models \Theory$. A sentence $\phi$ is said to be a consequence of $\Theory$, which we express by writing $\Theory \models \phi$, if and only if for any $\Language$-structure $\Structure$ with $\Structure \models \Theory$ it holds that $\Structure \models \phi$.
\end{definition}

To introduce a specific theory, one might proceed as follows. First, specify the language needed to describe the desired properties. Next, specify a set of statements $\Axioms$, referred to as \emph{axioms}  that describe relations between various symbols. It would normally considered a good thing if the list of $\Axioms$ is relatively short and effectively generated. Finally, consider the theory $\Theory$ consisting of all logical consequences of the accepted axioms: for a formula $\phi$, we have $\phi \in \Theory$ if and only if $\Axioms \models \phi$.

Before we proceed to discussing some examples, we need to make the following clarification.
\begin{remark}[Identity]
  We deliberately did not include the equality symbol ``$=$'' as a logical symbol. This goes against the current fashion, but was used for example by Robinson in \cite{models-by-robinson}. In all theories under consideration, there will instead be a binary relation $=$, corresponding to equality. For this relation to serve as equality, we need to ensure several properties. Firstly, we need it to be an equivalence relation, which is easy to ensure by adding axioms: 
  \begin{align*}
   &(\forall x)\ x = x, \\
  &(\forall x) (\forall y) x = y \imply y = x, \\
   &(\forall x) (\forall y)(\forall z) x = y \wedge y = z \imply x = z.
  \end{align*}
  
  Secondly, we need to ensure that equality behaves appropriately with functions, which is accomplished by adding for any $n$-argument function symbol $f$ axiom: 
$$(\forall x_1,\dots,x_n) (\forall y_1,\dots,y_n) x_1 = y_1 \wedge \dots \wedge x_n = y_n \imply f(x_1,\dots,x_n) = f(y_1,\dots,y_n).$$
  
  Finally, we need to ensure that equality behaves appropriately with relations, which is accomplished by adding for any $n$-argument relation symbol $R$ axiom: 
  $$(\forall x_1,\dots,x_n) (\forall y_1,\dots,y_n) x_1 = y_1 \wedge \dots \wedge x_n = y_n \imply (R(x_1,\dots,x_n) \iff R(y_1,\dots,y_n)).$$

  We refer to these axioms as the axioms of equality. Elements $a,b \in S$ with $a =_{\Structure} b$ will be indistinguishable within the theory, but may well be distinct elements of the set $S$. A model where the relation $=_{\Structure}$ is interpreted as the equality of set elements is called \emph{normal}, hence what we said amounts to acceptation of non-normal models.

  If $\Structure$ is a non-normal model of some theory $\Theory$, then there is a natural way to construct a normal model on the set $S/=_{\Structure}$. Thus, considering non-normal models does not provide more generality in any real sense. The reason for our treatment is that the ultraproduct construction is more elegant that way. We always assume that the considered theories have the binary relation $=$, and the axioms mentioned above belong to $\Theory$.
\end{remark}

\begin{example}[Sets]
  Let $\Language = \{=\}$, and the let axioms $\Axioms$ consist only of the identity axioms (which in this case amounts to the statement that $=$ is an equivalence relation). Then the corresponding theory describes sets.
\end{example}

\begin{example}
  If we take $\Language = \{\cdot,=\}$ and impose no additional axioms, except for the ones about identity, than the resulting theory describes groupoids (also known as magmas). 

  If we add the axiom of connectivity:
  $$(\forall x,y,z) (x\cdot y) \cdot z = x \cdot( y \cdot z),$$
  we get the theory of semigroups.

  If we add the axiom of existence of unit:
  $$(\exists e) (\forall x) (ex = x) \wedge (xe = x)$$
  then we get the theory of monoids.

  Depending on the taste, one could alternatively define monoids by adding a constant symbol for the unit $e$, and adding the shorter axiom:
  $$(\forall x) (ex = x) \wedge (xe = x)$$
    
  If we add the axiom that each element has an inverse: 
  $$(\forall x) (\exists y) (xy = e)$$
  then we get the theory of groups. Note that if we do not decide to add $e$ as a constant, we need to treat this sentence as a shorthand for a sentence similar to: 
  $$ (\exists e) ((\forall x) (ex = x) \wedge (xe = x)) \wedge ((\forall x) (\exists y) xy = e)).$$
\end{example}

\begin{example}
  Let us take $\Language = \{\cdot,<\}$. We impose, as always, axioms of identity. If we also include the axiom of transitivity:
  $$ (\forall x,y,z)\ (x < y) \wedge (y<z) \imply (x<z),$$
  and the axiom of strong asymmetry:
  $$ (\forall x)\ \neq(x < x),$$
  then the resulting theory describes partial orders. If we add the axiom of totality:
  $$ (\forall x,y)\ (x < y) \vee (y<x),$$
  then we get the theory of total orders. We may also add axioms for theory of dense orders:
  $$ (\forall x,y)\ ( (x < y) \imply (\exists z)\ x < z < y ).$$
\end{example}

  The results of the previous section concerned the issue of voting. One could consider each individual preference $\rel_x$ as a model for the theory\footnote{
    There is a slight technical difficulty here. It is not difficult to get the theory of total orders: we just need one binary relation $\rel$ in the language, and axioms of transitivity, asymmetry and totality. We can also impose the condition that $C$ is a subset of the set being ordered by adding a $\#C$ constant symbols to the theory, one for each element of $C$, say $\seq{c}{a}{C}$, and $\# \binom{C }{ 2}$ axioms ensuring that these constants are different: $\neg( c_a = c_b)$. The difficulty lies in ensuring that the universe is not larger. If $C$ is finite, we can add axiom saying that each element is equal to one of the introduced constants: $(\forall x)\ \bigwedge_{a \in C} (x = c_a)$ (where $\bigwedge$ is a shorthand for multiple application of $\wedge$). For infinite $C$, we can't of course form this sentence, and probably we have no way to ensure that there are no extra elements, hence we would be more correct to speak of the theory of orders on supersets of $C$ in this case. 
  }
$\Theory$ of total orders on $C$, or more precisely as the only piece of data we need to identify such a model. One of the result was that given the individual preferences (or a family of models for $\Theory$ indexed by $X$) and an ultrafilter $\cU$ on $X$, we were able to construct an aggregated preference, which was yet another model of $\Theory$. We now want to extend this approach to different theories. The case that will be of most interest will be when $X = \NN$ and the theory $\Theory$ has a specified standard model --- the outcome will be an introduction of a non-standard one.

\newcommand{\StructureP}{\mathsf{P}} 

\begin{definition}[Construction of ultraproducts]
\index{ultraproduct}
  Let $X$ be a set, let $\cU$ be a distinguished ultrafilter on $X$, and let $\Structure_x$ for $x \in X$ be an $\Language$-structure. Then we define the \emph{ultraproduct} $\StructureP := \prod^\cU_{x \in X} \Structure_x$ as follows.

  As the underlying set, we take the standard product $P := \prod_{x \in X} S_x$.

  For any $n$-argument function symbol $f$ in $\Language$, we define the corresponding function coordinatewise:
  $$f_\StructureP( a^1, a^2, \dots, a^n) := \left( f_{\Structure_x}(a^1_x,a^2_x,\dots,a^n_x \right)_{x \in X}.$$
  
  For any $n$-argument relation $R$ in $\Language$, we define 
  $$R_\StructureP( a^1, a^2, \dots, a^n) := \llim{\cU}{x}{R_{\Structure_x}(a^1_x,a^2_x,\dots,a^n_x)}.$$

  (In particular, if $\Language$ contains a relation symbol $=$, and its interpretation $=_{S_x}$ is the actual equality of set elements, then $=_{\StructureP}$ is the equality on $\cU$-many indices, as opposed to equality per se. Hence, $\StructureP$ is not normal.)

  If all structures $\Structure_x$ in the product are equal to some fixed structure $\Structure$, then the product is referred to as an \emph{ultrapower} of $\Structure$.
\end{definition}

\begin{theorem}[\L o\'s]\label{D:thm:los}
\index{model theory!\L o\'s}
\index{model theory!ultraproduct}
\index{model theory!interpretation}

  Let $\phi(v^1,\dots,v^n)$ be any formula over the language $\Language$ with $n$ free variables $v^1,\dots,v^n$, and let $a^1,a^2,\dots,a^n \in P$. Then, the following conditions are equivalent:
\begin{enumerate}
\item  $\StructureP \models \phi(a^1,a^2,\dots,a^n)$
\item $\Structure_x \models \phi(a^1_x,a^2_x,\dots,a^n_x)$  for $\cU$-many $x \in X$. 
\end{enumerate}

In particular, if $\phi$ has no free variables and $\StructureP$, and $\Structure_x \models \phi$ for all $x$, then also $\StructureP \models \phi$. If $\Structure_x = \Structure$ for some fixed structure $\Structure$, then $\Structure$ and $\Structure$ model the same sentences. If all structures $\Structure_x$ are models of a theory $\Theory$, then also $\StructureP$ is a model of $\Theory$.
\end{theorem}
\begin{proof}
  To keep notation concise, let $v = v^1,v^2,\dots v^n$ and $a = a^1,a^2,\dots, a^n$. We prove the characterisation by structural induction on the formula $\phi(v)$. The most primitive possible for of $\phi$ is when it is a relation symbol applied to terms. If $\phi$ is not of this form, then we may assume it is constructed from simpler formulas using the logical symbols: $\wedge, \neg$ and $\exists$. We consider the following cases:

  \begin{itemize}
	\item Suppose that $\phi (v) = R(t_1(v), t_2(v),\dots, t_n(v))$ for some relation symbol $R$ and some terms $t_1(v),\dots,t_n(v)$ dependent on $v$. Then the claim is an immediate consequence of how the interpretation $R_{\StructureP}$ is defined.

   \item Suppose $\phi(v) = \alpha(v) \wedge \beta(v)$. By the inductive assumption, the claim holds for $\alpha(v)$ and for $\beta(v)$. 

	Let $A$ be the set of $x \in X$ such that $\Structure_x \models \alpha(a_x)$, and likewise let $B$  be the set of $x \in X$ such that $\Structure_x \models \beta(a)$. It is clear that  $\Structure_x \models \alpha(a_x) \wedge \beta(a_x) = \phi(a_x)$ if and only if $x \in C$. 

	By the inductive assumption, we have $\StructureP \models \alpha(a)$ if and only if $A \in \cU$. Likewise, $\StructureP \models \beta(a)$ if and only if $B \in \cU$. Hence, $\StructureP \models \phi(a)$ if and only if $A,B \in \cU$. It is a general fact that $A, B \in \cU$ if and only if $A \cap B \in \cU$. Therefore, the condition $\StructureP \models \phi(a)$ is equivalent to $A \cap B \in \cU$, which in turn is just another way of stating that $\Structure_x \models \phi(a_x)$ for $\cU$-many $x$.

    \item Suppose $\phi(v) = \neg \alpha(v)$. By the inductive assumption, the claim holds for $\alpha(v)$. Let $A$ be the set of $x \in X$ such that $\Structure_x \models \phi(a_x)$ if and only if $x \in A^c$.

	By the inductive assumption, we have  we have $\StructureP \models \alpha(a)$ if and only if $A \in \cU$. Obviously, $\StructureP \models \phi(a)$ if and only if it is not true that $\StructureP \models \alpha(a)$. Hence $\StructureP \models \phi(a)$ if and only if $A \not \in \cU$. By a general rule, $A \not \in \cU$ if and only if $A^c \in \cU$. Combining these facts, we conclude that   $\StructureP \models \phi(a)$ if and only if $\Structure_x \models \phi(a_x)$ for $\cU$-many $x$.

   \item $\phi(v) = (\exists u) \alpha(v,u)$ for some $\alpha$. By the inductive assumption, the claim holds for the sentence $\alpha$ (with one more variable).

	Suppose that $\StructureP \models \phi(a)$. Then, there exists $b \in P$ such that $\StructureP \models \alpha(a,b)$. By the inductive assumption, $\Structure_x \models \alpha(a_x,b_x)$ for $\cU$-many $x$. Hence, we also have $\Structure_x \models \phi(a_x)$ for $\cU$-many $x$, as desired.

	Suppose conversely that $\Structure_x \models \phi(a_x)$ for $\cU$-many $x$, say for $x \in A$. Then let $b_x$ be such that $\Structure_x \models \alpha(a_x,b_x)$ for $x \in A$, and let $b_x$ be arbitrary for $x \in A^c$. Then we have $\Structure_x \models \alpha(a_x,b_x)$ for $\cU$-many $x$, and by the inductive assumption it follows that  $\StructureP  \models \alpha(a,b)$. In particular, $\StructureP \models \psi(a)$, which concludes the proof.

  \end{itemize}
\end{proof}

The reason for interest in ultraproducts, and in particular in ultrapowers, is the so called (countable) saturation property.

\begin{corollary}
\index{model theory!ultraproduct}
	Let $\cU \in \Ultrafilters{\NN}$ be a non-principal ultrafilter, let $\Theory$ be a theory with model $\Structure$, and let $\StructureP := \prod_{x \in X} \Structure$ be the ultrapower of $\Structure$. Let $\{\phi_i(v)\}_{i \in \NN}$ be a \emph{finitely satisfiable} sequence of sentences, i.e. for any finite $I \subset \NN$ there exists $a \in S$ such that $\Structure \models \phi_i(a)$ for $i \in I$. Then, there exists $a \in P$ such that $\StructureP \models \phi_i(a)$ for all $i \in \NN$.
\end{corollary}
\begin{proof}
  Let $a_n \in S$ be such that $\Structure \models \phi_i(a_n)$ for $i \leq n$. Let $a := \seq{a}{n}{\NN} \in P$. We claim that for all $i$ we have $\StructureP \models \phi_i(a)$. Indeed, for a fixed $i$ we know that $\Structure \models \phi_i(a_n)$ for $n \geq i$, so it fails to hold for at most finitely many $i$. Because $\cU$ is not principal, it contains no finite sets, and hence $\Structure \models \phi_i(a_n)$ for $\cU$-many $n$. By Theorem \ref{D:thm:los}, it follows that $\StructureP \models \phi_i(a)$.
\end{proof}

\begin{example} Throughout, let $\cU \in \Ultrafilters{\NN}$ be a fixed non-principal ultrafilter. Let $\RR$ be the standard real numbers, and $\RR^*$ be the ultrapower of $\RR$ with respect to $\cU$. It is obvious that any positive integer $n$, there exists $\varepsilon \in \RR$ such that $0 < \varepsilon < \frac{1}{n}$. Hence, there exists $\varepsilon \in \mathbb{R}^*$ such that  $0 < \varepsilon < \frac{1}{n}$ for \emph{any} positive integer $n$. Such $\varepsilon$ is often referred to as \emph{infinitesimal}. It is fairly easy to give an example of such a number, it suffices to take $\eps = \seq{\eps}{n}{\NN} \in \RR^*$ with $\lim_{n \to \infty} \eps_n = 0$. Usage of infinitesimals is the essence of non-standard analysis. 

  Note that although $\varepsilon$ is infinitesimal, it makes sense to apply all standard operations to it. For instance, it makes sense to form expressions such as $1 + 43\varepsilon + \varepsilon^2$ or $\frac{1}{\varepsilon}$. Also, if $f:\ \RR \to \RR$ is any function, then we have a natural way of extending it to $\RR^*$ by adjoining a corresponding function symbol to the language. Hence, it makes sense to consider expressions like $\sin \varepsilon$ or $\frac{f(x+\varepsilon) - f(x)}{\eps}$.
\end{example}

\section{Axiom of Determinacy and Axiom of Choice}
\index{axiom!axiom of choice|(}
\index{axiom!axiom of determinacy|(}
	Let us recall that the construction of ultrafilters relied on the Axiom of Choice, $\mathsf{AC}$. Although in most of mainstream mathematics $\mathsf{AC}$ is almost unilaterally accepted, there is still noticeable interest in axioms which are incompatible with $\mathsf{AC}$. Moreover, it is worth knowing which parts of theory really depend on Choice, and in which the dependence is only superficial. Throughout this section, we will be working with Zermelo-Fraenkel Axiomatization, $\mathsf{ZF}$, unless explicitly noted otherwise. Most of the results discussed here are apparently a part of mathematical folklore; an exhaustive treatment can be found in \cite{games-by-gradel}. A popular and extremely readable introduction, which was the first contact with ultrafilters for the author, can be found \cite{games-by-parys}. 

	It is by no means obvious that Axiom of Choice is independent of Zermelo-Fraenkel Axioms. In fact, it the proof of independence due to Paul Cohen uses a highly non-trivial method of forcing. A more accessible method of proving independence from Choice is by employing additional axiom which known to be consistent with $\mathsf{ZF}$  but false in $\mathsf{ZFC}$, and proving that this statement is not consistent with the result at hand. A possible choice for this purpose is the Axiom of Determinacy which we will now introduce. First, we need a definition, which we form in a slightly informal form.

\begin{definition}[$\omega$-game]
\index{w-game@$\omega$-game}
	An $\omega$-game is a two-player, perfect information, deterministic game of length $\omega$, played with integers. 

	In such game, there are two players, say $A$ and $B$. They take turns choosing integers, starting with $A$, knowing the choices made in previous turns. There are $\omega$ moves made, and hence the choices made by the players result in construction of an infinite sequence of integers, say $\seq{a}{i}{\omega}$, where $A$ chooses $a_0$, $B$ chooses $a_1$ knowing $a_1$, and so on. The game is determined by a set $X \subset \omega^\omega$: if $\seq{a}{i}{\omega} \in X$ then $A$ wins, else $B$ wins.

	A strategy (for player $A$) is a way to assign the next move of $A$ to a given position. Formally, it is a collection of maps $S = (S_i)_{i \in \omega}$ such that $S_i :\ \omega^{2i} \to \omega$. We say that $A$ follows the strategy $S$ if at $i$-th move he chooses $a_{2i} = S_i((a_j)_{j<2i})$. The strategy $S$ is said to be \emph{winning} if by following $S$, $A$ wins, regardless of how $B$ plays. Strategies for $B$ are defined analogously.

	A game is said to be \emph{determined} if either of the players has a winning strategy. 

\end{definition}

In the above definition, the requirement that the moves consist in choosing integers is not as restrictive as it might appear. In practice, it suffices that in each turn, the number of moves is at most countable: it is then possible to enumerate possible moves, and identify the choice of an integer with the choice of the corresponding move.

\begin{example}
	Consider a game of chess with the standard rules, but with the threefold repetition rule replaced by the rule that if the game proceeds indefinitely, then black wins. The resulting game is an $\omega$-game, although admittedly not a very interesting one. The same holds for checkers.
\end{example}

We are now ready to formulate the Axiom of Determinacy.

\begin{definition}
	\emph{Axiom of Determinacy} ($\mathsf{AD}$) is the statement that each $\omega$-game is determined.
\end{definition}

Justification of interest in $\mathsf{AD}$ lies in the following difficult theorem, which we cite without proof. We will not explicitly use it, but if it wasn't true, much of the subsequent considerations would be moot. The inquisitive reader is referred to \cite{axioms-by-kanamori} and \cite{sets-by-jech} for more details.

\begin{theorem}
  Axiom of Determinacy is consistent with Zermelo-Fraenkel Axioms.
\end{theorem}

Axiom of Determinacy has many surprising consequences, which are in contradiction with the standard result derived with use of the Axiom of Choice. Again, we give not give a proof, nor will we ever use them. Our only aim here is to give the reader a flavour of what mathematics looks like in $\mathsf{ZF} + \mathsf{AD}$.

\begin{theorem}
	Any of the following is a consequence of the Axiom of Determinacy:
	\begin{enumerate}
	\item Every subset of  $\RR$ is Lebesgue measurable.
	\item Every subset of  $\RR$ has the property of Baire.
	\item Every subset of  $\RR$ has the perfect set property.
	\item Every uncountable subset of $\RR$ has cardinality $\mathfrak{c}$.
	\item There is no Hamel basis of the $\RR$ over $\QQ$.
	\end{enumerate}
\end{theorem}

We will now present an elegant and non-trivial example of an $\omega$-game.

\begin{example}[The ultrafilter game.]\label{D:exple:game-ultrafilters}
 Suppose that $\cU \in \Ultrafilters{\NN}$ be an arbitrary non-principal ultrafilter. Consider the following game. Two players, Alice and Bob, take turns selecting consecutive terms of a sequence $\seq{a}{n}{\NN}$: first Alice selects any $a_0$, then Bob selects arbitrary $a_1 > a_0$, then Alice selects $a_2 > a_1$, and so on. After $\omega$ moves, the sequence $\seq{a}{n}{\NN}$ is constructed. We then define sets sets $A$ and $B$ are as:
  \begin{align*}
    A := \bigcup_{n \in \NN} [a_{2n-1}, a_{2n}), \qquad
    B := \bigcup_{n \in \NN} [a_{2n}, a_{2n+1}),
  \end{align*}
 where $a_{-1} := 0$ by convention. It is clear that $A \cap B = \emptyset$ and $A \cup B = \NN$, so exactly one of $A$ and $B$ belongs to $\cU$. Alice wins if $A \in \cU$, Bob wins if $B \in \cU$.
\end{example}

\begin{proposition}\label{D:prop:game-no-winning-strategy}
  Suppose that $\cU$ is an ultrafilter. Then the ultrafilter game described above in Example \ref{D:exple:game-ultrafilters}, neither player has a winning strategy.
\end{proposition}
\begin{proof}
  For a proof by contradiction, suppose that one of the players can ensure his victory. For concreteness, suppose that it is Bob. The considerations in the case when Alice has the winning strategy are entirely analogous.
  
  Consider two instances of the game being played in parallel; one with Alice and Bob generating sequence $\seq{a}{n}{\NN}$, the other one with Alice$'$ and Bob$'$ generating sequence $\seq{a'}{n}{\NN}$, where in both games Bob and Bob' play according to the hypothesized winning strategy. We will show that Alice and Alice$'$ can cooperate to win at least one of the games. Their joint strategy is as follows.

  First, Alice makes her initial move $a_0$ arbitrarily. Bob answers with some move $a_1$. Now, Alice$'$ makes her first move $a_0' := a_1$, to which Bob$'$ answers with $a_1'$. Then, Alice plays $a_2 := a_1'$, and waits for the move of Bob $a_3$. Alice$'$ then plays $a_2':= a_3$, and waits for Bob to play $a_3'$. They continue in this fashion. In general, suppose that after a number of turns it is the time for Alices to choose $a_{2n}$ and $a_{2n}'$. Alice moves first, choosing $a_{2n} := a'_{2n-1}$. Then Bob plays some $a_{2n+1}$. Once Bob's move is made, Alice$'$ chooses $a_{2n}' := a_{2n+1}$. After Bob$'$ makes his move, it is again the turn of the Alices, and the cycle is complete.

  Note that by construction $a'_n = a_{n+1}$ for any $n \in \NN$. In particular, if we denote the sets $A',B'$ in analogy to the sets $A,B$, we find that:
  $$ A' := \bigcup_{n \in \NN} [a'_{2n-1}, a_{2n}') 
    = [0,a'_0) \cup \bigcup_{n \in \NN} [a'_{2n+1}, a_{2n+2}') 
    = [0,a_1) \cup \bigcup_{n \in \NN} [a_{2n}, a_{2n+1}) = [0,a_1) \cup B 
  $$
  As a consequence, the symmetric difference $A' \triangle B$ is finite, and does not belong to $\cU$.

  Because of the assumption that Bob used a winning strategy, we know that he wins the game against Alice. It follows that $B \in \cU$. Since $A' \triangle B \not \in \cU$, we also have $A' \in \cU$. However, this means that Alice$'$ wins the game against Bob$'$, who was also assumed to play according to winning strategy. This is a contradiction, proving that the assumption of existence of winning strategy for Bob was false.
\end{proof}

\begin{corollary}\label{D:cor:AD-inconcistent-with-ultrafilters}
  Existence of non-principal ultrafilters on $\NN$ is inconsistent with $\mathsf{AD}$. In particular, it is consistent with $\mathsf{ZF}$ that no non-principal ultrafilters exist.
\end{corollary}
\begin{proof}
  Assume $\mathsf{AD}$ holds, and suppose that $\cU$ is a non-principal ultrafilter on $\NN$. Consider the ultrafilter game described in Example \ref{D:exple:game-ultrafilters}. On one hand, according to Proposition \ref{D:prop:game-no-winning-strategy}, neither of the player has a winning strategy for this game. On the other hand, this is an $\omega$-game, so $\mathsf{AD}$ implies the existence of a winning strategy. These two statements are contradictory, so the assumption that a non-principal ultrafilter on $\NN$ exists is inconsistent with $\mathsf{AD}$.
\end{proof}

In particular, we have just re-derived the following clasically known fact. Note that we are not dependent on any consistency results here.

\begin{corollary}\label{D:cor:AD+AC-inconsistent}

  The Axiom of Choice and the Axiom of Determinacy are incompatible within Zermelo-Fraenkel Axiomatization, in the sense that the theory $\mathsf{ZF} + \mathsf{AC} + \mathsf{AD}$ is inconsistent.
\end{corollary}

A practical consequence of Corollary \ref{D:cor:AD-inconcistent-with-ultrafilters} is that there is little hope of an explicit construction of an ultrafilter on $\NN$. We will not go into the details of what it precisely means for a construction to be ``explicit'', but such construction should clearly be possible carry out within $\mathsf{ZF}$. Hence, independence of existence of ultrafilters from $\mathsf{ZF}$ offers strong evidence that the construction is impossible\footnote{We choose not to formulate these results in a more decisive way for two reasons. Firstly, it is not entirely certain that each ``explicit'' construction is formalisable within the $\mathsf{ZF}$ framework. Secondly, it might possibly be the case that a \emph{construction} itself is possible within $\mathsf{ZF}$, and it is only the proof of correctness that requires stronger axioms.}.

\index{axiom!axiom of choice|)}
\index{axiom!axiom of determinacy|)} % A digression on Arrow's theorem.
  
%\input{ Chapter3}
%\input{ Chapter4} 
%\input{ Chapter5} 
%\input{ Chapter6} 
%\input{ Chapter7} 

%----------------------------------------------------------------------------------------
%	THESIS CONTENT - APPENDICES
%----------------------------------------------------------------------------------------

%\addtocontents{toc}{\vspace{2em}} % Add a gap in the Contents, for aesthetics

%\appendix % Cue to tell LaTeX that the following 'chapters' are Appendices

% Include the appendices of the thesis as separate files from the Appendices folder
% Uncomment the lines as you write the Appendices

%\input{./Appendices/AppendixA}
%	\input{./Appendices/AppendixB}
%\input{./Appendices/AppendixC}

\addtocontents{toc}{\vspace{2em}} % Add a gap in the Contents, for aesthetics

\backmatter

%----------------------------------------------------------------------------------------
%	BIBLIOGRAPHY
%----------------------------------------------------------------------------------------

\nocite{*}
\label{Bibliography}

\lhead{\emph{Bibliography}} % Change the page header to say "Bibliography"

\bibliographystyle{alpha} % Use the "unsrtnat" BibTeX style for formatting the Bibliography

\bibliography{Bibliography} % The references (bibliography) information are stored in the file named "Bibliography.bib"

\begin{thebibliography}{BHKM06}

\bibitem[AS03]{AutSeq}
Jean-Paul Allouche and Jeffrey Shallit.
\newblock {\em Automatic sequences}.
\newblock Cambridge University Press, Cambridge, 2003.
\newblock Theory, applications, generalizations.

\bibitem[BBH94]{partitions-of-words}
Vitaly Bergelson, Andreas Blass, and Neil Hindman.
\newblock Partition theorems for spaces of variable words.
\newblock {\em Proc. London Math. Soc. (3)}, 68(3):449--476, 1994.

\bibitem[Ber96]{Berg-RamseyErgo-update}
Vitaly Bergelson.
\newblock Ergodic {R}amsey theory---an update.
\newblock In {\em Ergodic theory of {${\bf Z}\sp d$} actions ({W}arwick,
  1993--1994)}, volume 228 of {\em London Math. Soc. Lecture Note Ser.}, pages
  1--61. Cambridge Univ. Press, Cambridge, 1996.

\bibitem[Ber03]{Berg-survey-minimal}
Vitaly Bergelson.
\newblock Minimal idempotents and ergodic {R}amsey theory.
\newblock In {\em Topics in dynamics and ergodic theory}, volume 310 of {\em
  London Math. Soc. Lecture Note Ser.}, pages 8--39. Cambridge Univ. Press,
  Cambridge, 2003.

\bibitem[Ber10]{Berg-survey-IP}
Vitaly Bergelson.
\newblock Ultrafilters, {IP} sets, dynamics, and combinatorial number theory.
\newblock In {\em Ultrafilters across mathematics}, volume 530 of {\em Contemp.
  Math.}, pages 23--47. Amer. Math. Soc., Providence, RI, 2010.

\bibitem[BFHK89]{vd-waerden-with-ultrafilters}
Vitaly Bergelson, Hillel Furstenberg, Neil Hindman, and Yitzhak Katznelson.
\newblock An algebraic proof of van der {W}aerden's theorem.
\newblock {\em Enseign. Math. (2)}, 35(3-4):209--215, 1989.

\bibitem[BFM96]{BFM-IP-limits}
Vitaly Bergelson, Hillel Furstenberg, and Randall McCutcheon.
\newblock I{P}-sets and polynomial recurrence.
\newblock {\em Ergodic Theory Dynam. Systems}, 16(5):963--974, 1996.

\bibitem[BHK96]{iterated-spectra}
Vitaly Bergelson, Neil Hindman, and Bryna Kra.
\newblock Iterated spectra of numbers---elementary, dynamical, and algebraic
  approaches.
\newblock {\em Trans. Amer. Math. Soc.}, 348(3):893--912, 1996.

\bibitem[BHKM06]{BergKnuMc2006}
Vitaly Bergelson, Inger~J. H{\aa}land~Knutson, and Randall McCutcheon.
\newblock I{P}-systems, generalized polynomials and recurrence.
\newblock {\em Ergodic Theory Dynam. Systems}, 26(4):999--1019, 2006.

\bibitem[BK12]{history-who-gave-C-W-tale}
Alexandre Borovik and Mikhail~G. Katz.
\newblock Who gave you the {C}auchy-{W}eierstrass tale? {T}he dual history of
  rigorous calculus.
\newblock {\em Found. Sci.}, 17(3):245--276, 2012.

\bibitem[BL96]{BergLeib-polySzemeredi}
V.~Bergelson and A.~Leibman.
\newblock Polynomial extensions of van der {W}aerden's and {S}zemer\'edi's
  theorems.
\newblock {\em J. Amer. Math. Soc.}, 9(3):725--753, 1996.

\bibitem[BL07]{BergLeib-GenPoly-distribution}
Vitaly Bergelson and Alexander Leibman.
\newblock Distribution of values of bounded generalized polynomials.
\newblock {\em Acta Math.}, 198(2):155--230, 2007.

\bibitem[Bla93]{Blass-equalities}
Andreas Blass.
\newblock Ultrafilters: where topological dynamics {$=$} algebra {$=$}
  combinatorics.
\newblock {\em Topology Proc.}, 18:33--56, 1993.

\bibitem[BM10]{BergMc2010-SzThm-GenPoly}
V.~Bergelson and R.~McCutcheon.
\newblock Idempotent ultrafilters, multiple weak mixing and {S}zemer\'edi's
  theorem for generalized polynomials.
\newblock {\em J. Anal. Math.}, 111:77--130, 2010.

\bibitem[CHS05]{hindman-survey-2}
Timothy~J. Carlson, Neil Hindman, and Dona Strauss.
\newblock Ramsey theoretic consequences of some new results about algebra in
  the {S}tone-\v {C}ech compactification.
\newblock {\em Integers}, 5(2):A4, 26, 2005.

\bibitem[Cla]{models-by-clark}
Pete~L. Clark.
\newblock Introduction to model theory and its applications.

\bibitem[CN74]{comfort-negre-book}
W.~W. Comfort and S.~Negrepontis.
\newblock {\em The theory of ultrafilters}.
\newblock Springer-Verlag, New York, 1974.
\newblock Die Grundlehren der mathematischen Wissenschaften, Band 211.

\bibitem[Com77a]{comfort-topology}
W.~W. Comfort.
\newblock Some recent applications of ultrafilters to topology.
\newblock In {\em General topology and its relations to modern analysis and
  algebra, {IV} ({P}roc. {F}ourth {P}rague {T}opological {S}ympos., {P}rague,
  1976), {P}art {A}}, pages 34--42. Lecture Notes in Math., Vol. 609. Springer,
  Berlin, 1977.

\bibitem[Com77b]{comfort-survey}
W.~W. Comfort.
\newblock Ultrafilters: some old and some new results.
\newblock {\em Bull. Amer. Math. Soc.}, 83(4):417--455, 1977.

\bibitem[Eis10]{Tanja-operators}
Tanja Eisner.
\newblock {\em Stability of operators and operator semigroups}, volume 209 of
  {\em Operator Theory: Advances and Applications}.
\newblock Birkh\"auser Verlag, Basel, 2010.

\bibitem[Ell58]{Ellis}
Robert Ellis.
\newblock Distal transformation groups.
\newblock {\em Pacific J. Math.}, 8:401--405, 1958.

\bibitem[Eng89]{engelking-topology}
Ryszard Engelking.
\newblock {\em General topology}, volume~6 of {\em Sigma Series in Pure
  Mathematics}.
\newblock Heldermann Verlag, Berlin, second edition, 1989.
\newblock Translated from the Polish by the author.

\bibitem[EW11a]{Ergo-by-Einsiedler-Ward}
Manfred Einsiedler and Thomas Ward.
\newblock {\em Ergodic theory with a view towards number theory}, volume 259 of
  {\em Graduate Texts in Mathematics}.
\newblock Springer-Verlag London Ltd., London, 2011.

\bibitem[EW11b]{ward-ergo}
Manfred Einsiedler and Thomas Ward.
\newblock {\em Ergodic theory with a view towards number theory}, volume 259 of
  {\em Graduate Texts in Mathematics}.
\newblock Springer-Verlag London Ltd., London, 2011.

\bibitem[FK85]{FK-IP-limits}
H.~Furstenberg and Y.~Katznelson.
\newblock An ergodic {S}zemer\'edi theorem for {IP}-systems and combinatorial
  theory.
\newblock {\em J. Analyse Math.}, 45:117--168, 1985.

\bibitem[Fur81]{Furstenberg-central}
H.~Furstenberg.
\newblock {\em Recurrence in ergodic theory and combinatorial number theory}.
\newblock Princeton University Press, Princeton, N.J., 1981.
\newblock M. B. Porter Lectures.

\bibitem[Gal]{arrows-by-galvin}
David Galvin.
\newblock Ultrafilters, with applications to analysis, social choice and
  combinatorics.

\bibitem[Ges13]{geshke}
Stefan Geschke.
\newblock Lecture notes on model theory.
\newblock 2013.

\bibitem[Gib73]{gibbard-by-gibbard}
Allan Gibbard.
\newblock Manipulation of voting schemes: a general result.
\newblock {\em Econometrica}, 41:587--601, 1973.

\bibitem[GKP94]{Knuth-ConcreteMathematics}
Ronald~L. Graham, Donald~E. Knuth, and Oren Patashnik.
\newblock {\em Concrete mathematics}.
\newblock Addison-Wesley Publishing Company, Reading, MA, second edition, 1994.
\newblock A foundation for computer science.

\bibitem[Gr{\"{a}}09]{games-by-gradel}
Erich Gr{\"{a}}del.
\newblock {\em Logic and Games}.
\newblock Mathematische Grundlagen der Informatik, RWTH Aachen. 2009.

\bibitem[Gr{\"{a}}11]{gradel-lectures}
Erich Gr{\"{a}}del.
\newblock Back and forth between logic and games.
\newblock In {\em Lectures in game theory for computer scientists}, pages
  99--145. Cambridge Univ. Press, Cambridge, 2011.

\bibitem[GT13]{golan-hindman}
Gili Golan and Boaz Tsaban.
\newblock Hindman's coloring theorem in arbitrary semigroups.
\newblock 2013.
\newblock cite arxiv:1303.3600.

\bibitem[HCB73]{Fibonacci-by-Hoggatt}
Verner~E. Hoggatt, Jr., Nanette Cox, and Marjorie Bicknell.
\newblock A primer for the {F}ibonacci numbers. {XII}.
\newblock {\em Fibonacci Quart.}, 11(3):317--331, 1973.

\bibitem[Hin74]{Hindman-original}
Neil Hindman.
\newblock Finite sums from sequences within cells of a partition of
  {$\mathbb{N}$}.
\newblock {\em J. Combinatorial Theory Ser. A}, 17:1--11, 1974.

\bibitem[Hin05]{hindman-survey-1}
Neil Hindman.
\newblock Algebra in the {S}tone-\v {C}ech compactification and its
  applications to {R}amsey theory.
\newblock {\em Sci. Math. Jpn.}, 62(2):321--329, 2005.

\bibitem[Hod97]{models-by-hodges}
Wilfrid Hodges.
\newblock {\em A shorter model theory}.
\newblock Cambridge University Press, Cambridge, 1997.

\bibitem[HS12]{hindman-strauss}
Neil Hindman and Dona Strauss.
\newblock {\em Algebra in the {S}tone-\v {C}ech compactification}.
\newblock de Gruyter Textbook. Walter de Gruyter \& Co., Berlin, 2012.
\newblock Theory and applications, Second revised and extended edition.

\bibitem[Jec78]{sets-by-jech}
T.J. Jech.
\newblock {\em Set theory}.
\newblock Pure and Applied Mathematics. Elsevier Science, 1978.

\bibitem[Kan08]{axioms-by-kanamori}
A.~Kanamori.
\newblock {\em The Higher Infinite: Large Cardinals in Set Theory from Their
  Beginnings}.
\newblock Springer Monographs in Mathematics. Springer, 2008.

\bibitem[Lei02]{Lei-Poly-Groups}
A.~Leibman.
\newblock Polynomial mappings of groups.
\newblock {\em Israel J. Math.}, 129:29--60, 2002.

\bibitem[Lei12]{Leib-GenPoly-canonical}
A.~Leibman.
\newblock A canonical form and the distribution of values of generalized
  polynomials.
\newblock {\em Israel J. Math.}, 188:131--176, 2012.

\bibitem[Mar02]{models-by-marker}
D.~Marker.
\newblock {\em Model Theory: An Introduction}.
\newblock Graduate Texts in Mathematics. Springer, 2002.

\bibitem[Par07]{games-by-parys}
Pawe\l{} Parys.
\newblock Gry niesko\'{n}czone.
\newblock {\em Delta --- matematyka, fizyka, astronomia, informatyka},
  September 2007.

\bibitem[Par12]{tao-blog-3}
Jonathan~R. Partington.
\newblock {\it {A}n epsilon of room, {I}: real analysis (pages from year three
  of a mathematical blog)}.
\newblock {\em Bull. Lond. Math. Soc.}, 44(1):203--205, 2012.

\bibitem[Ren01]{arrows-by-reny}
Philip~J. Reny.
\newblock Arrow's theorem and the {G}ibbard-{S}atterthwaite theorem: a unified
  approach.
\newblock {\em Econom. Lett.}, 70(1):99--105, 2001.

\bibitem[Rob65]{models-by-robinson}
A.~Robinson.
\newblock {\em Introduction to Model Theory and to the Metamathematics of
  Algebras}.
\newblock Studies in logic and the foundations of mathematics. North-Holland,
  1965.

\bibitem[Ros09]{AutomaticContinuity-by-Rosendal}
Christian Rosendal.
\newblock Automatic continuity of group homomorphisms.
\newblock {\em Bull. Symbolic Logic}, 15(2):184--214, 2009.

\bibitem[S{\'a}r78a]{Sarkozy-original-I}
A.~S{\'a}rk{\H{o}}zy.
\newblock On difference sets of sequences of integers. {I}.
\newblock {\em Acta Math. Acad. Sci. Hungar.}, 31(1--2):125--149, 1978.

\bibitem[S{\'a}r78b]{Sarkozy-original-II}
A.~S{\'a}rk{\"o}zy.
\newblock On difference sets of sequences of integers. {II}.
\newblock {\em Ann. Univ. Sci. Budapest. E\"otv\"os Sect. Math.}, 21:45--53
  (1979), 1978.

\bibitem[S{\'a}r78c]{Sarkozy-original-III}
A.~S{\'a}rk{\"o}zy.
\newblock On difference sets of sequences of integers. {III}.
\newblock {\em Acta Math. Acad. Sci. Hungar.}, 31(3-4):355--386, 1978.

\bibitem[Sch07]{Schnell}
Christian Schnell.
\newblock Idempotent ultrafilters and polynomial recurrence.
\newblock 2007.
\newblock cite arxiv:0711.0484Comment: 25 pages.

\bibitem[Tao]{arrows-by-tao}
Terence Tao.
\newblock Arrow's theorem.

\bibitem[Tar30]{Tarski}
Alfred Tarski.
\newblock Une contribution {\`a} la th{\'e}orie de la mesure.
\newblock {\em Fundamenta Mathematicae}, 15(1):42--50, 1930.

\bibitem[Zec72]{Fibonacci-by-Zeckendorf}
E.~Zeckendorf.
\newblock Repr\'esentation des nombres naturels par une somme de nombres de
  {F}ibonacci ou de nombres de {L}ucas.
\newblock {\em Bull. Soc. Roy. Sci. Li\`ege}, 41:179--182, 1972.

\bibitem[Zir12]{zirnstein-thesis}
Heinrich-Gregor Zirnstein.
\newblock Formulating {S}zemer\'{e}di's theorem in terms of ultrafilters.
\newblock 2012.

\bibitem[ZK12]{zorin}
Pavel Zorin-Kranich.
\newblock A nilpotent ip polynomial multiple recurrence theorem.
\newblock 2012.
\newblock cite arxiv:1206.0287Comment: 28 pages, v2: definition of polynomial
  and proof of Theorem 2.5 changed, minor corrections.

\end{thebibliography}

\addcontentsline{toc}{chapter}{Index} 

\index{FP|seealso{$\IP$-set}}
\index{FU|seealso{$\IP$-set}}
\index{FS|seealso{$\IP$-set}}
\index{topology!limit|see{limit}}
\index{topology!topological semigroup|see{semigroup!topological semigroup}}
\index{topology!ultrafilter|see{ultrafilter!topological structure}}
\index{symmetric discrete derivative|see{discrete derivative!symmetric}}
\index{T@$\TT$|see{torus}}
\index{IP@$\IP$|seealso{idempotent}}

\printindex

\end{document}